\documentclass[10pt. oneside]{amsart} 
\usepackage[lmargin=1in,rmargin=1in,
bmargin=1in, tmargin=1in]{geometry}
\usepackage[dvipsnames]{xcolor}

\usepackage{amsmath, amssymb,tikz}
\usepackage{tikz-cd}
\usepackage{tikz}
\usepackage{todonotes}
\usepackage{mathrsfs}
\usepackage{extarrows}
\usepackage{graphicx}
\usepackage[breaklinks, pagebackref]{hyperref}
\usepackage{IEEEtrantools}
\usepackage{mathrsfs}
\usepackage[utf8]{inputenc}
\usepackage[english]{babel}
\usepackage{comment}
\usepackage{csquotes}
\usepackage[shortlabels]{enumitem}
\usepackage{tikz-cd}
\usepackage{tikz}
\usetikzlibrary{decorations.pathreplacing,matrix,calc,positioning}
\usepackage{nicematrix}
\tikzset{ 
    table/.style={
        matrix of math nodes,
        row sep=-\pgflinewidth,
        column sep=-\pgflinewidth,
        nodes={rectangle,text width=3em,align=center},
        text depth=1.25ex,
        text height=2.5ex,
        nodes in empty cells,
        left delimiter=[,
        right delimiter={]},
        ampersand replacement=\&
    }
}
\usetikzlibrary{decorations.pathreplacing, calligraphy}

\def\smath#1{\text{\scalebox{0.9}{$#1$}}}

\def\sfrac#1#2{\smath{\frac{#1}{#2}}}

\hypersetup{
    colorlinks=true,
    linkcolor=blue,
    filecolor=red, 
    citecolor=red,          
    urlcolor=gray,
    pdftitle={},
    pdfpagemode=FullScreen,
    }

\usepackage{tikz}

\makeatletter
\newcommand*{\encircled}[1]{\relax\ifmmode\mathpalette\@encircled@math{#1}\else\@encircled{#1}\fi}
\newcommand*{\@encircled@math}[2]{\@encircled{$\m@th#1#2$}}
\newcommand*{\@encircled}[1]{%
  \tikz[baseline,anchor=base]{\node[draw,circle,outer sep=0pt,inner sep=.2ex] {#1};}}
\makeatother

\usepackage{pict2e,picture}

\makeatletter
\newcommand{\pnrelbar}{%
  \linethickness{\dimen2}%
  \sbox\z@{$\m@th\prec$}%
  \dimen@=1.1\ht\z@
  \begin{picture}(\dimen@,.4ex)
  \roundcap
  \put(0,.2ex){\line(1,0){\dimen@}}
  \put(\dimexpr 0.5\dimen@-.2ex\relax,0){\line(1,1){.4ex}}
  \end{picture}%
}

\newcommand{\precneq}{\mathrel{\vcenter{\hbox{\text{\prec@neq}}}}}
\newcommand{\prec@neq}{%
  \dimen2=\f@size\dimexpr.04pt\relax
  \oalign{%
    \noalign{\kern\dimexpr.2ex-.5\dimen2\relax}
    $\m@th\prec$\cr
    \noalign{\kern-.5\dimen2}
    \hidewidth\pnrelbar\hidewidth\cr
  }%
}
\makeatother

\DeclareFontEncoding{LS1}{}{}
\DeclareFontSubstitution{LS1}{stix}{m}{n}

\makeatletter

\newcommand*\textmathversion{\csname textmv@\math@version\endcsname}
  \newcommand*\textmv@normal{m}
  \newcommand*\textmv@bold{b}
\makeatother

\DeclareFontEncoding{LS1}{}{}
\DeclareFontSubstitution{LS1}{stix}{m}{n}

\makeatletter
\usepackage{capt-of}

\usepackage{sagetex}

\DeclareFontEncoding{LS1}{}{}
\DeclareFontSubstitution{LS1}{stix}{m}{n}

\makeatletter

\DeclareFontEncoding{LS1}{}{}
\DeclareFontSubstitution{LS1}{stix}{m}{n}

\makeatletter

\DeclareFontEncoding{LS1}{}{}
\DeclareFontSubstitution{LS1}{stix}{m}{n}

\makeatletter

\usepackage{setspace}\setdisplayskipstretch{}

\usepackage{setspace}
\usepackage[all]{xy}
\usepackage{textcomp}
\usepackage{amsbsy}
\usepackage{datetime}
\renewcommand{\dateseparator}{-}
\renewcommand{\today}{\the\year \dateseparator \twodigit\month
\dateseparator \twodigit\day}
\usepackage{amsmath,amsthm,amssymb}
\usepackage{colonequals}

\newcommand{\QQ}{\mathbb{Q}}

\newcommand{\ZZ}{\mathbb{Z}}

\newcounter{dummypart}



\newcommand{\CC}{\mathbb{C}}

\newcommand{\ch}{\mathrm{ch}}

\newcommand{\delT}{\mathbb{S}}

\newcommand{\Ab}{\mathbb{A}}   
    
\newcommand{\pr}{\mathrm{pr}}

\newcommand{\Art}{\mathrm{Art}}
\newcommand{\Res}{\mathrm{Res}}

\DeclareMathOperator{\Gal}{Gal}

\numberwithin{equation}{section}

\newtheorem{theorem}[equation]{Theorem}
\newtheorem{proposition}[equation]{Proposition}
\newtheorem{lemma}[equation]{Lemma}
\newtheorem{corollary}[equation]{Corollary}
\newtheorem*{corollary*}{Corollary}

\newcommand*\MapsTo{%
  \xrightarrow{\raisebox{-0.7 em}{\smash{\ensuremath{\,\sim\,}}}}%
}

\theoremstyle{definition}
\newtheorem{assumption}[equation]{Assumption}

\newenvironment{theorembis}[1]{%
    \renewcommand{\theequation}{\ref{#1} bis}%
    \begin{theorem}%
}{%
    \end{theorem}%
    \addtocounter{equation}{-1} % Prevents numbering increment
}

\newtheorem*{notation*}{Notation}

\theoremstyle{definition}

\newtheorem{definition}[equation]{Definition}

\theoremstyle{remark}
\newtheorem{remark}[equation]{Remark}
\newtheorem{example}{Example}[section]
\newtheorem{note}[equation]{Note}

\newtheorem{note*}[equation]{Note}

\tikzcdset{scale cd/.style={every label/.append style={scale=#1},
    cells={nodes={scale=#1}}}}

\usepackage{tikz-cd}
\usepackage{mathrsfs}

\DeclareFontFamily{U}{wncy}{}
\DeclareFontShape{U}{wncy}{m}{n}{<->wncyr10}{}
\DeclareSymbolFont{mcy}{U}{wncy}{m}{n}
\DeclareMathSymbol{\sha}{\mathord}{mcy}{"58}

\makeatletter
\renewcommand*\env@matrix[1][\arraystretch]{%
  \edef\arraystretch{#1}%
  \hskip -\arraycolsep
  \let\@ifnextchar\new@ifnextchar
  \array{*\c@MaxMatrixCols c}}
\makeatother

\newcommand{\negsmall}{\mkern-2mu}

\newcommand{\SL}{\mathrm{SL}}

\newcommand{\Gb}{\mathbf{G}}   
\newcommand{\Tb}{\mathbf{T}}

\newcommand{\Hb}{\mathbf{H}}

\newcommand{\Addresses}{{
  \bigskip
  \footnotesize
  
  \textsc{Department of Mathematics, University of California, Santa Barbara, CA 93106-3080}\par\nopagebreak
  \textit{E-mail address}: \texttt{swshah@ucsb.edu}

  \medskip
  \textsc{Department of Mathematics, Bilkent University, Ankara 06800, Türkiye}\par\nopagebreak
  \textit{E-mail address}: \texttt{swshah@bilkent.edu.tr}
}}

\usepackage{mathtools,bm}

\providecommand\given{} % ensure it exists
\newcommand\givensymbol[1]{%
  \nonscript\;\delimsize#1\allowbreak\nonscript\;\mathopen{}%
}
\DeclarePairedDelimiterX\Set[1]\{\}{%
  \renewcommand\given{\givensymbol{\vert}}%
  #1%
}

\usepackage{stmaryrd}

\DeclareMathOperator{\Hom}{Hom}

\newcommand{\GG}{\mathbb{G}}
\newcommand{\OO}{\mathcal{O}}
\newcommand{\GL}{\mathrm{GL}}

\newcommand{\et}{\mathrm{\acute{e}t}}

\newcommand{\RR}{\mathbb{R}}

\DeclareMathOperator{\Spec}{Spec}

\title{Norm relations for CM points on modular curves} 
\author{Syed Waqar Ali    Shah}    
\date{}

    \DeclareFontFamily{U}{wncy}{}
    \DeclareFontShape{U}{wncy}{m}{n}{<->wncyr10}{}
    \DeclareSymbolFont{mcy}{U}{wncy}{m}{n}
    \DeclareMathSymbol{\Sha}{\mathord}{mcy}{"58}

\begin{document}

\begin{abstract}   Kolyvagin introduced the method of Euler systems to study the structure of Selmer groups of elliptic curves. In this semi-expository  article, we prove the horizontal norm relations for the CM points on modular curves underlying Kolyvagin's Euler system, with a view toward higher-dimensional generalizations.    
\end{abstract}    

\maketitle

\tableofcontents
\label{amotivatingexample}  
\section{Introduction} 
\subsection{The BSD conjecture} 
Let $A$ be an elliptic curve over $\QQ$. The Mordell–Weil theorem guarantees that the group $A(\QQ)$ of rational points on $A$ is finitely generated. It is a long-standing problem in number theory to describe the structure of this abelian group. A deep result of Mazur \cite{Mazur} identifies the possible isomorphism classes of torsion subgroups that can occur in $A(\QQ)$. The rank of $A(\QQ)$, on the other hand, remains far more mysterious.

Let $L(A_{/\QQ}, s)$ denote the Hasse–Weil $L$-function of $A$ over $\QQ$. It is given by an infinite Euler product in the complex variable $ s $ that converges absolutely for $\mathrm{Re}(s) > \tfrac{3}{2}$ and thus defines a complex analytic function in that region. A consequence of the celebrated modularity theorem \cite{WilesFermat, Conradetal}   is that $L(A_{/\QQ}, s)$ admits an analytic continuation to the entire complex plane. The famous conjecture of Birch and Swinnerton-Dyer \cite{WilesClay} asserts that  
\begin{equation} \label{BSD} 
   \mathrm{ord}_{s = 1} \, L(A_{/\QQ}, s) = \mathrm{rank}_{\ZZ} \, A(\QQ),
\end{equation} 
where `$\mathrm{ord}$' on the left-hand side denotes the order of vanishing of a complex analytic function. This conjecture is wide open at present. One of the major obstacles to making progress is finding a systematic supply of non-torsion points in $A(\QQ)$ whose behaviour can be explicitly tied to $L(A_{/\QQ}, s)$.

However, if one assumes that $\mathrm{ord}_{s = 1} L(A_{/\QQ}, s) \leq 1$, then it is possible to construct such points over an imaginary quadratic extension and use them to establish (\ref{BSD}). The modularity theorem asserts that the elliptic curve $A$ admits a \emph{modular parametrization}. More precisely, if $N$ denotes the conductor of $A$, there exists a dominant morphism  
\begin{equation}   \label{modularpara}     \pi   : X_{0}(N) \to A          
\end{equation}    
where $X_{0}(N)$ denotes the compactified modular curve of level $ \Gamma_{0}(N)$, which is the moduli space of generalized elliptic curves endowed with a cyclic subgroup of order $ N$.  Suppose that $E$ is an imaginary quadratic field which satisfies the so-called \emph{Heegner hypothesis}: all primes dividing $N$ split in $E$. We will view $E$ and all its extensions inside $\CC$, the field of complex numbers.  
Let $E[1]$ denote the Hilbert class field of $E$. Then the moduli interpretation of $X_{0}(N)$ allows one to define a ``distinguished'' point  
$$ x_{1} \in X_{0}(N)(E[1]) $$  
known as a \emph{Heegner point}. More precisely, a Heegner point in $X_{0}(N)(\CC)$ is defined to be a non-cuspidal point that, under the moduli interpretation, corresponds to a cyclic $N$-isogeny $A_{1} \to A_{2}$ of elliptic curves such that both $A_{1}$ and $A_{2}$ have complex multiplication by the ring of integers $\mathcal{O}_{E}$ of $E$. Such points exist in  $ X_{0}(N)(\CC)$ if  the Heegner hypothesis is satisfied \cite[Proposition 3.8]{darmon}, and the theory of complex multiplication implies that the are all defined over $ E[1]$. If $W$ denotes     the group of automorphisms of $X_{0}(N)$ generated by the Atkin–Lehner involutions for each distinct prime $p$ dividing $N$, then the set of all Heegner points as defined above is a finite principal homogeneous space for $W \times \Gal(E[1]/E)$ \cite[\S 1.3]{GrossZagier}. If we fix a complex uniformization of $X_{0}(N)(\CC)$ by the extended upper half-plane (which we do),   we can make a choice in this finite set using an explicit isogeny  constructed by fixing an ideal of $\mathfrak{N}   \triangleleft  \mathcal{O}_{E} $ of index $ N $ \cite[\S 3]{grosskoly}. This is the sense in which the point $x_{1}$ is ``distinguished.''

Now let $p_{1} := \pi  (x_{1}) \in A(E[1])$, and let $p_{E} \in A(E)$ be the trace of $p_{1}$ down to $E$. Write $L(A_{/E}, s)$ for the Hasse–Weil $L$-function of $A$ over $E$. A consequence of the Heegner hypothesis is that the sign of the functional equation for $L(A_{/E}, s)$ is $-1$, which in turn forces $\mathrm{ord}_{s = 1} L(A_{/E}, s)$ to be odd. In particular, $L(A_{/E}, 1) = 0$. The Gross–Zagier formula \cite{GrossZagier} shows that when the discriminant $ D $ is odd,\footnote{i.e., $ D \equiv 1 \pmod{4} $}  the point  $p_{E} \in A(E)$ is of infinite order if and only if the derivative $L'(A_{/E}, 1)$ is non-vanishing. That is,  
$$ \mathrm{ord}_{s = 1} \, L(A_{/E}, s) = 1 \implies \mathrm{rank}_{\ZZ} \, A(E) \geq 1. $$  
The Birch and Swinnerton-Dyer conjecture for $A$ over the field $E$ similarly posits that the rank of $A(E)$ should be $1$ whenever $\mathrm{ord}_{s = 1} \, L(A_{/E}, s) = 1$. One therefore hopes to derive the \emph{upper bound} $\mathrm{rank}_{\ZZ} \, A(E) \leq 1$ under the assumption that $p_{E} \in A(E)$ is non-torsion. Since $A(\QQ)$ is a subgroup of $A(E)$, we also end up bounding the original group.     

In \cite{kolyvagin}, Kolyvagin introduced such a bounding argument using what he referred to as an \emph{Euler system} for $A$. Kolyvagin's argument hinges on the observation that the Heegner point $x_{1}$ does not come alone, but rather belongs to a family of such points defined over abelian extensions of $E$ that satisfy certain \emph{norm relations} (sometimes also called \emph{trace} or \emph{distribution relations}). More precisely, for each positive integer $m$, let $E[m]$ denote the ring class extension of conductor $m$. Then for each $m$ relatively prime to $N$, one has a ``distinguished'' Heegner point   $$   x_{m} \in X_{0}(N)(E[m])  $$   again constructed using the fixed complex uniformization of $X_{0}(N)(\CC)$ and an explicit isogeny  defined by  a lattice in $ \mathbb{C}$.   Such points are defined abstractly as before, except   that    $\mathcal{O}_{E}$   is    replaced by an \emph{order} in $\mathcal{O}_{E}$. This distinguished choice ensures that for any rational prime $\ell$ that is inert in $E$ and relatively prime to $ m N $,    we have  
\begin{equation} \label{kolyrelation} 
   T_{\ell}(x_{m}) = \mathrm{Tr}_{\ell}(x_{m \ell}). 
\end{equation}
Here $T_{\ell}$ denotes the standard ``self-dual"  Hecke correspondence of degree $\ell + 1$, and $\mathrm{Tr}_{\ell}$ denotes the trace map from $A(E[m \ell])$ to $A(E[m])$. See \cite[\S 6]{Gross}   
and \cite[Proposition 3.10]{darmon}.      

Kolyvagin's ingenious argument employs the norm relations (\ref{kolyrelation}) in conjunction with Galois cohomology techniques to show that $\mathrm{rank}_{\ZZ} \, A(E) = 1$ if $p_{E} \in A(E)$ is non-torsion \cite[Theorem A]{kolyvagin}. From this, the desired result over $\QQ$ can be obtained as follows. Observe that  
$$ L(A_{/E}, s) = L(A_{/\QQ}, s) L(A'_{/\QQ}, s), $$  
where $A'$ denotes the quadratic twist of $A$ with respect to $E$. It can be shown that if $ \mathrm{ord}_{s=1} L(A_{/\QQ}, s) \leq 1 $, then there exists an imaginary quadratic field $ E $ of odd discriminant such that the  Heegner hypothesis for $ A $ is satisfied and $ \mathrm{ord}_{s=1} L(A_{/E}, s) = 1 $ \cite[Theorem~1]{Murty}. We can therefore safely assume that  $ E $ satisfies all of   these conditions.  On the other hand, the Galois action of $\Gal(E/\QQ)$ on $A(E)$ can be used to identify $A(\QQ)$ with the `plus part' of $A(E)$ and $A'(\QQ)$ with the `minus part' of $A(E)$, from which one sees that  
$$ \mathrm{rank}_{\ZZ} \, A(E) = \mathrm{rank}_{\ZZ} \, A(\QQ) + \mathrm{rank}_{\ZZ} \, A'(\QQ). $$  
Finally, one argues that $p_{E}$ lies in $A(\QQ)$ (up to torsion) if and only if the sign of the functional equation of $L(A_{/\QQ}, s)$ is $-1$ \cite[Proposition 5.3]{grosskoly}, which is equivalent to $\mathrm{ord}_{s = 1} \, L(A_{/\QQ}, s)$ being odd. This proves (\ref{BSD}) when the left-hand side is at most one.  For a detailed exposition of the arguments sketched here, see \cite{darmon} and  \cite[\S 4.]{Miller}.    

\subsection{The Bloch–Kato conjecture} 
The bounding argument introduced by Kolyvagin has since been axiomatized and applies more generally in the context of global $p$-adic Galois representations \cite{Rubin, KatoEuler, Perrin-Riou}. This is partly motivated by a vast generalization of (\ref{BSD}), known as the \emph{Bloch–Kato conjecture} \cite{BlochKato}, which posits that the order of vanishing at integer values of the $L$-function of a global $p$-adic Galois representation is related to the dimension of a Galois cohomology group known as the \emph{Bloch–Kato Selmer group}. A very active area of research nowadays is the establishment of new instances of this conjecture, under the assumption that the order of vanishing of the relevant $L$-function is at most one. In many cases studied in recent years, a key step toward this goal is the construction of an Euler system for the underlying Galois representation. Such a construction is usually carried out by exploiting the geometry of a Shimura variety and is motivated by a period integral that establishes an intimate relationship between the $L$-values of the Galois representation and the ``bottom class'' of the Euler system. In the case of the elliptic curve $A$, the Galois representation is the $p$-adic Tate module of $A$, the Shimura variety is the modular curve, and the period integral relation is provided by the Gross–Zagier formula.

The relation between (\ref{BSD}) and the Bloch–Kato conjecture can be elaborated via Kummer theory. Let $\overline{\QQ}$ denote the algebraic closure of $\mathbb{Q}$ in $\mathbb{C}$. For $p$ a rational prime, let $A[p^{n}]$ for $n$ a positive integer denote the $p^n$-torsion subgroup scheme of $A$, and let $$\mathrm{T}_{p}(A)  = \varprojlim\nolimits_{n} A[p^{n}] (\overline{\QQ})     $$ denote the $p$-adic Tate module of $A$. The Kummer sequence associated with $A[p^{n}]$ for each $n$ gives rise to the familiar exact sequence  
\begin{equation} \label{exact} 
   0 \longrightarrow A(\QQ) \otimes \ZZ/p^{n}\ZZ \longrightarrow \mathrm{Sel}(\mathbb{Q}, A[p^{n}]) \longrightarrow \Sha(A_{/\QQ})[p^{n}] \longrightarrow 0
\end{equation} 
where $\mathrm{Sel}(\QQ, A[p^{n}]) \subset \mathrm{H}^{1}(\mathbb{Q}, A[p^{n}](\overline{\QQ}))$ denotes the classical $p^{n}$-Selmer group of $A$, $\Sha(A_{/\mathbb{Q}})[p^{n}]$ denotes the $p^{n}$-torsion of the Tate–Shafarevich group of $A$, and the first non-trivial map is the Kummer map. Let us denote  
$$ \mathrm{Sel}(\mathbb{Q}, \mathrm{T}_{p}(A)) := \varprojlim\nolimits_{n} \mathrm{Sel}(\mathbb{Q}, A[p^{n}]) .  $$  
This is a finitely generated $\mathbb{Z}_{p}$-module. It has been conjectured that the Tate–Shafarevich group $\Sha(A_{/\mathbb{Q}})$ is always finite.  Assuming this, and since each $A(\QQ) \otimes \ZZ/p^{n}\ZZ$ is finite, the inverse limit of (\ref{exact}) over all $n$ gives rise to an exact sequence  
\begin{equation} \label{exactSelmer}  
   0 \longrightarrow A(\QQ) \otimes \ZZ_{p} \longrightarrow \mathrm{Sel}(\mathbb{Q}, \mathrm{T}_{p}(A)) \longrightarrow \Sha(A_{/\QQ})[p^{\infty}] \longrightarrow 0  
\end{equation} 
where $\Sha(A_{/\mathbb{Q}})[p^{\infty}]$ denotes the $p$-primary component of the conjecturally finite group $\Sha(A_{/\mathbb{Q}})$.  
Thus, we expect that  
$$ \mathrm{rank}_{\ZZ} \, A(\QQ) = \mathrm{rank}_{\ZZ_{p}} \, \mathrm{Sel}(\QQ, \mathrm{T}_{p}(A)) , $$  
and we may instead replace the conjectural equality (\ref{BSD}) with
\begin{equation}  \label{BK} 
   \mathrm{ord}_{s=1}\, L(A_{/\mathbb{Q}}, s) = \mathrm{rank}_{\mathbb{Z}_{p}} \, \mathrm{Sel}(\QQ, \mathrm{T}_{p}(A)) . 
\end{equation} 
Now the Selmer group $\mathrm{Sel}(\mathbb{Q}, \mathrm{T}_{p}(A))$ above coincides with the \emph{Bloch–Kato Selmer group} $$ \mathrm{H}^{1}_{f}(\mathbb{Q}, \mathrm{T}_{p}(A)) $$ of the Galois representation $\mathrm{T}_{p}(A)$ as defined in \cite[Definition 5.1]{BlochKato}\footnote{The choice of the open set $U$ in  that     definition does not matter by eq.\ (3.11.2) of \emph{op.\ cit.}}, and this purely cohomological definition applies to any $p$-adic Galois representation.  It is also possible to define the (shifted) $L$-function $L(A_{/\QQ}, s+1)$ entirely in terms of  $\mathrm{T}_{p}(A)$, and one can generalize this definition to arbitrary ``motivic'' $p$-adic Galois representations \cite[Definition 5.5]{BlochKato}. However, the meromorphic continuation of these more general $L$-functions is unknown, except when one can identify these functions with the $L$-functions of certain  automorphic representations. Nevertheless, assuming this continuation, the Bloch–Kato conjecture posits an analogue of (\ref{BK}). See, e.g., \cite[Conjecture 1.2.3]{KingsBloch} for a precise statement and  the   unpublished notes  \cite{Bellaiche} for a user-friendly treatment of various   topics     surrounding this conjecture.                   

\begin{remark} 
In \cite{KingsBloch}, the Bloch–Kato conjecture for an elliptic curve $A$ would be stated in terms of $\mathrm{V}_{p}(A) := \mathrm{T}_{p}(A) \otimes_{\ZZ_{p}} \QQ_{p}$ and its Bloch–Kato Selmer group $\mathrm{H}^{1}_{f}(\mathbb{Q}, \mathrm{V}_{p}(A))$, which is a $ \QQ_{p} $-vector space. But by \cite[Proposition B.2.4]{Rubin} and \cite[eq.\ 3.7.3]{BlochKato}, it is easy to see that this Selmer group is just $\mathrm{Sel}(\mathbb{Q}, \mathrm{T}_{p}(A)) \otimes_{\ZZ_{p}} \QQ_{p}$, so that  
$$ \dim_{\QQ_{p}} \mathrm{H}^{1}_{f}(\mathbb{Q}, \mathrm{V}_{p}(A)) = \mathrm{rank}_{\ZZ_{p}} \, \mathrm{Sel}(\mathbb{Q}, \mathrm{T}_{p}(A)). $$  
We also remark that the term involving Galois invariants in the statement of the general Bloch–Kato conjecture vanishes unless the Galois representation contains the $p$-adic cyclotomic character $\mathbb{Q}_{p}(1)$ as a sub-representation. This additional term is included to account for the simple  pole of the Riemann zeta function,   and  can   otherwise     be   ignored.    
\end{remark}

\subsection{Euler  systems}   \label{Eulerintrosec}    Let us recall the definition of an Euler system modeled on \cite[Definition II.1.1]{Rubin}, in a special case. Suppose $V$ is a $p$-adic Galois representation of $\Gal(\overline{\QQ}/\QQ)$ that is unramified away from a finite set of primes $S$, and let $T \subset V$ be a Galois-stable lattice. That is, $ T $ is  a   $\ZZ_{p}$-submodule of $V$ of $\ZZ_{p}$-rank equal to $\dim_{\QQ_{p}} V$    which is invariant under $\Gal(\overline{\QQ}/\QQ)$. Such a lattice always exists \cite[\S 1.1.2]{FontaineOuyang}. Let $\mathcal{N}p^{\infty}$ denote the set of all integers of the form $np^{r}$, where $n$ is a square-free product of primes not in $S \cup \{ p \}$ and $r$ is a non-negative integer. For each $m \in \mathcal{N}p^{\infty}$, let $\QQ(\mu_{m})$ denote the cyclotomic extension of $\QQ$ generated by $\mu_{m}$, the group of $m$-th roots of unity. An \emph{Euler system} for $T$ is a collection of Galois cohomology classes  
$$ c_{m} \in \mathrm{H}^{1}(\QQ(\mu_{m}), T) $$
for each $m \in \mathcal{N}p^{\infty}$, such that for each prime $\ell$ with $\ell m \in \mathcal{N}p^{\infty}$,  
\begin{equation} \label{ESdef} 
   \mathrm{cores}^{\QQ(\mu_{m\ell})}_{\QQ(\mu_{m})}(c_{m\ell}) = 
   \begin{cases} 
      c_{m} & \text{if } \ell = p, \\
      P_{\ell}(\mathrm{Frob}_{\ell}^{-1}) \, c_{m} & \text{if } \ell \neq p .
   \end{cases} 
\end{equation} 
Here $P_{\ell}(X) := \det(1- \mathrm{Frob}_{\ell}^{-1} X \, | \, T^{\vee}(1))$ denotes the reverse characteristic polynomial of the geometric Frobenius at $\ell$ acting on the Cartier dual $T^{\vee}(1) : = T^{\vee} \otimes_{\ZZ_{p}}     \ZZ_{p}(1)     $ of $T$, $\mathrm{Frob}_{\ell}^{-1}$ denotes a choice of   geometric Frobenius above $\ell$ (which acts on $\mathrm{H}^{1}(\QQ(\mu_{n}), T)$ for $ \ell \nmid n $ via  inverse Frobenius substitution $ \mathrm{Fr}_{\ell}^{-1} \in  \Gal(\QQ(\mu_{n})/\QQ)$), and $\mathrm{cores}$ denotes the corestriction map in Galois cohomology. Note that the polynomial $P_{\ell}(X)$ is also used to define the Euler factor appearing in the $L$-function of $V^{\vee}(1)$, and its appearance in (\ref{ESdef}) is the motivation for the term “Euler system.” Under suitable hypotheses, a non-trivial Euler system imposes non-trivial bounds on the Selmer group of $T^{\vee}(1)$. Let us note that when $T = \mathrm{T}_{p}(A)$ for an elliptic curve $A$, the Weil pairing induces an isomorphism $$ T \simeq T^{\vee}(1) , $$ and we say that $T$ is \emph{self Cartier dual} or \emph{polarized}. For such representations, one may make the aforementioned definition entirely in terms of the Euler factors of $T$.  

Traditionally, the relations in the case $\ell = p$ are referred to as \emph{vertical norm relations} or \emph{wild norm relations}, whereas the relations for $\ell \neq p$ are referred to as \emph{horizontal norm relations} or \emph{tame norm relations}.\footnote{While these relations are strictly speaking not independent of each other, one can often work “one prime at a time” by parametrizing the Galois cohomology classes by a space that admits a restricted tensor product decomposition over all but finitely many places.} The class $c_{1} \in \mathrm{H}^{1}(\QQ, T)$ is called the \emph{bottom class} of the Euler system. One can also define such systems for abelian extensions of number fields $F$ different from $\QQ$. In practice, one often restricts to layers of abelian extensions of a particular type, as the classes that can be constructed to fit such a system are only norm compatible over special   extensions. For instance, the ring class extensions $E[m]$ introduced above are abelian extensions of $E$ that are \emph{anticyclotomic} over $\QQ$, i.e., $\Gal(E/\QQ)$ acts on $\Gal(E[m]/E)$ by inversion. A collection of classes defined only for layers in such extensions and satisfying analogous norm relations is referred to as an \emph{anticyclotomic Euler system}.  

From the perspective of the Euler system relations (\ref{ESdef}), the usefulness of (\ref{kolyrelation}) arises from the fact that the operator $T_{\ell}$ essentially determines the local $L$-factor of the $p$-adic Tate module of $A$ at the prime $\ell \neq p$. More precisely, if $ \tilde{A} $ denotes the reduction of $ A $ at the prime $ \ell $ and  $p_{m}$ denotes the rational point $ \pi (x_{m}) \in A(E[m])$ where $ \pi $ is as in (\ref{modularpara}), then the relation (\ref{kolyrelation}) specializes to  
\begin{equation}   \label{kolyspecialrelation}       
a_{\ell} \,  p_{m} = \mathrm{Tr}_{\ell}(p_{m \ell}),  
\end{equation}    
where $a_{\ell} := \ell + 1 - |\tilde{A}(\mathbb{F}_{\ell})|$.   The same relations then hold for  the  cocylce classes    $ c_{m}  \in \mathrm{H}^{1}(E[m], \mathrm{T}_{p}(A))$ obtained as images of $p_{m}$ under the Kummer maps  
$$ A(E[m]) \to A(E[m]) \otimes   _ { \ZZ  }     \ZZ_{p} \to \mathrm{H}^{1}(E[m], \mathrm{T}_{p}(A)) $$  
with $\mathrm{Tr}_{\ell}$ replaced by corestriction. That is,  
\begin{equation} \label{normrelcoho} 
   a_{\ell}  c_{m}   = \mathrm{cores}^{E[m\ell]}_{E[m]}(c_{m\ell}) 
\end{equation}    
for all positive integers $m$ and inert primes $\ell$ satisfying $(m, N) = (\ell, mN) = 1$. On the other hand, the   reverse     characteristic polynomial for the action of $\mathrm{Frob}_{\ell}^{-1}$ on the polarized  Galois representation $ \mathrm{T}_{p}(A)$ is  
\begin{equation} P_{\ell}(X) = 1 - a_{\ell} \ell^{-1} X + \ell^{-1} X^{2}   \label{PellX}   
\end{equation}    
It is possible to massage the classes $ c_{m} $ in such a way that $ c_{1}  $ remains unchanged and the Euler factor on the left-hand side of (\ref{normrelcoho}) becomes 
$$ P_{\ell}(\mathrm{Frob}_{\lambda}^{-1}) = 1 - a_{\ell} \ell^{-1} \mathrm{Frob}_{\lambda}^{-1} + \ell^{-1} \mathrm{Frob}_{\lambda}^{-2}, $$  
where $\mathrm{Frob}_{\lambda}^{-1}$ denotes a choice  of   geometric    Frobenius at the unique prime $\lambda$ of $E$ above $\ell  \neq p $. Notice that the Frobenius substitution at $ \lambda    $ is trivial in $\Gal(E[m]/E)$ for all $ m $ and inert $ \ell $ such that    $ \ell \nmid m $.   Thus the action of $P_{\ell}(\mathrm{Frob}_{\lambda}^{-1})$ on $\mathrm{H}^{1}(E[m], \mathrm{T}_{p}(A))$ coincides with multiplication by the scalar $P_{\ell}(1) = 1 - a_{\ell} \ell^{-1} + \ell^{-1}$. Since the degree of extension $E[m\ell]/E[m]$ is   $\ell + 1$, multiples of $\ell + 1$ in the $\ZZ_{p}$-module $\mathrm{H}^{1}(E[m], \mathrm{T}_{p}(A))$ are in the image of the corestriction map from level $E[m\ell]$.\footnote{In particular, the statement holds even in the case $p\mid (\ell+1)$, which is the case of primary interest.} Now observe that  
$$ (1 - a_{\ell} \ell^{-1} + \ell^{-1}) - a_{\ell} = \ell^{-1}(1+\ell)(1-a_{\ell}) $$  
is a $\ZZ_{p}$-multiple of $\ell + 1$.         Thus, if we define  
\begin{equation}  %\label{zellprimerelation} 
z_{\ell}   :=  c_{\ell} + \ell^{-1}(1-a_{\ell}) \, \mathrm{res}^{E[\ell]}_{E[1]}(c_{1}) \in \mathrm{H}^{1}(E[\ell], \mathrm{T}_{p}(A)) 
\end{equation}   
where $ \mathrm{res} $ denotes restriction,  we have   
\begin{align*}   P_{\ell}(\mathrm{Frob}_{\lambda}^{-1}) c_{1}  
& = (1 - a_{\ell} \ell^{-1} + \ell^{-1})  c_{1}  \\
& = a_{\ell} c_{1} + \ell^{-1}(1+\ell)(1-a_{\ell})c_{1}  \\
&   =  \mathrm{cores}^{E[\ell]}_{E[1]}(c_{\ell})  + \ell^{-1} (1-a_{\ell})   \,     \mathrm{cores}^{E[\ell]}_{E[1]}\left ( \mathrm{res}^{E[\ell]}_{E[1]}(c_{1}) \right ) \\
& = \mathrm{cores}^{E[\ell]}_{E[1]}(z_{\ell}) .  
\end{align*}   
More generally, for square-free $m$ relatively prime to $pN$,  we can define  
$$ z_{m}  := \sum_{n \mid m} \Big( \prod_{\ell \mid \sfrac{m}{n}} \ell^{-1}(1 - a_{\ell}) \Big)  \mathrm{res}^{E[m]}_{E[n]}(c_{n}   )   \in  \mathrm{H}^{1}(E[m], \mathrm{T}_{p}(A)) $$  
where the sum is over all divisors of $m$ and the product is over all prime divisors of $m/n$. Then $z_{1}  =  c_{1}$ and  
\begin{equation} \label{correctESHeeg} 
   P_{\ell}(\mathrm{Frob}_{\lambda}^{-1})  z_{m}   = \mathrm{cores}^{E[m\ell]}_{E[m]}(z_{m\ell}   ), 
\end{equation} 
for all inert primes $\ell $ that do not divide $  mN p   $. The norm relations  (\ref{correctESHeeg}) are then closer in spirit to the ones required in (\ref{ESdef}). See \cite[\S IX.6]{Rubin} for a similar ``massaging" trick for general Euler systems.    
\begin{remark}   \label{ESfactorremark}        
The original definition suggested by Kolyvagin in \cite[p.448]{KolyvaginES} (axiom AX1) insists on using $P_{\ell}(\mathrm{Frob}_{\lambda}^{-1})$ as the Euler factor for  norm    relations, and the bounding arguments go through with this choice. Note that we cannot literally use $ P_{\ell}(\mathrm{Frob}_{\ell}^{-1})$ as in (\ref{ESdef}), since the conjugacy class of $ \mathrm{Fr}_{\ell}^{-1} \in \Gal(E[m]/\QQ) $ can be that of complex conjugation (which is not a singleton  if $ \Gal(E[m]/E) $ is not 2-torsion) and elements of this class may have differing actions on $ \mathrm{H}^{1}(E[m], \mathrm{T}_{p}(A))$.   

On the other hand, if we only consider $ T = \mathrm{T}_{p}(A)$ as a $\Gal(\overline{\QQ}/E)$-representation, then it is more appropriate to use $P_{\lambda}(\mathrm{Frob}_{\lambda}^{-1})$, where   
\begin{equation}   \label{PlambdaX}   
\begin{split}
   P_{\lambda}(X) 
   &= \det(1 - \mathrm{Frob}_{\lambda}^{-1} X \, | \, \mathrm{T}_{p}(A)) \\ 
   &= 1 - (\ell^{-2} a_{\ell}^{2} - 2\ell^{-1}) X + \ell^{-2} X^{2}, 
\end{split}
\end{equation} 
is the reverse characteristic polynomial of $\mathrm{Frob}_{\lambda}^{-1}$  acting  on $ T  \simeq  T^{\vee}(1)$. Let $ c \in \Gal(\overline{\QQ}/\QQ) $ denote the complex conjugation and let $ T^{c} $ denote the representation of $ \mathrm{Gal}(\overline{\QQ}/E) $ on which $ \gamma \in \Gal(\overline{\QQ}/E) $ acts as $ c \gamma c^{-1}$. Then $ T^{c} \simeq T $ (complex conjugation provides an isomorphism) and therefore $$ T^{c} \simeq T^{\vee}(1) $$ as $ \mathrm{Gal}(\overline{\QQ}/E) $-representations. Such representations of $ \Gal(\overline{\QQ}/E) $ are often referred to as \emph{conjugate self-dual} in literature.     In many recent works, anticyclotomic Euler systems have been constructed for conjugate self-dual Galois representations of $\Gal(\overline{\QQ}/E)$ which may or may not descend to representations of $\Gal(\overline{\QQ}/\QQ)$. These works thus only use Euler factors over $E$. See \S \ref{ESoverE} for an analogue of (\ref{correctESHeeg}) that involves $P_{\lambda}(X)$. 
\end{remark} 
\begin{remark} 
Kolyvagin’s formulation in \cite{kolyvagin} also imposed a “congruence condition” (axiom AX3), but this can be replaced by the vertical norm relation requirement in the definition above \cite[Remark~II.1.5]{Rubin}. We refer the reader to \cite{loe} for a general machinery for establishing vertical norm relations that leverages  the theory of spherical varieties.
\end{remark} 

While the relation (\ref{kolyrelation}) suffices for Kolyvagin’s bounding argument, its form is not particularly representative of the    situation encountered in the setting of higher dimensional Shimura varieties.   In general, automorphic $L$-factors are computed via   the action of more than one Hecke operator. In fact, the totality of the operators required is packaged into what is known as a \emph{Hecke polynomial}. In the situation of modular curves, $T_{\ell}$ is the middle coefficient of a degree-two Hecke polynomial whose coefficients retrieve those of $P_{\ell}(X)$ (\ref{PellX}) as eigenvalues under the Hecke action on the eigenform associated with the elliptic curve $A$. In Kolyvagin’s case it suffices to work with $T_{\ell}$ alone, since the action of $P_{\ell}(\mathrm{Frob}_{\lambda}^{-1})$ corresponds to multiplication by $a_{\ell}$ modulo $\ell+1$, and, as explained above, one can derive the ``correct'' relations (\ref{correctESHeeg}) from the simplified relations (\ref{normrelcoho}). However, such simplifications do not exist for general automorphic Galois representations,  and one must establish the horizontal norm relations with the full Euler factor as,  for instance, required in (\ref{ESdef}).   

Accordingly, a more natural version of the Hecke-operator-valued norm relation (\ref{kolyrelation}) would involve the complete Hecke polynomial that directly specializes to (\ref{correctESHeeg}) and that also holds at primes $\ell$ which are \emph{split} in $E$. Indeed, Jetchev, Nekovář, and Skinner \cite{JNS} have proposed a framework in which only split relations are required to carry out Kolyvagin’s bounding argument. Their approach also  has the advantage of being applicable to conjugate self-dual Galois representations of $\Gal(\overline{\QQ}/E)$ that do not necessarily descend to representations of $\Gal(\overline{\QQ}/\QQ)$. Several examples of such “split” Euler systems have already been constructed (\cite{Anticyclo}, \cite{CRR}, \cite{SkinnerLai}, \cite{disegni}) and have been used to make significant progress towards the Bloch--Kato conjecture in a variety of settings. % \\[-1.5em] 

\subsection{Aims of this article}   \label{aimssection}      In this work, we revisit the setup of Heegner points (and more generally, CM points) on modular curves and establish horizontal norm relations with the full Hecke polynomial at all but finitely many  primes in the anticyclotomic tower of $E$ (see Theorem \ref{CMnormtheorem}). No prior knowledge of such relations is assumed, and the main arguments rely only on the combinatorics of two-dimensional lattices over local fields. In particular, we do not invoke the modular interpretation of these points, as it does not generalize to higher dimensional cycles. At inert primes, our relations can also be derived by a straightforward recasting of (\ref{kolyrelation}), though it is perhaps less immediate  at split primes. The latter case, however, offers a better view of the intricacy  of such relations for special cycles on general Shimura varieties.  

Another aim of this article is to reformulate the aforementioned norm relations in the language of adeles and smooth representation theory, which allows us to reduce the problem of   establishing horizontal norm relations  to  constructing  certain “integral test data’’ in purely local Schwartz spaces. This reformulation has played a key role in the construction of several new Euler systems, most notably in \cite{LSZ}, where such test data were first constructed in the setting of Siegel modular threefolds using \emph{local zeta integrals}.   Since many classical sources on Euler systems of Heegner points work in a non-adelic framework, we begin with a detailed review of the theory of modular curves and make an explicit translation between the classical and adelic languages. This also serves to address certain sign discrepancies   that arise from different choices of conventions, and provides an additional check on which conventions are mutually compatible.  We then proceed to establish the horizontal norm relations in a purely representation-theoretic setting. For comparison,  we   also  study these  local relations via the method of local zeta integrals developed in \cite{LSZ}, specialized here to the case of split primes.  

It should be noted, however, that the method of local zeta integrals relies crucially on the so-called \emph{multiplicity one hypothesis} for the associated period integrals, which does not hold in all situations of interest. More precisely, some automorphic $L$-functions can be represented by period integrals that admit motivic interpretations but unfold to so-called \emph{non-unique} models \cite{PollackShahGsp6}, \cite[p.~1798]{OWR}. To handle these situations,   an alternative approach to constructing the integral test data via Hecke polynomials was proposed by the author in \cite{CZE}. This method overcomes the failure of the aforementioned hypothesis and has been successfully used to construct Euler systems in the settings of Siegel modular sixfolds \cite{Siegel1} and certain unitary Shimura varieties of signature $(2,2)$ \cite{EulerGU22}, both of which lie outside the reach of the method of local zeta integrals. A third aim of this article is to elaborate on this more general method, with the hope of making the aforementioned works more accessible.

More recently, a promising connection between horizontal norm relations and the theory of spherical varieties has been explored in \cite{CaiFanLai}, although certain integrality issues currently limit the applicability of the   main ideas.  Via the examples of \S \ref{examplessec}, we  also aim to highlight certain congruence properties of the degrees of Hecke polynomials (and of their  twisted restrictions) that appear to underlie these norm relations, with the hope of stimulating further research in this direction.

\subsection{Outline} 
This article is divided into four sections. In \S \ref{modcurvesec}, we review the adelic theory of  modular curves.  In \S \ref{Eulersystemsec}, we establish the horizontal norm relations by introducing certain judiciously chosen elements in a space of Schwartz functions, whose elements parametrize divisors of CM points on modular curves.   In \S \ref{integraltestdatasec}, we formally define the notion of integral test data and elaborate on the methods of  \cite{CZE} and  \cite{LSZ}.      Finally in \S \ref{examplessec}, we reprove our norm relations at split primes using the both  methods.  An additional example  involving  $ P_{\lambda}(X) $ is also  included to illustrate the broader applicability of the method of \cite{CZE}.    

\subsection{Acknowledgements} This article is based on the author’s thesis work carried out at Harvard University. It was originally inspired by   a    combinatorial relationship between the test data constructed in \cite[\S 7]{Anticyclo} and the coefficients of the standard Hecke polynomial of $\GL_{n}$. The author is deeply grateful to Barry Mazur for his encouragement and for his careful reading of earlier drafts, and to Lillian Pierce for her insightful feedback on improving the exposition. While working through various sign–convention issues, the author benefited greatly from discussions with   Christophe Cornut, Andrew Graham, and Antonio Cauchi, and is grateful to them for their valuable insights. The author also wishes to thank the referees for their diligent reading of this paper and for their numerous helpful comments, which substantially improved the writing and broadened the scope of this work.    
\section{Modular curves}  \label{modcurvesec}  
In this section, we review the theory of modular curves in the spirit of \cite{DeligneTS}. Our primary goal is to present, in a simple setting, the terminology that appears in the study of higher-dimensional Shimura varieties. Although the material here goes beyond what is strictly required for establishing the norm relations in \S\ref{Eulersystemsec}, we include it to provide a fuller picture of the relationship between the adelic and classical descriptions of modular curves and to illustrate how one may translate between these two viewpoints.   This also serves as an  additional   check on our conventions and helps settle certain doubts regarding the definition of Hecke polynomials originally raised by Jan Nekovář in \cite{Nekovar}. In addition, since the literature employs two different Shimura data for $\GL_{2}$, we include a comparison of these choices throughout the section in the form of remarks and highlight how the associated conventions must be adjusted when translating statements  between them.

Throughout,  we let $\overline{\QQ} $  denote the algebraic closure of $\QQ$ in the field of complex numbers $ \mathbb{C}$. We fix $i \in \CC$ to be  choice of  a  root of $x^{2}+1 \in \mathbb{R}[x]$.   For a ring $ R $, we identify   $ R^{2} $ with $ \mathrm{Mat}_{2 \times 1}(R)  $ via $ (r_{1}, r_{2}) \mapsto \left ( \begin{smallmatrix} r_{1} \\ r_{2} \end{smallmatrix}  \right ) $  and let $ \GL_{2}(R) $ act on the left   of $ R^{2} \simeq \mathrm{Mat}_{2 \times 1}(R) $ via left  matrix multiplication.  For $ g \in \GL_{2}(R) $, we will denote by $ {}^{t}\negsmall g $ the transpose of $ g $. If $ H $ is a subgroup of $ \GL_{2}(R)$, we will let $   {}^{t}   \negsmall H    $ denote the group obtained by taking transposes of elements of   $ H   $.     If $ (e_{i}) $, $ (f_{j}) $ are two ordered basis for a free module $R$-module $M$ of finite rank,   the change of coordinates matrix from $ (e_{i}) $ to $ (f_{j}) $  is matrix of the identity map $ M \to M $ where the domain has basis $ (e_{i}) $ and the target has basis $ (f_{j}) $.  
\subsection{Shimura data} 
\label{SDdatumsec} 
The modular curves arise from what is known as a \emph{Shimura datum} for $\GL_{2,\QQ}$. For the sake of completeness, we first recall the general definition given in \cite[\S 2]{DeligneVar}   and  \cite[\S 5]{Milne}.

Let $\Gb$ be any connected reductive algebraic group over $\QQ$, and let $\delT$ denote the \emph{Deligne torus} $\mathrm{Res}_{\CC / \RR} \GG_{m}$, where `$\mathrm{Res}$' denotes Weil restriction of scalars. Recall \cite[\S 2]{Milne} that an algebraic representation of $\delT$ on a real vector space $V$ gives a Hodge structure on $V$, where the bigraded piece $V^{p,q}$ of the complexification $V_{\CC} := \CC \otimes_{\RR} V$ is the subspace
\begin{equation}   
V^{p,q} = \{ v \in V \otimes_{\RR} \CC \mid h(z)v = z^{-p}\bar{z}^{-q} v \text{ for all } z \in \mathbb{C}^{\times}\}   \label{Deligneconvention}
\end{equation}   
Thus a morphism $h : \delT \to \Gb_{\RR}$ determines a Hodge structure on the Lie algebra $\mathrm{Lie}(\Gb_{\RR})$ via the adjoint representation. The $\Gb(\RR)$-conjugacy class of $h$ is defined to be the set of all conjugated morphisms 
$
\{ g h g^{-1} \mid g \in \Gb(\RR) \}   $ where $ (g h g^{-1})(z)  :   = g h(z) g^{-1} $.
\begin{remark} The normalization for the Hodge bigrading used in (\ref{Deligneconvention}) is due to Deligne, and differs from the one used in Hodge theory. See \cite[Remarque 1.1.6]{DeligneVar} for a justification of this choice.
\end{remark} 
Let $\mathcal{X}$ be an arbitrary $\Gb(\RR)$-conjugacy class of homomorphisms $\delT \to \Gb_{\RR}$. We say that $(\Gb, \mathcal{X})$ is a \emph{Shimura datum} if for all $h \in \mathcal{X}$,
\begin{itemize} 
\item[(SV1)] the Hodge bigrading of the complex vector space $\mathrm{Lie}(\Gb)_{\CC}$ induced by the adjoint action of $\delT$ via $ h $  is contained in $\{ (-1,1), (0,0), (1,-1) \}$,  
\item[(SV2)] $\mathrm{ad}(h(i))$ is a Cartan involution of the derived group $\Gb^{\mathrm{der}}(\RR)$, i.e., the real Lie group
$$
\{ g \in \Gb^{\mathrm{der}}(\CC) \mid h(i)\bar{g}h(-i) = g \}
$$
is compact,  and   
\item[(SV3)] the adjoint group $\Gb^{\mathrm{ad}}$ has no $\QQ$-factor on which the projection of $h$ is trivial. 
\end{itemize}  
It is easy to see that these axioms hold for all elements in $\mathcal{X}$ if they do for a single $h \in \mathcal{X}$. A \emph{morphism} $  (\Gb',  \mathcal{X}')  \to  (\Gb, \mathcal{X}) $ of Shimura data is a morphism $ f :  \Gb' \to \Gb $ of algebraic groups over $ \QQ $   such that $ f_{\RR}  \circ  h'     \in  \mathcal{X} $ for any $ h ' \in  \mathcal{X}' $.     An  \emph{isomorphism} of Shimura data is a morphism such that  the  map   on algebraic groups is an isomorphism.       

Henceforth, we let $ \Gb $ denote the algebraic group $ \GL_{2, \QQ}   $.         
Let $ \mathcal{X}_{\mathrm{std}} $ denote the $ \Gb(\RR) $-conjugacy class of the homomorphism
\begin{equation}   \label{SDdatum}      
  h_{\mathrm{std}} :  \mathbb{S}  \to  \Gb   _  {   \RR    }       \quad  \quad  
  z =  a + b  i    \mapsto    
  \left ( \begin{smallmatrix}  a  &  b   \\  -  b  &   \,  \,     a  \end{smallmatrix}  \right  ) .
\end{equation}  
Then $ \mathcal{X}_{\mathrm{std}}  $ constitutes a Shimura datum  for $ \Gb $. This is \cite[Example 5.6]{Milne}, but we elaborate on some details. Axiom (SV1) is satisfied since  
$$ 
\mathfrak{gl}_{2, \CC}  
=   
\left\langle   
  \left( \begin{smallmatrix} -1 & i \\ i & 1 \end{smallmatrix} \right)   
\right\rangle  
\oplus  
\left\langle  
  \left( \begin{smallmatrix} 1 & 0 \\ 0 & 1 \end{smallmatrix} \right),   
  \left( \begin{smallmatrix} 0 & 1 \\ -1 & 0 \end{smallmatrix} \right)   
\right\rangle  
\oplus  
\left\langle    
  \left( \begin{smallmatrix} -  1 & -i \\ - i &   1 \end{smallmatrix} \right)   
\right\rangle .   
$$ 
is the desired Hodge decomposition. Since $ \Gb^{\mathrm{der}} = \SL_{2,\QQ} $, and  
$ 
g = \left( \begin{smallmatrix} a & b \\ c & d \end{smallmatrix} \right)  \in \SL_{2}(\CC) 
$  
satisfies $ h(i) \bar{g} h(-i) = g $ if and only if $ a = \bar{d}, \, b = -\bar{c} $, the Lie group defined by the involution $ \mathrm{ad}(h(i))$ is identified with real 3-sphere $ S^{3} $, so axiom (SV2) is verified. Finally, since $ \mathrm{PGL}_{2,\QQ} $ is simple and $ h_{\mathrm{std}} $ does not factor through the center of $ \Gb_{\RR} $, axiom (SV3) holds as well.

A consequence of the   axioms (SV1) and (SV2)   is that $ \mathcal{X}_{\mathrm{std}} $ has a natural structure of a complex Riemannian manifold \cite[\S 2.1]{DeligneVar},   \cite[Proposition 5.9]{Milne}.         Let  $ \mathbf{Z} $  denote the center of $ \Gb $. It is easy to see that  the    centralizer $ K_{\infty} $ in $ \Gb(\RR) $ of $h_{\mathrm{std}}(i) $  is the image $$    h_{\mathrm{std}}(\CC^{\times}) = \mathbf{Z}(\RR) \mathrm{SO}_{2}(\RR)   .   $$     Since $ K_{\infty} $ is abelian,  the  stabilizer of $ h_{\mathrm{std}} \in \mathcal{X}_{\mathrm{std}}  $ under the conjugacy action of $ \Gb(\RR) $  is  also  $  K_{\infty}  $.   Consequently,   we can identify $\mathcal{X}_{\mathrm{std}}$ with $\Gb(\RR)/ K_{\infty}  $ via $g \, \negsmall  h_{\mathrm{std}} \,   \negsmall     g^{-1} \mapsto [g]$ and furthermore, with the set of all complex structures $$ \mathrm{CS}(\RR^{2})   :     = \left \{ J \in \Gb(\RR) \, | \, J^{2} = -1 \right \} $$ on $ \RR^{2} $ via $ g \,   \negsmall  h_{\mathrm{std}} \,   \negsmall     g^{-1}  \mapsto  g h_{\mathrm{std}}(i) g^{-1} $. We can also identify these sets with $ \mathcal{H}^{\pm} := \CC \setminus  \RR   $   via     % via  $\mathcal{H}^{\pm} := \CC \setminus \RR$ via  
\begin{align*} 
\Gb(\RR)/K_{\infty} &\longrightarrow \mathcal{H}^{\pm}  \\ 
\left( \begin{smallmatrix} a & b \\ c & d \end{smallmatrix} \right) &\longmapsto \tfrac{ai + b}{ci + d} . 
\end{align*} 
and the resulting identification  of   $ \mathcal{X}_{\mathrm{std}} $ with $ \mathcal{H}^{\pm} $ respects the complex and Riemannian manifold structures. 
%Note that the first isomorphism is independent of the choice of the  ``base-point" $ h_{\mathrm{std}}  \in  \mathcal{X}_{\mathrm{std}}$ since it is given by $ h \mapsto h(i) $ for $ h \in \mathcal{X}_{\mathrm{std}}$.        
The left action of $\Gb(\RR)$ on $\mathcal{X}_{\mathrm{std}}$ (via conjugation) is then identified with left multiplication on $\Gb(\RR)/K_{\infty} $, with conjugation on $ \mathrm{CS}(\RR^{2}) $  and with Möbius transformations on $\mathcal{H}^{\pm}$, as defined in \cite[\S 1.2]{Shimurabook}.   The   following  diagram   summarizes the various  identifications.   
\begin{equation}   \label{variousiden}       
\begin{array}{cccccl}
\mathcal{X}_{\mathrm{std}}
 & \xrightarrow{\;\sim\;} &
 \mathrm{CS}(\RR^{2})
 & \xrightarrow{\;\sim\;} &
 \Gb(\RR)/  K_{\infty}    
 & \xrightarrow{\;\sim\;}   \, \mathcal{H}^{\pm} \\
[0.3em]
gh_{\mathrm{std}}g^{-1}
 & \longmapsto&
 gh_{\mathrm{std}}(i)g^{-1}
 & \longmapsto&
 g  K _ { \infty  }    
 &    \longmapsto \, \, \,  g\cdot i.
\end{array}
\end{equation}  
The choice of $i \in \CC$ made above allows us to designate the “upper” half-plane $\mathcal{H}^{+} \subset \mathcal{H}^{\pm}$ as the connected component of $ \mathcal{H}^{\pm} $  containing $ i $. Then $ g K _{\infty}  \in \Gb(\RR)/K_{\infty}    $ corresponds to a point in $\mathcal{H}^{+}$ if and only if the determinant $\det(g)$ is positive. Similarly, $ \mathcal{H}^{+}$  corresponds to the subset $  \mathcal{X}_{\mathrm{std}}^{+} \subset \mathcal{X}_{\mathrm{std}} $ of  conjugates of $ h_{\mathrm{std}} $ by $ \Gb(\RR)^{+} :=  \left \{ g   \in    \Gb(\RR) \, | \,  \det(g) > 0 \right \}$.

\begin{remark}   \label{invremark} Note that the first isomorphism  in      (\ref{variousiden}) depends only on the datum $ (\Gb , \mathcal{X}_{\mathrm{std}}) $,  since it can also be given by   the  evaluation   map     $ h \mapsto h(i) $,     $ h \in \mathcal{X}_{\mathrm{std}}$.          The remaining  identifications however are strictly speaking   \emph{not}     determined by the Shimura  datum and involves   additional choices.   For instance, the map  
\begin{equation}   \label{invmap}    \mathrm{inv} : \mathcal{H}^{\pm} \to \mathcal{H}^{\pm},  \quad     \tau \mapsto -1/\tau   
\end{equation}    is a holomorphic and isometric involution of $ \mathcal{H}^{\pm} $ that preserves $ i $ and the two connected components of $  \mathcal{H}  ^{\pm} $.
If we instead use the   identification  $$ \mathcal{X}_{\mathrm{std}} \rightarrow  \mathcal{H}^{\pm}  ,   \quad  \quad   
g h_{\mathrm{std}}  g^{-1}      \mapsto   \mathrm{inv}( g \cdot i  )  ,   $$    
then the conjugation action of $ \gamma  \in \Gb(\RR) $ on $ \mathcal{X}_{\mathrm{std}} $ is identified with the (left) action of $ {}^{t} \negsmall    \gamma^{-1} $ on $ \mathcal{H}^{\pm} $ (where ${}^{t}   \negsmall  \gamma^{-1} $ acts via usual Möbius transformations). In what follows, we will only use the identifications made in (\ref{variousiden}),  but   in order to keep our discussion intrinsic to the datum $ (\Gb,  \mathcal{X}_{\mathrm{std}})   $,  we will always    distinguish between  the  elements of $ \mathcal{X}_{\mathrm{std}} $ and  those of $ \mathcal{H}  ^  {   \pm     }   $.

A related observation is that the $ \Gb(\RR) $-conjugacy class $ \mathcal{X}_{\mathrm{std}}' $ of the map   
\begin{equation}   \label{weirdShimura}     h_{\mathrm{std}}' : \delT \to  \Gb     ,   \quad   \quad    z \mapsto {}^{t}    \negsmall ( h_{\mathrm{std}}(z))^{-1}   
\end{equation}    also gives a   Shimura  datum.\footnote{We may also define $ \mathcal{X}_{\mathrm{std}}' $ as the conjugacy class of $ z \mapsto h_{\mathrm{std}}(z)^{-1}$.}     We have  an isomorphism   
\begin{equation}  \label{weirdshimuraiso} \phi : (\Gb, \mathcal{X}_{\mathrm{std}})  \to  (\Gb, \mathcal{X}_{\mathrm{std}}')   
\end{equation} 
of Shimura data   induced by the map $    \Gb  \to  \Gb $, $ g  \mapsto  {}^{t}  \negsmall    g^{-1} $,   which induces the   holomorphic and   isometric  identification $  \mathcal{X}_{\mathrm{std}}   \xrightarrow{\sim}     \mathcal{X}_{\mathrm{std}}' $,  $     h  \mapsto  (z \mapsto  {}^{t}   \negsmall (h(z))^{-1}) $.   
This identification fits   into the commutative diagram 
\begin{equation}   \label{Xstd'compardiag}     
\begin{tikzcd}[sep = large]  \mathcal{X}_{\mathrm{std}}   \arrow[r, "\phi"]   \arrow[d]     &       \mathcal{X}_{\mathrm{std}}'    \arrow[d]         \\
\mathcal{H}^{\pm}       \arrow[r,  "{\mathrm{inv}}"]     &  \mathcal{H}^{\pm}  
\end{tikzcd}  
\end{equation}    
where the left  vertical map is the one in 
(\ref{variousiden}) and the right vertical map is (the holomorphic and isometric  isomorphism)    $$  \mathcal{X}_{\mathrm{std}}'  \ni  g   h_{\mathrm{std}}' g^{-1}    \mapsto   g \cdot i   \in   \mathcal{H}^{\pm}   ,   $$    where $ g \cdot i $     again    denotes the usual Möbius transformation.  The data (\ref{SDdatum})   and (\ref{weirdShimura})   give rise to isomorphic theories, and one may translate between them using the isomorphism (\ref{weirdshimuraiso}). However, we will carry out this translation explicitly at various junctures, since the datum $(\Gb, \mathcal{X}_{\mathrm{std}}')$ is used in parts of the literature (e.g., \cite[\S 5]{LSZ}, \cite{Carayol}), and this can be a potential source of confusion when citing results from sources that adopt different conventions.  We also refer the reader to \cite[\S 3.3]{CorVatsal2}, which      discusses the relation between these  two  data  at length. The reader should keep in mind however that the identification $ \mathcal{X}_{\mathrm{std}} \simeq \mathcal{X}_{\mathrm{std}}'$   used in  \emph{loc.~cit.} is \emph{anti-holomorphic} and in particular, not induced by the morphism $ \phi   $.    
\end{remark} 

\begin{remark} For general Shimura data $(\Gb', \mathcal{X})$, the conjugacy class $ \mathcal{X} $ can be endowed with a complex manifold  structure in such a way that makes each connected component of $ \mathcal{X} $ a \emph{Hermitian symmetric domain}. See   \cite[Proposition 5.9]{Milne}. 
\end{remark} 

To define certain algebraic points on modular curves, we need to introduce another Shimura datum. Let $E \subset \CC$ denote an imaginary quadratic field, and set $\mathbf{H} = \mathrm{Res}_{E/\QQ} \GG_{m}$. Fix an abstract isomorphism \begin{equation}  \label{varphi} \varphi : E \to \QQ^{2} 
\end{equation}  of $\QQ$-vector spaces, or equivalently, a choice of an ordered basis $(\omega_{1}, \omega_{2})    \in E \times E$ over $\QQ$.\footnote{Here,    $  \varphi(\omega_{1}) = (1, 0)  $ and $ \varphi(\omega_{2}) =  (0, 1) $.}  Given $  \omega  \in E$, multiplication by $\omega$ induces a $\QQ$-algebra endomorphism of $E$. In other words, the choice of $\varphi$ induces an inclusion $\iota : E \hookrightarrow \mathrm{Mat}_{2 \times 2}(\QQ)$ of $\QQ$-algebras and hence an embedding of algebraic groups  
\begin{equation} \label{iotaembedding} 
   \iota : \mathbf{H} \hookrightarrow \mathbf{G} 
\end{equation}  
over $\QQ$, whose $\Gb(\QQ)$-conjugacy class is independent of $\varphi$. Since $ E \subset \CC $, we have a natural identification     $ \CC  \simeq    E \otimes_{\QQ}  \RR $ of $\RR$-algebras, which induces  an    isomorphism $h_{0} : \delT \xrightarrow{\sim} \mathbf{H}_{\RR}$. The pair $(\mathbf{H}, \{h_{0}\})$ is then obviously a Shimura datum. Moreover, the mapping  
\begin{equation} \label{embedding} 
   \iota : (\Hb, \{h_{0}\}) \hookrightarrow (\mathbf{G}, \mathcal{X}_{\mathrm{std}})  
\end{equation}  
constitutes an (injective) morphism of Shimura data. This amounts to the claim  that the composition $\iota_{\RR} \circ h_{0}$ belongs to $\mathcal{X}_{\mathrm{std}}$, i.e.,  
\begin{equation}  \label{g0}    \iota_{\RR} \circ h_{0} = g_{0} \, \negsmall  h_{\mathrm{std}} \,  \negsmall    g_{0}^{-1}   
\end{equation}    for some $g_{0} \in \Gb(\RR)$. To check this, note that for each $z \in \CC$, multiplication by $z$ on $\CC$ is $\RR$-linear and $h_{\mathrm{std}}(z)$ is just the matrix of this transformation with respect to the ordered $ \mathbb{R}$-basis $(1,-i)$.\footnote{We can also use $ (i, 1 )$  as a  basis here here but the moduli description  we  give  later  on    is easier to state if the ordered basis associated to $ \tau \in \mathcal{H}^{\pm} $ is $ (1 , - \tau ) $. See  Remark  \ref{thebasisremark}.}     Similarly, $\iota_{\RR}(h_{0}(z))$ is the matrix of multiplication by $z$ with respect to the ordered $\RR$-basis $(1, \omega_{2}/\omega_{1} )$ of $\CC \simeq E \otimes_{\QQ} \RR$. The matrix $g_{0}$ can therefore be taken to be the change of coordinates matrix from $(1,-i)$ to $(1, \omega_{2}/\omega_{1} )$. One easily  checks  that  $ g_{0} \cdot i = - \omega_{2} /\omega_{1}  $, so  that 
\begin{equation}   \label{tau0}   \tau_{0} :=    -\omega_{2}/\omega_{1}  \in  \mathcal{H}^{\pm}   
\end{equation}    
is the point  corresponding to $ h_{0} \in  \mathcal{X}_{\mathrm{std}} $    under  (\ref{variousiden}).      

\begin{remark} \label{orient}  
Note that the point $h_{0}$ does not necessarily map to $h_{\mathrm{std}}$, since the choice $(\omega_{1}, \omega_{2})$ is arbitrary. In fact, $\iota_{\RR} \circ  h_{\mathrm{0}}$ belongs to the $\Gb(\QQ)$-conjugacy class of $h_{\mathrm{std}}$ if and only if $E = \QQ(i)$. It is also clear that $\iota _ { \RR } \circ h_{0}$ lies in $\mathcal{X}^{+}_{\mathrm{std}}$ if and only if $(1 , \omega_{2}/\omega_{1})$ is positively oriented with respect to $(1,-i)$. 
\end{remark}

From now on, we view $ \Hb $ as a subgroup of $ \Gb $  via $ \iota $, so that $ \Hb(R) \subset \Gb(R) $ for any $ \QQ$-algebra $ R $, and we regard $ h_{0} $ as an element of $ \mathcal{X}_{\mathrm{std}} $. 
For $ h \in \mathcal{X}_{\mathrm{std}} $, the \emph{complex conjugate} of $ h $ is the map $$ \bar{h} :  \delT \to  \Gb_{\RR}   ,   \quad   z \mapsto h(\bar{z})  .   $$    
If  $ \delta $ denotes $  \mathrm{diag}(1, -1)  \in \Gb(\RR) $ and $ h = g h_{\mathrm{std}} g^{-1} \in \mathcal{X}_{\mathrm{std}}$,  then $ \bar{h} = g \delta h_{\mathrm{std}} \delta^{-1} g^{-1}   $  also     lies in $ \mathcal{X}_{\mathrm{std}} $.   Under  the   identification made in    (\ref{variousiden}),  the    operation $ h \mapsto \bar{h} $ corresponds  to     complex conjugation on $ \mathcal{H}^{\pm}  $.     

\begin{lemma}   \label{uniqueCMpoint}    The only points of $\mathcal{X}_{\mathrm{std}}$ whose stabilizer in $\Gb(\QQ)$ is $\iota(E^{\times})$ are $h_{0}$ and $\bar{h}_{0}$.
\end{lemma} 
\begin{proof} Let us first show  that $ h_{\mathrm{std}}  $, $ \bar{h}_{\mathrm{std}} $ are the only two points in $ \mathcal{X}_{\mathrm{std}}  $ whose stabilizer in $ \Gb(\RR) $  is $ K_{\infty} $.  So suppose that $ K_{\infty} $ is the stabilizer of $   g h_{\mathrm{std}} g^{-1} $ for some $ g \in \Gb(\RR) $.    Then $K_{\infty}  =   g K_{\infty}   g^{-1}  $ and in particular,  $  g h_{\mathrm{std}}(i) g ^{-1}   \in  K_{\infty}.  $    From this, one can see by an explicit   matrix    calculation that $ g  \in  K_{\infty}   \cup   \delta  K_{\infty}   $.\footnote{Alternatively, note that since $ K_{\infty} $ is a maximal torus (or Cartan subgroup) in $ \Gb(\RR) $, the quotient of the normalizer $ N_{\Gb(\RR)}(K_{\infty}) $ by $ K_{\infty} $ is the Weyl group $W$ of $ \Gb(\RR)$ and $ \delta \in \Gb(\RR) $ is a representative for the non-trivial element in $W$.}      

Now let  $ g_{0} $ be as in  (\ref{g0}). 
Since $ \delta $ normalizes $ K_{\infty}   $, the conjugate $  \Hb  (  \RR  )   =    g_{0} K_{\infty} g_{0}^{-1} $ is the stabilizer in $ \Gb(\RR) $ for both $ h_{0}$ and $ \bar{h}_{0}   =  g_{0} \delta h_{\mathrm{std}} \delta^{-1} g_{0}^{-1} $.   
So the stabilizer in $ \Gb(\QQ) $ for each of them 
is  
$$ \Gb(\QQ) \cap  \Hb(\RR) = \Hb (\QQ) = \iota(E^{\times})  .  $$
If $ P   \in   \mathcal{X}_{\mathrm{std}}  $  is any other point with this property, then since $ E ^{\times} $ is dense $ \Hb(\RR) \simeq \CC^{\times} $, the stabilizer for $ P $ in  $ \Gb(\RR) $ would also be $ \Hb(\RR) $. The result  of the previous  paragraph easily implies that $   P   \in    \{  h_{0} ,  \bar{h}_{0}   \}   $. 
\end{proof} 

\subsection{Reflex fields}    
\label{reflexsec}    

Each Shimura datum $(\Gb', \mathcal{X})$ has an associated number field given as a subfield of $\CC$ that is called the \emph{reflex field} \cite[Definition 12.2]{Milne}, which is   defined as follows. The Deligne torus $\delT$ splits over $\CC$, i.e., $\delT_{\CC} \simeq \GG_{m,\CC} \times \GG_{m,\CC}$, and this isomorphism is uniquely determined by requiring that the inclusion
\[
\CC^{\times} = \delT(\RR) \hookrightarrow \delT(\CC) \simeq \CC^{\times} \times \CC^{\times}
\]
corresponds to
\[
z \longmapsto (z, \bar{z}).
\]
The reflex field of  $(\Gb', \mathcal{X})$ is defined to be the field of definition of the $\Gb'(\CC)$-conjugacy class of the   \emph{Hodge cocharacter}     
$$
\mu_{h} : \GG_{m,\CC} \to \Gb_{\CC}   ' 
$$ 
attached  to   any  $h \in \mathcal{X}$  by        restricting $h_{\CC} : \GG_{m} \times \GG_{m} \to \Gb_{\CC}'$ to the first component.  In practice, this    means that in the matrices $h(z)$ for $z \in \CC^{\times}$, one formally replaces $\bar{z}$ with $1$ and checks the smallest field over which an element in its conjugacy class can be defined. The reflex field is independent of the choice of $h$, as the $\Gb'(\CC)$-conjugacy class of $\mu_{h}$, denoted $\mu_{\mathcal{X}}$, is independent of $h \in \mathcal{X} $. 
 
Let us determine   these fields for the two Shimura data introduced in \S \ref{SDdatumsec}.    The cocharacter   
\begin{equation}   \label{muh0} 
\mu_{h_{0}} : \GG_{m} \to \Hb_{\CC} \simeq \GG_{m} \times \GG_{m}  ,   \quad \quad    z \mapsto (z,1)   
\end{equation}    associated with $h_{0}$ is defined over any field over which $\Hb$ splits and is clearly not defined over $ \QQ $, since $ \Gal(E/ \QQ) $ acts non-trivially on $ \mathbf{H}_{E}$.  So the  reflex field of $(\Hb, \{h_{0}\})$ is $ E $. For  $(\Gb, \mathcal{X}_{\mathrm{std}})$,  the reflex  field      is    $\QQ$. Indeed, the cocharacter 
$$
\mu_{h_{\mathrm{std}}} : \GG_{m,\CC} \to \Gb(\CC), \quad 
z \mapsto 
\begin{pmatrix} 
  \tfrac{z+1}{2} & \tfrac{z-1}{2i} \\[0.2em]     
  \tfrac{1-z}{2i} & \tfrac{z+1}{2} 
\end{pmatrix}
$$
when conjugated by 
$
\left( \begin{smallmatrix} i & 1 \\ -i & 1 \end{smallmatrix} \right),$ 
becomes 
\begin{equation} \label{GL2cocharacterconj}  
z \mapsto \left( \begin{smallmatrix} z & 0   \\[0.1em]     0 & 1 \end{smallmatrix} \right), 
\end{equation}  
which is itself defined over $\QQ$, and therefore so is the $\Gb(\CC)$-conjugacy class $\mu_{\mathcal{X}_{\mathrm{std}}}$ of $\mu_{h_{\mathrm{std}}}$. We denote the cocharacter (\ref{GL2cocharacterconj}) by $\mu_{\mathrm{std}}$.

\begin{remark}   \label{weirdhodgecocharacter}    
The  Hodge  cocharacter $ \mu_{\mathrm{std}}$ (or rather, its inverse)    is also used to define the Hecke polynomial alluded to in the introduction;     see  \S \ref{sanitycheck}. We  note for  later   that the $ \Gb(\CC)$-conjugacy class of $ \mu_{h_{\mathrm{std}}'}$ associated with  the   data     (\ref{weirdShimura}) equals the  conjugacy  class  of   $  \mu_{\mathrm{std}}^{-1} $.    
\end{remark}    
 
\subsection{Canonical models}   
\label{canonicalmodsec}       
Let $\Ab$, $\Ab_E$ denote the rings of adeles of $\QQ$ and $E$, respectively, and let $\Ab_f$, $\Ab_{E,f}$ denote their finite parts. For any algebraic  group  $ \Gb' $ over $ \QQ $, the adelic group   $ \Gb'(\Ab_{f}) $ is endowed with a natural topology inherited from the topology of $ \Ab_{f} $ that makes $ \Gb'(\Ab_{f}) $ a  \emph{locally profinite}   group   \cite{Weiltopology},           \cite{conradtopology}.\footnote{This resembles the process of topologizing $ \Ab_{f}^{\times}  =  \mathbb{G}_{m}(\Ab_{f}) $, whose topology is \emph{not} the subspace topology inherited from $ \Ab_{f}$.}  That is, $ \Gb   '    (\Ab_{f}) $  has a basis at identity given by subgroups    that are both compact (hence closed) and open  in $  \Gb'(\Ab_{f})    $. If $ K \subset \Gb'(\Ab_{f}) $ is a compact open  subgroup, then for all but  finitely  many  primes  $ \ell $, one can write $$ K = K_{\ell} K^{\ell} $$ where $ K^{\ell} $ is a subgroup of $ \Gb'(\Ab_{f}/\QQ_{\ell})$ and $K_{\ell}$ is the group of $ \ZZ_{\ell}$-points of a smooth reductive group scheme over $ \ZZ_{\ell}$ whose generic fiber is $ \Gb' $. If $ \ell $ is such a prime, we   say that $ K $ is \emph{unramified}   or   \emph{hyperspecial}     at   $ \ell $. Since $ K $ is open in $ \Gb'(\Ab_{f}) $,  the  quotient $ \Gb'(\Ab_{f})/K $ is discrete under the quotient topology inherited from $ \Gb'(\Ab_{f}) $.

With these general considerations in mind, let us denote by $U $  a compact open subgroup of $ \Hb(\Ab_{f}) = \Ab_{E}^{\times} $. Then the double coset     
$$ 
\mathcal{T}_{U}(\CC) := \Hb(\QQ) \backslash \Hb(\Ab_{f}) / U  = \Ab_{E, f}^{\times}/ (E^{\times} U) 
$$
is a finite (discrete) set that resembles the quotients    one sees in the adelic   formulation  of     class  field  theory.  Following Deligne,   we    can identify $ \mathcal{T}_{U}(\CC) $ with the $\CC$-points of an étale scheme over $\Spec E$ as follows. Let  
$ 
\mu_{h_{0}} : \GG_{m,E} \to \Hb_{E} 
$ 
be the cocharacter (\ref{muh0}) attached to $ h_{0}$.  The \emph{reciprocity law} for the Shimura datum $(\Hb,\{h_{0}\})$ is the morphism  
\begin{equation} \label{reciprocitylaw}        
r(\Hb,h_{0}) : \Res_{E/\QQ} \GG_{m} 
   \xrightarrow{\Res} 
   \Res_{E/\QQ}(\Hb_{E}) 
   \xrightarrow{\mathrm{Tr}} \Hb , 
\end{equation}    
where $ \mathrm{Res}  = \mathrm{Res}_{E/\QQ}(\mu_{h_{0}}) $ denotes restriction of scalars applied to $ \mu_{h_{0}} $ and $ \mathrm{Tr} = \mathrm{Tr}_{E/\QQ}$ is induced by the natural trace map $E \to \QQ$. Unwinding definitions,\footnote{We need to translate what the   trace    map looks like when we identify $\Res_{E/\QQ} \Hb_{E}$ with $\Res_{E/\QQ}(\GG_{m} \times \GG_{m})$, since the description of the map $ \mu_{h_{0}} :  \GG_{m} \to \Hb_{E} $ in  (\ref{muh0}) is given after identifying $ \Hb_{E} $ with $ \GG_{m, E } \times \GG_{m, E } $.}    this map is easily computed to be the identity map.   The Galois action of $\sigma \in \Gal(E^{\mathrm{ab}}/E)$ on $\mathcal{T}_{U}(\CC)$ is defined to be translation by $a_{f} \in \Ab_{E,f}^{\times}$ for any
$$
a = (a_{\infty}, a_{f}) \in \Ab_{E}^{\times}
$$
such that $a \mapsto \sigma$ under the \emph{Artin homomorphism}
\begin{equation} \label{CFTforE}
\Art_{E} : E^{\times} \backslash \Ab_{E}^{\times} \to \Gal(E^{\mathrm{ab}}/E),
\end{equation}
normalized in Deligne's convention, meaning that uniformizers are mapped to geometric Frobenii.  In other words,  the action of $\sigma   =    \mathrm{Art}_{E}(a)    $ on $\mathcal{T}_{U}(\CC)$ is via 
$$ 
[h_{f}] \mapsto [a_{f}h_{f}] \in \mathcal{T}_{U}(\CC). 
$$ 
This description of Galois action on $\mathcal{T}_{U}(\CC)$ determines an $E$-scheme that we denote by $\mathcal{T}_{U}$. In the language of \cite[Definition 3.13]{DeligneTS}, $\mathcal{T}_{U}$ constitutes the \emph{canonical model} for $\mathcal{T}_{U}(\CC)$.   

\begin{remark} \label{ArtinmapEremark}            
Since $E$ is imaginary, the infinite ideles $\CC^{\times} \hookrightarrow \Ab_{E}^{\times}$ are all in the kernel of the Artin map,  and  we  can   in   fact   view $\Art_{E}$ as an   isomorphism   
\begin{equation}   \label{CFTforEnoC}       
E^{\times} \backslash \Ab_{E,f}^{\times} = \Hb(\QQ) \backslash \Hb(\Ab_{f}) \xrightarrow{\;\sim\;} \Gal(E^{\mathrm{ab}}/E).   
\end{equation}    
See \cite[\S 2.1]{Lars} for more details.
\end{remark}
   
Let us now describe the corresponding objects for $(\Gb,\mathcal{X}_{\mathrm{std}})$. Let $K \subset \Gb(\Ab_{f})$ be a compact open subgroup, which we fix throughout the rest of this article.  We let $\Gb(\QQ)$ act diagonally on the left of $\mathcal{X}_{\mathrm{std}} \times \Gb(\Ab_{f})$ where $\Gb(\QQ)$ acts on $\mathcal{X}_{\mathrm{std}}$ via conjugation and on $\Gb(\Ab_{f})$ by left multiplication. We also let $K$ act on the right of $\mathcal{X}_{\mathrm{std}} \times \Gb(\Ab_{f})$ via right multiplication on the $\Gb(\Ab_{f})$-component and via  trivial action on $\mathcal{X}_{\mathrm{std}}$. Then the double coset  space   
\begin{equation}   \label{Shimuravar}  
\mathcal{S}_{K}(\CC) := \Gb(\QQ) \backslash (\mathcal{X}_{\mathrm{std}} \times \Gb(\Ab_{f})) / K 
\end{equation}    
is a    finite    disjoint union of (left)  quotients   of      $\mathcal{X}^{+}_{\mathrm{std}}   \simeq   \mathcal{H}^{\pm}        $ by certain subgroups of $\Gb(\QQ)^{+}  :  =  \Gb(\QQ) \cap  \Gb(\RR)^{+}$  \cite[Lemma 5.13]{Milne}. More precisely, we have an identification 
\begin{equation}  \label{shimuravardecom1}   
\begin{split}   \sqcup_{g}  \Gamma_{g} \backslash \mathcal{X}^{+}_{\mathrm{std}}    &    \MapsTo \mathcal{S}(K) (\CC)   \\
\Gamma_{g} x    &     \mapsto [x,   g]_{K}   
\end{split}
\end{equation}    
where $ g  \in \Gb(\Ab_{f}) $ runs over a set of representatives of the finite set $ \Gb(\QQ)^{+} \backslash \Gb(\Ab_{f})/ K $ and $ \Gamma_{g} $ denotes the twisted intersection $  \Gb(\QQ)^{+} \cap gK g^{-1} $. 
Note that $$ \det(\Gamma_{g}) \subseteq \QQ^{\times}_{\geq 0} \cap \widehat{\ZZ}^{\times} = \left \{ 1 \right \} . $$ Therefore, $ \Gamma_{g} = \SL_{2}(\QQ) \cap gKg^{-1} $ is a  \emph{congruence subgroup} of $ \SL_{2}(\QQ) $ \cite[Proposition 4.1]{Milne},   and in particular,  
\emph{Fuchsian of first kind}. By \cite[\S1.7]{Miyake} or \cite[\S1.3]{Shimurabook}, quotients of the upper half-plane by such groups can be naturally identified with finite  complements  of compact Riemann surfaces, which, by the Riemann existence theorem, are automatically smooth projective varieties. Thus $\mathcal{S}_{K}(\CC)$ is the set of $\CC$-points of a  (possibly disconnected)   smooth   algebraic curve $\mathcal{S}_{K, \CC} $. A consequence of the theory of moduli  of elliptic curves is that $ \mathcal{S}_{K, \CC} $ admits a  specific  model $ \mathcal{S}_{K}$ over the reflex field $\QQ$,  referred to as its \emph{canonical model} \cite[Proposition 4.20]{DeligneTS}. It is “canonical” in the sense that the Galois action on certain algebraic points on $\mathcal{S}_{K}(\CC)$ arising via the embeddings (\ref{embedding}) for all imaginary quadratic fields is dictated by the reciprocity law (\ref{reciprocitylaw}). 
See    
Note  \ref{DeligneCM} for more details.

\begin{remark}   \label{affineremark}       Since  congruence subgroups of $ \SL_{2}(\QQ) $ always contain parabolic (cuspidal) elements of the form $ \left ( \begin{smallmatrix} 1 & k \\  &   1 \end{smallmatrix} \right ) $ for $ k $ large enough, the  surfaces  $ \mathcal{S}_{K}(\CC) $ are  themselves never compact. Thus the algebraic curve $ \mathcal{S}_{K, \CC} $ is \emph{affine} \cite[\href{https://stacks.math.columbia.edu/tag/0A24}{Tag 0A24},   \href{https://stacks.math.columbia.edu/tag/0A28}{Tag 0A28}]{stacks-project},    and therefore so is its canonical model  $ \mathcal{S}_{K} $ \cite[p.~302]{Poonen}.    
\end{remark}

\begin{remark} Deligne’s convention in \cite{DeligneTS} for the double coset spaces $\mathcal{S}_{K}(\CC)$ is opposite to that  \cite{Milne} and \cite{DeligneVar}.    In Deligne’s original setup for $ \Gb = \GL_{2, \QQ}$, the group $\Gb(\QQ)$ would act on the right of $\mathcal{H}^{\pm}$, as in \cite[\S 2.1.3]{Beilinson}, and the compact open subgroup $K$ acts on the left of $\Gb(\Ab_{f})$. The conventions of \cite{DeligneVar}, which are also adopted in the present paper, have become the standard choice in much of the recent literature surrounding the Langlands program.\footnote{Though, see Remark \ref{errorremark}.}
\end{remark} 

In what follows, we will refer to   compact open subgroups of $ \Gb(\Ab_{f}) $ as \emph{levels}  and the canonical model $\mathcal{S}_{K}$ as the \emph{modular curve of level $K$}. If $ F $ is an extension of $ \QQ $ contained in $ \mathbb{C}$, we will write $$ \mathcal{S}_{K , F } = \mathcal{S}_{K} \times_{\Spec \QQ }  \Spec F  $$  for the base change of $ \mathcal{S}_{K} $ to $ F $.   We will denote points in the double coset $ \mathcal{S}_{K}(\CC) $  by $ [x , g]_{K} $ where $ x \in \mathcal{X}_{\mathrm{std}} $ and $ g \in \Gb(\Ab_{f}) $.  For any two levels $ L $, $ K $ with $ L \subset K $, the map 
\begin{equation}  \label{degeneracy} 
\begin{split}\mathrm{pr}_{L, K }(\CC) : \mathcal{S}_{L}(\CC)      & \to \mathcal{S}_{K}(\CC)  \\   
[x, g]_{L}     &     \mapsto [x, g]_{K} 
\end{split} 
\end{equation} 
extends uniquely to a finite holomorphic surjection of compactified Riemann surfaces, and  therefore arises from a $\CC$-morphism $   \pr_{L, K, \CC}  :    \mathcal{S}_{L, \CC} \to \mathcal{S}_{K, \CC} $.     The  theory of moduli of elliptic curves   also    implies that this  morphism   descends to   a  finite flat morphism $ \pr_{L, K} :  \mathcal{S}_{L} \to \mathcal{S}_{K} $ of canonical models.   We refer to it as the \emph{degeneracy map} induced by the inclusion $ L  \hookrightarrow   K    $.     Moreover for any $ g \in \Gb(\Ab_{f}) $, the holomorphic isomorphism   
\begin{equation}   \label{twisting}
\begin{split}    [g]_{K} (\CC) :   \mathcal{S}_{K}(\CC)   &  \to  \mathcal{S}_{g^{-1} K g} (\CC)   \\ 
[x , g_{1}] _{K}     &       \mapsto [ x, g_{1} g ] _ { g^{-1} K  g  } 
\end{split}     
\end{equation}    
also descends to an  isomorphism $ [g]_{K}     : \mathcal{S}_{K} \to \mathcal{S}_{g^{-1}Kg} $ which we refer to  as   the    \emph{twisting isomorphism} induced by $ g $ on level $ K   $.     If $ g $ normalizes $ K $, this is an automorphism of $ \mathcal{S}_{K} $.

\begin{remark}  We observe that $\mathcal{T}_{U}(\CC)$ can also be written as  
$$ 
\Hb(\QQ) \backslash (\{h_{0}\} \times \Hb(\Ab_{f})) / U,  $$ 
where the actions of $\Hb(\QQ)$ and $U$ on $\{h_{0}\} \times \Hb(\Ab_{f})$ are analogous to those defined for $ \Gb $.  Both $\mathcal{T}_{U}$ and $\mathcal{S}_{K}$ are examples of \emph{Shimura varieties} associated with their   respective  Shimura data. 
\end{remark}

\begin{remark}   \label{weirdshimuravarcompar}    For a level $ L \subset \Gb(\Ab_{f}) $, let us denote by $ \mathcal{S}_{L}'$ the canonical model associated with  the  alternative datum  (\ref{weirdShimura}), where $ \mathcal{S}_{L}'(\CC) =  \Gb(\QQ) \backslash ( \mathcal{X}_{\mathrm{std}}' \times \Gb(\Ab_{f}))/L $ 
and  the double coset actions are analogous.  Then the isomorphism (\ref{weirdshimuraiso}) induces an isomorphism \begin{equation}
\label{weirdshimuravariso}      \phi_{K}(\CC) : \mathcal{S}_{K}(\CC)   \to \mathcal{S}_{{}^{t}  \negsmall K}'(\CC)   ,  \quad  \quad   
[x, g]_{K}         \mapsto   [ \phi(x),  {}^{t}g^{-1}]_{{}^{t} \negsmall K}  
\end{equation}    
of Riemann surfaces.  The theory of canoncial model stipulates that $ \phi_{K}(\CC) $ arises from a $ \QQ $-isomorphism    $$ \phi_{K}: \mathcal{S}_{K} \to \mathcal{S}_{\, {}^{t}   \negsmall     K}'  $$   of   canonical  models, and that these   isomorphisms collectively commute with the corresponding degeneracy maps and twisting isomorphisms on the two sides. 

On the other hand,  we can also make  make the  identification $$   \phi' : \mathcal{X}_{\mathrm{std}}  \xrightarrow{\sim} \mathcal{X}_{\mathrm{std}}', \quad \quad  g h_{\mathrm{std}} g^{-1} \mapsto  g h_{\mathrm{std}}' g^{-1} .  $$ 
This is holomorphic and isometric as it arises via the identifications  $ \mathcal{X}_{\mathrm{std}} \to \mathcal{H}^{\pm} \leftarrow  \mathcal{X}_{\mathrm{std}}' $ used in    (\ref{Xstd'compardiag}).    This implies that the map 
\begin{equation}   \label{veryweirdiso}   \phi_{K}'(\CC)    :    \mathcal{S}_{K}(\CC)  \xrightarrow{\sim}   \mathcal{S}_{K}'(\CC)     \quad \quad  [x,    g]_{K} \mapsto [\phi'(x), g]_{K}
\end{equation} 
is   also     an  isomorphism  of  Riemann  surfaces.   However, this   isomorphism  \emph{does not} descend to a morphism of    the   underlying       canonical models.     See  Remarks     \ref{weirdcompactremark} and \ref{weirdCMdivisorremark}.    
\end{remark}

\begin{remark}   For a general Shimura data $(\mathbf{G}'  , \mathcal{X})$, the corresponding double coset spaces are  unions of quotients of Hermitian symmetric domains by arithmetic subgroups of $\mathbf{G}'(\QQ)$. By the theorem of Baily–Borel \cite{BailyBorel}, such quotients are quasi-projective algebraic varieties over $\CC$. In the 1960s, Shimura showed that a large class of these varieties admit models over explicit number fields, which he referred to as canonical models. Deligne later reformulated Shimura’s results by giving an axiomatic description of Shimura's canonical models  in terms of the axoims (SV1)-(SV3), and proved the existence of such models in  great generality \cite{DeligneTS,DeligneVar}. The general existence of canonical models for all Shimura data was subsequently established by  Borovoi-Milne-Shih  \cite{MilneBirk}.
\end{remark}

\subsection{Pullbacks   of divisors}   \label{pullsec}      We will need the following two results in \S \ref{Heckecorrsec}, for which we are unaware of a suitable reference. 

\begin{lemma}    \label{faithful}       Suppose $ L, K $ are two levels of $ \Gb(\Ab_{f}) $ such that $L \trianglelefteq  K $  and    $ L \cap \left \{ -1 \right \}  = K \cap \left \{ - 1  \right  \}    $. Then the right action of $ K / L $ on $ \mathcal{S}_{L}(\CC) $ by twisting isomorphisms  is faithful.  In particular, the  degree of $ \pr_{L, K} $ is $ [K : L ]   $.        
\end{lemma}

\begin{proof}   Suppose  $ k \in K $ fixes all points in $ \mathcal{S}_{L}(\CC) $.  Then for each $  g \in \Gb(\Ab_{f} )   $,  there exist $ \gamma = \gamma_{g} \in \Gb(\QQ) $ and $ l = l_{g}  \in L $ such that $$  \gamma h _ { \mathrm{std} } \gamma^{-1}  =   h _{\mathrm{std}}  \quad  \text {and }  \quad   g k  =  \gamma g l . $$
Thus $ \gamma \in \mathrm{Stab}_{\Gb(\QQ)}(h_{\mathrm{std}}) $ from the first equality and $ \gamma  = g k l^{-1} g ^{-1} \in g K g^{-1} $ from the second, which means that $ \gamma $ lies in the intersection $$ \Gamma := \mathrm{Stab}_{\Gb(\QQ)}(h_{\mathrm{std}}) \cap gKg^{-1} =  \Gb(\QQ) \cap K_{\infty} \cap g K g^{-1}    . $$ Since $ \Gb(\QQ) $ is a discrete subgroup of $ \Gb(\Ab)$, $ \Gamma $ is a discrete subgroup of $   K_{\infty}  =  \mathbf{Z}(\RR)  \mathrm{SO}_{2}(\RR)  $.  As $ \mathrm{SO}_{2}(\RR) $ is compact and the subgroup $ \langle \gamma \rangle  $ is discrete in  $ K_{\infty} $,  it must be that $ \gamma^{n} \in \mathbf{Z}(\QQ)  =  \mathbf{Z}(\RR) \cap  \Gb(\QQ)   $      for some positive integer $ n $.  Since  $  \gamma  \in   K_{\infty}    \cap \Gb(\QQ) $, it equals the matrix (in the basis $ (1, -i ) $) of an endomorphism  in $ \mathrm{End}_{\QQ}(\QQ(i)) $ given by  multiplication by some $   z   = z _ { \gamma }    \in \QQ(i)^{\times}  $. The  condition    $ \gamma^{n} \in \mathbf{Z}(\QQ) $ is   then    equivalent to  $ z^{n} \in \mathbb{Q}^{\times}  $. Write  $$ z = r \zeta  $$  where $ r = |z| $  and $ \zeta \in \CC^{\times} $ satisfies $ |\zeta| = 1 $. Then $ r^{2} \in \QQ^{\times} $ and  $ \zeta  ^ { 2  }        \in   \mathbb{Q}(i)     $ is  a root of unity. Since the only roots of unity in $ \mathbb{Q}(i) $ are $ \left \{\pm 1, \pm   i    \right \} $, we see that $ z^{2} \in \QQ^{\times} \sqcup \QQ^{\times} i $. As  $ z \in \QQ(i)^{\times} $, it is not hard to see that $ z $ is a $ \QQ^{\times} $-multiple of an element in $  \left \{ 1, i ,  1+i , 1 - i \right \}  $. Thus $$ \gamma \in \mathbf{Z}(\QQ) \sqcup   \mathbf{Z}(\QQ) J     \sqcup \mathbf{Z}(\QQ) J_{1}  \sqcup  \mathbf{Z}(\QQ) \, {}^{t}  J_{1} $$
where $$   J :  = \iota(i) =  \left (  \begin{smallmatrix}  & 1 \\  - 1 &  \end{smallmatrix}  \right ) \, \text{ and } \, J_{1}  :  =  \iota(1+i) =  \left (  \begin{smallmatrix}  1 & 1 \\  -  1 &   1 \end{smallmatrix}  \right ). $$       If $ \gamma = \gamma_{g} $ is central for some choice of $ g $, the equality $ gk  = \gamma g l $ implies that $ k = \gamma l  $. In this case,  $$ \gamma \in  K \cap \mathbf{Z}(\QQ)  =  K \cap \left \{ \pm 1 \right \} = L \cap \left \{ \pm 1 \right \} =  \mathbf{Z}(\QQ) \cap  L , $$  which forces $ k $ to be in $ L $ and we are done. So suppose that $ \gamma = \gamma_{g} $ is not central for any  $  g  \in \Gb(\Ab_{f}) $.  Choose a positive integer $ N $ such that for $\ell > N  $, both $ K $ and $ L $ are unramified at $ \ell $. The equality $ gk = \gamma g l  $  implies   that    $$ g^{-1}_{\ell} \gamma   g_{\ell} \in \GL_{2}(\ZZ_{\ell})  $$ 
for all $ \ell > N $ and $ g \in \Gb(\Ab_{f} )$,  where $ g_{\ell} $ denotes the component of $ g $ at $ \ell $. But if we take any $ g $ such that $ g_{\ell} =   \left (  \begin{smallmatrix}  \ell &  \\  &  1   \end{smallmatrix}  \right ) $ for some prime $ \ell  >    N $, we have $$ g_{\ell} ^{-1} J  g_{\ell}   =   \left (    \begin{smallmatrix}  & \ell^{-1} \\ -\ell &  \end{smallmatrix}  \right )      ,   \quad g_{\ell}^{-1} J_{1} g_{\ell}    =     \left (    \begin{smallmatrix}  1  &  \ell^{-1}   \\ -\ell &   1  \end{smallmatrix}  \right )      , \quad  g_{\ell}^{-1}   ( {}^{t}  J_{1}   )   g_{\ell}   =      \left (    \begin{smallmatrix}  1  &  - \ell^{-1}    \\ -\ell & 1  \end{smallmatrix}  \right )     $$
and none of these belong to $ \GL_{2}(\ZZ_{\ell}) \cdot \mathbf{Z}(\QQ) $.   
\end{proof}    

The next result is an  adelic version of \cite[Proposition 1.37]{Shimurabook}.

\begin{lemma}   \label{pullback}      Suppose $ L , K $ are two levels of $ \Gb(\Ab_{f}) $ with $ L \subset K $ such that $ L \cap \left \{ -1 \right \}  = K \cap  \left \{ -  1   \right \}    $. Then the pullback of $ [x, g]_{K}   \in   \mathcal{S}_{K}(\CC)    $  under $ \pr_{L, K }$ as a  divisor equals  $ \sum_{ K/ L} [x , g\gamma]_{L} $
\end{lemma} 
\begin{proof}  First assume that $ L \trianglelefteq K $. Then by  Lemma  \ref{faithful},  we have a faithful right action 
of $ \Gamma := K  / L $ on $ \mathcal{S}_{L}(\CC) $ by holomorphic automorphisms.   Let $ p \in \mathcal{S}_{L}(\CC) $ be any point. By  \cite[Theorem 3.4]{Rick} (applied to the component $ \mathcal{C}_{p}$ of $ \mathcal{S}_{L}(\CC)$ containing $ p $ and its stabilizer in $ \Gamma $), we see  that the ramification index of $ \pr =  \pr_{L, K} $ at  $ p $ is $ | \mathrm{Stab}_{p}(\Gamma)| $. Thus the pullback of $ q : = \pr_{L, K}(p) $  under  $   \pr_{L, K }$  is $$ \pr_{L, K}^{*}(q)  =   \sum_{ p \in \pr ^{-1}(q) }  | \mathrm{Stab}_{p}(\Gamma) | p   .  $$
By the orbit-stabilizer theorem, the right hand side above is $ \sum_{ \gamma \in \Gamma } p_{0} \cdot \gamma $ where $ p_{0} \in \pr^{-1}(q) $ is  any choice. 

To address the general case, choose a compact open subgroup $ L' \subset L $ such that $ L   '    $ is normal in $ K $ (e.g., take the intersection    of     $ K $ with all the  conjugates of $ L $ by $ K /  L $). Replacing $ L' $ with $ L' ( K \cap \left \{ \pm 1 \right \}) $, we can assume that $ L' \cap \left \{ -1 \right \} = K \cap  \left \{ - 1 \right \}   $ and we still have $ L' \subset L $,     $ L ' \trianglelefteq     K $.    
If $ p = [x, g]_{K} $, then  \begin{align*}    [L : L']  \cdot  \pr_{L, K}^{*} (p)    &      =      
( \pr_{L' , L , *} \circ \pr_{L', L}  ^{*}   )     \circ \pr_{L, K}^{*}  ( p  )    \\  &       =     \pr_{L', L  *}  \circ  \pr_{L', K}^{*}  (  p   )  \\
& =  \pr_{L', L, *} \Big (  \sum   \nolimits _{\gamma \in K/ L ' }  [x , g \gamma ]_{L'}  \Big )  \\      
& =  \sum  \nolimits    _{\gamma \in K/L'} [ x , g \gamma ] _  {   L   }   \\   
&    = [L : L' ]     \cdot     \sum   \nolimits    _{\gamma \in K / L } [x , g  \gamma ]_{L} 
\end{align*}
where  $ \pr_{L', L, *} $ denotes pushforward. This  establishes the claim in  general.  
\end{proof}

\begin{remark} Suppose   that    $ -1 $ is in $ K $  but not in $ L $. Define  $ L_{1} $ to be the product $ L \cdot   \left \{ \pm 1  \right  \}     $. 
Then $ \mathcal{S}_{L}(\CC) = \mathcal{S}_{L_{1}}(\CC) $ and $ \pr_{L, K}  = \pr_{L_{1}, K} $ has degree $ [K : L]/ 2 $. In this case, the  pullback formula    holds with $ L $ replaced by $ L_{1}$.  
\end{remark} 
  
\begin{remark}   \label{neatremark}  When working with Shimura varieties, it is common to assume that the levels are sufficiently small as in \cite[Definition 2.1]{Fouquet}, or more precisely, neat in the sense of \cite[\S 0.1]{PinkThesis}. If $ K  $    as above is neat, then the groups $\Gamma_{g}$ (\ref{shimuravardecom1}) (and even their images in $\mathbf{G}(\mathbb{Q})/\mathbf{Z}(\mathbb{Q})$) are torsion free, and the degeneracy map \(\pr_{L,K}\) (\ref{degeneracy}) is unramified (hence étale)   for any $ L \subset K  $  by \cite[Lemma 2.7.1]{CZE}.

For general Shimura data, the corresponding Shimura varieties need not be smooth unless the chosen levels  are neat. Smoothness is a crucial assumption needed to invoke Borel's theorem on  algebraicity of holomorphic maps between hermitian symmetric domains \cite[Theorem 3.14]{Milne} (cf., \cite[Theorem 2]{Kiernan}),   which is   needed    to establish the algebraicity of   certain    natural  maps  between Shimura varieties \cite[Theorem 5.16]{Milne}. The neatness assumption, however, is not needed in our context, since the modular curves admit a  smooth structure    for any level.  Assuming  neatness  also excludes some important level structures from  consideration; see Example \ref{Gamma0(N)}.
\end{remark}

\subsection{Moduli interpretation}    \label{modulisec}  Observe that the Hodge structure on $ V_{\mathrm{std}} : = \QQ \oplus \QQ $ induced by any $ h   \in   \mathcal{X}_{\mathrm{std}}     $ is of type   $$  \{ (-1, 0), (0, -1)   \} .   $$    Thus $(\Gb, \mathcal{X}_{\mathrm{std}}) $ is the so-called  \emph{Siegel   Shimura datum of genus one} \cite[\S 6]{Milne}, \cite[\S 1.3.1]{DeligneVar}.      Following these sources, we can give the  following  moduli  interpretation for $ \mathcal{S}_{K}(\CC)$.   
Consider the set  $ \mathcal{E} $ of  all pairs $ (A, \eta) $ where $ A $ is  an    elliptic curve over the complex numbers\footnote{To avoid set theoretic issues, we will think of all elliptic curves over $ \CC $ as  quotients of $ \CC $ by a $\ZZ$-lattice.}    and $$ \eta :   V_{\mathrm{std}} \otimes_{\QQ} \Ab_{f}  \to  \mathrm{H}_{1}(A(\CC), \ZZ) \otimes \Ab_{f}   $$     is an isomorphism of   $ \Ab_{f} $-modules. Recall that the singular   homology $ \mathrm{H}_{1}(A(\CC), \RR)  = \mathrm{H}_{1}(A(\CC), \ZZ) \otimes \RR $ is endowed with a unique complex structure arising from the Hodge decomposition on $  \mathrm{H}^{1}(A(\CC), \CC)$.  If $ A(\CC) = \CC / \Lambda $ for $ \Lambda $ a $ \ZZ $-lattice in $ \CC $, then  $$  \mathrm{H}_{1}(A(\CC)    , \ZZ) \simeq \Lambda $$ canonically and the     complex    structure  on    $ \mathrm{H}_{1}(A(\CC), \ZZ)  \otimes \RR  $ is identified with the  one  on  $ \Lambda \otimes \RR = \CC $ given by multiplication by $ i   $   \cite[Example 1.1.4]{DeligneVar}.  
For each    $  (A, \eta ) \in \mathcal{E} $,  pick an isomorphism $ \sigma :   \mathrm{H}_{1}(A(\CC), \QQ)  \to  V_{\mathrm{std}}    $ of $ \QQ$-vector spaces.   Let  $ J_{\sigma} $ be the  complex structure on $ \RR^{2} = V_{\mathrm{std}} \otimes_{\QQ} \RR  $ obtained by transport of structure along $ \sigma_{\RR}    $ and  let $ g_{\sigma} \in \Gb(\Ab_{f}) $ be the composition $$  V_{\mathrm{std}}  \otimes_{\QQ}  \Ab_{f}  \xrightarrow{ \,  \,  \eta   \, \,      }  \mathrm{H}  _   {1}(A(\CC)     , \ZZ) \otimes \Ab_{f}     \xrightarrow { \sigma \otimes 1 }   V_{\mathrm{std}}  \otimes_{\QQ}   \Ab_{f}    .    $$
Replacing $ \sigma $ by $ q \circ \sigma $ for $ q \in \Gb(\QQ )$ replaces $ J_{\sigma} $ with $ q J_{\sigma} q^{-1} $ and $ g_{\sigma} $ with $ q g_{\sigma} $. Thus,   each pair $ (A, \eta) $ determines a well-defined point  $$ [x_{\sigma}, g_{\sigma}] \in  \Gb(\QQ) \backslash  ( \mathcal{X}_{\mathrm{std}}  \times  \Gb (\Ab_{f})   )  $$ 
where $ x_{\sigma} \in \mathcal{X}_{\mathrm{std}} $ corresponds to $ J_{\sigma} \in \mathrm{CS}(\RR^{2}) $ under  the  canonical    identification  made   in       (\ref{variousiden}).   
Two pairs $ (A_{1}, \eta_{1}) $, $ (A_{2}, \eta_{2} )$ give the same point under this process if and only if there is an isogeny $ f :  A_{1} \to A_{2} $ such that $   ( f_{*} \otimes 1 ) \circ \eta_{1}  =  \eta_{2} $. This defines an equivalence relation $ \sim $  on $ \mathcal{E} $ and we have   a  bijection
\begin{equation}  % 
\label{moduliisominf}    
\mathcal{E}/\!  \!    \sim \, \,  \MapsTo   \, \,   \Gb(\QQ) \backslash  (  \mathcal{X}_{\mathrm{std}} \times \Gb(\Ab_{f})   )  .   
\end{equation}   
The right action of $ g\in \Gb(\Ab_{f})  $ on the right hand side of  (\ref{moduliisominf})     corresponds to the action  on   $ \mathcal{E} /\! \! \sim $ that sends the equivalence class of $ (A , \eta ) $ to that of $ (A, \eta \circ g ) $. Quotienting by $ K $, we obtain  an  identification    
\begin{equation}   \label{moduliisog}   (\mathcal{E}/\!\!\sim)/K  \,  \MapsTo   \,  \mathcal{S}_{K}(\CC).
\end{equation}   
For $ (A, \eta )$ as above,   the     $K$-orbit of $ \eta $  is referred to as a \emph{$K$-level structure} on $ A $. Thus (\ref{moduliisog}) says that $ \mathcal{S}_{K}(\CC) $ is a  parameter space for isogeny classes of elliptic curves  equipped with a $ K $-level structure.   %   See  \cite[\S 6]{Milne}.    
\begin{remark} 
The left-hand side of (\ref{moduliisog}) actually forms the set of $\CC$-points of a moduli functor that associates to any $\QQ$-scheme $S$ the set of isomorphism classes of elliptic curves over $S$ (up to isogeny)  equipped with a $K$-level structure, which is now   defined  in terms of local systems  arising from the  first \'{e}tale homology of the geometric fibers of the elliptic curve over $ S $.       See,  e.g.,    \cite[\S 2.6]{Ngo-Genestier} for a precise formulation. This functor can be shown to be representable by a coarse moduli scheme $M_{K}$ over $ \QQ$,\footnote{which is a fine moduli space if $ K $ is neat} whose $\CC$-points  are identified with $\mathcal{S}_{K}(\CC)$ via (\ref{moduliisog}). One then checks that $M_{K}$ (for varying $K$) satisfies all the properties required for it to serve as a canonical model for $\mathcal{S}_{K}(\CC)$, and this is the scheme  we  have denoted by $ \mathcal{S}_{K}$ above. See  also       Note~\ref{DeligneCM}.   
\end{remark}    

\begin{remark}    \label{weirdmoduli}    On the other hand, the alternative   datum  (\ref{weirdShimura})  induces the   dual    Hodge structure of type $$ \left \{(1, 0), (0, 1) \right \}   $$    on $ V_{\mathrm{std}}$. The canonical models for this datum can be constructed using  \cite[Crit\'{e}re 2.3.1]{DeligneVar}. More precisely, we apply Proposition 2.3.2 of \emph{loc.~cit.} with the dual representation $\rho^{\vee} : \Gb \to   \GL( V_{\mathrm{std}}^{\vee}    ) $ to embed this datum into the Siegel datum of genus one.\footnote{Shimura data that embed into the Siegel Shimura data (of some genus) are said  to  be    of         \emph{Hodge type.}} The  resulting embedding is then exactly the  inverse of the  isomorphism     (\ref{weirdshimuraiso}). The   % In the terminology of \cite[\S 7]{Milne}, the data  (\ref{weirdShimura}) is of  \emph{Hodge type}.  
Shimura varieties attached to (\ref{weirdShimura}) inherit their moduli interpretation form those of (\ref{SDdatum}), and this   comparison swaps the maps used to defined level structures  with their duals.      That is, $ \mathcal{S}_{K}'(\CC) $ parametrizes elliptic curves $ A $ (up to isogeny) equipped with a $ K^{t}$-orbit of isomorphisms $$   V_{\mathrm{std}}   \otimes_{\QQ   }  \Ab_{f}   \to  \mathrm{H}_{1}(A(\CC), \ZZ)  \otimes  \Ab_{f}     $$
or equivalently, a $ K $-orbit of isomorphisms $ \mathrm{H}^{1}( A(\CC), \ZZ) \otimes \Ab_{f}   \to   V_{\mathrm{std}}^{\vee}   \otimes_{\QQ}  \Ab_{f}      $.    
\end{remark} 

\subsection{Galois action on components}   \label{compsec}    
The curve $ \mathcal{S}_{K} $ is not geometrically connected in general, and one can describe its geometrically connected components as follows.  Let $ \det : \Gb \to \GG_{m} $ be the determinant map and $  \mathrm{sgn} : \mathcal{X}_{\mathrm{std}} \to \{ \pm 1 \} $ be the map  $ g    \,    \negsmall   h_{\mathrm{std}} \,   \negsmall     g^{-1} \mapsto \det(g)/|\det(g)|$. Then $ \mathrm{sgn} \times \det : \mathcal{X}_{\mathrm{std}}  \times \Gb(\Ab_{f}) \to  \left \{  \pm 1 \right \}  \times  \Ab_{f}^{\times} $ induces a surjective map 
\begin{equation}  \label{connmap} 
\mathcal{S}_{K}(\CC) \to \QQ^{\times} \backslash \big( \{ \pm 1 \} \times \Ab_{f}^{\times} \big) / \det(K)     =  \mathbb{Q}^{\times}_{>0} \backslash \Ab_{f}^{\times}/ \det(K)    
\end{equation} 
whose fibers are geometrically connected components of $ \mathcal{S}_{K}(\CC) $ \cite[Theorem 5.17]{Milne}. Thus  the   geometric curve $ \mathcal{S}_{K,  \overline{\QQ}}    $ decomposes as a disjoint union   of  curves    $ \mathcal{S}_{K,\alpha} $ indexed by  
$ 
\alpha \in \QQ^{\times} \backslash \big( \{ \pm 1 \} \times \Ab_{f}^{\times} \big) / \det(K)   $.   The components $ \mathcal{S}_{K,\alpha} $ are not necessarily   defined over $ \QQ $ but are defined on certain abelian extensions of $ \QQ $   inside $ \CC $,  which  are    determined by the Galois action on the components defined via  the   reciprocity law for the Shimura datum   
\begin{equation}   \label{compreci}    \det \circ  \, h_{\mathrm{std}} : \delT \to \GG_{m}   
\end{equation}  
for $ \GG_{m} $  similar to the one in \S \ref{canonicalmodsec}.    More precisely, let $ (\alpha_{\infty}, \alpha_{f}) \in  \{ \pm \}  \times \Ab_{f}^{\times} $ be a representative of $ \alpha $. Let  $ \sigma \in \Gal(\QQ^{\mathrm{ab}}/\QQ)$ and  pick   any  $ a = (a_{\infty}, a_{f}) \in \Ab^{\times} $  such  that   $ \sigma = \Art_{\QQ}(a) $, where 
\begin{equation}   \label{CFTforQ}    \Art_{\QQ} : \QQ^{\times} \backslash \Ab^{\times} \to \Gal(\QQ^{\mathrm{ab}}/\QQ)  
\end{equation}    denotes the Artin map,  normalized so that uniformizers are   mapped  to    geometric Frobenii.
Then  
$$ 
\sigma(\mathcal{S}_{K,\alpha}) = \mathcal{S}_{K,\alpha'} 
$$  
where $ \alpha' $ is represented by  $ 
(\mathrm{sign}(a_{\infty}) \alpha_{\infty}, a_{f}\alpha_{f}) \in \{ \pm  1\} \times \Ab_{f}^{\times}  $. Since this action is transitive, the   scheme     $ \pi_{0}( \mathcal{S}_{K}) $   is     identified with  the spectrum of   the     fixed field of $ \mathrm{Art}_{\QQ}(\QQ^{\times} \det(K) )$  and so   has a unique $ \QQ $-point. This implies  that    each      modular  curve    $ \mathcal{S}_{K} $  is  a smooth connected (hence integral)   scheme  over   $   \QQ   $.

\begin{remark}    Note  that the Galois action on components is independent of the identification of $ \pi_{0}(\mathcal{S}_{K})(\CC) $ with the quotient on the right   hand  side   of     (\ref{connmap}). That is,  if we replace $ \det : \Gb \to \GG_{m} $ with its inverse, the reciprocity law is also replaced by its inverse,  and we end   up      obtaining  the same Galois action on  the components of   $ \mathcal{S}_{K,   \overline{\QQ}}    $.    
\end{remark}  

\begin{remark}   \label{weirdcompactremark}    For  the alternative  datum   (\ref{weirdShimura}),   the reciprocity law on components uses $ \det \circ \, h _{\mathrm{std}}' $ which sends $ z \in \delT(\RR)   $ to $ (z\bar{z})^{-1}  \in   \mathbb{G}_{m}(\RR)   $. So while $ \mathcal{S}_{K}(\CC) $ and $ \mathcal{S}_{K}'(\CC) $ are isomorphic Riemann surfaces, the Galois action on their components   differ by a sign. For this reason alone, the  isomorphism (\ref{veryweirdiso})  cannot  descend to the   underlying    canonical models across all levels.  See also  Remark  \ref{weirdCMdivisorremark}.      
\end{remark}

\subsection{Classical modular curves}   \label{classicalmodular}           
One can obtain    classical modular curves from   adelic ones  as  follows. For $ \tau \in \mathcal{H}^{\pm}$, let $ x_{\tau} \in \mathcal{X}_{\mathrm{std}} $ be the point corresponding to $ \tau $ under (\ref{variousiden}). Given a representative $  (\alpha_{\infty}, \alpha_{f}) $ for $ \alpha    \in \QQ^{\times} \backslash \{ \pm  1 \} \times \Ab_{f}^{\times} / \det(K) $, let $ \beta_{f} \in \Gb(\Ab_{f}) $ be any element such that $ \det(\beta_{f}) = \alpha_{f} $. Then the map   
\begin{equation*}    
\begin{split}    
\mathcal{H}^{+}   &     \to \mathcal{S}_{K}(\CC)   ,        \quad  \quad      
\tau    \mapsto    [ x_{( \alpha_{\infty} \tau)} , \beta_{f} ]_{K}   
\end{split}   
\end{equation*}    
induces a holomorphic covering of $ \mathcal{S}_{K,\alpha}(\CC) $  that  factors through an  isomorphism  
\begin{equation}   
\label{quotient}    
  \Gamma \backslash \mathcal{H}^{+} \MapsTo  \mathcal{S}_{K,\alpha}(\CC),   
\end{equation}    
where $ \Gamma =   \Gamma_{\beta_{f} }  =   \Gb(\QQ)^{+} \cap \beta_{f} K \beta_{f}^{-1} .  $  
Of course,  replacing $ \beta_{f} $ with $ \beta_{f}  k   $ for $ k \in  K $ does not change the map  (\ref{quotient}).     However, replacing $ \beta_{f} $ with $ q \beta_{f} $ for some $ q \in \Gb(\QQ)^{+} $ changes  $ \Gamma $ to $ q \Gamma q^{-1} $, and the resulting   identifications    may be   different  even      when $  \Gamma =  q \Gamma q^{-1} $.  See  Remark \ref{diffembed}.

When $ K  $   is   contained  in   $    \GL_{2}(\widehat{\ZZ}) $,  the ``isogeny class" interpretation given in \S \ref{modulisec} can be rigidified to the more familiar     ``isomorphism class" interpretation   %  when  $ K \subset \GL_{2}(\widehat{\ZZ}) $ 
as  follows.    Suppose first  that   $$  K = K(N)  = \widehat{\Gamma}(N)  $$  is the (normal) subgroup of matrices in $ \GL_{2}(\widehat{\ZZ}) $  that reduce modulo $ N $ to identity. Then $ K $ is exactly the group   of    elements in $ \GL_{2}(\widehat{\ZZ}) $  whose  reductions modulo $ N $  act  trivially on $   ( \widehat{\ZZ}/N\widehat{\ZZ}) ^{2}   =   ( \ZZ / N \ZZ   )   ^{2}       $.    We  refer to $ K = K(N) $ as the \emph{principal  congruence subgroup of level $ N   $}.         Let $ \mathcal{E}(N) $ be the set of pairs $ (A, \nu) $ where $ A $ is an elliptic curve  over  $  \CC   $      and   
\begin{equation}   \label{numap} \nu : (\ZZ/ N\ZZ)^{2} \to A[N] (  \CC  )       
\end{equation}    is an isomorphism of $ \ZZ/N \ZZ $-modules that  we refer to as a \emph{full level $N  $ structure on $ A[N]$}.     Define an equivalence relation on $ \mathcal{E}(N) $ by declaring two pairs $ (A_{1}, \nu_{1}) $, $ (A_{2}, \nu_{2}) $ to be equivalent if there is an isomorphism $ f : A_{1} \to A_{2} $ of  elliptic   curves     satisfying $ f \circ  \nu_{1} =  \nu_{2} $.   Given  $ (A, \nu) \in \mathcal{E}(N) $,  one can choose an isomorphism   $$   \widehat{\nu} : \widehat{\ZZ}^{2} \to  \varprojlim \nolimits   _{N}  ( A[N](\CC)      )     $$     
whose reduction modulo $ N $ equals $ \nu $. Moreover, the set of all possible such choices constitutes    a $ K $-orbit.   
Let   $ \eta $ denote the map  $$   \widehat{\nu} \otimes 1 :  \Ab_{f}^{2}  \to  \big ( \varprojlim  \nolimits    _{N}  ( A[N]  ( \CC  )  )   \big     ) \otimes   _  {  \widehat{\ZZ} }     \Ab_{f}.   $$
Since the target of $ \eta $ is canonically identified with $ \mathrm{H}_{1}(A(\CC)    , \ZZ) \otimes  \Ab_{f} $,  the map $ (A, \nu) \mapsto (A, \eta) $ gives a map  from $  \mathcal{E}(N)/\!\!\sim $ to $   (\mathcal{E}/\! \! \sim)/ K  $ that is easily seen to be a bijection. Using (\ref{moduliisog}), we  obtain an  identification 
\begin{equation}    \label{moduliisom}  
\Psi  _   {   N    }     :    \mathcal{E}(N)/ \! \!  \sim \, \,  \,   \MapsTo \,  \,  \,    \mathcal{S}_{K}(\CC) 
\end{equation}   
As before, the twisting action   of    $ \kappa \in \GL_{2}(\widehat{\ZZ}) $ on $ \mathcal{S}_{K}(\CC) $ is identified under (\ref{moduliisom})  with the action on $ \mathcal{E}(N)/\!\!\sim $ that sends the class of $ (A, \nu) \in  \mathcal{E}(N) $ to that $ ( A, \nu \circ   \bar{\kappa}) $,  where $ \bar{\kappa} \in \GL_{2}(\ZZ / N \ZZ )$   denotes  the  reduction of  $ \kappa    $   modulo  $  N   $.

We can describe the  inverse of (\ref{moduliisom}) more explicitly. Observe that since $ \Gb(\QQ) \GL_{2}(\widehat{\ZZ}) = \Gb(\Ab_{f}) $, each class in $ \Gb(\QQ) \backslash \Gb(\Ab_{f})/ K $ contains a representative in $ \GL_{2}(\widehat{\ZZ}) $.  Since   $ \GL_{2}(\widehat{\ZZ})/K = \GL_{2}(\ZZ/N\ZZ) $,  we can write    $$ \mathcal{S}_{K}(\CC)    =  \GL_{2}(\ZZ)  \backslash   ( \mathcal{X}_{\mathrm{std}}   \times \GL_{2}(\ZZ / N \ZZ  )   ) .  $$ 
Given $ (x,  \kappa  ) \in   \mathcal{X}_{\mathrm{std}}  \times \GL_{2}(\ZZ / N  \ZZ    ) $, choose an element  $   g       \in \GL_{2}(\RR) $ such that $ x =  g  \,      \negsmall  h_{\mathrm{std}} \,    \negsmall      g^{-1} $  and   write    $ g ^ { - 1    }  =   \left (   \begin{smallmatrix}  a & b \\ c &  d \end{smallmatrix}     \right  )  $. 
Then the complex structure $$  J_{g} = g \,  \negsmall   h_{\mathrm{std}}(i)  g ^{-1} \in \mathrm{CS}(\RR^{2})  $$ on $ \RR^{2} \simeq \CC $  corresponding to $ x \in  \mathcal{X}_{\mathrm{std}}   $ under    (\ref{variousiden})  equals the matrix   of     multiplication by $ i $ in the ordered     $\RR$-basis $( a - ci , b - di ) $ of $ \CC $   and therefore also in the ordered   $\RR$-basis   
\begin{equation}    \label{thebasis}  (  1 , (b - di )/ (a - ci ) )  =   ( 1, - g \cdot i  )    
\end{equation}   
Let $ \tau = g \cdot i \in \mathcal{H}^{\pm} $ denote the point corresponding to $ x $ under (\ref{variousiden})  and    let
\begin{equation}   \label{lambdatau}   \Lambda_{\tau} :=  \ZZ +  \ZZ(-\tau) = \ZZ + \ZZ \tau     
\end{equation}   
be the $ \ZZ $-lattice spanned by   the   basis  $ (1,  -\tau   )   $ of $ \CC $.  Consider the complex elliptic curve $ A_{\tau} $  satisfying $ A_{\tau}(\CC)   =      \CC / \Lambda_{\tau}    $    endowed   with  the    full      level      structure % $ \nu : (\ZZ/ N\ZZ)^{2} \to A_{\tau}[N] $ via   
\begin{equation}   \label{levelstrucisom}    
\begin{split}   \nu  _ {  \kappa  }  :  (\ZZ / N \ZZ) ^{2}  &  \longrightarrow  N^{-1}\Lambda_{\tau} / \Lambda_{\tau}   =   A_{\tau}[N](\CC) \\  v     &  \mapsto   w_{1}/N  - w_{2} \tau/N + \Lambda_{\tau} 
\end{split}   
\end{equation}   
where $ w_{i} + N \ZZ = \pr_{i}( \kappa  v  ) $ denotes    the $ i$-th component of $  \kappa    v  $. Since the choice of the $ \ZZ$-basis $ (1, -\tau) $  for $ \Lambda_{\tau} $ corresponds to fixing an isomorphism $ \sigma_{\ZZ} :   \mathrm{H}_{1}(A_{\tau}(\CC), \ZZ) \to \ZZ^{2} $,  it is not   hard to  see that   class of $ (A_{\tau}(\CC),   \nu_{\kappa})  $ maps to   $ [x,    \kappa]_{K}    \in    \mathcal{S}_{K}(\CC) $ under  (\ref{moduliisom}). We observe that   the definition of full level structure (\ref{levelstrucisom}) and the moduli interpretation   obtained  here    matches with the one  stated   in     \cite[\S 4.2]{Scholl}. 
\begin{remark}   \label{thebasisremark}  Note that we can also work with the ordered basis $ ((a-ci)/(b-di), 1) = (-1/\tau, 1)  $. This gives us the lattice $ \Lambda_{1/\tau} =  \ZZ \cdot (-1/\tau) + \ZZ$, which is homothetic to $ \Lambda_{\tau}$ via multiplication by $ -\tau $.    
\end{remark}

Now suppose   that       $ K \subset \GL_{2}(\widehat{\ZZ}) $ is arbitrary.  Choose an integer   $ N \geq 1 $ such that $ K(N) $ is contained in $ K  $.  Then   the degeneracy map $$ \pr :  \mathcal{S}_{K(N) }(\CC)   \to   \mathcal{S}_{K}(\CC) $$  identifies $ \mathcal{S}_{K}(\CC) $ as a  quotient of $ \mathcal{S}_{K(N)}  (  \CC     )     $ by $ K/K_{N} $.  So   $ \mathcal{S}_{K}(\CC) $   parametrizes    isomorphism classes of    elliptic curves $ A $  endowed with a $ K / K_{N} $-orbit of isomorphisms $     ( \ZZ/ N \ZZ)^{2}   \to  A[N]  (  \CC   )     $. This interpretation can also be  obtained by noting that $$ \mathcal{S}_{K}(\CC) = \GL_{2}(\ZZ) \backslash (\mathcal{X}_{\mathrm{std}}   \times  \GL_{2}(\widehat{\ZZ}) )/K $$
and  writing the obvious integral counterpart of the discussion  in    \S \ref{modulisec}.

\begin{remark}   \label{orbitremark}  As evident, the data of a full level   $  N   $    structure $ \nu : (\ZZ/ N \ZZ )^{2}  \to A[N]   (   \CC    )     $ is the data of an ordered basis $ (e_{1},  e_{2}) $ for $ A[N]   (   \CC    ) $ given by  $$ e_{1}  = \nu(1, 0),  \quad   \quad        e_{2}    = \nu(0, 1) .   $$    In this interpretation,  the    action of $  \left (    \begin{smallmatrix} a & b \\ c &  d  \end{smallmatrix}   \right  )   \in   \GL_{2}(\widehat{\ZZ})     $ sends $ (A,  (e_{1}, e_{2})) $ to $ (A,    (e_{1}',  e_{2}')    ) $ where   
\begin{equation}   \label{basisrelation}        (e_{1}',  e_{2}')  =     (e_{1},  e_{2})    \left (   \begin{smallmatrix} a  & b  \\ c & d  \end{smallmatrix}  \right )   
\end{equation}    
That is,    $$  e_{1}'   = a e_{1} + c e_{2}  \quad  \text{ and }   \quad     e_{2}'    = b e_{1} + d  e_{2}    . $$  This  interpretation     can  be  used     to  give more explicit descriptions of $ \Gamma $-orbits of $ \nu $ for   certain    subgroups $ \Gamma $ of $ \GL_{2}(\ZZ / N \ZZ )$.         
\end{remark}

\begin{remark} One can similarly write a rigidified version of the moduli interpretation for the alternative datum (\ref{weirdShimura})  mentioned in Remark   (\ref{weirdmoduli}) for $ \widehat{\Gamma}(N)$ level  structures. If $ x \in \mathcal{X}_{\mathrm{std}}' $ corresponds to $ \tau \in \mathcal{H}^{\pm} $ via the right vertical arrow of (\ref{Xstd'compardiag}),  then the conventions of \S \ref{modulisec} force us to associate to $ (x, \kappa) \in \mathcal{X}_{\mathrm{std}}'   \times   \GL_{2}(\ZZ/N \ZZ )$ the ordered basis $ (\tau, 1) $ and the level structure  
\begin{align*} (\ZZ / N\ZZ)^{2}   & \to  N^{-1} \Lambda_{\tau}/\Lambda_{\tau} \\ 
v   &     \mapsto  w_{1}\tau/N + w_{1} / N + \Lambda_{\tau}  
\end{align*}   
where $ w_{i}  + N\ZZ = \pr_{i}(\kappa v) $.    
Moreover, the action of $ \gamma \in \GL_{2}(\widehat{\ZZ}) $ on level structures is via  pre-composition with $ {}^{t}   \negsmall \gamma^{-1} $. One then recovers  the moduli  interpretation   mentioned  in    \cite[Definition 5.1.1]{LSZ},  after rewriting the  analogue  of  the    relation (\ref{basisrelation})    in terms of column vectors.    
\end{remark}

\begin{example} \label{Gamma(N)}  Suppose $ K = \widehat{\Gamma}(N)    $. Since $ \det(K) =    \prod_{\ell} \ZZ_{\ell}^{\times}  \cdot     \prod_{\ell \mid N} (1 + N\ZZ_{\ell})   $, the  components  of   $ \mathcal{S}_{K}(\CC)   $ are indexed by $$ \QQ^{\times}    \backslash   (  \left \{   \pm 1  \right \}   \times    \Ab_{f}^{\times}  )  / \det(K)  \simeq    ( \ZZ / N \ZZ )^{\times } .  $$   
The    curve $ \mathcal{S}_{K}(\CC) $ thus  has $ \phi(N) $ connected components where $ \phi $ denotes the Euler totient function and the reciprocity law descibed in \S \ref{compsec} implies that each   component    is defined over the $N$-th cyclotomic extension $ \QQ(\mu_{N}) $ where $$ \mu_{N} = \{ e^{2\pi i   k /N} \, | \, 0 \leq   k    \leq N-1 \}  \subset  \CC .   $$       %are the roots of unity  in  $  \CC  $. 
Let us  consider the component  of $ \mathcal{S}_{K} $     indexed by the class of $  1 \in  ( \ZZ / N \ZZ)^{\times} $ and   take    $ \beta_{f} =   1  $ as a representative for the component.  Then   $ \Gamma(N)  =    \Gb(\QQ)^{+} \cap K  $    is the usual subgroup of matrices in $ \SL_{2}(\ZZ) $ that reduce   to identity modulo $ N   $ and we have an      embedding 
\begin{equation}   \label{twistbetaf}   
\begin{split} \Gamma(N) \backslash  \mathcal{H}^{+}     &     \hookrightarrow    \mathcal{S}_{K} (\CC)   ,    \quad  \quad    
 \Gamma(N)   \tau        \mapsto   [x_{\tau},  1 ]_{K}     .
 \end{split}    
\end{equation}    

By \cite[Theorem 1.5.1(c)]{Diamondmodular}, the  moduli space for $ \Gamma(N) \backslash \mathcal{H}^{+} $ is the set of isomorphism classes of elliptic curves $ A $ together with a basis $ (P, Q) $ for $ A[N] $ such that the Weil pairing sends the $ (P, Q) $ to $ e^{2\pi i  /N} $. It  identifies  with  the set   $$ \left \{  [     \CC/\Lambda_{\tau} ,  (\tau / N + \Lambda_{\tau} ,  1/N  +   \Lambda_{\tau} ) ]   \mid   \tau \in  \mathcal{H}   ^ {    +    }    \right  \}  : =   \mathrm{S}(N)  $$  via the obvious   map    
\begin{align*} \psi  _ {  N   } :   \mathrm{S}(N)   & \to  \Gamma(N) \backslash  \mathcal { H } ^{+ }  ,   \quad \quad 
[\CC/\Lambda_{\tau} ,   (\tau / N + \Lambda_{\tau} ,  1/N  +   \Lambda_{\tau} ) ]  %  &  
\mapsto   \Gamma(N)  \tau   .       
\end{align*} 
The embedding    (\ref{twistbetaf})  then    extends to a commutative  square       
\begin{equation*}  \label{diagram}       
\begin{tikzcd}[sep  =  large]        \mathrm{S}(N)    \arrow[r, hookrightarrow,  "\theta"]   \arrow[d, "\psi_{N} "']    &  \mathcal{E}(N)/   \!   \!    \sim       \arrow[d,  " \Psi_{N}    "]       \\ 
\Gamma(N)  \backslash   \mathcal{H}^{+}      \arrow[r, hookrightarrow]   &   \mathcal{S}_{K}(\CC)       
\end{tikzcd}    
\end{equation*}   
where the top   horizontal  map is \begin{equation}   \label{thetamap}    \theta :     [ \CC/\Lambda_{\tau}, (\tau/N + \Lambda_{\tau} ,  1/N + \Lambda_{\tau}  )  ]   \mapsto  [    A_{\tau}     ,  \nu_{\bar{\beta}_{f}}    )    ]  .  
\end{equation}    
That is, $ \theta $ sends the pair $ (A, (P, Q)) $ to the pair $ (A, (e_{1}, e_{2}) ) $ where $$ e_{1} = \nu(1,0) = Q , \quad    \quad    \quad      e_{2}  =     \nu(0, 1) = - P .   $$   
Recall also that  $ \Gamma(N)  \backslash  \mathcal{H}^{+} $ has a natural   \emph{left}    action of $ \gamma \in \SL_{2}(\ZZ/N\ZZ) $  given by  $$  \Gamma(N) \tau  \mapsto   \Gamma(N) \gamma(\tau)  $$  which agrees via $ \psi   _ {  N   }      $    with the obvious left action on $     \mathrm{S}    (N) $ that replaces $ \tau $ with $ \gamma(\tau) $   everywhere.  On the other hand,  the map $ \Psi_{N}   $   (\ref{moduliisom})      is equivariant with respect to the \emph{right} action of $ \GL_{2}(\widehat{\ZZ})/ K \simeq  \GL_{2}(\ZZ / N \ZZ ) $ and   its   subgroup $ \SL_{2}(\ZZ / N \ZZ) $ preserves the component of $ \mathcal{S}_{K}(\CC) $   indexed by $ 1 $. The reader is invited to check that the horizontal maps intertwine the  action   of $ \gamma \in \SL_{2}(\ZZ / N \ZZ)  $ on the domain with the action of $ \gamma^{-1} $ on the target    
i.e.,    $$  \theta( \gamma \cdot   (  -   ) ) =  \theta( - )  \cdot \gamma^{-1} .  $$   
See  also   Lemma     \ref{translation}. 
\end{example} 
\begin{remark}   \label{diffembed}  If we instead   use    $ \beta_{f} =    \left (  \begin{smallmatrix} & 1 \\ -1  \end{smallmatrix}  \right )  \in \GL_{2}(\ZZ)  \hookrightarrow \GL_{2}(\Ab_{f})  $ in the discussion  above, then   the  twisted  intersection  $ \Gb(\QQ)^{+} \cap \beta_{f} K \beta_{f}^{-1} $ is still the group      $ \Gamma(N) $. The  embedding 
\begin{align*}  \Gamma(N)  \backslash  \mathcal{H}^{+}   &    \hookrightarrow    \mathcal{S}_{K}(\CC) ,  \quad   \quad     
\Gamma(N) \tau       \mapsto [x_{\tau},  \beta_{f}]_{K}   
\end{align*}    now corresponds to  the   map      $$  \theta_{1} : \mathrm{S}(N)  \hookrightarrow  \mathcal{E}(N)/\!\! \sim   $$ that    sends the class of a pair $ (A, (P, Q)) $ to the class of $ (A, \nu) $ where $$   e_{1}   = \nu(1, 0) = P,   \quad   \quad   \quad         e_{2}   =     \nu(0, 1) = Q .   $$ 
One can check  that   $ \theta_{1} $ intertwines the action of $ \gamma \in \SL_{2}(\ZZ / N \ZZ ) $ on the domain with that of $   {}^{t}      \negsmall      \gamma   $ on the target.   
\end{remark} 

\begin{example} \label{Gamma0(N)}   For an integer $ n \geq 1 $, let $ L_{f, n} $ be the $ \widehat{\ZZ} $-lattice in $ \Ab_{E,f} $ spanned by $ \omega_{1} $ and $      n \omega_{2} $. We let $ \Gb(\Ab_{f}) $ act on the left of   $$ \Ab_{E,f} = \Ab_{f} \omega_{1} \oplus \Ab_{f} \omega_{2} \simeq \mathrm{Mat}_{2 \times 1} (\Ab_{f})   $$       
by left matrix multiplication.  Fix an integer $ N \geq 1 $ and let  $ 
K = \widehat{\Gamma}_{0}(N) $ be the set of all $ g \in \Gb(\Ab_{f}) $ such that $ g L_{f, 1} = L_{f, 1} $ and $ g L_{f, N} = L_{f, N }$. Then  
$$      
K =  \left\{ \left ( \begin{smallmatrix} a & b \\ c & d \end{smallmatrix}  \right  )  \in \GL_{2}(\widehat{\ZZ}) \,\middle|\, c \equiv 0 \pmod{N} \right\}.  
$$  
In this case, $ \det(K) = \widehat{\ZZ}^{\times} $ and $ \QQ_{>0}^{\times} \backslash \Ab_{f}^{\times} / \det(K) $ is a singleton. Thus  we have an identification     
\begin{equation}  \label{stdidenGamma0} 
\begin{split} 
 \Gamma_{0}(N) \backslash \mathcal{H}^{+}     &  \MapsTo   \mathcal{S}_{K}(\CC)  ,   \quad  \quad     \Gamma_{0}(N) \tau          \mapsto  [x_{\tau},   1]    %  
 \end{split}
 \end{equation}    
where $ \Gamma_{0}(N) = \SL_{2}(\QQ) \cap  K    $  is     the usual subgroup of matrices in $ \SL_{2}(\ZZ) $ whose reduction modulo $ N $ is upper triangular.  For a  free $ (\ZZ/ N\ZZ)$-module $   T $     of rank $ 2 $, the $ K / K_{N} $ orbit of an isomorphism $  \nu :  (\ZZ/ N  \ZZ)^{2}  \to    T     $ is uniquely determined by the data  of  the  rank $1$ sub-module $$  (\ZZ/ N\ZZ) \cdot  \nu(1,0) $$ spanned
by the   \emph{first}      basis element $ e_{1}  =  \nu(1, 0) $.\footnote{In the notation of \cite[\S 1.5]{Diamondmodular}, the map $   \mathrm{S}(N) \to \mathrm{S}_{0}(N)$ sends the class  of  $ (A, (P, Q)) $ to  that of $(A, \langle Q \rangle )$, i.e., the basis for $ A[N]   (   \CC   )   $    is mapped to the line spanned by the   \emph{second}   basis element, which is consistent with (\ref{thetamap}).} As $ \nu $ runs over all the set of all possible isomorphisms,  the  rank one sub-modules runs over  all cyclic subgroups of $   T    $ of order $ N   $. We recognize the curve $ \mathcal{S}_{K} $ as the smooth  geometrically connected affine modular curve commonly denoted as  $ Y_{0}(N) $,  which is a Zariski open subset of the   smooth     projective curve $ X_{0}(N)  $   from the introduction.

We also  observe    that since $ \Gamma_{0}(N) $ contains $ -1 $ for all $ N $,  $ Y_{0}(N) $ is only a \emph{coarse moduli space}. In fact, the image of $ \Gamma_{0}(N) $ in $ \SL_{2}(\ZZ)/\left \{ \pm   1  \right \} $ can also contain torsion elements for arbitrarily large $ N $. For instance, if we take $ N = a^{2} - a + 1 $ for some integer $ a  \geq 0$,   then     $$ \left ( \begin{smallmatrix} 1 \\ a & 1  \end{smallmatrix} \right )  \left (  \begin{smallmatrix}  &  -1 \\ 1 & \, \,   1  \end{smallmatrix}  \right  )   \left (  \begin{smallmatrix} 1 \\ -a & 1  \end{smallmatrix}  \right )   = \left (    \begin{smallmatrix} a  & - 1 \\ N  &   1 - a   \end{smallmatrix}   \right )  $$ 
is an order $6$ element of $ \Gamma_{0}(N) $. So $ \Gamma_{0}(N) $ is very far from being neat in general. 
\end{example}  

\begin{remark}    The group $ \Gamma^{0}(N) $ obtained by taking the transpose of the elements of $ \Gamma_{0}(N)$  gives another scheme $ Y^{0}(N)$ which is isomorphic to $ Y_{0}(N)$ and has the same moduli interpretation. However,    $ Y_{0}(N)(\CC)$ and $ Y^{0}(N)(\CC)$ are   \emph{not}     isomorphic as quotients of $    \Gamma(N)  \backslash   \mathcal{H}^{\pm} $   under the degeneracy maps induced by the inclusion of these groups in $ \Gamma(N)$.       
\end{remark}

\begin{example}  \label{Gamma1(N)} For $ N \geq 1 $, let  $$        
K =  \widehat{ \Gamma   }_{1}(N) =  \left\{  \left (  \begin{smallmatrix} a & b \\ c & d \end{smallmatrix}  \right  )    \in \GL_{2}(\widehat{\ZZ}) \,\middle|\,   a   \equiv  1,  c  \equiv 0\pmod{N} \right\}  
$$ 
Again, $ \det(K) = \widehat{\ZZ}  ^ {  \times   }    $ and  we have an identification   \begin{equation}   \label{stdidenGamma1} 
\begin{split} \Gamma_{1}(N) \backslash  \mathcal{H}^{+}    &  \MapsTo   \mathcal{S}_{K}(\CC)   ,  \quad   \quad    
\Gamma_{1}(N)\tau     \mapsto [x_{\tau}, 1]   
\end{split}   
\end{equation}    where $ \Gamma_{1}(N)  =   \SL_{2}(\QQ)  \cap   K    $ is the usual congruence subgroup of matrices in $ \SL_{2}(\ZZ) $ that reduce to $  \left ( \begin{smallmatrix} 1 & * \\ & 1  \end{smallmatrix}   \right   )     $ modulo $ N $. Given  a   free    $ (\ZZ/N\ZZ)$-module   $  T   $  of rank $ 2 $,  the $   K  /   K(N)    $-orbit of an isomorphism $ \nu :  (\ZZ/ N\ZZ) ^{2}  \to    T       $ is uniquely determined by the   \emph{first}      basis element  $$  e_{1}    =  \nu(1, 0)  .   $$     %\footnote{The reader should compare the effect of the corresponding map $ \mathrm{S}(N) \to \mathrm{S}_{1}(N)$ in \cite[\S 1.5]{Diamondmodular}.}  
We  recognize  the  curve $ \mathcal{S}_{K} $  as   the smooth geometrically  connected  affine curve over $ \QQ $ commonly  denoted by $ Y_{1}(N) $, which  parametrizes isomorphism classes of elliptic curve with a point of exact order $ N $. If $ N \geq 4 $, $Y_{1}(N)$ is a \emph{fine moduli space}.   
\end{example}   
\begin{remark}  \label{Gammaquot} We  note  for later    that $   \widehat{\Gamma}_{1}(N) \trianglelefteq     \widehat{\Gamma}    _{0}(N) $ and the quotient   is isomorphic to $   (  \widehat{\ZZ} /  N  \widehat{\ZZ}  ) ^{\times} =  ( \ZZ / N \ZZ )^{\times} $. The isomorphism is obtained  by extracting the top left entry of matrices in  $ \widehat{\Gamma}_{0}(N)$.     
\end{remark}    

\begin{remark}  \label{Gamma1'(N)}  The  interested reader may also wonder about the group 
  $$  K' = \widehat{\Gamma}_{1}'(N) =  \left \{   \left (        \begin{smallmatrix} a & b \\ c & d \end{smallmatrix}  \right  )    \in \GL_{2}(\widehat{\ZZ}) \,\middle|\,   d  \equiv  1,  c  \equiv 0\pmod{N}   \right \}   $$
which also satisfies $ \det(K') = \widehat{\ZZ}^{\times} $,   $\SL_{2}(\QQ)  \cap K'  =  \Gamma_{1}(N) $, and therefore  identifies $Y_{1}(N) $ with $ \mathcal{S}_{K'} $. 
However, the  moduli interpretation for $ \mathcal{S}_{K'}(\CC) $   is the set of isomorphism classes of triples $ (A, C, e + C) $ where $ A $ is an elliptic curve, $ C    \subset   A  (   \CC   ) $ is a cyclic subgroup of order $ N $ and $ e+C $ is a point of order $ N $ in $  (A / C)(\CC)     $.  To explain this discrepancy, denote   $ K = \widehat{\Gamma} _{1}(N) $ and   let      
\begin{equation}   w_{N} : =    \left (  \begin{smallmatrix} & -1 \\ N    \end{smallmatrix}   \right  )  \in 
\Gb(\QQ)^{+}   .
\end{equation}  
Then $w_{N}^{2}  \in \mathbf{Z}(\QQ) $ and $ w_{N} K  w_{N}^{-1} =  w_{N}^{-1} K w_{N}  =  K'  $.  This gives us a commutative diagram 
\begin{equation}   \label{atkin}      
\begin{tikzcd}[sep = large]      Y_{1}(N)   \arrow[r, "{j}"]   \arrow[d, "W_{N}"']    &       \mathcal{S}_{K}   \arrow[d, "{[w_{N}]_{K}}"]  \\
Y_{1}(N)    \arrow[r, "{j'}"]   &   \mathcal{S}_{K'}
\end{tikzcd}   
\end{equation}   
where $ j, j' $ are induced  by $  \tau \mapsto (x_{\tau} ,  1)  $ and $ W   _ {  N   }     $ is the \emph{Fricke involution} induced by $$ \mathcal{H}^{+} \to \mathcal{H}^{+}  ,  \quad  \tau \mapsto -1/(N\tau)  .    $$    
In the  moduli-theoretic terms,  the   effect of $ W_{N}    $ is via $ [A, Q] \mapsto [A/\langle Q \rangle , P + \langle Q \rangle ]  $,   where $ P \in A[N](\CC)     $ is any  point  that  satisfies the Weil pairing relation $ \langle P, Q \rangle = e^{2\pi i/N} $, where our pairing is normalized as  in \cite[p.~80]{Diamondmodular}, i.e., the basis of $ A_{\tau}[N]   (\CC)    $ corresponding to $ (\tau, 1) $ (or $(1, -\tau)$) is paired to $ e^{2\pi i/N} $.     Similarly,     $ j' $    sends $ (A', Q') $ to $ (A', \langle Q' \rangle, P'  +  \langle Q   '  \rangle  ) $ where $ P' \in A'[N]  (  \CC   )     $ is any point that satisfies $ \langle P',   Q' \rangle = e^{2\pi i/N }$.  
\end{remark} 

We end this   subsection by recording the following result, which makes the effect of degeneracy and twisting maps more explicit for  geometrically  connected    modular  curves. For a level $ L $, let $ \Gamma_{L} $ denote the intersection  $ \Gb(\QQ)^{+}  \cap  L $. We call   $$    \Gamma_{L}   \backslash  \mathcal{H}^{+}    \hookrightarrow     \mathcal{S}_{L}(\CC),  \quad  \quad   \Gamma_{L} \tau  \mapsto  [\tau, 1]_{L}  $$ 
the \emph{standard embedding}.    

\begin{lemma}  \label{translation}  Let $ L $, $ K $  be   two levels of $ \Gb(\Ab_{f}) $  such that  $  \mathcal{S}_{L} $ is geometrically connected and let $ g \in \Gb(\Ab_{f}) $ be  an element such $ K'= g^{-1}  K g  $ contains $ L  $.  %
Then   % $ \Gb(\QQ)^{+} \cap g^{-1}K  $ is non-empty and 
the composition $ [g^{-1}]_{K'} \circ \pr_{L, K'}  :  \mathcal{S}_{L}  \to \mathcal{S}_{K}$ on $ \CC $-points  is identified via the standard embeddings     with 
\begin{align*} \Gamma_{L}  \backslash  \mathcal{H}^{+}  & \to   \Gamma_{K}   \backslash  \mathcal{H}^{+} ,  \quad  \quad   
\Gamma_{L}  \tau    \mapsto   \Gamma_{K} q  \tau 
\end{align*} for any element $ q \in  \Gb(\QQ)^{+} \cap  K  g    $.  %In particular, the twisting isomorphism $ [g]_{K'}(\CC)  $ for any $ g \in \Gb(\QQ)^{+}$ is induced by the left action of $ g^{-1} $ on $ \mathcal{H}^{+}$.       
\end{lemma} 
\begin{proof}   Since $ \mathcal{S}_{L} $ is geometrically connected, so is $ \mathcal{S}_{K'} $ and $ \Gb(\QQ)^{+} \backslash \GL_{2}(\Ab_{f})/ K' $ is a singleton. In particular, $ 1  \in    \Gb(\QQ)^{+} g  K   '   = \Gb(\QQ)^{+} K  g    $. So we can write $ 1 = q^{-1} k g $ for some $ q \in \Gb(\QQ)^{+} $  and   $ k \in K $. Then $ q \in  K g $ and  $$   K'  =  q  ^{-1}  K    q  , \quad    \quad      [g^{-1}]_{K'}(\CC)  = [q^{-1}]_{K'} (\CC)  .  $$   Now  $ \Gb(\QQ)^{+} \backslash \Gb(\Ab_{f})/ K $ is a singleton as well  since  $ \mathcal{S}_{K'} \simeq   \mathcal{S}_{K}   $. So we know that  $$ \Gamma_{\star}\backslash \mathcal{H}^{+} \to \mathcal{S}_{\star}(\CC), \quad \quad   \Gamma_{\star}\tau \mapsto [\tau, 1]_{\star} $$ is an isomorphism   for each $ \star \in \left \{ K, K',   L    \right \} $.  
Using this, we see that $ \pr_{L, K'} $ is identified with $ \Gamma_{L} \tau \mapsto  \Gamma_{K'} \tau $  and    $ [g^{-1}]_{K'}(\CC)  =  [q^{-1}]_{K'}(\CC)  $ is identified with $   \Gamma_{K'} \tau \mapsto  \Gamma_{K}  q  \tau $.   
\end{proof}  

\subsection{CM  points}
\label{CMpointsubsection}   We now describe certain algebraic points on the modular curves that determine the ``canonicity" of the model $ \mathcal{S}_{K} $ in the  Deligne-Shimura   formalism.    

Let  $  P = [x, g]_{K}  \in \mathcal{S}_{K}(\CC) $ be a point. We say that $ P $ has \emph{complex multiplication}  (CM) by $ E  $   if one (and therefore any) pair $ (A, \eta) $ representing the class in $   ( \mathcal{E}/\! \! \sim)/ K  $ attached to   the  point     $  P  $  under   (\ref{moduliisog}) satisfies   $$   \mathrm{End}(A) \otimes_{\ZZ} \QQ = E.   $$    
If $ \tau  \in \mathcal{H}^{\pm} $ corresponds to $ x $  under  (\ref{variousiden}),    the     associated elliptic curve  $  A_{\tau} $  with $ \CC $-points $   \CC /  \Lambda_{\tau}   =       \CC / (\ZZ + \ZZ \tau ) $ has CM by $ E $ if and only if $ \QQ[\tau] = E $ \cite[Proposition 3.17]{MilneElliptic}.  Suppose this is the case.    Let $ g_{\tau} \in \Gb(\RR) $ denote the change of coordinates matrix from $ (1, -i) $ to $ (1, -\tau ) $. It is  easy to check that  $ \tau  = g_{\tau} \cdot i  $,  so that   
\begin{equation}   \label{gtau}    x =  g_{\tau} h_{\mathrm{std}}  g_{\tau}^{-1}   
\end{equation}    
Now if $ q \in \Gb(\QQ) $ denotes the change of coordinates matrix from $ (1, -\tau) $ to $ (1, \omega_{2}/\omega_{1}  ) $, then $  g_{0} =  q g_{\tau} $ is the matrix in (\ref{g0}).   So     $$   x  = g_{\tau} h_{\mathrm{std}} g_{\tau}^{-1}  = q ^{-1} h_{0}    q   $$
which implies that $ P  = [h_{0}, q   g]_{K} . $ Since $ \QQ[\tau] = \QQ[\bar{\tau}]$, the change of coordinates matrix from $ (1, -\tau) $ to $( 1, - \bar{\tau}    ) $ is in $ \Gb(\QQ) $ and  it  easily  follows that  $$ \bar{\tau} \in  \Gb(\QQ) \tau  .   $$     So   we  can also write    % and we can write 
$ P = [\bar{h}_{0}, g']_{K} $ for some $ g ' \in \Gb(\Ab_{f}) $.    Thus the set of points on $ \mathcal{S}_{K}(\CC) $ with CM by $ E $   is    % is  exactly     
\begin{equation}   \label{pkc}     \mathcal{P}_{K}   :  =   \left \{ [h_{0}, g]_{K}    \, | \,   g \in \Gb(\Ab_{f})  \right \}  =    \left \{ [\bar{h}_{0}, g]_{K}    \, | \,   g \in \Gb(\Ab_{f})  \right \}  
\end{equation}   
Lemma \ref{uniqueCMpoint} characterizes  the points $ h_{0}, \bar{h}_{0}\in \mathcal{X}_{\mathrm{std}}  $ in   terms of the morphism  (\ref{embedding}).

Since $ E $ is fixed in our discussion, we will refer to elements of $ \mathcal{P}_{K}  $ simply as \emph{CM points}.  We  observe  that $ \mathcal{P}_{K} $ depends only on the $  \Gb (\QQ) $-conjugacy class of $ \varphi $ (\ref{varphi}). %  
Indeed,  if we replace $ \varphi $ by $ q \varphi q ^{-1} $ for $   q \in   \Gb    ( \QQ ) $, then we  end up replacing $ h_{0} $ with  $  q h_{0} q^{-1} $.    Thus the set of points on $ \mathcal{S}_{K} $ that have CM by $ E $ depends only on the datum    (\ref{SDdatum}).    

We may  also  reinterpret  the set of CM-points as the images of all possible  \emph{twisted  embeddings}
\begin{equation}    \label{Deligneig}    
\begin{split}   \iota    _ {  g   }  (\CC)  : \mathcal{T} _{ H_{g}   } (  \CC  )   &      \hookrightarrow  \mathcal{S}_{K}   (  \CC  )    \\
    [ h]   &      \mapsto  [   h_{0}   ,     h    g  ] _{ K   }       
    \end{split}    
\end{equation}   
where $  H_{g} = H_{g, K}    :   =   \Hb(\Ab_{f}) \cap  g  K  g  ^  { - 1   }  $.    Each  $ \iota_{g}(\CC) $ is a morphism of underlying $ \CC $-schemes and the theory of  canonical models stipulates that it descends to a morphism $$ \iota_{g} : \mathcal{T}_{H_{g}} \to \mathcal{S}_{K, E} $$ of $ E $-schemes \cite[Corollary 5.4]{DeligneTS}. Hence the images $ [h_{0}  ,    h    g] $ are algebraic points on $ \mathcal{S}_{K}(\CC) $ whose field of definition can be computed using the explicit Galois action prescribed in \S \ref{canonicalmodsec}. More precisely,  if   $ \sigma \in  \Gal    (   E ^  { \mathrm{ab} }  /  E   )  $  and  $   h \in \Ab_{E, f}^{\times} $ is such that $  \Art _{E}   ( h )    =  \sigma $ under (\ref{CFTforEnoC}), then 
\begin{equation}  \label{GaloisonCM} \sigma [ h_{0}   , g  ]   _ { K }     =     [   h_{0}     ,   h   g  ]  _ { K  }  .     
\end{equation} 
Thus  $  [h_{0}, g]  _ { K  }   $ is defined over the field $ E _ { H_{g  }  }  $ that is associated  with  the group  $  H _ { g}   \subset  \Hb    ( \Ab _ {f }   ) $ via   (\ref{CFTforEnoC}), namely the fixed field of the subgroup $  \mathrm{Art}_{E}    (E^{\times} \backslash E^{\times} H_{g})    \subset  \Gal(E^{\mathrm{ab}} / E )$. The $ \Gal(\overline{\QQ}/E) $-orbit of $ [h_{0} , g] _ {  K   }  $ is then identified with the Galois set $ \mathcal{T}_{H_{g}}(\CC) $.   

\begin{remark} Suppose $ K $ contains the subgroup of $ \widehat{\ZZ}^{\times} $ of diagonal matrices in $ \GL_{2}(\widehat{\ZZ})$, e.g., $ K  = \widehat{\Gamma}_{0}(N) $.    Then     so    does $ H_{g}   $. Therefore the field $ E_{H_{g}} $   
is  fixed  by   the image of the  Verlagerung map $$ \mathrm{Ver} :   \Gal(\QQ^{\mathrm{ab}}/\QQ)   \to  \Gal(E^{\mathrm{ab}}/E) .   $$     
 Any such extension $ F $ of $ E $ is Galois over $ \QQ $ and its Galois group over $ \QQ $ is \emph{generalized dihedral}, i.e.,  the conjugation action of $ \Gal(E/\QQ)$ on $ \Gal(F/E) $ is via inversion and we have an isomorphism $$ \Gal(F/\QQ) \simeq  \Gal(F/E)  \rtimes   \Gal(E/\QQ)  $$
corresponding to each choice of a section of $ \Gal(F/\QQ) \to \Gal(E/\QQ) $. We refer the reader to \cite[\S 3.2]{Lars} for  more  detailed  results describing various interrelated extensions of this type.    
\end{remark}

\begin{note}  \label{DeligneCM}  In Deligne's formalism, the canonical model for the $ \CC $-scheme $ \mathcal{S}_{K, \CC} $ described by (\ref{Shimuravar}) is defined to be a scheme $ M_{K}  $ over the reflex field $ \QQ $ such that 
\begin{itemize} 
\item there are isomorphisms  $$  M_{K}  \times_{ \mathrm{Spec} \,  \QQ} \mathrm{Spec} \,  \CC \simeq \mathcal{S}_{K, \CC}  $$  that are ``compatible" for varying $ K$, and 
\item for every imaginary quadratic field $ E $ and $ g \in \Gb(\Ab_{f}) $,  there exist $ E $-scheme morphisms  $$ \iota_{g} : \mathcal{T}_{H_{g}} \to   M    _{K}   \times_{\Spec   \QQ}  \Spec   E   $$     (where the $ E$-scheme structure on $ \mathcal{T}_{H_{g}} $ is determined by  (\ref{reciprocitylaw})) whose base change to $ \CC $ is given by $ \iota_{g}(\CC) $ (\ref{Deligneig}) on $ \CC $-points.   
\end{itemize} 
For the precise meaning of the word “compatible,” see \cite[\S 3]{DeligneTS}.      
Deligne’s axiomatic characterization of the canonical models $ ( M_{K} )_{K}     $     allows the arithmetic properties of these varieties across all levels to be packaged in an efficient and elegant way. The \emph{existence} of the canonical model of modular curves, however, is    still     established by studying the moduli functors of elliptic curves with level structure. From  this    perspective, the Galois action described in~(\ref{GaloisonCM}) is essentially the theory of complex multiplication in disguise.
\end{note}

\begin{remark}   \label{weirdCMdivisorremark}  For the alternative  datum  (\ref{weirdShimura}), note that   the   embedding $  \iota ' $ given by the transpose inverse of the embedding  (\ref{iotaembedding})  also upgrades to a morphism $$ \iota'   :   (\Hb, \left \{ h_{0} \right \} ) \to  (\Gb ,  \mathcal{X}_{\mathrm{std}}')   $$ 
of  Shimura data.    
Let $ \mathcal{P}_{K}'   \subset \mathcal{S}_{K}'(\CC)  $ be the set of points that  have CM by $ E $ under the moduli interpretation mentioned   in  Remark  \ref{weirdmoduli}. Then $$ \mathcal{P}_{K}' =    \phi _ {{}^{t}   \negsmall K}  (\mathcal{P}_{{}^{t}   \negsmall    K})  =  \left \{ [h_{0}',  g]_{K} \, | \, g \in \Gb(\Ab_{f})  \right \} $$ 
where $ h_{0}' = \iota'_{\RR} \circ h_{0} \in \mathcal{X}_{\mathrm{std}}' $.  

To  see   that      the     map $ \phi' $ (\ref{veryweirdiso})   does not respect Galois actions, let us assume for simplicity that $ E = \QQ(i) $, and $ (\omega_{1}, \omega_{2}) = (1, -i) $, so that $ h_{0}  =  h_{\mathrm{std}}'  $ and $ h_{0}' = h_{\mathrm{std}}' $.  Then the map $ \phi'_{K}(\CC) $ given in  (\ref{veryweirdiso})  restricts to   
\begin{equation}   \label{weirdCMmap}   \mathcal{P}_{K}  \to   \mathcal{P}_{K}  '  ,  \quad  \quad     [h_{0}, g]_{K} \mapsto [h_{0}', g]_{K}  
\end{equation}   Now the theory of canonical models  requires that   \begin{align*} \iota'_{g}  :   \mathcal{T}_{H_{g}}(\CC)    \hookrightarrow  \mathcal{S}_{K}'(\CC) , \quad  \quad    
[h]  \mapsto  [h_{0}   '    ,   \iota'(h)    g]_{K}  
\end{align*}   
respects the  Galois action determined by the reciprocity law  (\ref{reciprocitylaw}). Since  the Galois action on $ \mathcal{P}_{K}$ is via (\ref{GaloisonCM}) and since  $    \iota(h)   \neq   \iota'(h)   $  in  general,  the mapping  (\ref{weirdCMmap})    cannot be    Galois  equivariant.      
\end{remark}   
  
\subsection{Heegner   points}   \label{Heegnersec}      Let us now  connect the  general CM points defined  in   \S \ref{CMpointsubsection}  with   the    Heegner points   from the introduction.     Suppose for all of this subsection  that $ K = \widehat{\Gamma}_{0}(N) $ for some $ N \geq 1 $ as defined in   Example      \ref{Gamma0(N)}.    
Fix a point $ P =  [   h_{0}       , g ]_ { K  } \in \mathcal{P}_{K} $. 
Then  $ H_{g}  =  \Hb(\Ab_{f}) \cap gKg^{-1} $ equals  
%The elliptic curve $  A_{q x_{\iota}} $ has endomorphism ring  
%
$$ \mathrm{Stab}_{\Hb(\Ab_{f})}(g L_{f,1}) \cap \mathrm{Stab}_{\Hb(\Ab_{f})}(g L_{f,N})
$$
where $ \Hb(\Ab_{f}) = \Ab_{E,f}^{\times} $ acts on the $ \widehat{\ZZ}$-lattices $ g L_{f,1}$ and $ g L_{f,N} $ inside $ \Ab_{E,f} $ by multiplication. % 
It follows that $ H_{g} $ is the group of units of the ring
$$
\widehat{\mathcal{O}}_{P} : = \left\{ a \in \Ab_{E,f} \, | \,  a   \cdot g L_{f,1}  \subseteq  g L_{f,1} \text{ and } a  \cdot   g L_{f,N} \subseteq g L_{f,N} \right\}   .
$$
It is not hard to see that $ \widehat{\mathcal{O}}_{P}$ equals the product (over all primes $ \ell $) of compact open subrings $ \mathcal{O}_{P, \ell}$  of $ E_{\ell} := E\otimes_{\ZZ}  \QQ_{\ell} $ that properly contain $ \ZZ_{\ell}$. Since $ \mathcal{O}_{\ell} :=  \mathcal{O}_{E} \otimes _{\ZZ} \ZZ_{\ell} $ is the unique maximal compact open subring of $ E_{\ell }$,  we see that $\mathcal{O}_{P, \ell} \subseteq  \mathcal{O} _{\ell} $. Thus   $ \widehat{\mathcal{O}}_{P} $ is a compact open subring of $  \widehat{\mathcal{O}}_{E} = \mathcal{O}_{E} \otimes \widehat{\ZZ}  $ that properly contains $ \widehat{\ZZ}$.    Since $ E $ is dense in $ \Ab_{E, f} $,  the intersection $$ \mathcal{O}_{P}  =   E  \cap \widehat{\mathcal{O}}_{P}   $$
is dense in $ \widehat{\mathcal{O}}_{P} $ (i.e.,   $ \widehat{\mathcal{O}}_{P}  =  \mathcal{O}_{P}  \otimes_{\ZZ}  \widehat{\ZZ} $),  and  we   have     $  \ZZ  \subsetneq  \mathcal{O}_{P} $.  The  upshot  is  that    $ H_{g}$ is the group of units of the profinite completion of  an   \emph{order} in $ \mathcal{O}_{E} $ (i.e., a subring of $ \mathcal{O}_{E} $ of rank $ 2 $ over $ \ZZ $)  and $ P $ is defined over the \emph{ring class extension of $ E $ associated    with      $ \mathcal{O}_{P}$}. Similarly,  the    compact open subrings $$ \widehat{\mathcal{O}}_{P}^{\dagger} :=  \left \{ a \in \Ab_{E, f} \, | \, a  \cdot g L_{f, 1}   \subseteq   L_{f, 1  }  \right \}  ,  \quad   \quad   \widehat{\mathcal{O}}_{P}^{\ddagger}  : = \left \{ a \in \Ab_{E, f} \, | \,  a   \cdot  g L_{f, N} \subseteq   g L_{f, N }  \right \}  $$ 
of $ \Ab_{E, f} $  respectively    arise        from orders $ \mathcal{O}_{P}^{\dagger} $, $ \mathcal{O}_{P}^{\ddagger} $ in $ \mathcal{O}_{E} $     obtained by taking   the  intersections of the adelic subrings with $ E  $.  Clearly,     $$  \mathcal{O}_{P}^{\dagger}  \cap  \mathcal{O}_{P}^{\ddagger}  =  \mathcal{O}_{P}  .    $$    Let us define    $$ \mathfrak{a}_{P} : =  g L_{ f, 1} \cap E  , \quad \quad   \mathfrak{b}_{P}  :  =  g L_{f,  N }  \cap  E.  $$
It is  straightforward to see  that   $ \mathfrak{a}_{P} $ and   $  \mathfrak{b}_{P}$ are proper (and therefore invertible) fractional ideals of $ \mathcal{O}_{P}^{\dagger} $ and  $ \mathcal{O}_{P}^{\ddagger} $   respectively.  Note that $ N \mathfrak{a}_{P} \subset \mathfrak{b}_{P}$ and   $   [  \mathfrak{b}_{P} :  N  \mathfrak{a}_{P} ] = N $.

\begin{lemma}   \label{CMX0N}      % The lattices $ \mathfrak{a}_{P}$, $ \mathfrak{b}_{P}$ are proper fractional id 
If $ A \to A' $ is the cyclic $ N $-isogeny representing the point $ P  $, then $ \mathrm{End}(A) = \mathcal{O}_{P}^{\dagger}$  and $ \mathrm{End}(A') =  \mathcal{O}_{P}^{\ddagger} $. Moreover, the point $P  $ can be represented by  the cyclic   $ N $-isogeny given on $ \CC $-points    by   $ \CC/N\mathfrak{a}_{P} \to  \CC/\mathfrak{b}_{P}    $, $ z +   N \mathfrak{a}_{P}  \mapsto   z   +   \mathfrak{b}_{P}$.    
\end{lemma}   
\begin{proof} Let $ (A, C) $ be the pair representing the class associated with $ P $,  where $ A $ is an elliptic curve and $ C $ is a cyclic subgroup of $ A(\CC) $ of order  $N $.   Recall that $ \tau_{0} $ (\ref{tau0})  denotes the point in  $ \mathcal{H}^{\pm} $ associated to $ h_{0} $.   Since $ \Gb(\QQ) K  =   \Gb (\Ab_{f})  $,   we can  write $ g =  q   \kappa $ for $ q \in  \Gb(\QQ) $ and $ \kappa \in K $ and  so $$ P  =  [  q ^{-1}  h_{0}    q  ,  \,   1    ]_{K}   .$$  Therefore,  $ A $ is isomorphic to the elliptic curve $ A_{\tau} $ where $ \tau  = q^{-1}  \tau_{0} \in  \mathcal{H}^{\pm} $ denotes the point associated to $ q ^{-1}   h_{0}     q     $.  Write  $ q    =   \left   (  \begin{smallmatrix}   a &  b \\  c  &  d   \end{smallmatrix}   \right   )        $   and   set  $  \varpi_{1}  = a \omega_{1} + c \omega_{2} $, $  \varpi_{2}      =   b\omega_{1} + d \omega_{2}  $.  Then    $ \mathfrak{a}_{P} = \ZZ \varpi_{1} + \ZZ  \varpi_{2} $ and   $$   \frac{\varpi_{2}}{\varpi_{1} }  =  \frac{d\tau_{0}-b}{c\tau_{0}-a} =  -  q^{-1} \cdot (\tau_{0}) =    -  \tau  .       $$ 
%and % By changing the basis $ (\omega_{1}, \omega_{2} ) $ (which changes neither $ K $, nor the set of CM points), we may assume that $ q = 1  $. 
 %and set  $ \tau := - (\omega_{2}/q\omega_{1} ) \in \mathcal{H}^{\pm }  $.  
So we see that 
\begin{align*}   \mathcal{O}_{P}   ^ { \dagger   }     &  =  \left \{ a \in \mathcal{O}_{E} \, | \, a   \cdot     \mathfrak{a}_{P}  \subseteq  \mathfrak{a}_{P}   \right \}   \\       
%\\ & =   \left \{ a \in \mathcal{O}_{E} \, | \, a    \cdot      q   (\ZZ  \omega_{1} + \ZZ  \omega_{2} ) \subseteq   q    (    \ZZ  \omega_{1}  +  \ZZ  \omega_{2}    )       %\omega_{2}   
%\right \}  \\
& =   \left   \{  a \in   \mathcal{O}_{E}  \, | \,  a   \cdot    (\ZZ   \varpi_{1}   + \ZZ   \varpi_{2}     )  \subseteq   \ZZ    \varpi_{1}      +   \ZZ    \varpi_{2}     \right \}   \\
& =    
\left   \{  a \in   \mathcal{O}_{E}  \, | \,  a   \cdot    \Lambda_{\tau}   \subseteq    \Lambda_{\tau} \right \}   \\
&   = 
\mathrm{End}(A_{\tau} ) . 
\end{align*}   
Now set $  \varpi_{1}'  = a \omega_{1} +   c N \omega_{2} $, $   \varpi_{2} '   =   b  \omega_{1} +  dN \omega_{2}  $.  Then $ \mathfrak{b}_{P} =  \ZZ \varpi_{1}'  +  \ZZ \varpi_{2} '  $  and     $  \varpi_{2}'  /    \varpi_{1}'   =  -N\tau   $    by  a   similar   computation.     From the discussion in Example  \ref{Gamma0(N)}, we see that $ (A, C) $ is isomorphic to  $  (A_{\tau}, C_{\tau}) $ where  $ C_{\tau}  = \langle 1/ N +  \Lambda_{\tau}   \rangle  $.  It  follows   that    $  A'  =    A/C $ is isomorphic   $ A_{N\tau} $  and      we    similarly deduce that 
\begin{align*}   \mathcal{O}_{P}^{\ddagger} 
& =   \left \{ a \in  \mathcal{O}_{E}  \, |  \,   a \cdot \Lambda_{N \tau}  \subseteq    \Lambda_{ N \tau  }   \right   \}      =    \mathrm{End}(A/C) 
\end{align*}
This proves the  first  claim.    
Since the isogeny $  A \to A/C $ is identified with the isogeny $ A_{\tau}  \to A_{N\tau} $  given on $ \CC $-points via  $$ \CC / \Lambda_{\tau}  \mapsto  \CC /  \Lambda_{N \tau} ,   \quad  z + \Lambda_{\tau}   \mapsto  N z  +  \Lambda_{N \tau} , $$    
the   second   claim    also    follows.   
\end{proof}

\begin{definition} We say the CM-point $P$ is a \emph{Heegner point}  if  $ \mathcal{O}_{P}^{\dagger} =  \mathcal{O}_{P}^{\ddagger} $.  The  endomorphism ring $ \mathcal{O}_{P}$ is then  called the \emph{order} of the Heegner point.     %  The \emph{order} of Heegner point $ P $ is the order $ \mathrm{\OO}_{P}$. 
\end{definition}
\begin{remark}  Suppose $ (A, C) $ represents a CM point $ P $ on $ \mathcal{S}_{K}(\CC)$.  Then  the   quotient $ A / C $ has endomorphism ring $ \mathcal{O}_{P}^{\dagger}  $ if and only if $ C  =    
A[\mathfrak{N}_{P}] $ for some invertible ideal $ \mathfrak{N}_{P}   \triangleleft \mathcal{O}_{P}^{\dagger}  $ (necessarily of index $ N $). For the maximal   order    $ \mathcal{O}_{E}$,  
ideals of index $ N $ exist   precisely when the discriminant $ D = \mathrm{disc}(E) $ (not assumed to be coprime to $ N $) can be written as $ B^{2} - 4 N A $ for integers $ A, B $ with $ \gcd(N, B, A) = 1 $ \cite[\S 2]{Gross}.   If  this is the case, then ideals of index $ N $ exist for all  orders in $ \mathcal{O}_{E}    $.    % A  sufficient criteria is the Heegner hypot 
\end{remark}   

For the next result,  we    assume  that the   \emph{Heegner hypothesis}  is satisfied, i.e., all primes dividing $ N $ are  split in $ E $.  For each $ \ell \mid N $, let $ \beta_{1} $, $ \beta_{2 }$ denote the two local idempotents in $  E _{\ell}   =   E\otimes_{\QQ} \QQ_{\ell}   \simeq  \QQ \oplus  \QQ_{\ell}     $ and let $ k_{\ell} \in  \Gb(\QQ_{\ell}) $ denote the change of coordinates matrix from $ (\beta_{1}, \beta_{2}) $ to $ (\omega_{1} \otimes 1,  \omega_{2} \otimes 1)  $. Define   
\begin{equation}   \label{gN}   g _{N}   \in  \Gb(\Ab_{f} ) 
\end{equation}    to be the element such that the component of $   g_{N}    $ at $ \ell $ is $ k_{\ell} $ if $ \ell \mid N $ and is $ 1 $   otherwise.       
\begin{lemma}   \label{Heegnerlemma}     The point $ [h_{0}, g_{N}]  _ {  K   }     $ is a  Heegner point of  maximal  order.      
\end{lemma}    
\begin{proof} For each $ \ell \mid N $, the map  $ k_{\ell} : E_{\ell} \to E _{\ell} $ is the $ \QQ_{\ell}$-linear map that sends $ \omega_{i} $ to $ \beta_{i} $. Hence, it sends the lattice $ \ZZ_{\ell} \omega_{1}  +    \ZZ_{\ell}  N   \omega_{2} $ to $ \ZZ_{\ell}  \beta_{1}  + \ZZ_{\ell}   N     \beta_{2} $.    It is  then  easy  to from this and Lemma  \ref{CMX0N} that    $ \mathcal{O}_{P} =  \mathcal{O}_{P}^{\dagger} =   \mathcal{O}_{P}^{\ddagger}  =  \mathcal{O}_{E} $.          
\end{proof}

\subsection{Adelic Hecke operators}   \label{Heckecorrsec}      A \emph{Hecke operator}  of level $ K $ associated with $ g \in \Gb(\Ab_{f}) $ is defined to be the characteristic function of the  double coset $ K g K $ and denoted $ \ch(KgK) $. That is, 
$ \ch(KgK) : \Gb(\Ab_{f}) \to \ZZ $ is the function $$ h \mapsto \begin{cases}  1 &  \text{if } h \in K g K \\ 
0 & \text{otherwise} 
\end{cases} 
$$ 
In particular, a Hecke operator is a compactly supported function on $ \Gb(\Ab_{f}) $ that   is     invariant under the left and right translation actions    of $ K $ on the domain. We denote the $ \ZZ $-module of all compactly supported $ K $-biinvariant functions by $$ \mathcal{H}_{\ZZ} (K \backslash \Gb(\Ab_{f})/ K) . $$
Clearly,  the   Hecke operators $ \ch(K g K ) $ for $ g $ running over  representatives  of $ K \backslash \Gb(\Ab_{f}) / K $ form a $ \ZZ$-basis for $ \mathcal{H}_{\ZZ}(K \backslash \Gb(\Ab_{f})  /  K) $.  This free  $ \ZZ $-module can  be endowed with a product operation known as \emph{convolution} as follows. Note that for each $ g \in \Gb(\Ab_{f}) $,  the  coset $ KgK/K $ is a finite set, since $ K g K $ is compact and the $K $-left cosets provide an open cover. Suppose  that $ KgK =  \sqcup_{i} \alpha_{i} K $ and $ K h K = \sqcup_{j} \beta_{j} K $ is a decomposition into left cosets.  We define the convolution operation   $   *   $     by  $$ \ch(KgK) * \ch(KhK)  : = \sum \nolimits _{ i , j }  \ch ( \alpha_{i} \beta_{j} K)  . $$
It is easy to see that the right hand side is independent of the choice of representatives $ \alpha_{i}$, $ \beta_{j} $ and the sum is a compactly supported function on $ \Gb(\Ab_{f}) $ that is $ K $ invariant under translations on both the left and the   right.  With the convolution operation, $ \mathcal{H}_{\ZZ}(K \backslash \Gb(\Ab_{f}) / K ) $ becomes a unital associative $ \ZZ$-algebra which is referred to as the \emph{Hecke algebra of level $K$}. Since $  K $ is fixed in our discussion, we will refer to  $  \mathcal{H}_{\ZZ}(K \backslash \Gb(\Ab_{f}) / K ) $ simply as the Hecke algebra.    Given an  operator $ \ch(KgK) $, its  \emph{transpose} is defined to be $$   \ch(K  g  K  ) ^ { t }  =   \ch(K g^{-1} K) .  $$    We can extend this operation $\ZZ$-linearly to the full Hecke algebra of level  $ K $, and it is easily  verified  that this induces   an    anti-involution on $  \mathcal{H}_{\ZZ}(K  \backslash   \Gb(\Ab_{f} )  /    K )    $.

\begin{remark} It is possible to define Hecke algebras in a more measure theoretic manner, e.g., see \cite[\S 4.1]{BushHenn} or \cite[\S 2.3]{CZE}. An alternative used in some sources (e.g., \cite[\S 3.4]{CorVastal})  is to consider certain endomorphisms of the $ \ZZ$-module $ \mathcal{C}_{\ZZ
}(\Gb(\Ab_{f})/K) $ of all right $ K$-invariant compactly supported functions on $ \Gb(\Ab_{f}) $.  This module has a left action of $ \Gb(\Ab_{f}) $ defined by $ g \cdot \ch(g_{1}K) = \ch(gg_{1}K )$ for $ g , g_{1} \in \Gb(\Ab_{f}) $ and one can consider the algebra $$\mathrm{End}_{\Gb(\Ab_{f})} \left ( \mathcal{C}_{\ZZ}(\Gb(\Ab_{f})/K)  \right ) $$ of 
all $ \Gb(\Ab_{f}) $-equivariant endomorphisms of $ \mathcal{C}_{\ZZ}(\Gb(\Ab_{f})/K) $.     Any such endomorphism is uniquely determined by its effect on $ \ch(K) $ and sends $ \ch(K) $ to an  element in  $  \mathcal{H}_{\ZZ}(K \backslash \Gb(\Ab_{f}) / K )$. The resulting $ \ZZ$-linear bijection gives an identification   $$ \mathrm{End}_{\Gb(\Ab_{f})} \left ( \mathcal{C}_{\ZZ}(\Gb(\Ab_{f})/K)  \right ) \simeq  \mathcal{H}_{\ZZ}(K \backslash \Gb(\Ab_{f}) / K )^{\circ}  $$
of $ \ZZ$-algebras, 
where $ \mathcal{H}_{\ZZ}(K \backslash \Gb(\Ab_{f}) / K )^{\circ} $ denotes the opposite algebra. See  \cite[\S 3.1]{VignerasBook}. 
\end{remark}

Recall that a divisor on an algebraic curve is a finite linear combination of its points. The group of complex divisors $ \ZZ\langle \mathcal{S}_{K}(\CC)\rangle $ admits  actions of $ \mathcal{H}_{\ZZ}(K \backslash \Gb(\Ab_{f}) / K ) $ via \emph{Hecke correspondences} in two possible ways.  Let $ g \in \Gb(\Ab_{f}) $ and denote $ L =  g^{-1} K g  \cap   K  $. Then we have a diagram of $ \QQ $-schemes 
\begin{equation}   \label{Heckecorrdiagram}    
\begin{tikzcd} &  \mathcal{S}_{L} \arrow[dl,  "\alpha"']   \arrow[dr, "\beta"] \\
\mathcal{S}_{K}   &  & \mathcal{S}_{K} 
\end{tikzcd}
\end{equation} 
where the  finite     maps $ \alpha $, $ \beta $ are defined on $ \CC $-points via 
\begin{align*} \alpha : [x, g_{1}]_{L}   &  \mapsto    [x , g_{1}]_{K} ,  \\ 
\beta : [x , g_{1}]_{L}   &      \mapsto [x , g_{1} g^{-1}]_{K}   
\end{align*}    
That is $$ \alpha = \pr_{L, K },   \quad \quad   \beta  =  [g^{-1}]_{K'}  \circ  \pr_{L,K'} $$
where $ K ' = g^{-1} K g $.    
The \emph{contravariant} and \emph{covariant} Hecke actions of $ \ch(K gK )$ on $ \ZZ \langle \mathcal{S}_{K}(\CC) \rangle $ are the maps    
$$  \ch(KgK)^{*} =  \beta_{*} \circ  \alpha^{*} , \quad \quad  \ch(KgK)_{*}  =  \alpha_{*} \circ  \beta^{*}  $$
respectively. Here, $ \alpha^{*}, \beta^{*} $ respectively denote the (flat) pullback of divisors induced by  $ \alpha $, $ \beta $ and  $ \alpha_{*}, \beta_{*} $ denote (proper)  pushforwards. The diagram   (\ref{Heckecorrdiagram}) can also be drawn as 
\begin{equation}   \label{Heckecorrdiagdiff} 
\begin{tikzcd}%[sep = large]  
  &   \mathcal{S}_{gLg^{-1} } \arrow[dl,"{\tilde{\alpha} = [g]}"']  \arrow[dr, "{\tilde{\beta} = \pr}"]       \\ 
\mathcal{S}_{K}&   &      \mathcal{S}_{K}    
\end{tikzcd}   
\end{equation} 
which allow us to define $$ \ch(K g   K )^{*} = \tilde{\beta}_{*} \circ  \tilde{\alpha}^{*} , \quad \quad  \ch(K g K ) _ { * } =   \tilde{\alpha}_{*}  \circ   \tilde{\beta}^{*}   .  $$
It is  clear   from these expressions that both contravariant and covariant Hecke actions depend only on  the class of $ g $ in $ K \backslash  \Gb(\Ab_{f}) /  K  $.   By replacing $ g $ with $ g^{-1} $ in (\ref{Heckecorrdiagdiff}), we recover diagram  (\ref{Heckecorrdiagram}) where the map $ \alpha $ (resp., $ \beta $) is drawn on the right (resp., left). It is  then  also          clear   that    
\begin{equation}  \label{duality} \ch(K g K)_{*} = \ch( K g ^{-1} K )^{*}  .  
\end{equation}  
The \emph{degree} of $ \ch(K g K)^{*} $ is defined to $ \deg(\beta)   = [ K : gK'g^{-1}] $ and that of $ \ch(KgK)_{*}$ to be $ \deg(\alpha) = [K : K'] $. Both of these  equal   $ | KgK/K|   $ by unimodularity of $ \Gb(\Ab_{f})$.      %Note that $ \alpha $ and $ \beta $ are not necessarily unramified. 
By Lemma \ref{pullback}, we find that 
\begin{align}  \label{badHecke1} \ch(K g K)^{*} \cdot [x, g_{1}] _ { K }   &  = \sum_{ \gamma \in K g^{-1} K / K } [ x , g_{1} \gamma ]  _ { K }  \\ 
 \label{badHecke2} \ch(K g K)_{*} \cdot [x, g_{1}]  _ { K } &  = \sum_{ \gamma \in K g K / K } [ x , g_{1} \gamma ] _ { K }  
\end{align} 
for all $ [x, g_{1}] \in \mathcal{S}_{K}(\CC) $.  
It is   easily verified from (\ref{badHecke1}), (\ref{badHecke2})  that  the contravariant action defines a \emph{left} action of $ \mathcal{H}_{\ZZ}(K  \backslash \Gb(\Ab_{f})/ K) $ on the group of divisors whereas the covariant action is a \emph{right} action. More precisely, $$  \ch(K h K)^{*}  \circ  \ch(K g K) ^{*}  = \big ( \ch ( K h K ) * \ch(K g K )  \big ) ^{*} $$ 
where the right hand side denotes the contravariant action of the convolution. 
If $ \mathcal{S}_{K} $ is  geometrically  connected,  then  so  is   $   \mathcal{S}_{ g^{-1} K g \cap K } $   and    one can use Lemma \ref{translation} to translate the effects of the aforementioned  Hecke operators   in terms of points on quotients of upper half plane.   
\begin{remark}    The expressions in (\ref{badHecke1}), (\ref{badHecke2}) can also be derived for certain special levels using the modular interpretation. See    \cite[Prop.~8, Prop.~9]{Rohrlich} 
\end{remark}

\begin{remark} Both covariant and contravariant actions are frequently  used in the literature, and it is important to pay attention to the conventions used in a given source, since results may depend   crucially   on this choice. See, for instance, \cite[\S 5.1]{RibetSerre}, where this distinction plays an important role.     We also refer the reader to \cite[p.~443]{Ribet}  and  \cite[\S 1.16]{NekovarCM} for a similar discussion of Hecke correspondences in the context of Jacobians of algebraic curves.  In the terminology of \cite{Ribet},  the action of $ \ch(KgK)^{*}$ would be in the ``Picard" convention and that of $ \ch(KgK)_{*}$ would be in the ``Albanese" convention.   
\end{remark}   

\begin{remark}   
We take this opportunity to   caution the reader that the expressions (\ref{badHecke1}), (\ref{badHecke2}) for the Hecke actions are somewhat peculiar to the case of zero cycles and do not generalize to cycles on higher dimensional Shimura varieties in the obvious way. See \cite[\S 1]{explicitdescent} for a discussion.   
\end{remark}

\subsection{Classical Hecke operators}   When working in the classical setting of quotients of the upper half-plane, one defines Hecke operators in a manner similar to \S\ref{Heckecorrsec}, except that only elements of the group $\Gb(\QQ)^{+}$ are used.   In the adelic setting, one prefers to work with Hecke operators corresponding to elements that are local at a prime.   The following two examples illustrate how one may express  some   important    classical operators in terms of local elements in $\Gb(\Ab_{f})$. In what follows, $\mathrm{diag}_{\QQ}(x,y)$ for a matrix $\left( \begin{smallmatrix} x \\ & y \end{smallmatrix} \right) \in \Gb(\QQ)$ denotes its image in $\Gb(\Ab_{f})$.

\begin{example}    \label{HeckeGamma0}   Suppose $ K = \widehat{\Gamma}_{0}(N) $ as in Example \ref{Gamma0(N)}. Let $ p $ be any prime  such that $ p \nmid N $ and  pick $$ g = 
\mathrm{diag}_{\QQ}(p, 1)   \in   \Gb(\Ab_{f})   .  $$
Then $   L = g^{-1} K g   \cap   K   =    \widehat{\Gamma}_{0}(Np)   $. 
So by Lemma \ref{translation} applied with $ q = g $, the  
diagram~(\ref{Heckecorrdiagram})   corresponds via the standard identification~(\ref{stdidenGamma0})  to the diagram  
\begin{center} 
\begin{tikzcd}  & Y_{0}( Np)   \arrow[dl,  "\alpha"']    \arrow[dr, "\beta"]       \\

Y_{0}(N)  &   &   Y_{0}(N) 
\end{tikzcd}
\end{center}
where the map $ \alpha $, $ \beta $ are respectively  induced by $ z \mapsto z $, $ z \mapsto pz  $ on $ \mathcal{H}^{+}$.      This is then  exactly the diagram in \cite[Ch.~5, \S7, p.~282]{MilneElliptic}. As in \emph{loc.~cit.},   we denote $  \ch(K g K)^{*} =  \beta_{*} \circ \alpha^{*}  $ 
by  $ T_{p}$. Note that  
$$ Kg^{-1}K = K \big ( \mathrm{diag}_{\QQ} (1, p) \cdot  \mathrm{diag}_{\QQ} (p, p)^{-1}\big )  K $$  
Since $ \mathbf{Z}(\QQ) $ acts trivially on $ \mathcal{S}_{K}(\CC) $ and  $  K (  \mathrm{diag}_{\QQ}(1, p) ) K =   K g K $,    
the relation (\ref{duality})  implies  that     
$$  T_{p} = \ch( K g K)^{*}  =    \ch(K g K )_{*},    $$      
and $ T_{p} $ is often referred to as \emph{self-dual} for this reason. In particular, there is little possibility of confusion when working with Hecke operators away from primes dividing $ N $ in the case of $ \Gamma_{0}(N) $ level structures, and one can define this  operator  entirely  locally  at $ p $ using, e.g.,  $ \mathrm{diag}(p, 1) \in \Gb(\QQ_{p}) $.   %   
We note that   the  degree of  the  operator   $ T_{p} $ is $ p + 1  $.        
\end{example} 

\begin{example}   \label{HeckeGamma1} Suppose $ K = \widehat{\Gamma}_{1}(N) $ as in Example \ref{Gamma1(N)}. As observed in Remark \ref{Gammaquot}, this group is normal in $ \widehat{\Gamma}_{0}(N) $ with quotient  isomorphic to $  (\ZZ /N \ZZ)^{\times}  $. For any integer $ d $ satisfying $ (d, N) = 1 $, let $  \gamma  =  \gamma_{d, N} \in  \widehat{\Gamma}_{0}(N)  $    be any matrix whose top left entry reduces to  $ d $ modulo $ N $.  Then the correspondence~(\ref{Heckecorrdiagram}) for $g = \gamma^{-1}  $ is just the twisting isomorphism
\[
\langle d \rangle : \mathcal{S}_{K} \to \mathcal{S}_{K}, \qquad [x, g_{1}]_{K} \longmapsto [x, g_{1}\gamma]_{K}.
\]
Let us identify $ \mathcal{S}_{K}$ with $  Y_{1}(N)  $ using  (\ref{stdidenGamma1}).  From the action noted in Remark \ref{orbitremark}  and   the  discussion   in  Example  \ref{Gamma1(N)},  it is clear that the right action of $ \gamma $ on the moduli space for $ \mathcal{S}_{K}(\CC) $ sends the class of  pair $ (A,  e_{1})$ to that of  $ (A, d e_{1} ) $.  Thus     $ \langle   d   \rangle $ is  exactly the map  defined in \cite[p.~175, (5.9)]{Diamondmodular}.    %and \cite[\S 5.2]{RibetSerre}.  
The operator $$ \ch(K \gamma  ^{-1}     K)^{*}  =  [\gamma ]_ { K ,   *}   =    \langle d \rangle _{*}    $$ 
is referred to as the \emph{diamond bracket operator} and depends only on the class of $ d \pmod{ N} $.  
An explicit choice of $ \gamma  =   \gamma_{d, N}    \in \widehat{\Gamma}_{0}(N)    $ is one where  the    component at a  prime   $ \ell $  is    
\begin{equation}   \label{explicitchoice}   (\gamma )_{\ell} = \begin{cases}  \mathrm{diag}(d , 1) & \text{ if } \ell  \mid N \\
\quad \quad 1  & \text{otherwise}.      
\end{cases}   
\end{equation}    
 Now let $ p $ be  a prime such that $  p  \nmid  N  $ 
and set   $$ g  =   \mathrm{diag} _{\QQ} (1,p) .   $$   
Then $ \SL_{2}(\QQ) \cap L = \SL_{2}(\QQ) \cap  g^{-1} K g \cap K $ is the subgroup  $  \Gamma_{1}^{0}(N, p)  :=  \Gamma_{1}(N) \cap   {}^{t}\Gamma_{0}(N)  $
where $ {}^{t}\Gamma_{0}(N)$ denotes the transpose of $ \Gamma_{0}(N)$.    
Therefore, Lemma    \ref{translation} applied with $ q = g  $ implies that under the standard  identification   (\ref{stdidenGamma1}),  diagram   (\ref{Heckecorrdiagram})  corresponds to
\begin{center}   
\begin{tikzcd}    &   Y_{1}^{0}(N, p)  \arrow[dl, "\alpha"']    \arrow[dr, "\beta"]   \\
Y_{1}(N) &   &   Y_{1}(N) 
\end{tikzcd}   
\end{center}
where $ Y_{1}^{0}(N, p)(\CC) =  \Gamma_{1}^{0}(N, p)   \backslash \mathcal{H}^{+} $ 
and $ \alpha $, $ \beta $ are respectively induced by the maps $ z \mapsto z $, $ z \mapsto p^{-1} z $ on $ \mathcal{H}^{+} $. The operator $  \ch(Kg K)^ { *   }  =  \beta_{*} \circ \alpha^{*} $  is then the  operator ``$ T_{p} $" defined in \cite[\S 5.2]{Diamondmodular}.\footnote{See  exercises 1.5.6(c)  and 5.2.10 in \cite{Diamondmodular}.} Following the comment on p. 397 of \emph{loc.~cit.}, we denote this operator  by  $ T_{p, *} $. If $$ \sigma_{p} :  = \mathrm{diag}(p, 1) \in \Gb(\QQ_{p})  \hookrightarrow  \Gb(\Ab_{f}) , $$  then   $ KgK = K\sigma_{p}K $  clearly and  so,    
\begin{equation}   \label{Tp_*}      T_{p, *}    =   \ch (  K \sigma_{p}  K  ) ^ {* }  =   \ch( K   \sigma_{p}^{-1}   K )_{*}  .  
\end{equation} 
Let us  denote $  \ch(K g K) _{ * }   =  \alpha_{*}  \circ  \beta^{*} $  by $ T_{p} ^ { * }   $. It is easy to see   $ K g^{-1} K =  K  c^{-1} \sigma_{p}  \gamma K  $     
where $ c $ denotes $ \mathrm{diag}_{\QQ}(p, p) $    and  $ \gamma = \gamma_{p, N} $ is as   in   (\ref{explicitchoice}).  Therefore, 
\begin{equation}   \label{Tprelation}    T_{p}^{*} =  \ch( K g^{-1} K)^{*}   =  \ch(K  \sigma_{p}   K) ^{*}  \circ \ch ( K \gamma  K )^{*}  =  T_{p, *} \circ  \langle p \rangle^{*}    
\end{equation}
which is  consistent with the notation of \cite[Theorem 5.5.3]{Diamondmodular} and agrees with    the relation mentioned in  \cite[\S 2.3.1.1]{RibetSerre}.     Finally, if set $$ \tau_{p} :=   \mathrm{diag}(p, p) \in \Gb(\QQ_{p})  \hookrightarrow  \Gb(\Ab_{f}) , $$ 
then since $ c K =  \tau_{p}  \gamma K $, we can write 
\begin{equation}  \label{diamond_*}   \langle p \rangle_{*}   =   [\gamma]_{K , * }  =     [\tau_{p}^{-1}]_{*},  
\end{equation}   
\end{example}    
\begin{remark}While this is not stated explicitly, the map denoted $\pi_{2}^{(p)}$ in \cite[\S 5.2]{RibetSerre} (in the case $p\nmid N$) appears to be induced by the map $\tau \mapsto \gamma \cdot p\tau$ on the upper half-plane, where $\gamma \in \SL_{2}(\ZZ)$ represents $\langle p\rangle$. The operator ``$T_{p,*}$’’ in \emph{loc.~cit.} thus  coincides with the ``$T_{p}$’’ of \cite{Diamondmodular} by (\ref{Tprelation}).   See    \cite[Exercise 7.9.3(a)]{Diamondmodular}.
\end{remark} 

\subsection{The Eichler–Shimura relation}      \label{EichlerShimurasec}  
Recall that each $ \mathcal{S}_{K} $ is a smooth  integral     $\QQ$-scheme   of dimension one.    By  \cite[\href{https://stacks.math.columbia.edu/tag/0BY1}{Tag 0BY1}]{stacks-project} or  \cite[Theorem 16.3.3]{Vakil},     $   \mathcal{S}_{K} $ is an open subscheme of a uniquely determined integral  projective $ \QQ $-scheme $  \overline{\mathcal{S}}_{K} $ that we  refer to as its   \emph{smooth  compactification}.   The   same  result  also  implies  that  the    degeneracy maps $ \pr_{L, K} $ (\ref{degeneracy}) and the twisting isomorphisms  $ [g]_{K} $   (\ref{twisting})     admit unique extensions to the smooth compactifications of their underlying   schemes.   Let $$   \mathrm{J}_{K}  =  \mathrm{Pic}^{0}(\overline{\mathcal{S}}_{K})_{/\QQ} $$    denote the Jacobian variety of $ \overline{\mathcal{S}}_{K} $  \cite{MilneJacobian}.    This is an abelian variety over $ \QQ $ of dimension twice the genus of $  \overline{\mathcal{S}}_{K} $.  One can define  a  right action of the Hecke algebra $\mathcal{H}_{\ZZ}(K \backslash \Gb(\Ab_{f}) / K)$ on $  \mathrm{J}_{K} $   using   covariant   Hecke correspondences    in  a   manner similar to  divisors.   More precisely, we can define the pullback and pushforward needed in the definition of Hecke actions via the Picard and Albanese functoriality of Jacobians, respectively \cite[p.~443]{Ribet}.    %An   important   consequence of the Eichler-Shimura congruence relation  is that for all primes $ \ell \nmid N p $,    
This action can also be defined on the $p$-adic Tate module $$   \mathrm{T}_{p,  K }     =     \varprojlim   \nolimits    _{n }   \mathrm{J}_{K}[p^{n}](\overline{\QQ})$$ 
for any prime $ p $.  This is a free $ \ZZ_{p}$-module of rank  twice the genus of $   \overline{\mathcal{S}}_{K}  $, and has a natural left action of $ \Gal(\overline{\QQ}/\QQ) $ which commutes with the  aforementioned  Hecke actions.

Suppose now that $ K = \widehat{\Gamma}_{i}(N) $ for $ i = 0 , 1 $ as in the Examples of \S \ref{classicalmodular}.      Then the   standard   identification  $ Y_{i}(N) \simeq \mathcal{S}_{K} $   extends   uniquely  to  an identification $ X_{i}(N)  \simeq     \overline{\mathcal{S}}_{K} $.   An   important   consequence of the Eichler-Shimura congruence relation   is that for all primes $ \ell \nmid N p $,        
%More precisely, we can define  the pullback and pushforward  needed in the definition of $ \ch(K g K)_{*} $  respectively via  Picard and Albanese functoriality of Jacobians \cite[p.~443]{Ribet}.  An   important   consequence of the Eichler-Shimura congruence relation  is that for all primes $ \ell \nmid N p $, 
\begin{equation}   \label{Eichlerclassical} \mathrm{Frob}_{\ell}^{2} -  T_{\ell, *}  \,     \mathrm{Frob}_{\ell}  +  \ell \, \langle \ell \rangle _{*}   = 0   
\end{equation}    
as an endomorphism of $ \mathrm{T}_{p, K}   $. See \cite[Theorem 2]{Rohrlich}, \cite[\S 5]{RibetSerre} or \cite[Theorem 9.5.1]{Diamondmodular}.   As noted in Examples \ref{Gamma0(N)} and \ref{Gamma1(N)},  we can  write  $$ T_{\ell, *} =  \ch(  K  \sigma_{\ell} ^ {-1} K )_{*} ,   \quad   \quad   \langle \ell \rangle _{ * }    =    \ch(K \tau_{\ell}^{-1} K)_{*   }   $$    
where   
\begin{equation}   
\label{sigmatau}\sigma_{\ell} : =  \left (  \begin{smallmatrix}  \ell \\  & 1   \end{smallmatrix}   \right  )   , \quad \quad   \tau_{\ell}  : =   \left (  \begin{smallmatrix}  \ell   \\   &   \ell   \end{smallmatrix}   \right )   
\end{equation}    
are as in  Example  \ref{HeckeGamma1}.   
%Recall that $ \langle \ell \rangle_{*} $.     % So  we  can  rewrite  (\ref{Eichlerclassical})  as       
This motivates  the  following general definition.       
\begin{definition}     \label{Heckepolydef}     %\label{Heckepolydef}     
Let $ K $ be any level and $ \ell $ be any prime such  that    $ K $ is unramified at $ \ell $.  The \emph{Eichler-Shimura Hecke polynomial} at the prime $ \ell $ is  defined  to   be   
\begin{equation}       \label{Heckepolyexp}           \mathfrak{H}_{\mathrm{ES}, \ell}(X) = \ch(K)      X^{2}     - \ch(K \sigma_{\ell}^{-1} K)   \,    X + \ell   \,     \ch( K \tau_{\ell}^{-1} K)  .   
\end{equation}    
%The \emph{geometric Hecke polynomial} is defined to be $ \mathfrak{H}^{-}_{ \ell} (X)  =   X^{2}  \cdot     \mathfrak{H}_{\ell}(1/X) $
%These are 
considered  as  an   element   of $ \mathcal{H}_{\ZZ}(K \backslash \Gb(\Ab_{f})/ K)[X] $.   
\end{definition}  
In this notation,   relation (\ref{Eichlerclassical}) can be restated as follows.     
\begin{theorem}[Eichler–Shimura]   \label{EichShimteo}    For  every positive integer  $   N   $ and $ \ell $ a prime such that $ \ell \nmid N   p     $,  the Hecke-Frobenius endomorphism $ \mathfrak{H}_{\mathrm{ES},  \ell, *}  (\mathrm{Frob}_{\ell}) $  on   $ \mathrm{T}_{p, K  } $  vanishes    for      $  K =    \widehat{\Gamma}_{0}(N) $, $ \widehat{\Gamma}_{1}(N)  $.   
\end{theorem} 

We can  reformulate this relation   in terms of   $p$-adic     \'{e}tale cohomology. % Let $ X $ be a smooth projective scheme over $ \QQ $. 
By   \cite[\href{https://stacks.math.columbia.edu/tag/03RQ}{Tag 03RQ}]{stacks-project} and Poincar\'{e} duality   for  smooth   projective     curves over $ \overline{\QQ} $,    we   have  a  canonical isomorphism  $$  \mathrm{T}_{p, K}^{\vee}     \xrightarrow{\sim}   % \mathrm{H}^{1}_{\et}(J_{K, \overline{\QQ}},  \ZZ_{p}) \xleftarrow{\sim}    
  \varprojlim \nolimits _{n}  \mathrm{H}^{1}_{\et} (\overline{\mathcal{S}}_{K, \overline{\QQ}}, \ZZ/ p^{n} \ZZ) =:\mathrm{H}^{1}_{\et}(\overline{\mathcal{S}}_{K,  \overline{\QQ}}, \ZZ_{p} ) $$
of     $ \Gal(\overline{\QQ}/\QQ) $-representations. If $ L $ is a compact open subgroup of $ K $, then  these isomorphisms commute with the \emph{dual} of the  maps  induced by  Albanese (resp., Picard) maps on the  dual   Tate modules and pullback (resp., pushforward) on \'{e}tale  cohomology. Similarly for twisting  isomorphisms.   So  these  isomorphisms are also  equivariant  with respect to   the natural   covariant and contravariant Hecke actions one can  define using said maps.      On  the  other hand,   the   natural  pairing  $$ \langle -, - \rangle :  \mathrm{T}_{p, K}^{\vee}   \times  \mathrm{T}_{p,  K }   \to   \ZZ_{p} $$
induces an \emph{adjoint} Hecke action on $ \mathrm{T}_{p, K}^{\vee} $  induced by the covariant action on $   \mathrm{T}_{p,  K}   $.    One easily checks that this adjoint action on $  \mathrm{T}_{p, K}^{\vee} $ matches the  contravariant action that we can define directly.   So we also have  the following.    
\begin{theorembis}{EichShimteo}   \label{EichShimbisteo}   For every positive integer $ N $ and $ \ell $ a prime such that $ \ell \nmid Np  $,  the Hecke-Frobenius endomorphism $ \mathfrak{H}_{\mathrm{ES}, \ell}^{*}  (\mathrm{Frob}_{\ell}^{-1}) $  of $ \mathrm{H}^{1}_{\et}(\overline{\mathcal{S}}_{K ,  \overline{\QQ} } ,  \ZZ_{p})     $ vanishes for $  K =    \widehat{\Gamma}_{1}(N) $, $ \widehat{\Gamma}_{0}(N)  $.        
\end{theorembis}   
See \cite[Theorem 4.9]{Delignemodular}, where this result is  proved  for   the         \emph{interior cohomology}\footnote{the image of compactly supported cohomology $\mathrm{H}^{1}_{\et   ,    c     }(\mathcal{S}_{K}, \ZZ_{p}) $ in $ \mathrm{H}^{1}_{\et}(\mathcal{S}_{K}, \ZZ_{p})$}  of $ \mathcal{S}_{K}$ for principal   congruence  level      $  K = \widehat{\Gamma}(N)$.  Note that $ \mathrm{H}^{1}_{\et, c}(\mathcal{S}_{K}, \ZZ_{p}) \to  \mathrm{H}^{1}_{\et}(\mathcal{S}_{K}, \ZZ_{p}) $ factors as $$     \mathrm{H}^{1}_{\et, c}(\mathcal{S}_{K},  \ZZ_{p}) \to    \mathrm{H}^{1}_{\et}(\overline{\mathcal{S}}_{K}, \ZZ_{p}) \to \mathrm{H}^{1}_{\et}(\mathcal{S}_{K},    \ZZ_{p}).  $$ 
Now the second map above is injective by \cite[Remark 5.4]{Milneetale} and the first map is surjective by Poincar\'{e} duality. This implies that the interior cohomology is (Hecke and Galois equivariantly) isomorphic to $ \mathrm{H}^{1}_{\et}(\overline{\mathcal{S}}_{K}, \ZZ_{p}) $.  Thus  the   cohomological  Eichler-Shimura  relation above also holds for $ K = \widehat{\Gamma}(N)$.   One can    then    use this to establish the Eichler-Shimura relation for \emph{any}   level $ K $ that is unramified at the prime $ \ell   \neq   p    $ as follows. Choose a principal congruence level $ L  = \widehat{\Gamma}(N) $ contained in $ K $.     Since $ K $ is unramified at $ \ell $, we can assume that $ \ell \nmid N $.  Consider the   Galois  equivariant    pullback $$  \pr_{L ,  K  }  ^ { * }  :    \mathrm{H}^{1}_{\et}(\overline{\mathcal{S}}_{K, \overline{\QQ}}, \ZZ_{p} )   \to   \mathrm{H}^{1}_{\et}(\overline{\mathcal{S}}_{L, \overline{\QQ}},  \ZZ_{p}) .  $$ This is injective,   since cohomology is torsion free, and the post  composition  with $  \pr_{L, K ,* }$  induces multiplication by $ [K : L]$.   % Since $ K $ is unramified at $ \ell $, we can choose $ L $ and therefore $ \ell \nmid N $. 
Now one can  easily  verify that $$ \pr_{L, K}^{*} \circ  \ch(KgK)^{*}  = [K: L] \cdot \ch(LgL)^{*}  $$ 
for any $ g \in \Gb(\QQ_{\ell}) \hookrightarrow \Gb(\Ab_{f}) $ \cite[Corollary 2.4.3]{CZE}.  The vanishing of Hecke-Frobenius endomorphism for level $ K $ therefore follows from the  the corresponding  vanishing  for   level      $ L $.

\begin{remark} Since $ \mathrm{T}_{p, K}  \simeq \mathrm{H}^{1}_{\et}(\overline{\mathcal{S}}_{K, \overline{\QQ}   }    ,  \ZZ_{p}(1)) $, we  see that  $ \mathfrak{H}_{   \mathrm{ES},   \ell, *}(\ell \cdot \mathrm{Frob}_{\ell}) $  also vanishes on $ \mathrm{H}^{1}_{\et}(\overline{\mathcal{S}}_{K,  \overline{\QQ}} , \ZZ_{p})    $. This may also be deduced by noting that the constant term operator $$ c_{0} := \ell \cdot  \ch(K \tau_{\ell}^{-1} K ) $$  of $ \mathfrak{H}_{\mathrm{ES}, \ell}(X) $ is invertible in $ \mathcal{H}_{\ZZ[1/\ell]}(K \backslash \Gb(\Ab_{f}) / K ) $  with respect to the convolution operation and that    
\begin{equation}   \label{dualityHecke}       \mathfrak{H}_{\mathrm{ES}, \ell}(X)  =  (c_{0}^{-1})^{t} \cdot   X^{2}  \cdot   \mathfrak{H}_{\mathrm{ES}, \ell}^{t}(\ell/X)   ,   
\end{equation}
where $ \mathfrak{H}_{\mathrm{ES}, \ell}^{t}(Y) $ denote the polynomial in $ Y $ whose coefficients are transposes of the coefficients of $ \mathfrak{H}_{\mathrm{ES} , \ell }(Y) $.            
\end{remark}   
\begin{remark} The Eichler-Shimura congruence  relation is  established    more generally in   \cite[\S 10]{Carayol}  for Shimura curves arising from quaternion  algebras over totally real fields. Note  however    that the Shimura data used in \emph{loc.~cit.} coincides  with    (\ref{weirdShimura}) in the case of modular curves. One can use  the isomorphism (\ref{weirdshimuraiso}) to translate between the two conventions, as noted in Remark   \ref{weirdshimuravarcompar}.  In particular, if $ K $ equals its own transpose (e.g., $ K = \widehat{\Gamma}   (N)$),  then $ \ch(K gK ) $ in our convention corresponds to $ \ch(K ({}^{t}   \negsmall    {g}^{-1})K) $ in Carayol's convention. We also observe that  Carayol's  reciprocity law in \cite[\S 1.2]{Carayol} for geometrically connected components is the \emph{inverse}  of the one described in \S \ref{compsec}, which is consistent with  what we observed   in   Remark   \ref{weirdcompactremark}.\footnote{In particular,  the erroneous sign convention  noted in Remark \ref{errorremark} seems to not have affected Carayol's work.}       See  also    Remark  \ref{fuckthesetorsors}.  
\end{remark}   
\begin{remark} The vanishing  discussed above actually  extends to all degrees of \'{e}tale cohomology, i.e.,   $ \mathfrak{H}_{   \mathrm{ES},    \ell}^{*}   (   \mathrm{Frob}_{\ell}^{-1})     $ vanishes    on both $ \mathrm{H}^{0}_{\et} $   and    $ \mathrm{H}^{2}_{\et}$.  See the next subsection for a proof.     
This vanishing phenomenon is part of a far reaching  generalization proposed by Langlands for arbitrary Shimura varieties, who was motivated by the problem of computing the Hasse-Weil   zeta functions  of  these   varieties.      See \cite{BlasiusRogawski} for a   discussion.   
\end{remark}     

\subsection{A sanity check}   \label{sanitycheck}    As  is     evident from the discussion  so  far,  one has to reckon with a multitude of $(\ZZ/2\ZZ)$-torsors of conventions\footnote{This terminology is due to Christophe Cornut.} when working with adelic modular curves and, more generally, Shimura varieties. For instance, one must choose whether to work with arithmetic or geometric Frobenii (in addition to fixing the normalization of the Artin map used in the reciprocity laws), whether the Hecke action is taken to be covariant or contravariant, and  whether to use left or right action on level  structures.      Fortunately, most recent literature has largely converged on a common set of conventions, and these are the ones adopted in the present article.   

However, the use of alternative conventions in earlier works (both classical and adelic) introduces considerable potential for confusion, and the most relevant  in the context of Euler systems concerns the definition of the Hecke polynomial for a Shimura datum. In the appendix to \cite{Nekovar}, Jan Nekov\'{a}\v{r} suggested that with the standard choices,\footnote{i.e., Frobenii are geometric, the Artin map is normalized in Deligne’s convention, Hecke actions are contravariant, the level structures are acted on from the right, etc.} it is the Hecke polynomial associated with the \emph{inverse} of the  Hodge  cocharacter $\mu_{\mathcal{X}}$ for a Shimura datum $(\Gb',\mathcal{X})$ that should appear in the conjectural generalization of the Eichler–Shimura  relations  on the \'etale cohomology of the Shimura varieties attached to $(\Gb',\mathcal{X})$. While we have not explained how one attaches Hecke polynomials to cocharacters, the reader can accept our claim that  this  polynomial is $ \mathfrak{H}_{\mathrm{ES},   \ell}(X) $ for  the  datum     $ (\Gb, \mathcal{X}_{\mathrm{std}})$   by comparing our expression with  \cite[(A1.6.1)]{Nekovar}. This stands in contrast with \cite[\S 6]{BlasiusRogawski}, whose conventions appear to align with the standard ones, but where the conjectural congruence relation is stated using the Hecke polynomial for  $\mu_{\mathcal{X}}$.     For  $ ( \Gb , \mathcal{X}_{\mathrm{std}}) $,  this  polynomial is   
\begin{equation}   \label{wrongHeckepoly}      \mathfrak{H}_{\mathrm{BR}, \ell}(X)   =  \ch(K)  X^{2} - \ch(K \sigma_{\ell} K ) \, X  + \ell \,  \ch( K \tau_{\ell} K )   ,    
\end{equation}    whose coefficients are \emph{transposes} of the coefficients of (\ref{Heckepolyexp}). See the reverse characteristic polynomial denoted ``$P_{r}(X)$" on  \cite[p.~536]{BlasiusRogawski}, which satisfies $ \mathfrak{H}_{\mathrm{BR},\ell}(X) = X^{2} \cdot  P_{r}(\ell^{\frac{1}{2}} \cdot 1/X)$ when the measure of $ K $   equals    one.        

\begin{remark}
Although this is not explicitly stated in \cite[p.~527]{BlasiusRogawski}, the Frobenius ``$\Phi_{v}$'' used  throughout    is geometric. This follows from their proof of Proposition~6.1, which invokes Deligne’s theorem on the absolute values of the eigenvalues of geometric Frobenii. See also the introductions of \cite{Bultel} and \cite{Wedhorn}.
\end{remark}

The purpose of this subsection is to provide directly verifiable evidence supporting Nekov\'{a}\v{r}'s claim by determining which of the two Hecke polynomials (evaluated at \emph{geometric} Frobenii) vanishes on the zeroth \'etale cohomology of modular curves.   We show that the endomorphism induced by $ \mathfrak{H}_{\mathrm{ES}, \ell}(X)$, formulated using the standard conventions, always vanishes, whereas the  endomorphism  induced by $ \mathfrak{H}_{\mathrm{BR},\ell}(X) $    does not. To make this subsection as self-contained as possible for readers who simply wish to check this computation themselves, we recall below the relevant notation and conventions used in  our  computation.

Let $(\Gb, \mathcal{X}_{\mathrm{std}})$ be the standard Shimura datum (\ref{SDdatum}). 
For each compact open subgroup $K \subset \Gb(\Ab_{f})$, let $\mathcal{S}_{K}$ denote the corresponding canonical model, whose $\CC$-points are given in (\ref{Shimuravar}). 
The  modular curve  $\mathcal{S}_{K}$ is a smooth   integral    affine $\QQ$-scheme and admits a unique smooth compactification over $\QQ$, which we denote by $\overline{\mathcal{S}}_{K}$. As noted in \S  \ref{compsec},  the geometrically connected components of $\mathcal{S}_{K}$ are parametrized by the double quotients
\begin{equation} \label{doublequotforcomp}
\Gb(\QQ)^{+} \backslash \Gb(\Ab_{f}) / K \;\xrightarrow{\sim}\; 
\QQ^{\times}_{\ge 0} \backslash \Ab_{f}^{\times} / \det(K),
\end{equation}
where the isomorphism between the two sides is induced by the determinant map $\det : \Gb \to \GG_{m}$.
For each $h \in \Gb(\Ab_{f})$, let $z_{K}(h)$ denote the geometrically connected component of $\mathcal{S}_{K}$ indexed by $h$, which is a quotient of the upper half-plane by a congruence subgroup of $\SL_{2}(\QQ)$. 
We regard $z_{K}(h)$ as a $\overline{\QQ}$-scheme. 
Clearly,
\[
z_{K}(h) = z_{K}(q h k) 
\qquad\text{and}\qquad
z_{K}(h h') = z_{K}(h' h)
\]
for all $q \in \Gb(\QQ)^{+}$, $k \in K$, and $h, h' \in \Gb(\Ab_{f})$.  The quotients (\ref{doublequotforcomp}) also describe   the      geometrically connected components of $\overline{\mathcal{S}}_{K}$: the component indexed by $h$ is simply the smooth compactification $\overline{z}_{K}(h)$ of $z_{K}(h)$, which we also view as a scheme over $ \overline{\QQ}$.  For a  scheme  $X$ over $ \QQ $ and a prime $p$, we  denote the $i$-th  $p$-adic \'etale cohomology of the base change of 
$X$ to  $ \overline{\QQ}$   by     
$$  \mathrm{H}^{i}_{\et}(X_{\overline{\QQ}}, \ZZ_{p}    ) ,   $$    
which is endowed with a    left $ \Gal(\overline{\QQ}/\QQ)  $-action in the usual way. For a set $ Y $, we let $ \mathcal{C}_{\ZZ_{p}}(Y) $   denote  $ \ZZ_{p}$-module of all $  \ZZ_{p}$-valued  functions on $ Y $ that have finite support.  Then we   have canonical isomorphisms
\begin{equation}    \label{isos} 
\begin{array}{ccccl}
\mathcal{C}_{\ZZ_{p}}(\Gb(\QQ)^{+} \backslash \Gb(\Ab_{f}) / K  )     
& \xrightarrow{\:\sim\:} & 
\mathrm{H}^{0}_{\et}(\mathcal{S}_{K,\overline{\QQ}}, \ZZ_{p}) 
& \xrightarrow{\:\sim\:} & 
\mathrm{H}^{0}_{\et}(\overline{\mathcal{S}}_{K,\overline{\QQ}}, \ZZ_{p}) 
\\[6pt]
\ch(\Gb(\QQ)^{+} h K) 
& \longmapsto & 
z_{K}(h) 
& \longmapsto & 
\bar{z}_{K}(h)
\end{array}
\end{equation}
of $ \ZZ_{p}$-modules.   We can endow the leftmost module  in   (\ref{isos}) with a left $ \Gal(\overline{\QQ}/\QQ)  $-action  that   factors  through $ \Gal(\QQ^{\mathrm{ab}}/\QQ)$ using (\ref{doublequotforcomp}) and (\ref{CFTforQ}). This is normalized so that the geometric Frobenius $ \mathrm{Frob}_{\ell}^{-1} $  at a  prime $ \ell  $    acts via $$   \ch(\Gb(\QQ)^{+}hK)  \mapsto  \ch(\Gb(\QQ)^{+}ahK)   $$    
for any element $ a \in \Gb(\Ab_{f}) $ that  has determinant $ \ell \in \QQ_{\ell}^{\times} \hookrightarrow \Ab_{f}^{\times} $.  Then the Deligne-Shimura reciprocity law described in \S \ref{compsec} (and functoriality of \'{e}tale  cohomology)         implies  that  the isomorphisms in (\ref{isos}) are equivariant with respect to Galois actions.  If $ L $ is a compact open subgroup of $ K $, there are natural pullback and pushforward   morphisms    on all of these modules  induced  by   the    finite flat degeneracy map $ \mathrm{pr}_{L, K}$    (\ref{degeneracy}). Similarly for the twisting isomorphisms (\ref{twisting}). It is straightforward to verify that  the  isomorphisms  (\ref{isos}) are also compatible with  respect   to   these     induced    maps.  So  one can verify Nekov\'{a}\v{r}'s claim on any of  these  modules.    

For each $g \in \Gb(\Ab_{f})$, we have a Hecke correspondence diagram 
(\ref{Heckecorrdiagram}).  
Using the functorial pullbacks and pushforwards of \'etale cohomology 
induced by  the  finite   flat  degeneracy maps and  twisting isomorphisms  on modular curves, we can define the contravariant Hecke action
\[
\ch(K g K)^{*} : 
\mathrm{H}^{0}_{\et}(\mathcal{S}_{K,\overline{\QQ}},  \ZZ_{p} )
\longrightarrow 
\mathrm{H}^{0}_{\et}(\mathcal{S}_{K,\overline{\QQ}},  \ZZ_{p}    )
\]
as the map 
$   \big   (    [g^{-1}]_{K'}  \circ \pr_{L, K'   }     \big      )   _{*}       \circ \pr_{L,K}^{*}$, 
where $ K ' = g^{-1} K g $ and $ L  = K \cap K' $. 
\begin{lemma}
$\;\ch(K g K)^{*} \cdot z_{K}(h)
= \lvert K g K / K \rvert \cdot z_{K}(hg^{-1})$.
\end{lemma}

\begin{proof}
This is \cite[Example 4.2]{explicitdescent}.  
The neatness assumption on levels $K$ used in \emph{loc.~cit.} can be removed in light of 
the results of \S\ref{pullsec}.  
One can also verify the statement directly by comparing the degrees of the 
components of $\mathcal{S}_{K \cap g^{-1} K g}(\CC)$ over  the   components   of    
$\mathcal{S}_{K}(\CC)$ as in Lemma~4.8 of \emph{loc.~cit.}
\end{proof}

\begin{remark} A quick check on our result is that the pullback action of $ g $ on the function  $ \ch(\Gb(\QQ)^{+}   h    K) $ is via right translation on domain,   which gives  $  \ch( \Gb(\QQ)^{+}   h    K g^{-1}   ) $. If $ g $ normalizes $ K $,   this  is obviously  equal to   $ \ch(\Gb(\QQ)^{+}  h    g^{-1} K) $.   
\end{remark}

We can now carry out our  verification.  Let  $ \ell   \neq   p    $   be    a prime where $ K $ is unramified and let  $ \mathfrak{H}_{   \mathrm{ES},   \ell}(X) $ be as in  Definition  \ref{Heckepolydef}. Let us take the local element $ \sigma_{\ell} = \mathrm{diag}(\ell , 1) $ as in (\ref{sigmatau}) to represent $ \mathrm{Frob}_{\ell}^{-1} $.    Then  for any $   h    \in \Gb(\Ab_{f} )$, we have 
\begin{align*} \mathfrak{H}_{\mathrm{ES}, \ell}^{*}(\mathrm{Frob}_{\ell}^{-1})  \cdot  z_{K}(h)   &  =    z_{K}(h \sigma_{\ell}^{2})   - (\ell +1 )  z_{K}(h \sigma_{\ell}^{2})  +  \ell   z_{K}(h \tau_{\ell}) \\
& =  z_{K}(h   \sigma_{\ell}^{2}) - (\ell + 1) z_{K}(h \sigma_{\ell}^{2})  +  \ell   z_{K}(h    \sigma_{\ell}^{2})  =  0   .        
\end{align*} 
To handle cohomology in degree $ 2 $, note that the  endomorphism   $$ \mathfrak{H}_{\mathrm{ES}, \ell,   *}(\ell \cdot \mathrm{Frob}_{\ell})   =     \ell^{2}    \, \ch(K)_{*}   \,     \mathrm{Frob}_{\ell}^{2}     - \ell \,  \ch(K \sigma_{\ell}^{-1} K)_{*}   \,   \mathrm{Frob}_{\ell}    + \ell   \,     \ch( K \tau_{\ell}^{-1} K)_{*}       $$    
also     vanishes on $   \mathrm{H}^{0}_{\et}(\overline{\mathcal{S}}_{K,   \overline{\QQ}} ,   \ZZ_{p}) $ by   (\ref{dualityHecke}).     Therefore,     $  \mathfrak{H}_{\mathrm{ES}, \ell, *}( \mathrm{Frob}_{\ell}   )    $    vanishes on $      \mathrm{H}^{0}_{\et}(\overline{\mathcal{S}}_{K,   \overline{\QQ}} ,   \ZZ_{p}(1)) $. Since $$  \mathrm{H}^{0}_{\et}(\overline{\mathcal{S}}_{K,   \overline{\QQ}} ,   \ZZ_{p}(1))     \simeq  \mathrm{H}^{2}_{\et}(   \overline{ \mathcal{S}}_{K,  \overline{\QQ}} ,   \ZZ_{p}  )^{\vee}    $$
by Poincar\'{e} duality, we obtain the   vanishing  of  $ \mathfrak{H}_{\mathrm{ES},  \ell}^{*}(\mathrm{Frob}_{\ell}^{-1}) $ on  $     \mathrm{H}^{2}_{\et}(\overline{\mathcal{S}}_{K,  \overline{\QQ}},  \ZZ_{p})  $ by dualizing.

On the other hand,   
\begin{align*}     \mathfrak{H}_{\mathrm{BR}, \ell}^{*}   (  \mathrm{Frob}_{\ell}^{-1})  \cdot  z_{K}(  h   )    &    =     \mathfrak{H}_{\mathrm{ES},\ell, *}( \mathrm{Frob}_{\ell}^{-1}) \cdot  z_{K}(h)   \\
& =   z_{K}( h \sigma_{\ell}^{2})   - (\ell +1 )  z_{K}(h )   +  \ell   z_{K}(h \tau_{\ell}^{-1}) \\
& =  z_{K}( h   \sigma_{\ell}^{2}) - (\ell + 1) z_{K}(h )  +  \ell   z_{K}(  h     \sigma_{\ell}^{-2})   ,        
\end{align*} 
which is clearly not zero if $ K^{\ell} = K /\GL_{2}(\ZZ_{\ell})  $ is chosen appropriately.\footnote{It is of course zero if $ \mathcal{S}_{K} $ is geometrically   connected.}  For instance, we  can    take $ h = 1 $, $ K = \widehat{\Gamma}(N) $ for any  $ N \geq   3    $ and $ \ell $    any    prime such that $ (\ell , N ) =   1      $   and    $ \ell^{2} \not \equiv   \pm   1    \pmod{N} $.      Note however that the endomorphism $$ \mathfrak{H}_{\mathrm{BR}, \ell}^{*}( \mathrm{Frob}_{\ell})   =    \mathfrak{H}_{\mathrm{ES}, \ell, *}( \mathrm{Frob}_{\ell})   $$     does vanish on the zeroth \'{e}tale  cohomology.  
\begin{remark}   \label{fuckthesetorsors}  
One can  similarly check that for the alternative    Shimura  data (\ref{weirdShimura}), it is  $ \mathfrak{H}_{\mathrm{BR}, \ell}^{*}(\mathrm{Frob}_{\ell}^{-1})$ that vanishes on the zeroth \'{e}tale  cohomology. This   is    consistent with the fact that the Hodge cocharacter for the data (\ref{weirdShimura}) is the inverse  of the Hodge cocharacter for (\ref{SDdatum}),  as  noted  in  Remark  \ref{weirdhodgecocharacter}.      
\end{remark}    
\begin{remark}    
The choice of the inverse Hodge cocharacter for Eichler–Shimura relations  is   noted in \cite[Remark 2.1.3]{SiYing}.  See also \cite[Remark 4.1.3]{Morel}, \cite[Corollary 9.2]{Scholze-Shin} and \cite[\S 2.2]{CaraianiShin}, where these inverse cocharacters make an appearance.
\end{remark}

\begin{remark}      \label{errorremark}   
As Christophe Cornut has explained to the author, the discrepancy in \cite{BlasiusRogawski} may well have its origins in Deligne’s sign error in his    Corvallis article \cite{DeligneVar}. The mistake went unnoticed for more than a decade before being identified by Milne in 1990 and subsequently acknowledged by Deligne \cite{Milne1990b}. To clarify,  this    sign error appears in the extra inverse occurring in the reciprocity morphism in \cite[\S 2.2.3]{DeligneVar}.   The remaining conventions used by Deligne   \emph{must still be used}    after correcting this   error in order to  obtain a valid  theory of canonical models. 
\end{remark}

\section{The  horizontal Euler  system}    
\label{Eulersystemsec}      
We maintain the notations and conventions introduced in \S \ref{SDdatumsec}-\ref{canonicalmodsec} and \S \ref{CMpointsubsection}.   In particular, $K$ denotes the fixed compact open subgroup of $\Gb(\Ab_{f})$ from \S\ref{canonicalmodsec}.  
If $n$ is a square-free positive integer, we let $[n]$ denote the set of primes dividing $n$,  
$\Ab_{f,[n]} := \prod_{\ell \mid n} \QQ_{\ell}$,  
and $\Ab_{f}^{[n]} := \Ab_{f} / \Ab_{f,[n]}$ denote the ring of finite adeles away from the primes dividing $n$.      Let  $   R       $ be the set of all rational primes $ \ell $ such that the following conditions are satisfied.  
\begin{enumerate} 
\item[(C1)] $ \ell $ does not divide the  discriminant $ \mathrm{disc}(E) $.    
\item[(C2)] The   $ \ZZ_{\ell}  $-lattice generated by $  \omega_{1} \otimes 1 ,   \omega _{2}  \otimes 1 $ inside $ E \otimes _ { \QQ }  \QQ_{\ell } $ is $  \OO_{\ell}  = \mathcal{O}_{E, \ell}   :  =     \OO_{E} \otimes _ { \ZZ }  \ZZ_{\ell}   $.
\item[(C3)]  $   K  $  is unframified at $ \ell $,
\item[(C4)]   $ K ^ { \ell } = K/\GL_{2}(\ZZ_{\ell}) $ contains the element   $   \mathrm{diag}    ( \ell , \ell )  \in 
\Gb ( \QQ )  \hookrightarrow \Gb ( \Ab_{f}/\QQ_{\ell})  $  if  $  \ell $ is  inert.      
\end{enumerate}   
Condition (C1) implies that $ \ell $ is unramified in $ E $.   If $ \ell \in R $ is inert in $ E  $, we let $ \lambda $ denote the unique prime in $ E $ above $ \ell $. If   $   \ell \in R  $ is  split   in  $   E   $, we  let   $ \lambda $ be any  one   of the two  primes above $  \ell    $  in  which  case we denote the conjugate of $ \lambda $     by  $  \bar{\lambda}   $.      Let  $   \Lambda   $   be  the  set of all  primes  $ \lambda $ of  $ E  $ above $ R $ obtained by this procedure and set $  \mathcal{N}  $  to  be     the set of all square-free  products  of primes in $  R   $. We consider $ 1 \in \mathcal{N} $ as the empty product. For $ n \in \mathcal{N}$, we can write   $$  K =  K^{[n]   } K_{[n]}  $$ where   
$ K_{[n]}    :     = \prod_{\ell  \mid n } K_{\ell} $ and $ K^{[n]}   =  K  /  K_{[n]}     \subset \Gb ( \Ab_{f}^{[n]}) $.   The first condition  also  implies that $ \Hb $ admits a smooth model over $ \ZZ_{\ell} $, whose group of $ \ZZ_{\ell} $-points equals  the group of units  $ \mathcal{O}_{\ell}^{\times} $, which is  the   unique maximal compact open subgroup of $ \Hb(\QQ_{\ell})$. 
\begin{remark}   If $ (\omega_{1}, \omega_{2}) $ forms   a $ \ZZ$-basis for $   \mathcal{O}_{E}   $, condition (C2) is redundant.   Condition (C4)  is  imposed to reflect the behavior of the Frobenii above  inert primes in the  anticyclotomic extensions of $ \QQ $.  For applications to Euler systems, we would also like $ R $ to contain infinitely many primes.  This is clearly true if $K $ contains the diagonal group   $ \widehat{\ZZ}^{\times}    \hookrightarrow \GL_{2}(\widehat{\ZZ}) $ and in particular,  for $ K = \widehat{\Gamma}_{0}(N) $. %If $ \omega_{1}, \omega_{2} $ form a $ \ZZ$-basis for $ E $, condition (C2) is redundant. 
If $ K = \widehat{\Gamma}(N) $ or $ \widehat{\Gamma}_{1}(N) $, then $ R $ contains all but finitely many primes that are congruent to $ 1 $ modulo $ N  $ and   are    inert in $ E $, and contains all but finitely many primes that are split in $ E  $.                
\end{remark}   
\subsection{CM divisors}   
\label{CMdivisors} Recall that $ \mathcal{P}_{K} $ (\ref{pkc}) denotes the set of points $ \mathcal{S}_{K}(\CC) $ that have CM by $ E $.     Consider the free $ \ZZ $-module 
$$ \mathcal{Z} = \mathcal{Z}_{K} := \ZZ\langle  \mathcal{P}_{K}\rangle .  $$
It admits a $\ZZ$-linear left action of the Galois group $\Gal(\overline{\QQ}/E)$ as defined in \S \ref{CMpointsubsection}, which is equivalently described by the left action of $\Hb(\Ab_{f})$. Explicitly, elements of $\Hb(\Ab_{f})$ act by left multiplication on the second component of   the    points   in    $\mathcal{P}_{K}$, i.e.,
$$ 
h   \cdot     [h_{0}, g]_{K} = [h_{0}, hg]_{K} 
$$
for all $h \in \mathbf{H}(\Ab_{f})$ and $g \in \Gb(\Ab_{f})$. If $ V \subset \Hb(\Ab_{f}) $ is a compact open subgroup, we let   $$   \mathcal{Z}(V) :  = \mathcal{Z}^{V}   $$    denote the $ \ZZ $-submodule of all $ V $-invariant linear combinations. This is then precisely  the subgroup of CM divisors  that are defined over the field $ E_{V} $ associated to $ V $ via (\ref{CFTforE}).  
We say that a divisor $ \xi = \sum_{\gamma} a_{\gamma} [h_{0}, \gamma] \in \mathcal{Z} $ is \emph{unramified} at a prime $ \ell \in R  $ if its stabilizer in $ \Hb(\Ab_{f}) $ contains the subgroup $ \OO_{\ell}^{\times} $ of units of $ \OO_{\ell} $, where $ \OO_{\ell}^{\times} $ is viewed as a subgroup of $ \mathbf{H}(\Ab_{f}) $ via 
$$ 
\OO_{\ell}^{\times}  =  \Hb(\ZZ_{\ell})  \hookrightarrow  \mathbf{H}(\QQ_{\ell}) \hookrightarrow \mathbf{H}(\Ab_{f}). 
$$ 
We say that $ \xi \in \mathcal{Z} $ is \emph{unramified at} $ n \in \mathcal{N} $ if it is unramified at all $ \ell \mid n $.   We denote by $ \mathcal{Z}_{[n]}    \subset \mathcal{Z} $ the $ \ZZ $-submodule of all elements in $ \mathcal{Z} $ that are unramified at $ n $. 

As evident  from  the   expression (\ref{badHecke2}), the group of CM divisors also admits a \emph{right} Hecke action by covariant Hecke  operators.   
We collectively denote the Galois and Hecke actions by  \begin{equation}   \label{HeckeGaloisaction}       ( h  ,   \mathrm{ch} ( K g  K   )_{*}  )   \cdot   [  h_{0}  ,    g_{1}   ]   _ {   K   }             =   \sum _  {  \gamma  \in  Kg K  / K  }    [   h_{0}    ,    h  g_{1} \gamma  ] _ { K   }
\end{equation} 
where $ h  \in  \Hb  ( \Ab  _ {  f }    )  $ and $   g  ,   g_{1}    \in    \Gb    ( \Ab_{f} ) $.     Since the point $ h_{0} \in \mathcal{X}_{\mathrm{std}}    $ does not play any role in the definition of Galois and Hecke actions, we can describe these actions  in a  more representation theoretic way. For a topological space $ X $, let $ \mathcal{C}_{\ZZ}(X)$ denote   the  set   of   $ \ZZ$-valued function on $ X $ with finite support.  Define $$ \mathcal{F} : =   \mathcal{C}   _{\ZZ}( \Hb(\QQ) \backslash \Gb(\Ab_{f}) / K ) $$
where $ \Hb(\QQ) $ is viewed as   a subgroup of $ \Gb(\Ab_{f}) $ via $ \Hb(\QQ) \xrightarrow{ \iota} \Gb(\QQ) \hookrightarrow \Gb(\Ab_{f}) $. Then $ \mathcal{F} $ is identified with the $ \ZZ$-module of functions on $ \xi : \Gb(\Ab_{f} ) \to \ZZ $ that are compactly supported modulo $ \Hb(\QQ) $ and  invariant by $ K $ under the right translation action on the domain. The left action of $ h \in \Hb(\Ab_{f}) $ and the right action of $ \ch(KgK) \in \mathcal{H}_{\ZZ}( K \backslash \Gb(\Ab_{f}) / K )  $ on $ \xi \in \mathcal{F  }  $   are    defined via $$  \xi \mapsto  \sum_{\delta  \in K \backslash  K g K  } \xi ( h^{-1} ( - ) \delta^{-1}) .   $$ 
Note that even  though the individual summands are only right invariant under $ \delta^{-1}K\delta $, the whole sum  is right invariant under translation by $ K $, so the action is well-defined.   Now since the  stabilizer of $ h_{0}  \in \mathcal{X}_{\mathrm{std}}   $ in $ \Gb(\QQ) $ is $ \Hb(\QQ) $, there is a $ \ZZ $-linear bijection  
\begin{equation}   \label{psiparamodular}    
\begin{split} \psi \, : \,    \mathcal {  F  }     &   \to  \mathcal{Z}  \\     \mathrm{ch}(E^{\times} g_{1}  K   )    &   \mapsto [   h_{0}      ,   g   _{1}   ]  _     {  K   }   
\end{split}     
\end{equation} 
Clearly, $ \psi $ respects $ \Hb(\Ab_{f}) $-actions and one can verify that it also respects   the    Hecke  actions \cite[\S 1.1]{explicitdescent}.    For any $ \xi  \in   \mathcal{C}_{\ZZ}\big(  \Gb(\Ab_{f})/ K  \big )  $,  we  let $ [\xi] \in \mathcal{F}   $   
denote the image of $  \xi   $ under the map 
\begin{equation}   \pr :  \mathcal{C}_{\ZZ} \big ( \Gb (\Ab_{f}) / K  \big )  \to   \mathcal{F}   
\end{equation}    
induced by  the   projection $ \Gb(\Ab_{f})/K \to \Hb(\QQ)  \backslash  \Gb(\Ab_{f})/K   $.   
Explicitly, if $  \xi   = \ch(gK )$, then $ [\xi    ] = \ch(E^{\times} g K ) $. 
Then  $    \mathrm{pr} $ is also   equivariant with respect to the $ \Hb(\Ab_{f})$ and Hecke actions  defined   similarly.

\subsection{The Hecke  polynomial}    \label{HeckepolyGL2sec}  
Recall that for $ \ell $ a prime, we denote \begin{equation}    
\sigma _ { \ell }  : =   \left ( \begin{smallmatrix}   \ell & \\  &  1   \end{smallmatrix}   \right ) ,  \quad \quad   \tau _ { \ell } :  =  \left ( \begin{smallmatrix}  \ell   &   \\    &  \ell   \end{smallmatrix} \right ) \end{equation} which we view as elements of both $   \Gb    ( \QQ_{\ell} ) $  and  also   $   \Gb    ( \Ab_{f} ) $ via $  \Gb ( \QQ_{\ell}   )   \hookrightarrow  \Gb   ( \Ab_{f} )  $. 
\begin{definition}  The normalized \emph{reverse   geometric  Hecke  polynomial}  at a  prime $ \ell \in R $ is 
\begin{equation}  \mathfrak{H} _{\ell} (X)   :      =    \ell   \,   \mathrm{ch} (  K   ) -         \mathrm{ch} ( K   \sigma _ { \ell }  ^{-1}   K   )  X   +   \mathrm{ch} ( K  \tau    _  {  \ell  } ^{-1}        K)   X ^ { 2 }     
\end{equation} in  the   polynomial      ring    $  \mathcal{H}_{\ZZ}(K \backslash   \Gb    ( \Ab_{f} ) / K )   [  X ] $.  
\label{HeckepolyGL2}    
\end{definition}  

By our discussion in the previous subsection, the  expression   $ \mathfrak{H}  _ { \ell  ,    * }  (   \gamma   )   $ for $ \gamma \in \Gal(E^{\mathrm{ab}   }  /  E ) $  acts on the  module   $   \mathcal{Z} = \mathcal{Z}_{K}   $ via  the  commuting actions    of  covariant     Hecke operators and   the Galois   group.     If $  \gamma  =  \mathrm{Frob}_{\lambda}^{-1} \in   \Gal ( E ^ { \mathrm{ab} }   /   E )$ is a choice of (geometric) Frobenius element  at  a prime $  \lambda   \in   \Lambda     $, then for any  abelian extension $ F / E  $ in which  $  \lambda   $ is  unramified,  $  \mathrm{Frob}_{\lambda}^{-1} $  restricts to the  inverse Frobenius  substitution     $   \mathrm{Fr}_{\lambda}^{-1}    \in   \Gal ( F  /  E )   $.  The action of $ \mathrm{Frob}_{\lambda}^{-1}             $ on $ \mathcal{Z}_{[n]} $     for   $   \lambda  \nmid  n $ is  then independent of  this  choice.   
\begin{remark}   \label{normalizeremark}     Suppose $ K = \widehat{\Gamma}_{0}(N) $ for some $ N  \geq  1 $.     Let  $ A $ be an elliptic curve of conductor $ N $, $ \tilde{A} $ denote its reduction at a prime $ \ell \in R $ and $ a_{\ell} = \ell + 1 - \tilde{A}(\mathbb{F}_{\ell}) $ denote the quantity from introduction. The modularity theorem implies that $ A$  appears as a quotient of $ J_{K}   $ in such a way that under the induced map on Tate modules, the relation (\ref{Eichlerclassical})  specializes  to  
\begin{equation}   \label{Frobcong} \mathrm{Frob}_{\ell}^{2} - a_{\ell} \mathrm{Frob}_{\ell} + \ell = 0   \in  
 \mathrm{End}_{\ZZ_{p}}(\mathrm{T}_{p}(A)) . 
\end{equation} Therefore, the (not necessarily zero) endomorphism of $ \mathrm{T}_{p}(J_{K}   )$ that specializes to the reverse characteristic polynomial of $ \mathrm{Frob}_{\ell}^{-1} $ acting on $ \mathrm{T}_{p}(A)  \simeq \mathrm{T}_{p}(A)^{\vee}(1) $ under this quotient map is
\begin{equation}   \label{Heckepolyreal}     \ch(K)_{*} -  \ell^{-1}  \ch(K \sigma_{\ell}  ^{-1}     K)_{*}   \cdot   \mathrm{Frob}_{\ell}^{-1}   +  \ell^{-1} \ch(K \tau_{\ell}   ^ {  -  1   }     K )_{*}    \cdot   \mathrm{Frob}_{\ell}^{-2}.   
\end{equation} 
For aesthetic reasons,  we have scaled this expression by $ \ell $, so that its coefficients all lie in the Hecke algebra with coefficients in $ \ZZ $. This is harmless, since horizontal norm relations are useful only at primes $ \ell \neq p $, and we can always scale the classes back to match the Euler factor given by (\ref{Heckepolyreal}) after dividing by $ \ell \in \ZZ_{p}^{\times} $.   See  \S  \ref{projection}.     
\end{remark} 

\subsection{Frobenii matrices}   \label{Frobmatrixsec}     We would like to explicitly  describe    elements    in $ \Gb(\Ab_{f}) $ that correspond    via the embedding $ \iota $ to the Frobenii elements in $ \Hb(\Ab_{f}) $.  This is simple if $ \ell $ is  inert  since $ \lambda $ is the unique prime above $ \ell $ and multiplication by $ \ell $ on $ E $ corresponds to diagonal matrix in any basis.  For split  $  \ell  $,  note that    $ \Hb ( \QQ_{\ell } )   \cong  \QQ_{\ell} ^ {  \times }  \times \QQ_{\ell} ^ { \times } $, but the local embedding  $ \iota _ { \ell }   :     \Hb ( \QQ_{\ell} )  \hookrightarrow  \GL_ {  2  }     ( \QQ_{\ell} ) $ is not  diagonal.   To  remedy this, let        $$  \beta_{1}  ,  \beta _{2}    \in  \mathcal{O}_{E}  \otimes \ZZ_{\ell}    $$ be the  two    local idempotents, with $ \beta_{1}$ corresponding to our choice of $ \lambda $ above $ \ell $.   Recall  that  for  any    $   \omega   \in E_{\ell} $, $  \iota_{\ell}(\omega) $ is the matrix of multiplication by $  \omega  $ in the ordered   basis $ (\omega_{1}, \omega_{2}) $. Since $ (\omega_{1} , \omega_{2} )$ and $ (\beta_{1}, \beta_{2}) $ are both bases of $ \mathcal{O}_{\ell} $ by (C2), we can write  $ \beta_{1} = a  \omega_{1} + c  \omega _{2} $, $  \beta _{2} =  b \omega _{1} + d   \omega   _{2} $ for some $ a, b, c, d \in \ZZ_{\ell }$, so that 
\begin{equation}   \label{kell}   k_{\ell} :=  \left ( \begin{smallmatrix}  a  & b  \\ c   &  d  \end{smallmatrix} \right )     \in  \GL_{2  }  ( \ZZ_{\ell  }    )    .  
\end{equation}
is the change of coordinates matrix from $ (\beta_{1}, \beta_{2}) $ to $ (\omega_{1}, \omega_{2}) $,  
Then %then   % $ k_{\ell}^{-1} \cdot  [M_{r}]_{(e_{1},e_{2})   }  \cdot k _ {\ell }  =  [M_{r}]_{(f_{1}, f_{2})} $ and so 
$  k_{\ell}^{-1}  \iota_{\ell}   k_{\ell}   :   \Hb ( \QQ_{\ell} )   \hookrightarrow   \GL_{2} ( \QQ_{\ell} ) $  is
diagonal  with the  top  left   corner  entry   corresponding  to  $  \beta  _ { 1   }  $.  Consequently, the action of  geometric  Frobenius $ \mathrm{Frob}_{\lambda}^{-1} $  corresponds, via   (\ref{CFTforE}), to the action of $ h_{\ell} $  where \begin{equation}   \label{hellfrobenius}  \Hb (\QQ_{\ell}  )  \ni  h_{\ell} =  \begin{cases}   \mathrm{diag}   (                         \ell ,   \ell  )  &   \text{ if } \ell \text{ is inert},  \\   k _ { \ell }   \cdot        \mathrm{diag}  ( \ell   ,  1 )   \cdot    k _  { \ell   }   ^ {  - 1  }     &   \text{ if }  \ell  \text{ is  split}.    
\end{cases}   
\end{equation}

\subsection{The layers $ E[n]$}    Throughout, we  fix a compact open subgroup  
\begin{equation}  \label{groupU}   U \subseteq    \Hb(\Ab_{f})  \cap K 
\end{equation} such that  $ U $ is unramified at all primes $ \ell \in R $, i.e., $ U = U^{\ell} U_{\ell} $ where $ U_{\ell} =  \mathcal{O}_{\ell}^{\times}  $.      \label{modularlayers}     
For each $ \ell \in  R  $, set   \begin{equation}   \label{gell}     \Gb  ( \QQ _  { \ell }  )   \ni     g _ { \ell }        : =  \begin{cases}   \quad  \left  ( \begin{smallmatrix}    1/\ell   &   \\[0.02em] &   1   \end{smallmatrix} \right )   &   \text{ if } \ell \text{ is inert}    \\ k _  { \ell  }  \left ( \begin{smallmatrix}    1   &  \\[0.05em]   1/\ell   \,   & 1   \end{smallmatrix}  \right  )    k _  { \ell   }   ^ {  - 1   }   &  \text{ if  }  \ell  \text { is split}         \end{cases} 
\end{equation}    
and let $$  H_{\ell, g_{\ell}} : =  \mathbf{H}  (  \QQ_ {  \ell   }    )  \cap    g _  { \ell  }  K _ { \ell}  g  _ { \ell }  ^ { -    1    } .  $$      We  note  that $  H_ {  \ell,    g_{\ell }}   \subseteq   \OO _ { \ell }  ^  { \times  }  $ necessarily, since $ \mathcal{O}_{\ell}^{\times} $ is the unique maximal  compact  open subgroup of $ \mathbf{H}(\QQ_{\ell} )  $. We set     $    \Delta_{\ell} : = \OO_{\ell} ^ { \times } /  H_{\ell,   g_{\ell} }  $  and define $ \Delta_{1}   $   to  be     the trivial group. 
For  $  n   \in   \mathcal{  N    }  $, we  denote  $$ g_{n}  : = \prod_{\ell \mid n } g_{\ell} \in   \Gb ( \mathbb{A}_{f,[n]})  $$    
where $ g_{1} = 1 $ by convention. Abusing notation, we  consider $ g_{n} $ as elements of $ \Gb (\mathbb{A}_{f})   $ via the natural inclusion $   \Gb  (   \mathbb{A}_{f,[n]}   ) \hookrightarrow   \Gb  (  \mathbb{A}_{f}   ) $.    For each $ n \in \mathcal{N} $, set   
\begin{equation}   \label{Hgn}   U_{g_{n}} : =  U   \cap g_{n} K g_{n} ^{-1}   
\end{equation}    Then $  U_{g_{n}} $ are compact open subgroups of $ \mathbf{H}(\Ab_{f}) $ and  \begin{equation}    \label{Hnmodular}  
U_{g_{n}}   =     \left (  U^{[n]}\cap K^{[n]} \right )    \cdot    \prod   \nolimits     _{   \ell |  n }   H_{\ell,    g_{\ell}} 
\end{equation}
where $ U^{[n]} = U / \prod_{\ell \in [n]} \mathcal{O}_{\ell}^{\times} $.   
The groups $  U _{g_{n}} $ form a lattice (in the sense of order theory)  where      $ m  \mid    n    $ implies $  U _{g_{m}  } \supset  U _{g_{n}} $.   Moreover,     $  U _{g_{m}} / U _{g_{n}}  \simeq \Delta_{n/m} $
where   $$ \Delta _{k } : = \prod   \nolimits    _{ \ell | k } \Delta_{\ell} $$     for $ k \in \mathcal{N}     $.     For $ n \in  \mathcal{N}  $,   let     $ E[n]         $   be       the abelian   extension of $ E $ corresponding to $ U_{g_{n}} $ via (\ref{CFTforEnoC}), i.e.,   $ E[n] $ is   the     field such that $ \Gal(E^{\mathrm{ab}}/ E[n]) $ is identified with $  E^{\times} \backslash E^{\times} U _{g_{n}}   %   E^{\times} 
\subset \mathbf{H}(\QQ) \backslash \mathbf{H}(\Ab_{f})    $ via the Artin map.   Clearly,  $ E[m] \subset E[n] $ for $ m  \mid  n $.   In order to describe $  \Gal(  E[n]/E[m]  ) $, we need to take the units of $ \OO_{E}  $ into account. Let  \begin{equation} \nu_{n} : =  U_{g_{n}} \cap E^{\times}  \subset \OO_{E}^{\times}  \label{nunmodular} 
\end{equation} and  set    $  v   _     { n } : = | \nu_{n} | $.  Note that $ v _{n} \in \left \{ 1, 2 , 4, 6  \right \} $  since the the possible orders of the  group of units of imaginary quadratic fields   are     $     2 $, $ 4 $ or $ 6 $. Again, the groups    $ \nu_{n} $ form a lattice and $ m  \mid    n $ implies that $\nu_{m} \supset  \nu_{n} $.   Set   
\begin{equation}   \label{numn}   \nu^{m}_{n}  : = \nu_{m} /  \nu_{n}   
\end{equation}    
Then $  \nu^{m}_{n} = \nu_{m} U_{g_{n}} / U_{g_{n}}   $ is a subgroup of $ U_{g_{m}} / U_{g_{n}}   \simeq   \Delta_{   n  /  m     }  $.   
\begin{lemma}   \label{ESlayers}              For all $ m , n \in \mathcal{N} $ with $ m \mid n $,  the  Galois   group      $ \Gal ( E[n] / E[m] ) $ is isomorphic to $ (  \Delta_{ n /  m  }  )    /  \nu^{m}_{n} $. In particular,  the degree of extension $ E[n]/E[m] $ is $ | \Delta_{n/m}|\cdot  (v_{m}/v_{n})^{-1}         $.    
\end{lemma}
\begin{proof} We have   $ \Gal(   E ^ {  \mathrm{ab} }  / E[a] )   \simeq U_{g_{a}} E^{\times} / E^{\times} $ for any $  a  \in \mathcal{N} $.  Therefore 
\begin{align*} \Gal( E[n]/ E[m] )  &   \simeq U _{g_{m}} E^{\times} / (  U _{g_{n}} E^{\times}   )  \\  
&    \simeq (  U _{g_{m}}  \cdot    U _{g_{n}} E^{\times}    )/  ( U _{g_{n}} E^{\times}   ) \\
& \simeq  U _{g_{m}} / ( U _{g_{m}} \cap   U _{g_{n}}    E ^ {  \times   }    ).    \\ 
%& \simeq  H_{m} / ( \nu_{m}  H_{n} )  % (\Delta _ { m /  n }  ) / \nu^{m}_{n}  \\
& \simeq  U _{g_{m}} / (  \nu_{m} U _ {g_{n}} )  \\
&  \simeq (U _{g_{m} } / U _{g_{n}}  ) / ( \nu_{m} U _{g_{n}} / U _{g_{n}} ) \\
&   \simeq    \Delta_{n/m} / \nu^{m}_{n} .
\end{align*} 
The claim on cardinality is then immediate.  
\end{proof}

\begin{remark}Note that $ H_{\ell, g_{\ell}} = \Hb(\QQ_{\ell}) \cap g_{\ell} K_{\ell} g_{\ell}^{-1} $ 
coincides with $ U_{\ell} \cap g_{\ell} K_{\ell} g_{\ell}^{-1} $, so we may also 
denote this group by $ U_{\ell, g_{\ell}} $ in line with our notation. For more general Shimura data where $ \Hb $ is not necessarily a torus, 
the local group $ H_{\ell, g_{\ell}} $ rarely equals $ U_{\ell, g_{\ell}} $, and is also   not  necessarily a  subgroup of  $   U _{\ell }   $ (or  even  its  conjugates by $ \Gb(\QQ_{\ell})  $).              
%and 
This discrepancy leads to  significant additional technical difficulties  in   establishing  horizontal  norm   relations for the method described in   \S       \ref{LSZsec}. See Remark  \ref{volumeremark}.     
\end{remark}

\subsection{Lattice  Counting}  
In this  subsection, we recall  some  basic facts   on   lattices     and  establish a  combinatorial  lemma on trace maps with  respect  to    $ \Delta_{\ell} = \OO_{\ell}    ^   {   \times   }     /  H_{\ell, g_{\ell}   }   $.   
\begin{definition}   
Let $ \ell $ be any rational  prime and $ V  =  V_{\mathrm{std} }  =  \QQ_{\ell }  \oplus  \QQ  _ { \ell  }     $  be the standard vector space of dimension $ 2 $. A  \emph{lattice} in $ V $ is a $ \ZZ_{\ell} $-submodule spanned  by a  $  \QQ_{\ell} $-basis for    $ V $.  We   let  $  \mathcal{L}   $ denote the set of   all     lattices in $ V  $. The  \emph{standard lattice} $   L _ { \mathrm{std}    } \in    \mathcal{L} $ is the lattice generated by the standard basis.   
 \end{definition}   
Each $ g \in  \Gb ( \QQ_{\ell}  )  $ acts on $ V $ by linear transformations   and      sends a lattice to a lattice, thus  giving us a  left   action $ \Gb ( \QQ_{\ell}  )  \times  \mathcal{L}   \to   \mathcal{L}  $. The  stabilizer of the standard lattice   $  L _  {   \mathrm{std }   }    $     is precisely $  K _ {\ell }  =  \GL_{2}    ( \ZZ_{\ell} ) $   and therefore one obtains a  bijection   
\begin{align}   \label{latticebijection}    \Gb ( \QQ_{\ell} )   /   K_{\ell}   \xrightarrow{\sim}   \mathcal{L}  \\
g K_{\ell}  \mapsto  g    \cdot     L_{\mathrm{std}}  \notag    
\end{align}  
For $   \ell  \in  R $, consider $  E_{\ell}  =   E \otimes  \QQ_{\ell} $ as the standard  vector space with basis $  \omega_{1} \otimes 1,   \omega _{2} \otimes  1  $. The  standard  lattice then coincides with $  \OO_{\ell}  : =  \OO_{E}  \otimes  _ { \ZZ }    \ZZ_{\ell } $.     We note that $  \OO_{\ell} =  \OO_{\lambda}   $ (the ring of integers of $ E_{\lambda} $) if $ \ell $ is inert,  and $$  \OO_{\ell}  =    \ZZ _ { \ell }   \omega_{1}  \oplus   \ZZ_{\ell}   \omega_{2}   =    \ZZ_{\ell}  \beta _ { 1 }  \oplus  \ZZ_{\ell }  \beta _ { 2 }  =  \OO_{\lambda}  \oplus  \OO_{  \bar{\lambda}   }    $$        if  $   \ell  $   is   split.           For  $ \ell \in  R $,   let 
 \begin{equation}  \xi_{0} :   =   \sum     _  {   \gamma  \in     \Delta  _ {  \ell }    }   \ch (   \gamma   g  _  {  \ell   }    K _ { \ell } )  
 \end{equation} 
considered as an element of $ \in   \mathcal{C}_{\ZZ} ( \Gb (  \QQ_{\ell }    )    / K_{\ell } ) $.  Via  (\ref{latticebijection}), $ \xi_{0}$ represents an element in $ \mathbb{Z}\langle \mathcal{L}  \rangle  $.      %   For  $ \kappa  \in  \mathbb{F}_{\ell} $, let $ [\kappa]$ denote a lift to $ \ZZ_{\ell}$.       
 \begin{lemma}   \label{formallattice}  The element  of    $ \mathbb{Z}\langle \mathcal{L} \rangle   $ corresponding to $ \xi_{0} $ is  the formal sum of all lattices $ \eta  \cdot L_{\mathrm{std}}  $ where           
\begin{enumerate}       [label = \normalfont (\alph*)]      
\item  $      \eta \in   \Set* { \left (    \begin{smallmatrix}  1/ \ell   &      \\    i/\ell &    1   \end{smallmatrix}  \right )    \given i = 0,  \ldots, \ell - 1 }    \cup      \Set*{  \left (  \begin{smallmatrix} 1   & \\  &  1/\ell    \end{smallmatrix} \right )  }     $  if  $  \ell  $  is  inert,    \\[-0.7em]        
\item  $     \eta \in   \Set*{  k    _   { \ell }    \left (  \begin{smallmatrix} 1  &  \\   i /  \ell       &  1   \end{smallmatrix} \right )   k_{\ell}^{-1}        \given    i = 1, \ldots, \ell - 1 }    $  if  $  \ell  $   is   split.
\end{enumerate} 
In particular, $ | \Delta _{\ell} | $ equal $ \ell + 1 $ if $ \ell $ inert and $ \ell - 1 $ if $ \ell $ is split. 
\end{lemma}

\begin{proof}   First   observe  that  $   \OO _{  \ell } ^ { \times }   = \Hb ( \QQ_{\ell} )  \cap  K_{\ell} $ is  the  stabilizer  in $  \Hb    ( \QQ_{\ell } )  $ of the standard lattice $  L  _ { \mathrm{std } } =  \OO_{E}  \otimes   _ { \ZZ }  \ZZ_{\ell}  $, where $  \mathbf{H}(\QQ_{\ell}) $ acts on $ \mathcal{L} $ via $ \iota $.   Similarly,    $ H_{\ell,   g_{\ell}}   =  \Hb(\QQ_{\ell} ) \cap g_{\ell} K_{\ell} g_{\ell}   ^ {  - 1   }   $ is the stabilizer in $  \Hb    ( \QQ_{\ell} ) $ of the lattice $$  L _ { g_{\ell}  }   :  =  g_{\ell}  (  L _ {\mathrm{std} }  ) =   \mathbb{Z}_{\ell}      \langle g_{\ell}  \omega _{1}  ,  g_{\ell}  \omega _{2}  \rangle    \in   \mathcal{L} .  $$
Therefore, $ \xi_{0} $ represents the formal sum of lattices in the  $ \OO_{\ell}^{\times} $-orbit of $ L_{g_{\ell} } $. \\

\noindent a) If $ \ell $ is inert,  $ L  _ { g_{\ell} } =  \langle  \ell^{-1}   \omega_{1}  ,   \omega_{2}   \rangle $.  Since $ \ell \OO_{\ell} \subsetneq \ell   L   _   {  g_{\ell }    }     \subsetneq  \OO_{\ell}  , $ 
the lattices $ L $  in the orbit of $ L _ { g_{\ell} } $ under the action of $ \OO_{\ell}^{\times}    $  must also  satisfy  $   \ell  \OO_{\ell}   \subsetneq  \ell  L   \subsetneq     \OO_{\ell} $. As $ \OO_{\ell} =  \OO _ { \lambda}  =  
L_{\mathrm{std} } $ by  our convention, the lattices $ \ell L $ thus obtained  correspond to a subset of the  
set of one dimensional $\mathbb{F}_{\ell}$-vector subspaces of  $$   \mathbb{F}_{\lambda}  : =  \OO_{\lambda} / \ell \OO_{\lambda}  =   \mathbb{F}_{\ell}   [\omega_{1}]  \oplus  \mathbb{F}_{\ell}   [\omega_{2}]    $$
where $  [\omega_{i}]$ denotes the reduction of $ \omega_{i} \in \mathcal{O}_{\lambda} $ modulo $ \lambda $.       Since $ \OO _ { \ell } ^ { \times }  = \OO_{\lambda} ^ { \times }  $ acts transitively on $ \mathbb{F}_{ \lambda } ^  { \times } $,  the $ \mathcal{O}_{\ell}^{\times}$-orbit of $ L_{g_{\ell}}  $  is   the  set of  \emph{all} the lattices $ L   \in   \mathcal{L}   $ such that $ \ell \OO_{\ell} \subsetneq    \ell    L  \subsetneq     \OO_{\ell} $. Now the number of one dimensional $ \mathbb{F}_{\ell}$-vector subspace in  $ \mathbb{F}_{\lambda} $ is exactly $ | \mathbb{F}_{\lambda} ^{\times } / \mathbb{F}_{\ell} ^{\times} | = \ell + 1    $, since  each element $ \vec{x} \in \mathbb{F}_{\lambda}^{\times} $ spans the subspace  $ \mathbb{F}_{\ell} \vec{x} $ and any $ \vec{y} \in  \mathbb{F}_{\ell}^{\times} x $ determines the same subspace.  These $  \ell + 1 $ subspaces are spanned by $$ [\omega_{1}] , \quad  [\omega_{1}]  +  [ \omega_{2}], \,  \,     \ldots  ,   \,  \,    [\omega_{1}] + (\ell - 1 )   [\omega_{2}] \,   \,  \text{ and }  \,  \,  [\omega_{2}] . $$
Therefore, the  $ \ZZ_{\ell}$-lattices spanned by $$   \{  \ell^{-1}  \omega_{1}    ,   \omega_{2}  \}  , \quad  \{  \ell^{-1}  \omega_{1}  +     \ell^{-1} \omega_{2} ,   \omega_{2}     \}   ,  \,       \ldots  ,     \,    \{ \ell^{-1}  \omega_{1}  +  \ell^{-1} (\ell- 1)   \omega_{2} ,   \omega_{2}   \}  \,   \text{ and }  \,  \{  \omega_{1} ,  \ell^{-1}  \omega_{2}  \}     $$
represent the orbit of $ \mathcal{O}_{\ell}^{\times} $ on $ L_{g_{\ell}}$.  These are exactly the lattices $ \eta \cdot L_{\mathrm{std}} $  as in the  claim.     \\

\noindent  b) If $  \ell  $  is    split on the other hand,  the  group    $ H  _ { \ell,   g_{\ell   }    } $ is the stabilizer in $ \OO_{\ell}  ^ { \times  }  $ of the lattice $   L  _ {   g _ { \ell }   }    =  \langle \beta   _ {  1  }      + \ell ^{-1}   \beta _  {2} ,    \beta_{2}  \rangle $.  Now   $ \OO_{\ell} ^ { \times } =   \OO_{\lambda} ^ {\times } \times  \OO_{ \bar{ \lambda}   } ^ { \times }  \cong  \ZZ_{\ell} ^  { \times }  \times  \ZZ_{\ell } ^  { \times } $ acts on $ \OO_{\ell} = \ZZ_{\ell} \beta_{1}  \oplus  \ZZ_{\ell}  \beta_{2} $ componentwise.  So   if    $   \gamma  =    (\gamma_{1},  \gamma_{2})  \in   \OO_{\lambda} ^{\times }  \times  \OO_{   \bar{  \lambda } } ^ {  \times  }   $,  then   $$   \gamma   \cdot       L _ {   g  _ {  \ell  }    }   =   \langle  \gamma _ { 1 }   \beta  _ { 1  }    +    \ell^{-1}      \gamma_{2}  \beta_{2} ,     \gamma _ { 2}  \beta _ { 2 }  \rangle  =    \langle   \beta _ { 1 }  +   \ell  ^ { - 1 }   \gamma _ { 1 } ^{-1}   \gamma  _{  2 }     \beta _ {   2    }  ,   \beta_{2}     \rangle . $$
This lattice is   equal to $ L _  {g_{\ell}  }$ if and only if $ \gamma _{  1} \gamma_{2} ^ {  -1 } \in 1 + \ell \ZZ_{\ell} $.  Thus,  there are  exactly   $  \ell -  1  =  | \ZZ_{\ell} ^ { \times }  /  ( 1 + \ell \ZZ_{\ell}  )  | $  distinct lattices in the orbit of $ \OO_{\ell}^{\times} $ on $ L_{g_{\ell} } $ and we find representatives by taking $ \gamma_{1} = 1 $ and $ \gamma_{2} =  i $  for $ i = 1, \ldots, \ell - 1 $.    
\end{proof}   
In what follows, we denote $ \gamma_{i}  :  =  \left ( \begin{smallmatrix} 1 / \ell  & \\   i /  \ell     &     1   \end{smallmatrix}   \right )  $ for $  i = 0  , \ldots, \ell - 1 $ and   $   \gamma _{\ell}:  =  \left (  \begin{smallmatrix} 1   \\ &    1 /  \ell      \end{smallmatrix}   \right ) $.    
\begin{lemma} \label{Telldecompose}  For all $ \ell \in R $,   $ \ch ( K _{\ell} \sigma_{\ell}^{-1}   K  _{\ell }  ) \in   \mathcal{C}_{\ZZ}(\Gb(\QQ_{\ell}   )    /K_{\ell}) $ corresponds   to    $ \sum_{i=0}^{\ell} \gamma_{i} \cdot L_{\mathrm{std}}  $ in $ \ZZ\langle \mathcal{L}  \rangle   $.     
\end{lemma}

\begin{proof} This amounts to describing the orbit of $ K_{\ell} $ acting on the lattice $ \langle \ell^{-1}   \omega_{1} ,   \omega_{2} \rangle $ which leads to  a similar  argument as in part (a) of Lemma  \ref{formallattice}.          
\end{proof}

\subsection{Norm Relations}   \label{normrelationmodularsec}                  
For any prime   $ \ell \in  R   $,   define   \emph{local test data}            
\begin{equation}   \label{zetamodularvector}       \zeta_{ \ell }   :    =   \mathrm{ch}(K_{\ell} )  -  \mathrm{ch} ( g    _  { \ell } K _ { \ell }  )    , \quad \quad  \zeta_{0, \ell} = \ch(K_{\ell})  
\end{equation} 
in $ \mathcal{C}_{\ZZ}\big(\Gb(\QQ_{\ell})/K_{\ell} \big) $.  For  $ n \in \mathcal{N}   $, set 
\begin{equation} \zeta_{n} : = \otimes _ { \ell | n }  \zeta_{\ell} \in   \mathcal{C}    _{\ZZ}  \big ( \Gb(\QQ_{[n]})   / K_{[n]}   \big   ) 
\end{equation} 
which consists of $ 2^{ \#     [ n ] } $ terms of the form $ \mathrm{ch}(gK) $ with   coefficients in $  \left \{ \pm 1  \right \}   $.    
Denote by $ \Ab_{f}^{R} $ the restricted tensor product of $ \QQ_{\ell}$ for $ \ell \notin R $ and write $ K = K_{R} K^{R} $, $ U = U^{R} U_{R} $ where $ K_{R} = \prod_{\ell \in R} \GL_{2}(\ZZ_{\ell}) $ and $ U_{R} =  \prod_{\ell \in R }  \mathcal{O}_{\ell}  ^  {  \times  }    $.    
Fix any  $$ \zeta^{R} \in    \mathcal{C}_{\ZZ} \big (  \Gb(\Ab_{f}^{R})/K^{R}  \big   )   $$
that is invariant under the action of $ U^{R}  $
%Denote $ K_{R} = \prod_{\ell \in R}\GL_{2}(\ZZ_{\ell}) $, $ K^{R} = K / K_{R} $.  Fix any  $$  \zeta^{R} \in    \mathcal{C}_{\ZZ}( \Gb(\Ab_{f}^{R}/K^{R})  . $$
and  set  $$   \zeta_{n  ,   f   }  : =  \zeta^{R} \otimes \zeta_{0, R}^{n}  \otimes   \zeta_{n}  \in  \mathcal{C}_{\ZZ}(\Gb(\Ab_{f})/K) $$ 
where $ \zeta_{0, R}^{n}= \otimes_{ \ell \in R \setminus [n] }   \zeta_{0, \ell}   $.  
%where $ \Ab_{f}^{R} $ denotes the restricted tensor product of $ \QQ_{\ell} $ for $ \ell \notin R $ and $ K^{R} = K / K_{R} $.  
\begin{definition} For $ n \in \mathcal{N} $,  the  $ n $-th \emph{Euler system divisor class}  is  defined   to    be  $$ y_{n}   =  \psi \big (  v_{1} v_{n}^{-1}      \cdot [ \zeta _{n, f} ] \big)   \in \mathcal{Z}_{K}      $$
where $ \psi $ is as in (\ref{psiparamodular}) and $ v_{n} $ denotes the cardinality of $ \nu_{n} $   (\ref{nunmodular}).         We call $ y_{1} $ the \emph{bottom class}  of the  system.
\end{definition}  
The CM   divisors $  y _{n}  $  are defined over $ E [ n ] $, i.e.,  $ y_{n} \in \mathcal{Z} ( U_{g_{n}} ) $.   Indeed, $    (U^{[n]} \cap K^{[n]}) $ acts trivially on $  \zeta^{R} \otimes \zeta_{0, R}^{n} $   by  assumption      and   $   U   _  { \ell ,   g_{\ell}   }    =   H_{\ell,   g_{\ell} }     \subset  \mathcal{O}_{\ell} ^ {\times}      $ stabilizes $ \zeta_{\ell} $     by  construction.      Moreover, for any $ \lambda \in \Lambda $ above  a   prime    $ \ell \in  R   \setminus [n] $, the  class   $  y _{n} $ is unramified over $ \lambda $ as $  U _{g_{n}} $ can be written as $  \mathcal{O}_{\ell}^{\times} U^{\ell} $ for some subgroup $  U^{\ell} $ of $ \Hb(\Ab_{f}/\QQ_{\ell}) $.      Thus,   the action of the  geometric  Frobenius $ \mathrm{Frob}_{\lambda}^{-1}      $ at  a prime $ \lambda $ on the  divisor     $ y_{n} $ is well-defined for any such $ \lambda $.   Let $$ \mathrm{Tr}^{E[n\ell]}_{ E[n]}  : \mathcal{Z}(U_{g_{n\ell}})  \to   \mathcal{Z}(U_{g_{n}}) $$ 
denote the trace map induced by summing over conjugates by  elements in  $ \Gal(E[n  \ell    ]  /  E[n] )   $.  
\begin{theorem}   \label{CMnormtheorem}            For all $ \ell \in R $ and $ n  \in  \mathcal{ N }  $ such that $ \ell  \nmid  n  $,     we have  \begin{equation*}  \mathfrak{H}_{\ell, *} (  \mathrm{Frob}_{\lambda}^{-1} )  y_{   n   }  =   \mathrm{Tr}^{E[n\ell]}_{E[n] }  (   y    _{n   \ell   } )  
\end{equation*}
as elements of $ \mathcal{Z}(U_{g_{n}}   )     $.    
\end{theorem}

\begin{proof} 
By the properties of the isomorphism  $ \psi   $   (\ref{psiparamodular}),  it suffices to establish that    
\begin{equation}   \label{reduction1}  \mathfrak{H}_{\ell, *} ( h_{\ell} ) \cdot  [ \zeta _{n,  f}   ]     \overset{?}{=}    \sum  _ { \gamma \in \Gal(E[n\ell]/E[n])}      \gamma \cdot [   v_{n}/v_{n\ell }    \cdot   \zeta _{ n \ell   ,   f     } ]  \end{equation}    
in $ \mathcal{F} $, where  $ h_{\ell} $ is as  in    (\ref{hellfrobenius}).   Since $ E^{\times} = \Hb(\QQ) $ acts trivially on $ \mathcal{F} $ and since $ \nu_{n\ell} \subset \nu_{n} \subset E^{\times} $,  we have $$ [ v_{n \ell}/ v_{n} \cdot \zeta_{n\ell, f} ] =  \sum_{  \delta \in \nu^{n}_{n \ell}      }   \delta   \cdot     [    \zeta_{n\ell, f}]  $$ 
where $ \nu^{n}_{  n     \ell} $ is as in (\ref{numn}).   
So Lemma   \ref{ESlayers} and  the reciprocity law (\ref{GaloisonCM}) imply that    (\ref{reduction1}) is equivalent to    
\begin{equation}   \label{reduction2}  \mathfrak{H}_{\ell, *} ( h_{\ell} ) \cdot  [ \zeta _{n,  f}   ]     \overset{?}{=}    \sum  _ { \gamma \in   \Delta_{\ell} }          \gamma \cdot [  \zeta _{ n \ell   ,   f     } ]  \end{equation}  
Now observe that the components of $ \zeta_{n} $ and $ \zeta_{n \ell} $ agree  away from $ \ell $. Since  both   $ \mathfrak{H}_{\ell, *}(h_{\ell}) $ and $ \Delta_{\ell} \subset  \mathcal{O}_{\ell}^{\times }   \subset    K_{\ell} $ only affect  the   components at $ \ell $,  relation    (\ref{reduction2}) would follow  from
\begin{equation}   
\label{reduction3} 
%\begin{split}   
\Big ( \ell   \cdot          \mathrm{ch}   (  K )  _ {  *  }      - ( h_{\ell} ,   \mathrm{ch} (  K     \sigma  _  {   \ell  }  ^{-1}  K     ) _ { *  }   )
+            (  h_{\ell } ^ {  2  } ,          \,    \mathrm{ch} ( K \tau  _ {  \ell  }  ^{-1}       K  ) _{*}  )  \Big  )   \cdot [\ch(K ) ]      \overset{?}{=}  |\Delta_{\ell} | \cdot  [  \ch(K)  ]  -   \sum_{\gamma \in \Delta_{\ell}}   [\ch(\gamma   g_{\ell}  K  ) ]     
%\end{split}   
\end{equation}
in $ \mathcal{F} $.\footnote{We could replace $ K $ by $ K_{\ell}$ everywhere and attempt to prove this relation in $ \mathcal{C}_{\ZZ}(\Gb(\QQ_{\ell})/K_{\ell})$  at this stage, but the resulting equality doesn't hold at inert primes. We have yet to use  the  fact  that  the      geometric     Frobenius $ h_{\ell} $ for $ \ell $ inert  acts trivially.}  
As in Lemma \ref{Telldecompose}, we denote $ \gamma_{i} :  =  \left ( \begin{smallmatrix} 1 / \ell  & \\   i /  \ell     &     1   \end{smallmatrix}   \right )  $ for $  i = 0  , \ldots, \ell - 1 $ and   $   \gamma _{\ell}:  =  \left (  \begin{smallmatrix} 1   \\ &    1 /  \ell      \end{smallmatrix}   \right ) $.    \\    %   Denote $ \gamma _{0}:  =  \left (  \begin{smallmatrix} 1 \\ &  \ell \end{smallmatrix}   \right ) $ and $ \gamma_{i} :  =  \left ( \begin{smallmatrix} \ell  &  i \\ &    1   \end{smallmatrix}   \right )  $ for $  i = 1 , \ldots, \ell . $ 
%Then $$ \ch  ( K_{\ell} \sigma_{\ell}^{-1}  %K_{\ell} ) = \sum\nolimits _{i=0}^{\ell} \ch( \gamma_{i} K _ { \ell  }    ) $$ 
%by Lemma \ref{Telldecompose}. \\   

\textit{Case 1: $\ell$ is inert.}  Recall that (C4)  requires  $ K^{\ell} $  to contain  the element     $   \chi^{\ell} : =    \mathrm{diag}(\ell,   \ell) \in \Gb(\QQ) $ embedded diagonally in $ \Gb(\mathbb{Ab}_{f}/\QQ_{\ell}) $. So  $ h_{\ell} K = h_{\ell}  \chi^{\ell} K  $ and  clearly, $ h_{\ell}  \chi^{\ell}  = \mathrm{diag}(\ell ,\ell)   \in   \mathbf{Z}(\QQ)  \subset \Hb(\QQ)   $.  Therefore $$  h_{\ell} \cdot [ \ch(K  )  ]  = [ \ch ( h_{\ell} K )]  = [ \ch ( h_{\ell} \chi^{\ell} K)] = \ch[K]  .     $$ 
So (\ref{reduction3}) would follow from the equality   
\begin{equation}     \label{reduction4intert}
\ell   \cdot  \mathrm{ch}   (  K )   -  \bigg (  \sum  \nolimits _{ i = 0 }   ^ {  \ell   }      \ch ( \gamma_{i} K   )  \bigg  )          +   \mathrm{ch} ( K )  
 \overset{?}{=}         (  \ell + 1  )  \cdot  \mathrm{ch} ( K  )  -  \sum   \nolimits   _  { \gamma   \in   \Delta  _ { \ell   }       }   \ch (    \gamma   g _ { \ell }  K    ).
\end{equation} 
in     $   \mathcal{C}_{\ZZ}( \Gb (\Ab_{f})     / K ) $.
Canceling  $  (\ell+1)  \cdot   \ch(K)     $ on both sides of (\ref{reduction4intert}),  we are reduced to showing that     $$  \sum   \nolimits  _{ i = 0 } ^ {  \ell  }   \ch ( \gamma_{i} K )   \overset{?}{=}     \sum   \nolimits    _  { \gamma  \in   \Delta _ { \ell }  }      \ch (    \gamma    g _ { \ell }    K ). $$   
But this follows from the local equality established in Lemma \ref{formallattice} (a).    \\

\noindent  \textit{Case 2: $ \ell $ is split.}    Arguing similarly as in the inert case,  (\ref{reduction3}) would follow from the local equality 
\begin{align}   \label{reduction4split}        \ell   \cdot          \mathrm{ch}   (  K    _ { \ell   }  )   -      \mathrm{ch} ( h_{\ell} K _{ \ell  }  \sigma   _  {   \ell  }   ^ { - 1 }    K  _  {  \ell  }   ) +    \mathrm{ch} (  h_{\ell} ^  { 2 }     K  _ { \ell  }   \tau   _ { \ell }   ^ {  - 1   }     K )   \overset{?}{=}   (  \ell -  1 )  \cdot  \mathrm{ch}  ( K _ { \ell  }   )         -       \sum  _ {  \gamma   \in   \Delta _ { \ell }          }   \ch (    \gamma  g_{\ell}  K   _ {  \ell  }   ) 
\end{align} 
in $ \mathcal{C}_{\mathbb{Z}}( \Gb(\QQ_{\ell} ) / K_{\ell} )   $.  Since the matrix $ k_{\ell }$ (\ref{kell}) lies in $K_{\ell} $, we see  from  Lemma  \ref{Telldecompose}    that $ \ch(K_{\ell} \sigma_{\ell}  ^ {  - 1   }    K_{\ell} )  = \sum _{ i = 0 } ^ { \ell } \ch ( k_{\ell} \gamma_{i}   k_{\ell}^{-1} K_{\ell} ) $  as  well.      Now note that $$   h_{\ell}  (  k_{\ell} \gamma_{0} k_{\ell}^{-1}  )  K_{\ell}  =  K_{\ell}  , \quad   \text{  and   }      \quad  h_{\ell}  (  k_{\ell} \gamma_{\ell}   k_{\ell}^{-1}    )        K_{\ell}  =   k_{\ell}   \left (    \begin{smallmatrix}  \ell &    \\ & 1  / \ell   \end{smallmatrix}   \right  )     K_{\ell} =     h_{\ell}^{2}   \tau_{\ell}   ^ {  -  1   }     K_{\ell}   .    $$    
So the   left hand side of (\ref{reduction4split})   equals   $ (  \ell - 1 ) \cdot \ch( K_{\ell}  ) - \sum _{ i = 1}^{\ell   -   1    }  \ch \big  ( h_{\ell} ( k _ { \ell }  \gamma _{i} k_{\ell } ^{-1}    )  K_{\ell} \big  )   .   $
Thus   (\ref{reduction4split})   would follow  if $$  \sum  \nolimits    _ { i =  1 } ^ { \ell   -   1     }  \ch ( h_{\ell} (  k _ { \ell }   \gamma _  { i }  k_{\ell } ^{-1}   )      K _ { \ell }  )   \overset{?}{=}      \sum   \nolimits      _  {      \gamma   \in        \Delta _{\ell }   } \ch ( \gamma g_{\ell } K_{\ell}  )     .   $$  But since $ h_{\ell} k_{\ell} \gamma_{i} k_{\ell}^{-1}       k_{\ell}^{-1}     K_{\ell}  =   k_{\ell}   \left  (  \begin{smallmatrix}  1 \\  i / \ell &  1   \end{smallmatrix}   \right   )  K_{\ell} $, this   is a consequence of    Lemma  \ref{formallattice}(b). 
\end{proof}

\subsection{Projection to Galois cohomology}   \label{projection}    Let us now recover the norm relations from the  introduction with split  primes   incorporated.     Suppose that  $ K = \widehat{\Gamma}_{0}(N) $ (see Example \ref{Gamma0(N)})  where $ N   $ is  any  positive  integer.   Then $ R $ is the set of all primes $ \ell $ that do not divide $ N \cdot \mathrm{disc}(E)$. %  We  assume that $ (\omega_{1}, \omega_{2})$ is positively oriented with respect to $ (1, -i)$, so that the point $x_{\iota}$ in Lemma \ref{uniqueCMpoint} equals $ h_{0} $.       
We   assume that the Heegner hypothesis is satisfied, i.e., all primes dividing $ N $ are  split in $ E $.   % Set  $ \ch(g_{N} K ) $ is invariant under $ \Hb(\Ab_{f}^{R}) \cap K^{R}  \subset $.   
Set   $$  U = \Hb(\Ab_{f}) \cap g_{N} K g_{N}^{-1}  $$ % g_{N}^{-1} $  $$ \zeta^{R}  = \ch(g_{N} K )  $$
where $ g_{N} $ is as in (\ref{gN}) and pick $ \zeta^{R} = \ch(g_{N}K^{R} ) $.  By   Lemma   \ref{Heegnerlemma},   the    bottom class $ y_{1} $  in our Euler system is   a Heegner point in $ \mathcal{S}_{K}(E[1])$ where $ E[1] $ equals the Hilbert class field of $ E $. As noted in \S \ref{Heegnersec},   the   field  $ E[n] $ is the  ring class extension   corresponding to   the    adelic order  whose groups of units     equals     $ U_{g_{n}} =  U \cap g_{n}Kg_{n}^{-1} = \Hb(\Ab_{f}) \cap  g_{n} g_{N}   K (g_{n}g_{N})^{-1}  $.

Now let $ A $ be an elliptic curve of conductor $ N $  as  in   the introduction.  Identify  $ \mathcal{S}_{K} $ with $ Y_{0}(N) $ via (\ref{stdidenGamma0}), so that the unique compactification   $ \overline{\mathcal{S}}_{K}$ is identified with $X_{0}(N) $.   Recall that $ \mathrm{J}_{K}$  denotes the Jacobian variety of $ \overline{\mathcal{S}}_{K}$.   Let $ x_{0} \in X_{0}(N)(\CC) $ denote cusp corresponding to the class of $ \infty $.   Then $ x_{0} $ is defined over $ \QQ $  \cite[\S 1.2]{Rohrlich}, and there is a unique morphism  $   \overline{\mathcal{S}}_{K} \to   \mathrm{J}_{K} $ 
of $ \QQ $-schemes which sends $ x \in \mathcal{S}_{K}(F) $ to the class of $  x - x_{0} $ in $  \mathrm{J}_{K}(F)    $    
for any extension $ F $ of $ \QQ $   \cite[\S 2]{Milne}. We let $$   \jmath_{F} :  \ZZ \langle   \overline{ \mathcal{S}   }    _{K}(F) \rangle  \to \mathrm{J}_{K}(F)     $$
denote its unique extension to divisors defined on $ F $.     Let  $ \pi :   X_{0}(N) \to A     $    be the dominant map guaranteed by the modularity theorem, which sends the  rational   cusp $ x_{0}$  to the zero element of $ A $  and  let   $$  \mathrm{J}(\pi) :  \mathrm{J}_{K} \to A $$ be the unique morphism  induced by the universal property  of   Jacobians   \cite[Proposition 6.1]{Milne}.  Fix $ p $ to be any rational  prime.   For  each $  n  \in   \mathcal{N} $, we have a $ \Gal(E[n]/E) $-equivariant  composition   
\begin{equation}   \label{longcomposition}     \mathcal{Z}(U_{g_{n}})   \hookrightarrow  \ZZ  \langle  \overline{\mathcal{S}}_{K}(E[n])   \rangle   \xrightarrow{\jmath}  \mathrm{J}_{K}(E[n])   \xrightarrow{\mathrm{J}(\pi)}    A(E[n]) \to  \mathrm{H}^{1}(E[n],   \mathrm{T}_{p}(A) )    
\end{equation}   
Let $ \mathcal{N}^{p}$ denote the set of $ n \in  \mathcal{N} $ not divisible by $ p $. For any $ n \in  \mathcal{N}^{p}$,   we    define   $$  z_{n} \in \mathrm{H}^{1}(E[n],  \mathrm{T}_{p}(A))  $$  to be $  1/n    $    times   the image of $ y_{n} $ under this map.    % The  following result  is a  consequence  of   Theorem  \ref{CMnormtheorem} and the discussion in  Remark  \ref{normalizeremark}.         
%By  construction   
\begin{corollary} \label{heegcorollary}      For all $  n \in  \mathcal{N}^{p}  $ and $ \ell $   a prime with $ \ell n \in \mathcal{N}  ^  { p   }    $, we have $$ 
P_{\ell}(\mathrm{Frob}_{\lambda}^{-1}) ( z_{m} ) =   \mathrm{cores}^{E[n\ell]}_{E[n]} (z_{m\ell}) $$  
where $ P_{\ell}(X) $ denotes the   reverse    characteristic polynomial of $ \mathrm{Frob}_{\ell}^{-1} $ acting on $ \mathrm{T}_{p}(A) $.     

\end{corollary}

\begin{proof} This follows by Theorem \ref{CMnormtheorem} and the fact that pre-composition of $T_{\ell}$ for $ \ell \nmid n $ with  (\ref{longcomposition}) equals multiplying (\ref{longcomposition}) with  $ a_{\ell}  $ (see the proof of \cite[Proposition 3.7]{grosskoly}).   
\end{proof} 
\begin{remark} A partial result of this type is stated in \cite[Proposition 3.10]{darmon}, which says that given  a  class $ y $ at level $ E[n\ell] $, there exists another class $ y ' $ at level $ E[n] $ such that the trace of of $ y $ down to $ E[n] $ equals the image of $ y' $ under an appropriate Euler factor.   It is  however  unclear  from  the  statement  alone   if one can use this to construct an infinite system as in Corollary   \ref{heegcorollary}.       %to an infinite  system.   
\end{remark}

\begin{remark}By assuming that $a_{p} = p + 1 - \tilde{A}(\mathbb{F}_{p})$ is invertible in $\mathbb{Z}_{p}$, one can extend this system along the anticyclotomic $\mathbb{Z}_{p}$-extension of $E$ and thereby obtain a genuine Euler system, as,  for instance,   required    in \cite{JNS}. See \cite{loe}, which provides a fairly general method for carrying out this extension.          
\end{remark} 

\subsection{Cohomological formulation}   \label{cohoform}      The horizontal Euler system of Theorem \ref{CMnormtheorem} is formulated in terms of divisors on modular curves, since the Tate modules of Jacobians provide “access’’ to the Galois representation $\mathrm{T}_{p}(A)$. For higher-dimensional Shimura varieties, interesting (irreducible) Galois representations occur in the middle-degree $p$-adic étale cohomology,    just as $\mathrm{T}_{p}(A)   $ appears as a quotient of $   \mathrm{T}_{p, K} \simeq   \mathrm{H}^{1}_{\et}(X_{0}(N),\mathbb{Z}_{p}(1))$. In these higher-dimensional settings, however, there is no analogue of  Jacobian that  serves as    a   replacement  for     \'{e}tale     cohomology.  Consequently,  one  must carry out all constructions at the level of    cohomology itself.      Let us briefly explain how this may be done in the case of modular curves, so that the reader can  see the parallel with higher dimensions  more  easily.    

Let $ \mathbf{T}  \subset \Hb $ denote the torus of norm one elements, i.e., $ \mathbf{T}(\QQ) = \left \{   \omega   \in E \, | \,  \omega  \bar{\omega } = 1 \right \} $ where $ \bar{\omega} $ denotes the  complex  conjugate of $ \omega $.  There is a  norm  map $  \nu :  \Hb \to  \Tb $ which on $ \QQ $-points sends $  \omega  \in E $ to $  \omega/\bar{\omega} \in \Tb(\QQ) $.\footnote{The corresponding quotient map on adelic quotients corresponds to anticyclotomic extensions of $ E$.}   Let us denote $ \tilde{\Gb} = \Gb \times \mathbf{T} $. Then the diagonal   map $$ \tilde{\iota} = \iota \times  \nu  : \Hb  \to \tilde{\Gb}  . $$
extends to a morphism of Shimura data, where the underlying $ \tilde{\Gb}(\RR) $-conjugacy class $ \tilde{\mathcal{X}} $ of cocharacters is the class of $ h_{\mathrm{std}} \times (\nu \circ h_{0}  )    $.   The reflex field of the datum  $ ( \tilde{\Gb}, \tilde{\mathcal{X}}) $ is then $ E $.    Given a compact open subgroup $ \tilde{L} \subset \tilde{\Gb}(\Ab_{f}) $,  we let $ \tilde{\mathcal{S}}_{\tilde{L}} $ denote the corresponding canonical model of the Shimura variety attached to 
$ (\tilde{\Gb},   \tilde{\mathcal{X}}   ) $.   If $ \tilde{L} = KC $ where $ K \subset  \Gb(\Ab_{f}) $ and $ C \subset \Tb(\Ab_{f}) $, then we have a canonical isomorphism $$ \mathcal{\tilde{S}}_{\tilde{L}}   \xrightarrow{\sim}   \mathcal{S}_{K, E_{C} }    =   \mathcal{S}_{K  ,  E }  \times_{\Spec   E    } \Spec E_{C}   $$ 
of $ E $-schemes,   where $ E_{C} $ is a finite  dihedral extension of $ E $ determined by a  Shimura-reciprocity law for the datum $ (\Tb,  \left \{ \nu \circ  h_{0 }    \right \}   ) $  similar to the  one    in  \S \ref{canonicalmodsec}.        
For each $ \tilde{g} \in  \tilde{\Gb}(\Ab_{f})$ and compact open subgroups $ V \subset  \Hb(\Ab_{f}) $, $  \tilde{L}  \subset \Gb(\Ab_{f}) $ satisfying $ V \subset \tilde{g}L\tilde{g}^{-1} $,     we have a finite morphism $$    \tilde{\iota}_{\tilde{g}, \tilde{V}, \tilde{L}}  =   [\tilde{g}] \circ \iota_{ V, \tilde{g}\tilde{L}\tilde{g}^{-1} }  :   \mathcal{T}_{V}  \to     \tilde{\mathcal{S}}_{\tilde{g}\tilde{L}  \tilde{g}^{-1}}    \to  \tilde{\mathcal{S}}_{\tilde{L}}  $$
analogous to the  map  (\ref{Deligneig}),      
%let us denote $$ N(V) :=  \mathrm{H}^{0}_{\et}(\mathcal{T}_{V},  \ZZ_{p}), \quad \quad  \quad  M(\tilde{L})  =  \mathrm{H}^{2}_{\et}(\tilde{\mathcal{S}}_{\tilde{L}}, \ZZ_{p}(1))  $$
%the \emph{arithmetic}   \'{e}tale cohomology\footnote{} of $ \mathcal{T}_{V}$. Whenever $ V \subset \tilde{K} $, there is a  
which induces a   \emph{Gysin pushforward}   
\begin{equation}   \label{Gysinmap}   \tilde{\iota}_{\tilde{g}, \tilde{V}, \tilde{L} ,  * }  :   \mathrm{H}^{0}_{\et}(\mathcal{T}_{V} , \ZZ_{p} ) \to   \mathrm{H}^{2}_{\et}(\tilde{\mathcal{S}}_{\tilde{L}}, \ZZ_{p}(1))   
\end{equation}   
on  \emph{arithmetic}  $p$-adic  \'{e}tale  cohomology.   
Since each $ \tilde{\mathcal{S}}_{\tilde{L}} $ is an affine scheme over $ E $, the Hoschild-Serre spectral sequence\footnote{It is more appropriate to work with \emph{continuous} \'{e}tale cohomology \cite{Jannsen1988}, since taking inverse limit does not commute with spectral sequences in general.} induces a map 
\begin{equation}    \label{AbelJacobi}     \mathrm{AJ}_{\tilde{L}} :         \mathrm{H}^{2}_{\et}(\tilde{\mathcal{S}}_{\tilde{L}}, \ZZ_{p}(1))     \to  \mathrm{H}^{1}(E,  \mathrm{H}^{1}(\tilde{\mathcal{S}}_{\tilde{L}, \overline{\QQ}   } ,  \ZZ_{p}(1))   
\end{equation}   
referred to as the \emph{Abel-Jacobi} map.

Suppose $ \tilde{L}  $ is of the  form   $   KC $ from now on. Then we have an isomorphism $ \tilde{\mathcal{S}}_{\tilde{L}, \overline{\QQ}}  \simeq  \bigsqcup \nolimits _{\sigma}   \mathcal{S}_{K, \overline\QQ}  $
where $ \sigma $ runs over $ \Gal(E_{C}/E)$ and  we  have a  $  \Gal(\overline{\QQ}/E) $-equivariant  isomorphism 
$$  \mathrm{H}^{1}(\tilde{\mathcal{S}}_{\tilde{L}, \overline{\QQ}   } ,  \ZZ_{p}(1))   \simeq   \mathrm{H}^{1}\big (\tilde{\mathcal{S}}_{\tilde{L}, \overline{\QQ}   } ,  \ZZ_{p}(1)  \big  )  \otimes_{\ZZ_{p}} \ZZ_{p} [\Delta_{C}] $$ 
where $ \Delta_{C} = \Gal(E_{C}/C) $ and $ \ZZ_{p}[\Delta_{C}]$ denotes the group algebra of $ \Delta_{C} $. An application of Shapiro's  lemma  gives a canonical  isomorphism   
\begin{equation}   \label{Shapiro}   \varsigma_{\tilde{L}} :  \mathrm{H}^{1}_{\et}  \big  ( E,    \mathrm{H}^{1}(\tilde{\mathcal{S}}_{\tilde{L}, \overline{\QQ}   } ,  \ZZ_{p}(1)  \big )    \xrightarrow{\sim}     \mathrm{H}^{1}_{\et} ( E_{C}  ,       \mathrm{H}^{1}\big (\mathcal{S}_{K, \overline{\QQ}   } ,  \ZZ_{p}(1)  ) \big )
\end{equation}   
So the composition $  \varsigma_{\tilde{L}} \circ   \mathrm{AJ}_{\tilde{L}}  \circ \tilde{\iota}_{\tilde{g} ,  \tilde{V} ,  \tilde{L}, *} $ gives us a  map  $$  \mathrm{H}^{0}_{\et}(\mathcal{T}_{V},  \ZZ_{p}  )   \to \mathrm{H}^{1}_{\et}\big(E_{C},  \mathrm{H}^{1}_{\et}(\mathcal{S}_{K, \overline{\QQ}}, \ZZ_{p}(1) \big ) . $$
One may then pose the question of constructing a system of cocyle   classes   
\begin{equation}   \label{Galoiscohoetale}  c_{n}   \in      \mathrm{H}^{1}(E_{C_{n}},     \mathrm{H}^{1}_{\et}( \mathcal{S}_{K, \overline{\QQ}}, \ZZ_{p}(1)    )     \big )    
\end{equation}   for    an infinite lattice of compact open subgroups $  ( C_{n} )_{n}  $, which satisfy   
\begin{equation}   \ell^{-1} \mathfrak{H}_{\ell,  *} (\mathrm{Frob}_{\ell}^{-1}) ( c_{n} ) =   \mathrm{cores}^{C_{n\ell}}_{C_{n}}(c_{n\ell})    
\end{equation}   
where $ \mathrm{cores}^{C_{n\ell}}_{C_{n}} $ denotes the corestriction map from  the Galois cohomology at $ E_{C_{n\ell}}$ to the cohomology at  $ E_{C_{n}} $ and $ \mathfrak{H}_{\ell}(X) $ is as in  Definition \ref{HeckepolyGL2}.  This can be done by making suitable choices of $ \tilde{g} $ mirroring the choice of the  local element (\ref{zetamodularvector}).    We will explain how one can verify the \emph{existence} of these   local  choices by certain congruence conditions in \S \ref{examplessec}.    The  global  construction can then be carried out  as in  \cite[\S 3.4]{CZE} 

\begin{remark}   As noted in \S \ref{EichlerShimurasec}, there is an injection  $  \mathrm{T}_{p, K}  \simeq    \mathrm{H} ^{1}_{\et   } ( \overline{\mathcal{S}}   _{K , \overline{\QQ} } ,  \ZZ_{p}(1) )   \hookrightarrow    \mathrm{H}^{1}_{\et}(\mathcal{S}_{K, \overline{\QQ}} ,  \ZZ_{p}(1)    )  $. Thus the Galois representations that appear in the cohomology of the compactified curve all appear in the cohomology of the open curve. If  $$   \mathrm{H}^{1}_{\et}( \mathcal{S}_{K, \overline{\QQ}} , \ZZ_{p}(1)) \to   (\pi^{\vee})^{K} \otimes V^{\vee}   $$  is a projection to a Galois-automorphic piece (where $ V^{\vee} $ is a two-dimensional Galois representation over some finte extension of $ \QQ_{p}$) and $ \varphi \in \pi^{K} $ is a non-zero element,    then we can construct cocycles $ z_{n} \in \mathrm{H}^{1}(E_{C_{n}}, V^{\vee})$ by pairing the projection of $ c_{n} $   to   $ (\pi^{\vee})^{K}  \otimes V^{\vee}   $ with $ \varphi $.   The cocycles $ z_{n} $ lie in the Galois stable  $ \ZZ_{p}$-lattice of $ V^{\vee} $ given by the image of $ \mathrm{H}^{1}_{\et}(\mathcal{S}_{K, \bar{\QQ}},  \ZZ_{p}(1)) $ in $ V^{\vee} $.  Under certain technical hypothesis,  one can pull these classes back to the Galois cohomology of the lattice inside $ V^{\vee} $.    
%Let us also explain why it is $ \ell^{-1} \mathfrak{H}_{\ell, *}(\mathrm{Frob}_{\ell}^{-1}) $ that we would like as Euler factor in (
\end{remark} 
\begin{remark}  The author learned the idea of  introducing  a larger group $ \tilde{\Gb} $ in \cite{loe}.  It gives a more flexible control on the Galois variation of classes by intertwining Hecke and Galois actions on the target, and allows us to convert the problem of norm relations to one involving only Hecke operators.  This is essentially equivalent to introducing the Hecke and Galois action (\ref{HeckeGaloisaction}) in terms of Hecke operators of a single group, except we used the action of $ \Hb \times  \Gb  $. One may of course   replace $ \nu $ with maps to  other  tori.       For instance,  one may take $  \nu $ to be   the identity map $ \Hb \to \Hb  $ and try to construct a ``full" Euler system going up full tower of abelian extensions of $ E $. This   however   does  not  turn  out to  be  feasible.  % So the choice of the torus and the map $ \nu $
\end{remark}
\section{Integral test data}   \label{integraltestdatasec} 
In this section, we put the choice of the test data  (\ref{zetamodularvector}) on a more conceptual  footing. We will work abstractly in the setting of locally profinite groups   and formulate an abstract norm relation problem  in  the  spirit  of    Theorem  \ref{CMnormtheorem}.    Since the main goal is to illustrate how to prove norm relations rather than describe an actual construction of an Euler system, we will only make brief remarks on how the  abstract formalism applies  to  the  cohomology of Shimura varieties and hope the reader can make the connection  concrete  by referring to \S \ref{cohoform}.  The  notations  of this  section  are independent of  ones  introduced  in  the previous  ones.       %  However, the reader is invited  to refer to  \S \ref{cohoform} to see how the formalism  may be applied.    
\subsection{Abstract pushforwards}  \label{abspushsec}   Let $ G $ be a  unimodular  locally profinite group and $ M $  be a $ \ZZ_{p}$-module that is a smooth left representation of $ G $, i.e., any element in $ M $ is fixed by a compact open subgroup of $ G $. For brevity, we will refer to compact open subgroups of $ G $ as levels.   For each level $ L $ of $ G $, we let $ M(L) = M^{L} $ denote the $L$-invariants of $ M $. If $ L \hookrightarrow K $ is an inclusion  of   levels,    we have two maps   
\[
\begin{aligned}
\pr_{L,K}^{*} : M(K)    &   \longrightarrow M(L) &  &   &  
\pr_{L,K,*}  : M(L)   &     \longrightarrow M(K) \\ 
x &\longmapsto x,  &  &  \quad   &  x &\longmapsto \sum   \nolimits    _{\gamma \in K / L} \gamma x .
\end{aligned}
\]
that we refer to as  \emph{restriction} and \emph{induction}   respectively.  Moreover, for any $ g \in G $, we have \emph{conjugations}   
\[
\begin{aligned}
[g]^{*}_{K} : M(g^{-1}Kg)    &   \longrightarrow M(K) &  &   &  
[g]_{K,*}  : M(K)    &    \longrightarrow M(g^{-1}Kg) \\ 
x &\longmapsto g\cdot x,  &  &  \quad    &  x &   \longmapsto       g^{-1} \cdot x .
\end{aligned}
\]
These maps then model the behaviour of the cohomology of a Shimura variety over varying levels, and the representation $ M $ can be thought of as the direct limit of the cohomology over all levels.  For any two levels $ K, K' $ and $ g \in G $, we have a covariant Hecke correspondence $$ [K  g  K ' ]_{*} : M(K) \to M(K') $$ 
defined as the composition $$ M(K) \xrightarrow{\pr^{*}} M ( K \cap g K ' g^{-1}) \xrightarrow{\pr_{*}}  M( gK' g^{-1})   \xrightarrow{[g]_{*}}     M ( K' )   . $$   The  \emph{degree}  of this operator    is defined to be  $$   \deg  [K  g K ' ] _  {  *   }   =      |  K    ' \backslash  K '   g        K   |    $$ and we extend this notion   linearly   to linear  combinations of Hecke correspondences $ M(K') \to M(K)$.   We  define $ [K'gK]^{*} : M(K) \to M(K') $ as $ [K g^{-1} K']_{*} : M(K) \to M(K') $.   
\begin{remark}  Notice that when working with $ \ZZ_{p}$-coefficients, the cohomology of a Shimura variety at a finite level cannot be recovered by taking invariants of the direct limit over all levels.  This failure of ``Galois descent" introduces additional technical difficulties  that we will ignore for the purposes of  our  discussion.  For a detailed treatment of this issue, we refer the reader to \cite[\S   2]{CZE}.     
\end{remark} 

Suppose now   that   $ H $ is  another  unimodular profinite group and $$ \iota : H \to G $$ is a  closed embedding, via which we view $ H $ as a subgroup of $ G $.  Let     $ N $  be   a smooth representation of $ H $.  Suppose   that for each level $ L $ of $ G $ and $ V $ a level of $H $ contained in $ L $, we have  a morphism $$ \iota_{V , L , * } : N(V)  \to  M(L) $$ 
that satisfies   the obvious compatibility  conditions with  respect to the the restrictions, inductions and conjugations by elements of  $ H $ on the two sides  of  this    map.     We refer to the collection of the maps   $ \iota_{V, L, *} $ as a \emph{pushforward} and denote this collection   informally      by $ \iota_{*} : N \to M $. We moreover require that this pushforward satisfies Mackey's double coset axiom.  That is, for any levels $ K , L   \subset G $ and $ U \subset H $ satisfying $ U, L \subset K $,  we have a commutative diagram   
\begin{center}
    \begin{tikzcd}  [column sep = large]  \bigoplus_ { \gamma }  N ( U _ {\gamma}   )   \arrow[r,  "{ \sum [\gamma]_{*}}"]   &  
  M ( L  )  \\ 
 N   ( U )  \arrow[r, "{\iota_{*}}"]   \arrow[u, "{\oplus\,\pr^{*}}"] &  M  (K)   \arrow[u,"{\pr^{*}}", swap]  
    \end{tikzcd}
\end{center}
where $ \gamma \in U \backslash K  /  L $ is a fixed  set of representatives, $ U_{\gamma} = U \cap \gamma L \gamma^{-1} $ and  $ [\gamma]_{*} :  N   (  U  _ { \gamma }  )  \to  M  ( L )  $  denotes  the    composition   $$    N   ( U_{\gamma })  \xrightarrow {  \iota_{* } }  M   (  \gamma L \gamma ^{-1} )   \xrightarrow { [ \gamma ] _ { * }  }  M    (   L   ) .   $$       This  condition is independent of the choice of representatives $ \gamma $.  
The pushforward $ \iota_{*} $ then models the pushforwards in the cohomology of Shimura varieties obtained by an embedding of Shimura  data.    When pushing cycle classes, we can take $ N $ to be the $ \ZZ_{p}$-span of fundamental cycles of Shimura varieties in $ \mathrm{H}^{0}_{\et} $, i.e.,  the trivial  representation.   
\subsection{Completed pushforwards}   \label{comppushsec}    
We would like to encode the data of a pushforward $ \iota_{*} $ into a single representation. One way of achieving this is by working modulo $ \ZZ_{p}$-torsion. For any $ \QQ_{p}$-algebra $ R$,  let us denote $$ N_{R} = N \otimes_{\ZZ_{p}}
R , \quad \quad    M_{R} = M \otimes_{\ZZ_{p}} R  . $$
Abusing notation, we denote the induced map $ N_{R}(U) \to M_{R}(L) $ on invariants by $ \iota_{V, L, *} $ as well.   
Fix $ \QQ$-valued Haar measures $ \mu_{H}, \mu_{G} $ on $ H $, $ G $ respectively. Let $ \mathcal{H}_{R}(G) $ denote the  full Hecke algebra of $ G $ with coefficients in $ R $, which  is  the set of all $ R$-valued functions $ \xi : G \to \QQ_{p} $ that are  locally constant and compactly supported. It equals the union of $ \mathcal{H}_{R}(K \backslash G / K )$ over all levels $ K $, and the union is endowed with a convolution operation   that   equals $ \mu(K) $ times the convolution operation defined on $ \mathcal{H}_{R}(K \backslash G / K  ) $ in \S \ref{Heckecorrsec}. The  representation $ M_{R} $ then becomes a  left-module over $ \mathcal{H}_{R}(G) $, where  the action  satisfies   \begin{align*}   \ch(K ' gK) \cdot x   &  =  \mu(K)   \cdot     [K'gK]^{*}(x)    \\
&   =  \mu(K)   \cdot     \sum   \nolimits   _{ \gamma \in K' g K / K }   \gamma    \cdot  x     
\end{align*}    
for all  $ \ch(K   ' gK) \in  \mathcal{H}_{R}(G) $ and $ x \in M_{R}(K)$.      In what follows, we   will view $ N_{R} \otimes_{R} \mathcal{H}_{R}(G) $ and $ M_{R}$ as representations of $ H \times G $ in the following way:  
\begin{itemize}
    \item $ (h, g ) \in H \times G $ acts on $ x \otimes \xi \in _{R}       N_{R}   \otimes     \mathcal{H}_{R}(G) $ via $ x \otimes  \xi   \mapsto   hx \otimes \xi( h^{-1} (-) g) $,  
    \item $ (h, g) \in H \times G $ acts on $ y \in M_{R}$ via $ y \mapsto g \cdot y $.    
\end{itemize}
Recall that the smooth dual of a representation of a locally profinite group is the set of all dual vectors that are invariant under some compact open subgroup.      Let $ M_{R}^{\vee} $ denote the smooth dual of $ M_{R} $ and $$ \langle -  ,   -   \rangle  : M_{R}  ^ {  \vee }  \times   M_{R}    \to  R   $$ denote the induced pairing.   We  will   consider $ N_{R} \otimes   M_{R}^{\vee}          $ as a smooth representation of $ H $ where $ h \in H $ acts on $ N_{R}   \times M_{R}^{\vee} $ via $ x \otimes f \mapsto  h x \otimes h f $.    
\begin{proposition}  There is a unique intertwining map $  \hat{\iota}_{*}   :  N_{R} \otimes _ { R   }   \mathcal{H}_{R   }(G)  \to   M_{  R    }        $
of $ H \times G $-representations 
such that for any level $ L \subset G $, any level $ V \subset H $  that is contained in $ L $ and any element $ x \in N_{R}(V) $, we have $  \hat{\iota}_{*}(x \otimes \ch(L) ) =   \mu_{H}(V)  \cdot \iota_{V, L , *}(x)   $ 
in $ M_{\QQ_{p}}$.    
\end{proposition}   
\begin{proof} See \cite[Proposition  2.13]{Anticyclo}.
\end{proof}    
The following result is  version of Frobenius  reciprocity  for  smooth  representations.    
\begin{proposition}   \label{Frobreci}     For any intertwining  map  $  \mathfrak{Z}   :   N_{R} \otimes \mathcal{H}_{R}(G)  \to  M_{R}$ of $ H \times G  $-representations,  there is a unique  intertwining   map   $  \mathfrak{z} :   N_{R} \otimes M_{R}^{\vee} \to  R     $ of  $ H $-representations         %smooth  $ H $-representations 
such that $$ \langle   \varphi   ,   \mathfrak{Z} ( x \otimes \xi )  \rangle  =   \mathfrak{z}    (  x \otimes ( \xi \cdot   \varphi   ) ) $$    
for all $ x \in   R  $, $   \varphi   \in  M_{R}^{\vee}   $ and $ \xi \in   \mathcal{H}_{R}(G)   $. The mapping $ \Psi \mapsto  \psi $  thus   defined    induces a bijection   between     $  \mathrm{Hom}_{H \times G }   \big   ( N_{R} \otimes \mathcal{H}_{R}(G),  M_{R}   \big    ) $ and $ \mathrm{Hom}_{H} ( N_{R} \otimes   M_{R}^{\vee}, R    ) $. 
\end{proposition}   
\begin{proof}     See  \cite[Lemma  2.13]{Anticyclo}.     
\end{proof}   
\subsection{Integral test data} Throughout this subsection, we fix levels $ K , L \subset G $ and $ U \subset H $ such that $ U \subset K $ and $ L \triangleleft K $.     
Fix also an element $ x_{U} \in N(U) $ and let $ y_{K}  \in M(K) $ denote the pushforward  $ \iota_{U, K, *}(x_{U})  $.  Suppose we are given    a    $ \ZZ_{p}$-linear combination  $ \mathfrak{H} \in \mathcal{H}_{\ZZ_{p}}(K \backslash G / K ) $ of Hecke operators.    We would like to study conditions such that there exists a  class $ y_{L} \in M(L) $ such that   
\begin{equation}   \label{strongnormrel}       \mathfrak{H}_{*}( y_{K} ) = \pr_{L, K, *} ( y_{L} ) 
\end{equation}    
For applications to Euler systems, it suffices to establish such an equality modulo the $ \ZZ_{p}$-torsion in $ M(K) $. So we instead content ourselves with  describing conditions such that $    \mathfrak{H}_{*} (y_{K}) - \pr_{L, K, *} (y_{L}) \in M(K)_{\mathrm{tors}}  $.   
Equivalently, we wish to construct a class $ y_{L} \in M(L) $ such that   
\begin{equation}   \label{weaknormrel}   \mathfrak{H}_{*}(y_{K, \QQ_{p}}  )  = \pr_{L, K, *}(y_{L, \QQ_{p} }   )   
\end{equation}    
where $ y_{K, \QQ_{p}} \in M_{\QQ_{p}}(K) $, $ y_{L, \QQ_{p}}  \in   M_{\QQ_{p}}(L) $  are the images of $ y_{K} $, $ y_{L} $   respectively.    Inspired by  the construction in  \S \ref{Eulersystemsec}, we would like the class $ y_{L} $ to be given by a finite sum of  maps    
\begin{equation}   \label{mixedHecke}      [Vg L]_{*}   :      N(V) \xrightarrow{\iota_{*}} M(gLg^{-1})  \xrightarrow{[g]_{*}} M(L).   
\end{equation}  
applied to elements $ x_{V} \in N(V) $ for some levels $ V   \subset  g L g^{-1} $.   % We refer to  (\ref{mixedHecke})    as a \emph{mixed Hecke  correspondence}.   
That is, we would like our class $ y_{L} $ to satisfy   
\begin{equation}    \label{testpushclass} y_{L} = \sum   \nolimits   _{ \alpha }    [V_{\alpha}  g_{\alpha} L ] _{*}(x_{V_{\alpha}})   
\end{equation}    
for  a  finite  collection of twisting  elements   $ g_{\alpha}   \in G $ (which are not required to be  distinct for distinct $ \alpha $), compact open subgroups $ V_{\alpha} \subset g_{\alpha } L g_{\alpha}^{-1} $  and candidate classes $ x_{V_{\alpha} } \in   N(V_{\alpha})$ where  $ \alpha $  runs  over   some  finite   indexing set $ A $.      
\begin{lemma}   \label{norminlimitlemma}     The equality (\ref{weaknormrel}) holds with $ y_{L}$ as in (\ref{testpushclass})  if and only if  \begin{equation}    \label{norminlimit}       \hat{\iota}    _{*}( x_{U, \QQ_{p}} \otimes \mathfrak{H}  ) =   \sum \nolimits    _{\alpha \in A}   [U : V_{\alpha}]  \cdot    \hat{\iota}_{*}\big (x_{V_{\alpha}, \QQ_{p}} \otimes \ch( g_{\alpha} K) \big )    
 \end{equation}
 where $ [U : V_{\alpha}]$ denotes $ \mu_{H}(U)/\mu_{H}(V_{\alpha}) $.   
\end{lemma} 
\begin{proof} See  \cite[Note 3.1.2]{CZE}.
\end{proof}
Let $ \mathcal{C}(G/ K ,  N_{\QQ_{p}}) $ denote the $ \QQ_{p}$-vector space of all functions $ \xi :   G/ K \to N_{\QQ_{p}}$ that have finite support. The  input of $ \hat{\iota}_{*}$ on the  right hand side of (\ref{norminlimit}) can be viewed as the element of $ \mathcal{C}(G/ K ,  N_{\QQ_{p}})$  that sends $  g_{\alpha}K $ to a normalized linear combination of elements of $ N_{\QQ_{p}}$.  We  call  $ \xi $ a  \emph{test data} for our norm relation problem  (\ref{weaknormrel}). The form of   the   input    on   the    right hand side of  (\ref{norminlimit}) imposes an \emph{integrality} constraint on  our  test  data.  The definition below captures this condition.     
\begin{definition}   \label{integraldatadef}         An element   $  \xi  \in    \mathcal{C}_{\QQ_{p}}(G/ K ,  N_{\QQ_{p}}) $ is said to be  \emph{$ \ZZ_{p}$-integral  at level $L$} if for each $ g \in G $,  there exists a  finite collection $ \left \{ V_{i} \, | \,  V_{i} \subset g L g^{-1} \right \}$ of levels of $H $ and    classes    $ x_{V_{i}} \in N(V_{i}) $  for  each  $   i   $     such that $$ \xi( gK ) =  \sum   \nolimits    _{i \in I }    [U : V_{i}   ]  \cdot x_{V_{i}, \QQ_{p}}$$ 
where $ [U : V_{i}] = \mu_{H}(U)/ \mu_{H}(V_{i}) $.    
\end{definition}
This definition guides the choices of test data in $ \mathcal{C}(G/K,  N_{\QQ_{p}})$ that we hope to feed in the limit map $ \hat{\iota}_{*} $ in order to solve  (\ref{weaknormrel}) by an element of the form (\ref{testpushclass}).   % It is an integrality   
It however says nothing about the equality (\ref{norminlimit})  itself.   To remedy 
this,  note that $ x_{U, \QQ_{p}} \otimes \mathfrak{H} $ is also an element of $ \mathcal{C}(G/K,  N_{\QQ_{p}})$. Since $ \hat{\iota}_{*}$ is $ H $-equivariant  and the taget   of $ \hat{\iota}_{*}$ has  trivial $H$-action, one way of proving (\ref{weaknormrel}) is to require that the inputs of $ \hat{\iota}_{*}$ in (\ref{norminlimit}) have equal $ H $-coinvariants, where the $ H $ action on $ \mathcal{C}(G/K , N_{\QQ_{p}}) $ is as in   \S     \ref{comppushsec}.  This motivates the  following.    

\begin{definition} A \emph{zeta element} for $ (x_{U},  \mathfrak{H}, L) $ is an element of $ \mathcal{C}_{\QQ_{p}} (G/ K   ,    N_{\QQ_{p}}   )  $ that is $ \ZZ_{p}$-integral  at  level $ L $  and  lies in the $ H $-coinvariant class of $ x_{U,   \QQ_{p}    } \otimes \mathfrak{H}  $.    
\end{definition}

So constructing a zeta element amounts to proving (\ref{weaknormrel}).  Note that the existence of such an element  solves the norm relation problem (\ref{weaknormrel}) with respect to  \emph{any} representation $M$ and \emph{any} pushforward from  $N $ to  $M$,  since  our definition  is  independent of these  two   objects.    A key result of \cite[\S   3]{CZE} is a necessary and sufficient criteria for the existence  of such an element in terms of certain operators on $ H $ derived  directly from $ \mathfrak{H}$.       For  $ g \in G $,  define    the \emph{$ g $-twisted $H$-restriction} of $ \mathfrak{H}$ to be the function   
\begin{equation}   \label{twistedrestriction}       \mathfrak{h}_{g} : H \to \ZZ_{p},  \quad \quad h \mapsto     \mathfrak{H}  (   hg    )    
\end{equation}   
Let $ A   $  denote  the  finite set  $   H \backslash  H \cdot \mathrm{Supp} (\mathfrak{H})/K $.     For each $ \alpha \in A $, pick any representative $ g_{\alpha} \in G $ for $ \alpha $, and denote (abusing notation) $ H_{\alpha} = H \cap g_{\alpha}  K g_{\alpha}^{-1} $, $ V_{\alpha} =  H_{\alpha} \cap g_{\alpha} L g_{\alpha}^{-1} $ and $ \mathfrak{h}_{\alpha} $ the $ g_{\alpha}$-twisted $ H$-restriction of $ \mathfrak{H} $.  Note that each $ \mathfrak{h}_{\alpha} $ is an element of $ \mathcal{C}_{\ZZ_{p}}(  U \backslash  G /  H_{\alpha}) $ and can be viewed as a  covariant  correspondence   $$  \mathfrak{h}_{\alpha, *} :  N(U) \to N(H_{\alpha})  .   $$    In particular,   we    can define its degree.        
\begin{theorem}    \label{zetateo}  Suppose $ N $ is the trivial representation.  Then  a zeta element for $ (x_{U},  \mathfrak{H}, L) $ exists if and only if $ \deg(\mathfrak{h}_{\alpha, *}) \in [H_{\alpha} : V_{\alpha} ]  \cdot \ZZ_{p}$ for all $ \alpha \in  A $.     
\end{theorem}
It is straightforward to see that this criteria is  independent of the choice of representatives $ g_{\alpha} $. In fact, it   also implies the stronger relations (\ref{strongnormrel}), i.e., the desired norm relations hold without modding out by  $ \ZZ_{p}$-torsion.  A more general version that applies to arbitrary  representations $N $ can be found in \cite[\S3]{CZE}.    We will give two examples of  this  criteria      in \S  \ref{examplessec}.     %  The case of trivial representation is applicable to pushforwards of cycle classes of the source Shimura variety.     
\begin{remark} This method outlined has been further strengthened in \cite{compactinduction} as follows. Given $ x_{U} $ and $ L $, one can ask for the set of all $ \mathfrak{H} $ such that the criteria of Theorem \ref{zetateo}  is satisfied. Under the assumption that $G$ is the product of a group $G_{0}$ with a torus (and that both $K$, $L$ also have this form),  this set can be shown to be an ideal of $ \mathcal{H}_{\ZZ_{p}}(K \backslash G / K ) $, and one can study this ideal  via  its   Satake transform.  This is advantageous, since Hecke polynomials are defined  as inverse Satake transforms  of ``Satake-polynomials".   So  one   can establish the  norm relation problem (\ref{strongnormrel}) by showing that the Satake polynomial lies in this ideal.   % by finding  a set of generators.    % To  establish this  inclusion,  one can   simpler operators that generate this ideal, which is useful when the Hecke polynomial required to produce a specific  $L$-factor  is  complicated  to  write  down.           
\end{remark}

\subsection{The method   of   local    zeta integrals}   \label{LSZsec}      The method of \cite{LSZ} intends to prove the a weaker version of (\ref{weaknormrel}).   More   precisely,     it  aims    to prove analogues of  (\ref{correctESHeeg})     directly at the level of Galois cohomology.     Let us explain this strategy  in our abstract  formalism. For a concrete application of this strategy, we refer the reader to \cite[\S 8]{Anticyclo}.

Suppose, as is the case with cycles, that $ N $ is the trivial representation $ \ZZ_{p}$  and $ x_{U} = 1 \in \ZZ_{p} $. Let $ R $ denote a $ \QQ_{p}$-algebra.  For this subsection only,  we  will   assume that the representation $ M_{R}^{\vee}    $ is \emph{unramified}, i.e.,  the  $ K $ invariants $ M_{R}^{\vee}(K) $ form a one-dimensional module over $ R $.\footnote{The one-dimensionality of $ M_{R}^{\vee}(K) $ is a harmless assumption, since one eventually projects   the Galois cohomology classes  (\ref{Galoiscohoetale})    to automorphic pieces of geometric \'{e}tale cohomology, and one may project them to a one-dimensional Galois cohomology piece on which the Galois group acts by a  finite   order  character. In our abstract scenario, we are assuming that $M _{R}^{\vee}(K) $ is equal to this Galois-automorphic piece.  See \cite[\S 8]{Anticyclo}  for a specific  scenario. }  Fix    a non-zero   element   $ \phi^{\circ}    \in  M_{R}^{\vee}(K) $.     We wish to construct a class $ y_{L} \in M(L) $ such that (\ref{weaknormrel})  holds   after pairing with $  \phi^{\circ}   $. That is, we wish to verify that   
\begin{equation}   \label{pairednormrelation}    \langle  \phi^{\circ}   , \mathfrak{H}_{*}(y_{K, R})  \rangle  \rangle  =  \langle   \phi^{\circ}  ,   \pr_{L, K , * }  (y_{L, R})  \rangle    
\end{equation}    
As before, we would like the class $ y_{L} $ to be of the   form   (\ref{testpushclass}). For each such $ y_{L} $, there is a corresponding   integral  test data  $ \xi \in \mathcal{C}_{\QQ_{p}}(G/K, \QQ_{p}) $.   Lemma \ref{norminlimitlemma}  states  that the  equality     (\ref{pairednormrelation}) is equivalent to 
\begin{equation}   \label{pairednormrelationlimit}    \langle    \phi^{\circ}   , \hat{\iota}_{*} ( \mathfrak{H}  ) \rangle  =   \langle   \phi^{\circ}     ,  \hat{\iota}_{*}(\xi)  \rangle  .      
\end{equation}     
By Proposition \ref{Frobreci}, we see that  (\ref{pairednormrelationlimit})    is equivalent to   \begin{equation}  \label{LSZnormrelation}   \mathfrak{i} ( \mathfrak{H}   \cdot  \phi^{\circ}    ) =   \mathfrak{i}    (\xi \cdot   \phi^{\circ}   )    
\end{equation}    
where $ \mathfrak{i}  \in  \mathrm{Hom}_{H}( M_{R}^{\vee}, R) $ is the unique  element corresponding to the completed map $ \hat{\iota}_{*} : \mathcal{H}_{\QQ_{p}}(G)  \to  M_{\QQ_{p}}$.    Since $ M_{\QQ_{p}}^{\vee}(K) $ is one-dimensional, $ \mathfrak{H}  \cdot   \phi^{\circ}    = L(\mathfrak{H}) \cdot   \phi^{\circ}    $  where $ L(\mathfrak{H}) \in   R^{\times} $ is a scalar.\footnote{This constant will be the inverse of an   appropriate    $L$-factor in applications.}  One now imposes the following crucial assumption.     % and $ \xi $ is the integral test data corresponding to $ y_{L} $ in (\ref{testpushclass}).
%\begin{remark} The one-dimensionality of $ M_{\QQ_{p}}^{\vee}(K) $ is a harmless assumption, since one eventually projects   the Galois cohomology classes \ref{Galoiscohoetale} to automorphic pieces of geometric \'{e}tale cohomology twisted by a Dirichlet character, and we are assuming that $M _{\QQ_{p}}(K) $ is equal to this Galois-automorphic piece.  See \cite[\S 8]{Anticyclo}  for a specific  scenario.   
%\end{remark}    
%We  must also  introduce the following  crucial assumption holds.
\begin{assumption}   \label{keyassumption}    The space $ \Hom_{H}(M_{\QQ_{p}}^{\vee}, \QQ_{p}) $ is one-dimensional.  
\end{assumption}
Then   $ \mathrm{Hom}_{H}(M_{R}^{\vee}, R) $ is one-dimensional for any $ R $. The advantage of this assumption is that the relation (\ref{LSZnormrelation}) may be verified with respect to \emph{any} non-zero element of $ \mathrm{Hom}_{H}(M_{R}^{\vee} , R) $ and  any $ \QQ_{p}$-algebra $ R $. In particular, we may use $ R = \bar{\QQ}_{p} \simeq  \CC $.     The  strategy is then to construct a specific basis $ \mathfrak{z}$ of $ \mathrm{Hom}_{H}( M_{\CC}^{\vee}, \CC) $ using   holomorphy  factors  arising from    \emph{local zeta integrals} and verify that   
\begin{equation}   \label{LSZnormrelagain}  L(\mathfrak{H})    \cdot    \mathfrak{z}(    \phi^{\circ}    )     =  \mathfrak{z}( \xi \cdot   \phi^{\circ}    )   
\end{equation}    
for some choice of integral test data $ \xi $.   Note that this condition is  independent of the choice of $   \phi^{\circ}    \in  M_{\CC}^{\vee}(K) $.  For applications to Euler systems, we would need  the  test  data   $ \xi $ to be  independent of $ M_{R}^{\vee} $ (or at least  independent of  $ M_{R}^{\vee}$ twisted by finite order characters of $G$).   We   will give one  example   of this method  in  \S  \ref{examplesecLSZ}.    
\begin{remark}      The method is also applicable to more general representations   $N$ satisfying the obvious analog of Assumption \ref{keyassumption}, but   is    slightly more involved to state.  The  main  difficulty of this method lies in  identifying the data $ \xi $ and  in controlling denominators of the coefficients of $ \xi $ that are required  by   the    integrality condition  in Definition \ref{integraldatadef}.   
\end{remark}   
\section{Examples}
\label{examplessec}  In this section, we illustrate the two methods for proving horizontal norm
relations described in \S\ref{integraltestdatasec} in the setting of modular
curves. Throughout, we use only the notation from
\S\ref{integraltestdatasec} and  introduce   any further notation as needed.
\subsection{Example 1.}   \label{examplezetasec}      In this subsection, we study the local norm relation problem that arises through the setup of   \S \ref{cohoform} at a prime $ \ell   \neq  p    $ that is split in the imaginary quadratic extension.  Throughout this subsection, let   $$ H = \QQ_{\ell}^{\times}   \times   \QQ_{\ell}^{\times}   , \quad  \quad     G =  \GL_{2}(\QQ_{\ell}) \times \QQ_{\ell}^{\times} . $$ 
and   define the embedding \begin{align*}   \iota : H  & \longrightarrow    G   \\  
 (h_{1}, h_{2})   %  \mapsto k^{-1}      &     
 & \longmapsto    ( \mathrm{diag}(h_{1}, h_{2})   \cdot      , h_{1}/h_{2} )    
\end{align*}   
via  which  we  consider  $  H  $  as  a   subgroup  of   $   G   $.       
We let  $$   U =  \ZZ_{\ell}^{\times} \times \ZZ_{\ell}^{\times},  \quad   K  = \GL_{2}(\ZZ_{\ell}) \times \ZZ_{\ell}^{\times} ,   \quad       L = \GL_{2}(\ZZ_{\ell}) \times (1 + \ell \ZZ_{\ell})   $$      
and $$ \mathfrak{H} =   \ell  \,  \ch(K)  -  \ch(K \sigma  ^  {  - 1  }    K)  +  \ch(K \tau ^ { - 1  }  K ) $$   
where $ \sigma = (\mathrm{diag}(\ell, 1), \ell)  $, $ \tau = (\mathrm{diag}(\ell, \ell), \ell^{2}   )   $.   Note  that both $ \sigma $ and $ \tau  $   lie  in  $   H    $.    
\begin{remark}   \label{notdiagembedding}   As noted in \S \ref{Frobmatrixsec}, the local embedding  arising from the  Shimura  data is not diagonal.  Suppose $ H' =  k H   k ^{-1} \subset G $ is the conjugate of $ H $ some $ k \in K $.  Let $ \mathfrak{h}_{g} $ (resp., $\mathfrak{h}_{g}'$) denote is the $g$-twisted $H$-restriction (resp., $ H'$-restriction) of $\mathfrak{H}$. Then $$ \mathfrak{h}'_{kg}(khk^{-1}) = \mathfrak{H}(khg) =   \mathfrak{H}(hg)  =     \mathfrak{h}_{g}(h) $$
for all $ h \in H $, $ g \in G $. So if $ \mathfrak{h}_{g} = \sum_{i} c_{i} \,  \ch(U h_{i} H_{g}) $, then $ \mathfrak{h}_{kg}' = \sum_{i} c_{i} \,  \ch(kUh_{i}H_{g}k^{-1})$ and it   easily follows that $ \deg(\mathfrak{h}_{g,*}') = \deg(\mathfrak{h}_{g,*}) $.  It  therefore  suffices to work with  the  diagonal  embedding  for  the  purposes  of  verifying   Theorem  \ref{zetateo}.         
%As $g$ runs over    representatives  of $ H\backslash H \cdot \mathrm{Supp}(\mathfrak{H})/K $, $ kg $ runs over representatives of $ H' \backslash H' \cdot   \mathrm{Supp}(H)/K $.    Thus it  suffices  to work with the diagonal embedding.     
\end{remark} 

\begin{remark} We are using $ \ell $ in the second component of $ \sigma $, since the action of $ \mathfrak{H}$ on cohomology is covariant and the (right) action of $ \ell^{-1} $ in the covariant convention corresponds to the (left) action of   geometric   Frobenius at the place corresponding to the first component of $H  $. Note   also     that in anticyclotomic extensions, the geometric Frobenius at one of  the places above a split prime $ \ell $ equals the arithmetic Frobenius at the other  place,   so this choice does not seem   too important in the proposed framework of \cite{JNS}. In fact, our criteria also gives an affirmative answer when $ \mathfrak{H}$ is replaced by  
\begin{equation}  \label{checkH} \ell  \,   \ch(K)   -  \ch ( K  \check {\sigma}^{-1} K)    +  \ch ( K   \check{\tau}  ^{-1} K ) 
\end{equation}   
where  $ \check{\sigma}  =   (\mathrm{diag}(\ell,1), \ell^{-1}) $, $    \check{\tau}    =   ( \mathrm{diag}(\ell ,  \ell ) ,  \ell^{-2}) $.          
\end{remark}  
If  $ \xi $ denotes the function $ \ch(U \eta \gamma K )  : G \to \ZZ $ for some $ \eta \in H $, $  \gamma  \in G $, then the twisted restriction $ \xi_{g} :    H \to \ZZ $, $ h \mapsto \xi(hg) $ is zero unless $ H g K   =    H \gamma K  $, and $ \xi_{\gamma} = \ch ( U \eta H_{\gamma} ) $.  So to compute the twisted restrictions of $ \mathfrak{H }$,   we first write $ \mathfrak{H} $ as an element of $ \mathcal{C}_{\ZZ_{p}}(U \backslash G / K ) $.   Let  us  denote    $$ \sigma_{1} =   \left (     \left  (   \begin{smallmatrix}  1  /  \ell  & \\   &   1 \end{smallmatrix}  \right )    ,   1/\ell  \right ) ,   \quad      \sigma_{2} =   \left   (   \left  (   \begin{smallmatrix} 1/\ell \\ 1/ \ell  &   1     \end{smallmatrix}    \right  )   ,    1/\ell  \right  )       \quad   \sigma_{3} =   \left (  \left (  \begin{smallmatrix} 1 \\ &  1 /  \ell       \end{smallmatrix}   \right  )    ,   1  /   \ell   \right   )  ,   \quad      \sigma_{4} =     \left (    \left   (   \begin{smallmatrix}  1/\ell \\  &   1/\ell       \end{smallmatrix}   \right   )   ,   1 /  \ell^{2}    \right   )     .$$ 
\begin{lemma}   \label{Uorbitsplit}     A set of representatives for $ U \backslash  K \sigma^{-1} K  / K $ is $ \left \{  \sigma_{1}, \sigma_{2}, \sigma_{3}  \right \} $.    
\end{lemma}   
\begin{proof} This is easily established  by studying the $ U$-orbits on the coset  space $ K \sigma^{-1} K/K $, which was   % which was 
described in  Lemma  \ref{Telldecompose}. It is easy to see that $ U \sigma_{i}K $ are pairwise disjoint for $ i =1 , 2,  3 $.    
\end{proof}    
\begin{corollary}  A set of representatives for $ H \backslash H \cdot  \mathrm{Supp}(\mathfrak{H})/ K  $    is $  \left \{ 1_{G},   \left    (    \left  (  \begin{smallmatrix} 1 \\ 1/\ell  &  1 \end{smallmatrix}   \right )      , 1   \right     ) ,  (1,  1/ \ell^{2}) \right \}   $.   
\end{corollary}   
\begin{proof} Since $ H \sigma_{1} K = H K $, $ H \sigma_{2} K =     \left    (    \left  (  \begin{smallmatrix} 1 \\ 1/\ell  &  1 \end{smallmatrix}   \right )      , 1   \right     ) $, and $ H\sigma_{3} K = H \sigma_{4} K  =  H (1, 1/\ell^{2})    K $,        %  and  $ H \sigma  _ { i } K $,   
Lemma  \ref{Uorbitsplit} implies that the  representatives are contained in the claimed set.  Now $H\gamma_{1} K = H \gamma_{2} K $ for $ \gamma_{1}, \gamma_{2} \in G $ if and only if there is an $ h \in H $ such that $ \gamma_{1}^{-1} h \gamma_{2} \in K$. Using this, one easily sees that the elements represent distinct cosets in $ H \backslash G /  K  $.    
\end{proof}
Using Lemma  \ref{Uorbitsplit}, we see that    
\begin{equation}   \label{explodedhecke}    \mathfrak{H} =   \ell \,  \ch(UK) -    \Big (  \ch(U \sigma_{1} K ) +  \ch(U \sigma_{2} K )  +  \ch( U \sigma_{3} K )   \Big     )    +    \ch  (  U \sigma_{4} K )    
\end{equation} 
Set  $$ g_{0} = 1_{G}, \quad g_{1} =    \left    (    \left  (  \begin{smallmatrix} 1 \\ 1/\ell  &  1 \end{smallmatrix}   \right )      , 1   \right     )  ,   \quad   g_{2}  =       (1,  1/ \ell^{2}    )    . $$ 
and let  $ H_{g_{i}} = H \cap g_{i} K g_{i}^{-1}  $ and $ V_{g_{i}} = H \cap g_{i} L g_{i}^{-1} $.   Then $ H_{g_{0}} = H_{g_{2}} = U $. From the expression  (\ref{explodedhecke}), we see that 
\begin{align*}   \mathfrak{h}_{g_{0}}    &  =    \ell  \,   \ch(U)   -       \ch(U ( \ell^{-1},   1) U )    \\
\mathfrak{h}_{g_{1}}  &  =    \ch(U \,  (\ell^{-1},  1) \,  H_{g_{1}})   \\   
\mathfrak{h}_{g_{2}}  &  =     \ch(U   (  \ell^{-1} ,   1)   U )  -            \ch(U (\ell^{-1}, \ell^{-1}   )  U ) .   
\end{align*}
Since $ H $ is abelian and $ H_{g_{1}} \subset U $,   it is easy  to  see  that   $$   \deg(\mathfrak{h}_{g_{0}, * }) = (\ell -1 ) \quad  \quad   \deg(\mathfrak{h}_{g_{1}, * }) =   1    ,   \quad  \quad   \deg(\mathfrak{h}_{g_{2}, *}) =  0  .   $$ 
Since $ [H_{g_{i}} : V_{g_{i}}] $ divides $ \ell - 1 $, the condition on the degree of $ \mathfrak{h}_{g_{0}, *}$ and $ \mathfrak{h}_{g_{2}, *} $  required in  Theorem  \ref{zetateo} are immediate.  As for  $ \mathfrak{h}_{g_{1}} $,   we simply verify that $ H_{g_{1}} =  V_{g_{1}} $, i.e., if $ (h_{1}, h_{2}) \in H_{g_{1}}  \subset U $, then $ h_{1} \equiv h_{2} $ modulo $ \ell -  1   $.

\begin{remark} A zeta element in this scenario is $   (\ell -1  ) \left  ( \ch(K) - \ch( g_{1} K)   \right   )    $, which is  essentially  the element $\zeta_{\ell}$  (\ref{zetamodularvector})  at  a split prime scaled by $ (\ell - 1) $.   If in the discussion above, we   replace the term $  - \ch(K  \sigma ^{-1}   K ) $  in  $ \mathfrak{H}$  with, say, $   - \ell^{2}   \,     \ch( K  \sigma^{-1}  K    ) $,   a zeta element still exists but it now spanned by $ \ch(g_{i}K) $     for $i = 0 ,1, 2  $  with  non-zero  coefficients  for  each  $ i   $.      The vanishing of the third ``twist" $ \ch(g_{2}K) $ is a consequence of the vanishing of $   \deg( \mathfrak{h}_{g_{2}, *}   )   $. 
\end{remark}  
\begin{remark}   \label{alternativezeta}   We invite the reader to verify  that the criteria of Theorem  \ref{zetateo} also holds if  $ \mathfrak{H}$  is taken to be   as  in     (\ref{checkH}). In this case, $$ (\ell-1)  \left (\ch(K) - \ch({}^{t}\negsmall g_{1}  K)  \right  )   $$ 
is a zeta element.  
\end{remark} 

\subsection{Example 1 bis}   \label{examplesecLSZ} In this subsection, we illustrate the method of \cite{LSZ} in the same setting as the previous subsection. We continue to use the notation $H$, $G$, $U$, $K$, and $L$ for the groups introduced above, and we view $H$ as a subgroup of $G$ via the embedding $\iota$.    
We  will  also consider the groups $$  G_{0} = \GL_{2}(\QQ_{\ell}) ,  \quad  K_{0} = \GL_{2}(\ZZ_{\ell}) $$  and  the  embedding $$  \iota_{0}  : H \to G_{0}  ,  \quad  \quad   (h_{1}, h_{2}) \to \mathrm{diag}(h_{1}, h_{2}) .  $$ We will write $ H_{0}$ for the image of $ \iota_{0}$. If $ \chi  :  \QQ_{\ell}^{\times}  \to  \CC  ^ { \times}    $ is a  character, we will write   \[
\begin{aligned}
\chi_{G} : G    &  \to   \CC^{\times}  & & \quad   \quad    &  
\chi_{H_{0}}  :  H_{0}  &   \to   \CC^{\times} \\ 
(g,  a ) &  \mapsto  \chi(a)  ,  &  &  \quad  \quad       &    \iota_{0}    (h_{1},  h_{2})   &    \mapsto  \chi( h_{1}^{-1}  h_{2} ) 
\end{aligned}
\] Abusing  notation, we  denote the the space underlying the one-dimensional representations given by $ \chi_{G} $, $ \chi_{H_{0}} $ by   the   same   symbols.  We fix a Haar measure on $ G_{0} $ that gives $ K_{0}$ measure one and  a   %measure a    
multiplicative measure on $ \QQ_{\ell}^{\times }$ that gives $ \ZZ_{\ell}^{\times} $ measure one.  We  assume that the Haar measure $ \mu_{G} $ in \S  \ref{comppushsec} is chosen so that $ \mu_{G} $ equals to the product of these two measures.  In  particular,  $ \mu_{G}(K)  =  1  $.                  

Recall that we require $ M_{\CC}^{\vee}(K) $ to be a $ 1$-dimensional $ \CC $-vector space. Fix an  irreducible admissible  unramified principal series  representation $ \Pi    $ of $  G_{0}  =     \GL_{2}(\QQ_{\ell})$.  We will assume that $  M^{\vee}_{\CC} $  belongs to the family of representations  $$   \mathrm{Tw}(\Pi )    =  \left \{  \Pi    \otimes \chi_{G} \, | \, \chi : \QQ_{\ell}^{\times} \to \CC^{\times} \text{ is a finite order unramified character}  \right \}   $$
and our goal is to construct  an integral test data $ \xi $ that satisfies   (\ref{LSZnormrelagain}) for \emph{every}    $ M_{\CC}^{\vee} \in  \mathrm{Tw}(\Pi)$   with  respect  to   $$ \mathfrak{H} = \ell \,  \ch(K) -   \ch(K  \check{\sigma}  ^ { - 1 }  K )  +  \ch(  K  \check{\tau}   ^ { - 1  }     K   ) $$
where $ \check{\sigma}  =   (\mathrm{diag}(\ell,1), \ell^{-1}) $,    $    \check{\tau}    =   ( \mathrm{diag}(\ell ,  \ell ) ,  \ell^{-2}) $  are as  in  (\ref{checkH}).   %\footnote{We already know one such choice that works for every $ M_{\CC}^{\vee}$ by our discussion in \S  \ref{examplezetasec}.} %and a  finite  order  unramified  character $ \chi  :     \QQ_{\ell}^{\times} \to \CC  $  (which we view as a character of $ G $ via the projection $ \GL_{2}(\QQ_{\ell}) \times  \QQ_{\ell}^{\times}  \to  \QQ_{\ell}^{\times} $ onto the second component)   such that $$ \pi = \sigma \otimes \chi  . $$
For any  $  \Pi \otimes \chi_{G}   \in  \mathrm{Tw}(\Pi)$ and $  \phi^{\circ}  \in  (\Pi \otimes  \chi_{G})^{K}   $,   $$ \mathfrak{H}   \cdot  \phi^{\circ}     = \ell \cdot  L(\tfrac{1}{2}   ,    \Pi^{\vee}   \otimes  \chi_{G} )   ^ {  -   1    }      \cdot   \phi^{\circ}      $$
where $   L(s,  \Pi^{\vee} \otimes\chi_{G})   $  denotes the standard   $L$-factor  of  $ \Pi^{\vee} \otimes \chi_{G}  $    in the complex variable $ s $,  i.e.,   if $ \alpha $, $ \beta $ denote the  Satake  parameters for $ \Pi $, then  $  L(s,  \Pi^{\vee} \otimes  \chi_{G}  ) ^{-1}   =  ( 1 -  \alpha  ^ { - 1 }   \chi_{G}(\ell)   \ell^{-s} ) ( 1 -   \beta ^ { - 1  } \chi_{G}(\ell)     \ell^{-s})  $.   Thus    $$ L(\mathfrak{H}) = \ell \cdot L(\tfrac{1}{2}, \Pi^{\vee} \otimes \chi_{G})^{-1} $$ is the constant needed in  our  norm  relation (\ref{LSZnormrelagain}).      
%where $ \alpha $, $ \beta $  denote  the    Satake parameters of $ \pi^{\vee} $.   

Recall   also  that the pushforward $ \hat{\iota}_{*} $  gives as element of $ \mathrm{Hom}_{H}(M_{\CC}^{\vee},  \CC)$.  % The restriction of $ \chi $ to $  \iota(H) $ allows us to define a character $ \chi_{0} : H_{0} \to \CC^{\times} $, which sends $  \mathrm{diag}(h_{1},  h_{2})  \in \iota_{0}(H) $ to $  \chi(h_{1} / h_{2}) $. Then   
If $  M_{\CC}^{\vee} =  \Pi \otimes \chi_{G} \in   \mathrm{Tw}(\Pi)$ for some $ \chi $,   then    $$ \mathrm{Hom}_{H}(M_{\CC}^{\vee}, \CC) =   \mathrm{Hom}_{H}(\Pi \otimes \chi_{G} , \CC)   \simeq    \mathrm{Hom}_{H_{0}}(\Pi,  \chi_{H_{0}}) $$  
By  \cite[Theorem B]{ChenSun} and the discussion on ``good"  characters following it,  we  know  that    \begin{equation}   \label{Chensunbound}     \dim_{\CC}  \mathrm{Hom}_{H_{0}}(\Pi,  \chi_{H_{0}}     ) \leq  1 \end{equation}   
If this  dimension is  zero,    the relation (\ref{pairednormrelationlimit}) hold trivially for the corresponding choice of $ M_{\CC}^{\vee} $.    Thus the case of interest is when $  \mathrm{Hom}_{H_{0}}(\Pi , \chi_{H_{0}}) $ is exactly  one-dimensional.   %for some choice of $ \chi $.    
The non-vanishing of  % Note that the non-vanishing of this space for any singe $ \chi $ forces $ \Pi $ to have trivial central character, and %The non-vanishing $ \mathrm{Hom}_{H_{0}}(\Pi, \CC) $ is closely r % If this is the case  Note that this forces $ \Pi $ to have trivial central character. % For trivial $ \chi $,  non-vanishing of
these spaces  
is   closely related to the existence of a certain \emph{model} for $ \Pi $, as we now explain.       
%\begin{remark}   The one-dimensionality $ \mathrm{Hom}_{H}(\pi, \CC)$ is also closely related  to the non-vanishing of the bottom class of our Euler system in Galois cohomology and is a natural condition to impose.  
%\end{remark}   

%Before we proceed, let us  explain a closely related notion of a \emph{Shalika model}.   
We write $v_\ell      $ for the normalized
$\ell$-adic valuation  on    $   \QQ_{\ell}   $       satisfying $v_\ell(\ell)=1$, and $|\cdot|$ for the
associated absolute value on $\mathbb{Q}_\ell^\times$, given by
$|x|=\ell^{-v_\ell(x)}$.   Let   $$  \psi_{\mathrm{std}} : \QQ_{\ell} \to \CC^{\times}   $$    denote the standard character that is trivial on $ \ZZ_{\ell}$ and satisfies $ \psi_{\mathrm{std}}(1/\ell^{n}) =  \mathrm{exp}(2 \pi i/\ell^{n}) $ for  $   n \geq  1   $. Then any other character $ \psi : \QQ_{\ell} \to \CC^{\times} $ is of the form $ x \mapsto \psi_{\mathrm{std}}(ax )$ for some $ a \in \QQ_{\ell} $ and   has  conductor $ \ell^{-\delta} $ where $ \delta = v_{\ell}(a) $. We fix a non-trivial   additive character $ \psi $, and also fix  a    multiplicative character $ \eta : \QQ_{\ell}^{\times}  \to  \CC  $.           % $ e^{2\pi i/p^{n}} $ on $ 1/p^{n}$ for $ n \geq 1 $.    
Let $ S_{0}  \subset G_{0} $ denote the subgroup of all elements of the form    $  \left (   \begin{smallmatrix}  x  &  x a  \\  &  x  \end{smallmatrix}   \right  ) $ where $ x \in \QQ_{\ell}^{\times} $, $ a  \in  \mathbb{Q}_{\ell} $. We can  define a character of $ S_{0}  $ via  
\begin{align*}    \Theta   : S_{0}   &  \to  \CC^{\times}   \\  \left ( 
\begin{smallmatrix}  x & x a  \\   &  x  \end{smallmatrix}   \right )     &   \mapsto      \eta(x)    \psi(a)   .   
\end{align*}  
Let $ \mathrm{Ind}_{S_{0}}^{G_{0}}(\Theta) $ denote the representation of $ G $ induced from $   \Theta    $. That is, elements of $ \mathrm{Ind}_{S_{0}}^{G_{0}}(\Theta) $ are locally constant functions $    W :  G_{0}  \to   \CC $ such that    $$  W (\gamma g) = \Theta    (\gamma)  \cdot     W    (g)    $$      for all $ \gamma \in S_{0} $, $ g \in G_{0} $ and the action of $ G_{0}    $ on   $   W    $  is via right translation on the domain of $   W    $.   
An $(\eta, \psi)$-\emph{Shalika model}   for $  \Pi    $  is  a non-zero (but not necessarily smooth)  linear  map  $$   \mathcal{S}   :  \Pi   \to  \mathrm{Ind}_{S_{0}}^{G_{0}} (\Theta)  . $$ 
Since $   \Pi    $ is  irreducible,  such a map is necessarily  an  embedding.  For any $ W  \in  \mathrm{Ind}_{S_{0}}^{G_{0}}(\Theta)   $   and  quasi-character $ \chi : \QQ_{\ell}^{\times}   \to  \CC^{\times} $,     we can  define  a  \emph{local    zeta  integral}           $$   \zeta    (s ;   W,  \chi    ) =  \int_{\QQ^{\times}   _  {  \ell   }   }    W   \left  [  \left  (   \begin{matrix} x \\ & 1   \end{matrix}   \right   )   \right  ]   \chi( x ) \,  |x|^{s - \frac{1}{2} }    \, d^{\times}x   $$
where $ d^{\times} x $ denotes the    multiplicative   measure on $  \QQ_{\ell}^{\times}  $ that gives $ \ZZ_{\ell}^{\times} $ measure one.         
The   following   result  is  taken   from \cite[\S 3.2]{Dimitrov}.          
\begin{theorem}   \label{FJShalikateo}      Suppose   that   $  \Pi   $ admits an $ (\eta, \psi) $-Shalika model $ \mathcal{S} $ as above.  Then  for  each $ W \in   \mathcal{S} (\Pi) $,  the integral $ \zeta(s; W, \chi) $  converges  absolutely for $ \mathrm{Re}(s) $  large  enough and there exists a holomorphic   function $ P (s; W, \chi  )   $  such   that   $$      \zeta(s; W, \chi)  =     P (s; W,  \chi  )   \cdot     L(s,   \Pi    \otimes \chi_{G} )   $$
Moreover,  there exists a spherical  vector $  W^{\circ}  \in   \mathcal{S}  (\Pi^{K_{0}})   $ satisfying $ W^{\circ}(1_{G_{0}}) = 1 $ such that $ P(s; W^{\circ}, \chi)   =     ( \ell^{s   -   \frac{1}{2}} \chi(\ell)) ^{\delta} $ for all $ s \in  \mathbb{C} $.    
\end{theorem}
%\begin{remark}
\begin{remark} 
When $ \eta $ is the trivial character and $ \Pi  $ is unitary, we recover the setup studied in \cite{FriedbergJacquet} for the group    $ \GL_{2} $.    
\end{remark}   
Let    us    now assume that   $ \mathrm{Hom}_{H}(M_{\CC}^{\vee}, \CC) $ does not vanish for some $ M_{\CC}^{\vee}   \in   \mathrm{Tw}(\Pi)$.   Then $ \Pi $ is forced to have  trivial  central  character and the notion of a $(1, \psi)$-Shalika model for $ \Pi $  coincides  with the notion of a $\psi$-Whittaker model. By   \cite[Theorem 4.4.3]{Bump}, any irreducible   admissible  infinite dimensional  representation  of  $ G_{0}  $     admits a $ \psi$-Whittaker model.  In particular,  $ \Pi $ admits a $ (1, \psi_{\mathrm{std}})$-Shalika model. In what follows, we   fix such a model $ \mathcal{S} $, so that the character $ \Theta $ is  defined using $ \psi_{\mathrm{std}} $ and trivial $ \eta  $.            % denot   
Let $ \chi $ be  an  arbitrary  finite order  unramified character  that  we  fix  for  the  rest of the  discussion.     Consider the map
\begin{align}   \label{zmap}       \mathfrak{z}_{\chi}    :  \Pi    \to \CC    ,  \quad    \phi  \mapsto  P(\tfrac{1}{2},   \mathcal{S}(\phi), \chi )    
\end{align} 
where $ P(s, W, \chi) $ denotes the holomorphy factor in Theorem \ref{FJShalikateo}. For each fixed $ a , b \in \CC $ and $ W_{1}, W_{2} \in  \mathrm{Ind}_{S_{0}}^{G_{0}}(\Theta)$,   it is easy to see that    $$ \zeta(s;    aW_{1} + bW_{2}, \chi)  = a \zeta(s; W_{1}, \chi ) + b   \zeta(s; W_{2} ,  \chi )   $$   
for all $ s \in \mathbb{C} $ with $ \mathrm{Re}(s) $ large enough. Thus the same property  holds for $ P(s; -, \chi) $. Since $ P(s; W, \chi) $ is holomorphic for  each $ W \in \mathcal{S}(\Pi)$, we see that the linearity of $ P(s;   -   ,     \chi) : \mathcal{S}(\Pi) \to \CC $   holds for all $ s \in  \CC $.     In particular,  $ \mathfrak{z}_{\chi} $ is     $ \CC$-linear.     
\begin{proposition} $ \mathfrak{z}_{\chi}  $ is a basis for $ \Hom_{H_{0}}(\Pi, \chi_{H_{0}}) $. In particular, $ \Hom_{H}(M_{\CC}^{\vee}, \CC) $ is one-dimensional for  all  $ M_{\CC}^{\vee} \in   \mathrm{Tw}(\Pi)   $.    
\end{proposition}    
\begin{proof}  For any $  h =  \iota _ {  0  }    (h_{1}, h_{2}) \in H_{0}$, $ W \in  \mathcal{S}(\Pi) $ and $ s $ sufficiently large enough,     
\begin{align*}   \zeta ( s ; h \cdot W , \chi )    &  =   \int_{\QQ^{\times}   _  {  \ell   }   }    W   \left  [  \left  (   \begin{matrix} h_{1} x \\ &  h_{2}   \end{matrix}   \right   )   \right  ]   \chi( x ) \,  |x|^{s - \frac{1}{2} }    \, d^{\times}x  \\
&  =       \int_{\QQ^{\times}   _  {  \ell   }   }    W   \left  [  \left  (   \begin{matrix} h_{1}  h_{2}^{-1}  x \\ &   1    \end{matrix}   \right   )   \right  ]   \chi( x ) \,  |x|^{s - \frac{1}{2} }    \, d^{\times}x    \\
& =         \int_{\QQ^{\times}   _  {  \ell   }   }    W   \left  [  \left  (   \begin{matrix} y  \\ &   1    \end{matrix}   \right   )   \right  ]   \chi( h_{1}^{-1}h_{2} y ) \,  |h_{1} ^{-1} h_{2} y |^{s - \frac{1}{2} }    \, d^{\times}y   \\
%& =  \chi(h_{1}^{-1}h_{2})  | h_{1}^{-1} h_{2}|^{s - \frac{1}{2}}  \cdot   \zeta(s; W, \chi ) \\
& =  \chi_{H_{0}}(h) |h_{1}^{-1}h_{2} | ^{s - \frac{1}{2}}   \cdot \zeta(s; W,  \chi ) . 
\end{align*}   
where in the third equality, we used   the  change of  variables $ y =  h_{1} h_{2}^{-1}  x   $.    
Dividing both sides by $ L(s, \Pi \otimes \chi) $, we see that   
\begin{align}    \label{holofactorequal}  P(s   ; h \cdot  W ,  \chi )  =       \chi_{H_{0}}(h) |h_{1}^{-1}h_{2} | ^{s - \frac{1}{2}}   \cdot    P(s; W, \chi) .    
\end{align}   
Since $ P $ is holomorphic,  we can plug $ s = \frac{1}{2}  $ in   (\ref{holofactorequal}) to obtain $$   P(\tfrac{1}{2} , h \cdot W , \chi )  = \chi_{H_{0}}(h)  \cdot   P(\tfrac{1}{2}; W, \chi )  .  $$
As  $ \delta = 0 $ for $ \psi_{\mathrm{std}}$,  the  second claim of  Theorem   \ref{FJShalikateo} implies that $ P(\frac{1}{2}, W^{\circ} , \chi)  =  1 $ for some $ W^{\circ} \in \mathcal{S}(\Pi^{K_{0}})$.  Consequently, $  \mathfrak{z}_{\chi}    $ is non-zero.   The claim  now follows by the  bound  $    \mathrm{dim}_{\CC} \,     \mathrm{Hom}_{H_{0}}( \Pi,  \chi_{H_{0}})  \leq 1  $ discussed  above.    
\end{proof} 
Let $ W^{\circ} \in  \mathcal{S}(\Pi) $ be the vector  given by Theorem  \ref{FJShalikateo}. Consider the function $$ f_{W^{\circ}} :  \QQ_{\ell}^{\times}  \to   \CC   , \quad \quad x  \mapsto  W   ^{\circ}   \left [ \left (\begin{smallmatrix} x \\ & 1 \end{smallmatrix}   \right )   \right ] . $$
\begin{lemma}   \label{fwlemma}       $ f_{W^{\circ}} $ is supported on $ \ZZ_{\ell} \setminus \left \{ 0  \right \} $ and equals identity on $ \ZZ_{\ell}^{\times} $.     
%Note that $ f_{W^{\circ}} $ is non-zero, since $ 
\end{lemma} 
\begin{proof}   For any  $ a \in \ZZ_{\ell} $,  the  element    $   \gamma_{a} := \left  (  \begin{smallmatrix}  1 & a \\ &  1   \end{smallmatrix}   \right )   \in K_{0}$  fixes $ W^{\circ} $. Therefore   
\begin{equation}   \label{fwequality}    f_{W^{\circ} }(x) =      W  ^ { \circ}    \left [ \left (\begin{smallmatrix} x \\ & 1 \end{smallmatrix}   \right )  \gamma_{a}  \right ]   =        W  ^   { \circ}   \left [ \left (\begin{smallmatrix} x  &  xa \\ &   1  \end{smallmatrix}   \right )  \right ] =  \psi_{\mathrm{std}}(ax)   \cdot    f_{W^{\circ} }(x).   
\end{equation}    
%Since $ W^{\circ} $  is  determined by its values on $ S_{0} \backslash G_{0} / K_{0} $, its support $ \mathrm{Supp}(W^{\circ}) $ is a union of double  cosets  in   $ S_{0}  \backslash  G_{0}    /  K_{0}   $.      Since any element in $ G_{0}/K_{0} $ has a representative in the standard upper triangular Borel subgroup (Iwasawa decomposition), it is easy to see that any  double  coset  in    $ S_{0} \backslash G_{0}  /  K $ can be represented by an element of the form $  \mathrm{diag}(x_{0}, 1) $  where $ x_{0} \in \QQ_{\ell}^{\times} $. So $ f_{W^{\circ}} $ is non-zero. 
If $ x_{0} \in \QQ_{\ell} \setminus \ZZ_{\ell} $ lies in the support of $ f_{W^{\circ}} $, then taking $ a = \ell^{-1} x_{0}^{-1} \in \ZZ_{\ell}  $  in (\ref{fwequality}) implies that  $$ f_{W^{\circ} }(x_{0}) = \psi_{\mathrm{std}}(1/\ell) \cdot f_{W^{\circ}}(x_{0}) .   $$    
Since $   \psi_{\mathrm{std}}(1/\ell) \neq   1    $, we see that $ f_{W^{\circ}}(x_{0}) = 0 $. Thus $ f_{W^{\circ}} $ is supported on $ \ZZ_{\ell} \setminus \left \{  0  \right  \}   $.     Since $ W^{\circ}(1_{G_{0}}) = 1 $ and $ W^{\circ}$ is $  K _ {   0  }     $-invariant, the second  claim   is  obvious.         
\end{proof}    
%\begin{lemma} If $ W^{\circ} \in \mathcal{S}(\Pi^{K_{0}}) $ is contained in $ \ZZ_{\ell} \setminus \left \{ 0  \right \} $. 
%\end{lemma}   
%\begin{proof}    
%\end{proof}    
Since $ \Pi $ has  trivial   central  character, it is induced by two characters that are inverses of each other  and  it   easily   follows  that   $ \Pi \simeq \Pi^{\vee} $.   Consequently,   
\begin{equation}   \label{selfdualLH}   L (  \mathfrak{  H } )   =    \ell \cdot   L( s,  \Pi^{\vee}  \otimes  \chi_{G} )  =   \ell  \cdot    L( s  ,   \Pi \otimes \chi_{G}   ) .
\end{equation}    
Define the  integral test data 
\begin{align}   \label{guesstestdata}     \xi =   (\ell-1) \left  (  \ch(K) - \ch( \check{g} K)  \right  )   \in  \mathcal{C}(G/K,  \QQ_{p} )   
\end{align}
where $   \check{g}   =  \left    (    \left  (  \begin{smallmatrix} 1  &  1/ \ell    \\   &  1 \end{smallmatrix}   \right )      , 1   \right     )  $. We will   write   $ \check{g}_{0} $ for the  first   component of $ \check{g} $ and let  $ \xi_{0} : G_{0}/K_{0} \to \CC $   denote  the map obtained by  restricting  $ \xi $ to $ G_{0}/K_{0}$, where $ G_{0} \hookrightarrow G $ is given by $ g \mapsto (g, 1)  $.   Then $ \xi_{0} $ is an element of the Hecke algebra $ \mathcal{H}_{\CC}(G_{0}) $ and therefore acts on $ \mathcal{S}(\Pi)$.  We wish to compute  $$ P(s, \xi_{0} \cdot W^{\circ} , \chi)  .  $$             % Let $  \varphi^{\circ}  \in  \Pi^{K} $ be the vector corresponding to $ W^{\circ}  $  specified by  Theorem  \ref{FJShalikateo}.  
To this end, note that Theorem \ref{FJShalikateo} and  Lemma \ref{fwlemma} imply  that        $$   L(s, \Pi \otimes \chi_{G})  = \zeta(s; W^{\circ} ,  \chi  )  =     \int_{ \ZZ_{\ell} \setminus   \left \{  0   \right \}   }   f_{W^{\circ}}(x)   \chi(x)  |x|^{s - \frac{1}{2}}     d ^{\times} x       %  =  L(s, \Pi \otimes \chi_{G})    
$$ 
for all $ s $ where the  zeta  integral is  absolutely    convergent.  Moreover, 
\begin{align*}  \zeta ( s ,  \ch(   \check{g}_{0}   K_{0})  \cdot   W^{\circ}    , \chi )   &   =   \zeta ( s ,  \check{g}_{0} \cdot  W^{\circ}   ,  \chi  )  \\ 
&  =    \int_{\QQ_{\ell}^{\times}}  W^{\circ}       \left  [  \left  (   \begin{matrix} x   & x /  \ell  \\ &   1   \end{matrix}   \right   )   \right  ]      \chi( x ) \,  |x|^{s - \frac{1}{2} }     d^{\times}x     \\
&   =      \int_{\QQ_{\ell}^{\times}} f_{W^{\circ}}(x) \,  %W^{\circ}       \left  [  \left  (   \begin{matrix} x   &    \\ &   1   \end{matrix}   \right   )   \right  ]
\psi_{\mathrm{std}}(x/\ell)  \,    \chi( x ) \,  |x|^{s - \frac{1}{2} }     d^{\times}x    %  \\
%\end{align*} 
%& 
%& = \int_{\ZZ_{\ell}\setminus \left \{0\right \} }  f_{W^{\circ}}(x) \psi_{\mathrm{std}}( x/\ell  )  \,    \chi( \ell y) \,  | x |^{s - \frac{1}{2} }     d^{\times}   
\end{align*}
for $ \mathrm{Re}(s) $ large enough. 
%&  =    \chi(\ell) \ell^{s  - \frac{1}{2}}  \int_{\QQ_{\ell}^{\times}}   f_{W^{\circ}}(\ell y )  \,  % W ^{\circ} \left  [  \left  (   \begin{matrix} \ell y   &    \\ &   1   \end{matrix}   \right   )   \right  ]
%\psi_{\mathrm{std}}( y )  \,    \chi( y) \,  | y|^{s - \frac{1}{2} }     d^{\times}   y   \\
%&  =
%
%Since %$ \psi_{\mathrm{std}} $ is trivial  on $ \ZZ_{\ell}$ and 
%$ f_{W^{\circ}} $ is supported on $ \ZZ_{\ell} \setminus \left \{ 0 \right \} $, the integral above is supported on  $ \ZZ_{\ell}^{\times}$ and $  f_{W^{\circ}}(\ZZ_{\ell}^{\times}) =  \chi(\ZZ_{\ell}^{\times}) = 1 $.  Therefore %  Since $ W^{\circ}(1_{S}) = 1 $ and $ W^{\circ}$  is invariant under $ K^{\circ} $, $ f_{W}^{\circ}(\ZZ_{\ell}^{\times}) = 1 $.   Therefore}
%\begin{align*}  
%&= %\zeta ( s ,  \ch(   \check{g}_{0}   K^{\circ})  \cdot   W^{\circ}    , \chi )     &   =  
Again by Lemma \ref{fwlemma},  the  integral above is supported on $ \ZZ_{\ell} \setminus \left \{ 0  \right \}  $.  We break the integral into the sum $ I_{1} $ and $ I_{2} $ where $ I_{1} $ is the integral over $ \ZZ_{\ell}^{\times} $ and $ I_{2}$ is over $ \ell \ZZ_{\ell}  \setminus  \left \{ 0 \right \}  $.  % Then   %    
Let $ \omega_{\ell}   \in \CC^{\times} $  denote a primitive $ \ell$-th root  of  unity.     Then 
\begin{align*}      
I_{1}  %  &    
=   \int_{\ZZ_{\ell}^{\times}  }  \psi_{\mathrm{std}}(x/\ell)   \,      d^{\times}x     % \\
& = \sum   \nolimits    _{ a \in (\ZZ / \ell \ZZ )^{\times}}  \psi_{\mathrm{std}}(a/\ell) \cdot \mathrm{vol}(1 + \ell\ZZ_{\ell}) \\
%&  =  \sum   \nolimits   _{a \in (\ZZ/\ell\ZZ)^{\times}   } \psi_{\mathrm{std}}(a/\ell) (\ell-1)^{-1} \\
& =  \sum\nolimits   _{i=1}^{\ell-1} \omega_{\ell}^{i} (\ell-1)^{-1}  
=  -  (\ell-1)^{-1}   . 
\end{align*}  
%where $ \omega_{\ell} $ denotes a primitive $ \ell$-root of unity.
On the other  hand,   
\begin{align*}   I_{2}   &   =   \int_{ \ell \ZZ_{\ell} \setminus \left \{ 0  \right \}    } f_{W^{\circ}}(x ) \chi(x) |x|^{s-\frac{1}{2}} d^{\times   }    x   \\
% &  =  \int_{ \ZZ_{\ell} \setminus   \left \{  0   \right \}   }   f_{W^{\circ}}(\ell y)   \chi(\ell y)  |\ell y|^{s - \frac{1}{2}}     d ^{\times} y     \\
 &  = \zeta ( s, W^{\circ} , \chi) - 
  \int_{ \ZZ_{\ell}^{\times} }   f_{W^{\circ}}(x)   \chi(   x)  | x|^{s - \frac{1}{2}}     d ^{\times} x   =  % \chi(\ell) \ell^{s - \frac{1}{2}} \cdot   
  L(s, \Pi   \otimes     \chi_{G})  -   1 
\end{align*} 
%by changing the variable $ x = \ell y $.  
Therefore,   
\begin{align*}   \zeta  (  s  ,   \xi   _ { 0 }    \cdot  W^{\circ}  , \chi  )   &  =  (\ell-1)  \cdot \left (   L( s,  \Pi   \otimes  \chi_{G} )  - I_{1}  - I_{2}   \right   )    %  \\
%  &  
=  \ell    .     
\end{align*}   
Dividing both sides  by $ L(s, \Pi \otimes \chi_{G})$ and remembering that $  P(s, W^{\circ} ,  \chi )  \equiv 1 $,  we  see that \begin{equation*}   \label{LSZrelationinPs}   P (s, \xi  _ { 0 } \cdot W^{\circ} , \chi ) =  \ell   \cdot       L(s, \Pi  \otimes  \chi_{G})^{-1}    \cdot   P(s, W^{\circ} ,  \chi )   
\end{equation*}    
for $ \mathrm{Re}(s) $ large enough.  Since $ P(s, W ,\chi) $ is holomorphic  for any $ W \in  \mathcal{S}(\Pi) $, the  equality  above holds for all $  s \in  \mathbb{C} $.   Plugging $ s = \frac{1}{2} $,   we   find   that           
\begin{equation}   \label{LSZrelationP1/2}   P (\tfrac{1}{2} ,  \xi  _ { 0 } \cdot W^{\circ} , \chi ) =  \ell      \cdot     L(\tfrac{1}{2} , \Pi  \otimes  \chi_{G})^{-1}   \cdot      P(\tfrac{1}{2} , W^{\circ} ,  \chi )     
\end{equation}    
%  Dividing both sides by $ L(s, \Pi \otimes \chi_{G}) $, we   get   that   
%$$ P(s,  \xi  _ {  0   }   \cdot  W^{\circ}  ,  \chi   ) =  (\ell-1)   +  L(s, \Pi \otimes \chi_{G})^{-1}  =    $$ 
%View $ \mathfrak{z}_{\chi} $ as an  element of $ \Hom_{H}(\Pi \otimes \chi_{G} , \CC) $. Let $    \varphi  ^{\circ} \in \Pi^{K_{0}} $ denote the element corresponding to $ W^{\circ} $, and let $  \phi ^{\circ} = \varphi ^{\circ} \otimes 1 $ denote the element in $   ( \Pi \otimes \chi_{G} )  ^ { K   }  $.  
%Thus $$ \zeta (s, \xi \cdot W^{\circ} , \chi ) = (\ell-1) L(s, \Pi  \otimes  \chi_{G})   $$ 
Let  $ \varphi^{\circ} \in \Pi^{K_{0}}$ denote the element such that $ \mathcal{S}(\varphi^{\circ}) = W^{\circ} $.  If  we view $ \mathfrak{z}_{\chi} $ as an element of $  \mathrm{Hom}_{H}(\Pi \otimes \chi_{G} , \CC) $ and let  $ \phi^{\circ} \in \Pi  \otimes  \chi_{G} $ denote the vector $  \varphi^{\circ} \otimes 1 $,    % where $ \varphi^{\circ} \in \Pi^{K_{0}}$ corresponds to $ W^{0}$,   
then   we can rewrite  (\ref{LSZrelationP1/2}) as 
\begin{align*}   \mathfrak{z}_{\chi}( \xi \cdot  \phi^{\circ}   )  %  & =  P(\tfrac{1}{2} , \mathcal{S}( \xi \cdot \phi^{\circ})  , \chi)  \\  
%& =     \ell \,  L(\tfrac{1}{2} , \Pi \otimes  \chi_{G} ) ^{-1}  \cdot   P( \tfrac{1}{2} , \mathcal{S}(  \phi^{\circ}) ,  \chi)  \\
%& =  \\ 
&  =  L(\mathfrak{H}) \cdot  \mathfrak{z}_{\chi}(\phi^{\circ}) . 
\end{align*} 
Thus the relation  (\ref{LSZnormrelagain}) is verified in our  setting.   

\begin{remark}
As noted in \S\ref{aimssection}, the failure of the multiplicity-one hypothesis (Assumption~\ref{keyassumption}) constitutes a major limitation of this method. We also observe that the central ingredient in this approach is the choice of the integral test data~\eqref{guesstestdata}. Unlike Theorem~\ref{zetateo}, however, this method provides no insight into how one might \emph{a priori} identify this data, and instead requires proceeding by pure   guesswork.  See also  Remark  \ref{hardtowrite}.      
% The failure of multiplicity one requirement (Assumption \ref{keyassumption}) is a major limit of this method.  % This guesswork can get considerably tricky for higher dimensional groups. 
\end{remark} 
\begin{remark}   \label{volumeremark} An additional difficulty for this method (not encountered in the example at hand)  arises from the fact that verifying the integrality condition of the test data requires computing volumes of   non-parahoric subgroups of the source groups. For the situation considered in \cite{Anticyclo}, this turns out to be manageable, essentially because the Hecke polynomial computing the standard $L$-function is “deceptively simple” to describe.\footnote{See the introduction to \cite{GrossSatake}.} For the $L$-factors arising in the settings studied in \cite{Siegel1} and \cite{EulerGU22}, the author is not aware of any method for computing all the required twisted volumes.\footnote{See however  \cite{CornutnormI}, \cite{CornutnormII}  where a method for computing such volumes  is   described  for a class of    orthogonal  groups.}  On the other hand, the twisted restrictions~\eqref{twistedrestriction} are much more amenable to computation, owing to a “geometric” recipe for decomposing parahoric double cosets originally  discovered by  \cite{Lagsansky} for Chevalley  groups   and generalized in \cite[\S  5]{CZE}. 
\end{remark}

\subsection{Example 2.}   
\label{ESoverE}

In this section, we establish the local  norm relations in the setting   of   \S      \ref{cohoform}  at an  inert prime $ \ell  \neq p  $   that specialize to the Euler factor $  P_{\lambda}(\mathrm{Frob}_{\lambda}^{-1}) $  given in     (\ref{PlambdaX})  in Galois cohomology when the level of the modular curve is $ \widehat{\Gamma}_{0}(N)$.    %\footnote{when the level of the curve is $ \widehat{\Gamma}_{0}(N)$ for some $ N $ not divisible by $ \ell$.}     
The notation of this subsection is independent of \S \ref{examplezetasec}   and   \S  \ref{examplesecLSZ}.

Let $ E $ denote the  unique unramified extension of $ \QQ_{\ell}$ of degree $ 2 $. We choose a $\ZZ_{\ell}$-basis $ (1, \delta) $ for the ring of integers $ \mathcal{O}_{E} $  where   $ \delta $ is  a  trace  zero  element in $  \mathcal{O}_{E}$. This determines an embedding $$ \iota_{0} :  E^{\times} \hookrightarrow  \GL_{2}(\QQ_{\ell})   ,   \quad   \quad   a  +  b  \delta  \mapsto  \left (   \begin{smallmatrix}  a   &  b \delta^{2}  \\    b  &  a   \end{smallmatrix}   \right )     $$ Let $ T \subset E^{\times} $ denote the  group of elements $ h $ such that $ h \bar{h}  = 1 $, where $ \bar{h} $ denotes the conjugate  of  $ h $ under the non-trivial  element of $ \Gal(E/\QQ_{\ell})$.  Set   $$ H = E^{\times}, \quad    G = \GL_{2}(\QQ_{\ell}) \times T  $$ 
and   let   
\begin{align*}   
\iota :   H    & \to  G    , \quad \quad  ( \iota_{0}(h) ,  h/ \bar{h})   
\end{align*}   
%where $ \bar{h} \in E^{\times} $ denotes the conjugate of $ h \in E^{\times} $ under the action of the non-trivial element in $ \Gal(E / \QQ) $. 
An  application of Hilbert 90  implies that  the map $ \nu : H \to T $, $ h \mapsto  h/\bar{h} $ is surjective and induces an isomorphism $$  E^{\times} /\QQ_{\ell}^{\times}    =  \mathcal{O}_{E}^{\times}/\ZZ_{\ell}^{\times}  \simeq  T  .   $$
In  particular,  $ T $  is compact. Set $$ T_{1} :   =      \nu    ( \ZZ_{\ell}^{\times}  +  \ell  \mathcal{O}_{E})   \subset   T .   $$      Note   that  $  [T :   T_{1} ]  =  \ell +   1   $.    
We set  $$ U =  \mathcal{O}_{E}^{\times} ,   \quad    K  =  \GL_{2}(\QQ_{\ell} )  \times  T  ,   \quad     L   =   \GL_{2}(\ZZ_{\ell})  \times  T_{1}   . $$
We define $$ \mathfrak{H} =   \ell^{2}  \ch(K)   -      \left   (   \ch ( K  \varsigma ^ { - 1 }   K )  -  (\ell - 1 ) \ch(K \tau ^  { - 1  }    K )   \right   )    +   \ch(K \upsilon ^{-1}  K )   $$ 
where $$    \varsigma   =   \left (  \left  (  \begin{smallmatrix}  \ell^{2}  \\ & 1   \end{smallmatrix}   \right )  , 1  \right )  , \quad   \tau =  \left (   \left (   \begin{smallmatrix}  \ell  \\ &  \ell   \end{smallmatrix}   \right ) ,  1     \right  ) ,        \quad     \upsilon  =  \left  (   \left ( \begin{smallmatrix}  \ell^{2}  &   \\[-0.1em]       &   \ell^{2}   \end{smallmatrix}  \right  ) ,  1   \right   )  .  $$
\begin{remark} The local emmbedding  arising from the global embedding in \S \ref{cohoform} will in general not agree with the embedding considered here. However,  the  content    of   Remark  \ref{notdiagembedding}   also  applies   here.    
\end{remark}
\begin{remark}  We are using $ 1 $ in the second component of the elements $ \varsigma $, $ \tau $, $ \upsilon $ since  the geometric Frobenius at a place of an imaginary  quadratic field above an inert rational prime  $ \ell  $   restricts to   the     trivial element in the Galois group of any anticyclotomic  Galois  extension that is  unramified at $ \ell $.      
\end{remark}

\begin{remark} If we assume that   % Note that if $ \Pi $ is the local piece of an automorphic representation attached to the Tate module $ \mathrm{T}_{p}(A) $ from the introduction, then the reverse characteristic polynomial
the Euler factor (\ref{PellX}) factors  as   $$ P_{\ell}(X) =  ( 1 - \alpha^{-1}\ell^{-\frac{1}{2}}X) ( 1 - \beta ^{-1}   \ell^{-\frac{1}{2}} X)   $$ over $\CC$, then  the  Euler   factor  (\ref{PlambdaX}) equals  \begin{align*}  P_{\lambda}(X)   &     =  ( 1 - \alpha^{-2} \ell^{-1}X)( 1 -   \beta^{-2}  \ell^{-1} X )    \\
&  =   1 - \ell^{-1}  (\alpha^{-2} + \beta^{-2}   )   X    +  \ell^{-2}  (\alpha\beta)^{-2}  X^{2}   .    
\end{align*}     Let $  \Pi $  denote   the  unramified principal  series  representation with Satake  parameters $ \alpha $, $ \beta $.  Consider $  \Pi  $ as a representation of $ G $ where the action of $ T $ is trivial   and  let  $  \phi^{\circ}  \in  \Pi^{K}$ denote a non-zero  element.       Then $ \mathfrak{H}$ above satisfies  $$ \mathfrak{H}  \cdot  \phi^{\circ}   =   \ell^{2}  \cdot   (1 - \alpha^{-2}  \ell^{-1}  )(1 -  \beta^{-2}  \ell^{-1} )    \cdot  \phi^{\circ} .   $$     
The normalized expression $ \mathfrak{H}$ is obtained by inverting the Satake transform, and the reader can find the  relevant  computations  in   \cite[\S 4.5]{CZE}.   
\end{remark}   
Recall that for $ g \in G $, we denote $ H_{g} = H \cap g K g^{-1} $ and $ V_{g} =  H \cap gLg^{-1}$.   
\begin{lemma}   \label{degreelemma}       $ H_{g} = V_{g}   $ if $ Hg K \neq H K  $.       
%\begin{itemize}
%\end{itemize}   
\end{lemma} 
\begin{proof}     Note that $ H_{g} \subset U $ for any $ g \in G $, since $ U $ is the unique maximal compact open subgroup of $H  $.   By Iwasawa decomposition for $ G $, any coset in $ H\backslash G/K    $  has  a   representative  of the   form  $$  g(u,  x)   =  \left (  \begin{smallmatrix} \ell^{u} & x\\[0.1em] &    1    \end{smallmatrix}  \right ) $$ 
where  $ u \in \mathbb{Z} $ and $ x \in \QQ_{\ell}$. Since the equality of  $ H_{g} $ and $  V_{g} $ does not depend on the class of $  g  $ in    $  H  \backslash  G /  K  $,     it  suffices to verify the claim for $ g = g(u,x) $.  
%Recall that for $ g \in G $, we denote $ H_{g} = H \cap g K g^{-1} $ and $ V_{g} =  H \cap gLg^{-1}$.   
So let $ h \in H_{g(u,  x) }  \subset   U $. Then $  g( u  , x  ) ^{-1} h g(u, x)    \in  K  $  by  definition.  Let us  write    $ h   =   \left  (  \begin{smallmatrix}  a  &  b  \delta^{2}   \\[0.05em]        b &  a  \end{smallmatrix}  \right ) $ where $ a , b \in \ZZ_{\ell} $.    Now $$     g(u, x)^{-1} h g( u , x  )    K    =   \left (   \begin{matrix}  a  - bx    &   b \ell^{-u}( \delta^{2} - x^{2}) \\ b \ell^{u}    &  a + b x   \end{matrix}   \right   )  K  . $$
If $ u < 0 $ or $ x \in \QQ_{\ell} \setminus \ZZ_{\ell} $,  then $ b $ must be in $ \ell \ZZ_{\ell} $ for the displayed  matrix to be in $ K $,  which  implies  $ H_{g} =  V_{g}  $ in this case.  If  $  u  > 0 $, then since   $ \delta^{2} - x ^{2} $ is either in $ \ZZ_{\ell}^{\times} $ or $ \QQ_{\ell} \setminus \ZZ_{\ell} $,  we  see that  $ b \ell^{-u} (\delta^{2} - x^{2}) $ cannot be in $ \ZZ_{\ell} $ unless $ b \in \ell \ZZ_{\ell} $, and so $ H_{g} = V_{g} $  in this  case  too.  So  the  only   possibility for $ H_{g}  $ to not be equal to $ V_{g} $ is $ u = 0 $ and $ x \in \ZZ_{\ell} $, in which case  $ H g(u,x) K = HK $.    
\end{proof}   
For $ g \in G $,  let $  \mathfrak{h}_{g} : H \to \ZZ $ denote the $  g $-twisted  $   H   $-restriction of $  \mathfrak{H}   $.  Lemma  \ref{degreelemma}  implies that the criteria of Theorem  \ref{zetateo}  is  trivially  satisfied for all $ g $ unless $ Hg K = H K $.  Now we have the decomposition $$     K  \varsigma^{-1}  K / K  =  \bigsqcup_{i = 0}^{\ell  -  1 }     \left (  \begin{smallmatrix}  1 / \ell^{2}  \\  i / \ell    &  1     \end{smallmatrix}      \right )    K   \sqcup  \bigsqcup_{j =  0  }^{\ell^{2} -1}    \left  (  \begin{smallmatrix}  1   &    j / \ell^{2}  \\ &  1 / \ell^{2}    \end{smallmatrix}  \right  )    K    $$
and none of the cosets of $ G/ K$  appearing in this decomposition  map   to $ H K \in H \backslash G/ K $ under   the  projection   $ G/ K  \to H \backslash G / K  $, $ g K \mapsto HgK $.   Therefore, 
$$ \mathfrak{h}_{1_{G}}  =      \ell^{2}  \,  \ch(U)    + ( \ell - 1 )  \, \ch( U\ell ^{-1}  U )  +  \ch( U \ell^{-2} U)  .  $$
Since $$   \deg (  \mathfrak{h}_{1_{G}, *}  ) =   \ell(\ell + 1)   \in   [H_{1_{G}} : V_{1_{G}}]  \cdot  \ZZ_{p} =  (\ell + 1 )  \cdot  \ZZ_{p} $$ 
a  zeta  element  exists in this scenario.   

\begin{remark}   \label{hardtowrite}     The actual zeta element  is  quite complicated to write down even in this simple situation, since the volumes of the twisted intersections
$ 
V_{g} = V \cap g L g^{-1} $ for $ g \in H \backslash H \cdot \mathrm{Supp}(\mathfrak{H}) / K,
$
are given by intricate polynomial expressions in $\ell$. %In higher-dimensional settings,  the abstract criteria
The abstract criteria  of Theorem~\ref{zetateo}   %   allows   
provides   many   similar advantages  in  higher  dimensional  settings.  
\end{remark}     
\bibliographystyle{amsalpha}    
\bibliography{refs}

@incollection {Gross,
    AUTHOR = {Gross, Benedict H.},
     TITLE = {Heegner points on {$X_0(N)$}},
 BOOKTITLE = {Modular forms ({D}urham, 1983)},
    SERIES = {Ellis Horwood Ser. Math. Appl.: Statist. Oper. Res.},
     PAGES = {87--105},
 PUBLISHER = {Horwood, Chichester},
      YEAR = {1984},
   MRCLASS = {11G05 (11G16 11G40)},
  MRNUMBER = {803364},
MRREVIEWER = {Loren D. Olson},
}

@article {GrossZagier,
    AUTHOR = {Gross, Benedict H. and Zagier, Don B.},
     TITLE = {\href{https://doi.org/10.1007/BF01388809}{Heegner points and derivatives of {$L$}-series}},
   JOURNAL = {Invent. Math.},
  FJOURNAL = {Inventiones Mathematicae},
    VOLUME = {84},
      YEAR = {1986},
    NUMBER = {2},
     PAGES = {225--320},
      ISSN = {0020-9910},
   MRCLASS = {11G40 (11F11 11G05 14G10)},
  MRNUMBER = {833192},
MRREVIEWER = {Loren D. Olson},
       DOI = {10.1007/BF01388809},
       URL = {https://doi.org/10.1007/BF01388809},
}

@book {Diamondmodular,
    AUTHOR = {Diamond, Fred and Shurman, Jerry},
     TITLE = {A first course in modular forms},
    SERIES = {Graduate Texts in Mathematics},
    VOLUME = {228},
 PUBLISHER = {Springer-Verlag, New York},
      YEAR = {2005},
     PAGES = {xvi+436},
      ISBN = {0-387-23229-X},
   MRCLASS = {11Fxx},
  MRNUMBER = {2112196},
MRREVIEWER = {Henri Darmon},
}

@article {Miller,
    AUTHOR = {Miller, Robert L.},
     TITLE = {\href{https://doi.org/10.1112/S1461157011000180}{Proving the {B}irch and {S}winnerton-{D}yer conjecture for
              specific elliptic curves of analytic rank zero and one}},
   JOURNAL = {LMS J. Comput. Math.},
  FJOURNAL = {LMS Journal of Computation and Mathematics},
    VOLUME = {14},
      YEAR = {2011},
     PAGES = {327--350},
   MRCLASS = {11G40 (11G05 14G10 14G25)},
  MRNUMBER = {2861691},
MRREVIEWER = {Remke Kloosterman},
       DOI = {10.1112/S1461157011000180},
       URL = {https://doi.org/10.1112/S1461157011000180},
}

@article {Mazur,
    AUTHOR = {Mazur, B.},
     TITLE = {\href{http://www.numdam.org/item?id=PMIHES_1977__47__33_0}{Modular curves and the {E}isenstein ideal}},
      NOTE = {With an appendix by Mazur and M. Rapoport},
   JOURNAL = {Inst. Hautes \'{E}tudes Sci. Publ. Math.},
  FJOURNAL = {Institut des Hautes \'{E}tudes Scientifiques. Publications
              Math\'{e}matiques},
    NUMBER = {47},
      YEAR = {1977},
     PAGES = {33--186 (1978)},
      ISSN = {0073-8301},
   MRCLASS = {14G25 (10D05)},
  MRNUMBER = {488287},
MRREVIEWER = {M. Ohta},
       URL = {http://www.numdam.org/item?id=PMIHES_1977__47__33_0},
}

@incollection {WilesClay,
    AUTHOR = {Wiles, Andrew},
     TITLE = {The {B}irch and {S}winnerton-{D}yer conjecture},
 BOOKTITLE = {The millennium prize problems},
     PAGES = {31--41},
 PUBLISHER = {Clay Math. Inst., Cambridge, MA},
      YEAR = {2006},
   MRCLASS = {11G40},
  MRNUMBER = {2238272},
}

@incollection {kolyvagin,
    AUTHOR = {Kolyvagin, V. A.},
     TITLE = {Euler systems},
 BOOKTITLE = {The {G}rothendieck {F}estschrift, {V}ol. {II}},
    SERIES = {Progr. Math.},
    VOLUME = {87},
     PAGES = {435--483},
 PUBLISHER = {Birkh\"{a}user Boston, Boston, MA},
      YEAR = {1990},
   MRCLASS = {11R34 (11G05 11G40 11R29)},
  MRNUMBER = {1106906},
MRREVIEWER = {Karl Rubin},
}

@article {CRR,
    AUTHOR = {Alonso, Ra\'{u}l and Castella, Francesc and Rivero, \'{O}scar},
     TITLE = {\href{https://doi.org/10.1017/S1474748023000221}{The diagonal cycle {E}uler system for {$\rm GL_2\times GL_2$}}},
   JOURNAL = {J. Inst. Math. Jussieu},
  FJOURNAL = {Journal of the Institute of Mathematics of Jussieu. JIMJ.
              Journal de l'Institut de Math\'{e}matiques de Jussieu},
    VOLUME = {24},
      YEAR = {2025},
    NUMBER = {5},
     PAGES = {1591--1653},
      ISSN = {1474-7480,1475-3030},
   MRCLASS = {11R23 (11F85 14G35)},
  MRNUMBER = {4959511},
       DOI = {10.1017/S1474748023000221},
       URL = {https://doi.org/10.1017/S1474748023000221},
}

@misc{SkinnerLai,
      title={\href{https://arxiv.org/abs/2408.01219}{Anti-cyclotomic {E}uler system of diagonal cycles}}, 
      author={Shilin Lai and Christopher Skinner},
      year={2024},
      eprint={2408.01219},
      archivePrefix={arXiv},
      primaryClass={math.NT},
      url={https://arxiv.org/abs/2408.01219}, 
}

@misc{CZE,
      title={\href{https://arxiv.org/abs/2409.03517}{On constructing zeta elements for Shimura varieties}}, 
      author={Syed Waqar Ali Shah},
      year={2024},
      eprint={2409.03517},
      archivePrefix={arXiv},
      primaryClass={math.NT},
      url={https://arxiv.org/abs/2409.03517}, 
}

@misc{CaiFanLai,
      title={\href{https://arxiv.org/abs/2410.18392}{Euler systems and relative {S}atake isomorphism}}, 
      author={Li Cai and Yangyu Fan and Shilin Lai},
      year={2024},
      eprint={2410.18392},
      archivePrefix={arXiv},
      primaryClass={math.NT},
      url={https://arxiv.org/abs/2410.18392}, 
}

@misc{Siegel1,
      title={\href{https://arxiv.org/abs/2409.03738}{Horizontal norm compatibility of cohomology classes for {$\rm GSp_6$}}}, 
      author={Shah, Syed Waqar Ali},
      primaryClass={math.NT},
      year = {2024},
      url={}, 
}

@misc{EulerGU22,
      title={Euler systems for  exterior square motives}, 
      author={Cauchi, Antonio and Graham, Andrew and Shah, Syed Waqar Ali},
      primaryClass={math.NT},
   year = {2026}, 
      note = {in preparation}, 
      url={}, 
}

@misc{disegni,
      title={\href{https://arxiv.org/abs/2410.08419}{Euler systems for conjugate-symplectic motives}}, 
      author={Daniel Disegni},
      year={2024},
      eprint={2410.08419},
      archivePrefix={arXiv},
      primaryClass={math.NT},
      url={https://arxiv.org/abs/2410.08419}, 
}

@incollection {BlochKato,
    AUTHOR = {Bloch, Spencer and Kato, Kazuya},
     TITLE = {{$L$}-functions and {T}amagawa numbers of motives},
 BOOKTITLE = {The {G}rothendieck {F}estschrift, {V}ol. {I}},
    SERIES = {Progr. Math.},
    VOLUME = {86},
     PAGES = {333--400},
 PUBLISHER = {Birkh\"{a}user Boston, Boston, MA},
      YEAR = {1990},
   MRCLASS = {11G40 (11G09 14C35 14F30 14G10)},
  MRNUMBER = {1086888},
MRREVIEWER = {Ehud de Shalit},
}

@book {Rubin,
    AUTHOR = {Rubin, Karl},
     TITLE = {\href{https://doi.org/10.1515/9781400865208}{Euler systems}},
    SERIES = {Annals of Mathematics Studies},
    VOLUME = {147},
      NOTE = {Hermann Weyl Lectures. The Institute for Advanced Study},
 PUBLISHER = {Princeton University Press, Princeton, NJ},
      YEAR = {2000},
     PAGES = {xii+227},
      ISBN = {0-691-05075-9; 0-691-05076-7},
   MRCLASS = {11R23 (11G40 11R34 11R42)},
  MRNUMBER = {1749177},
MRREVIEWER = {Jan Nekov\'{a}\v{r}},
       DOI = {10.1515/9781400865208},
       URL = {https://doi.org/10.1515/9781400865208},
}

@incollection{DeligneTS,
     author = {Deligne, Pierre},
     title = "\href{https://doi.org/10.1007/BFb0058700}{Travaux de Shimura}",
     booktitle = {S\'{e}minaire Bourbaki\ :\ vol. 1970/71, expos\'{e}s 382-399},
     series = {S\'eminaire Bourbaki},
     publisher = {Springer-Verlag},
     number = {13},
     year = {1971},
     pages = {123-165},
}

@article {Lars,
    AUTHOR = {K{\"{u}}hne, Lars},
     TITLE = {\href{https://doi.org/10.4064/aa180717-9-6}{Intersection of class fields}},
   JOURNAL = {Acta Arith.},
  FJOURNAL = {Acta Arithmetica},
    VOLUME = {198},
      YEAR = {2021},
    NUMBER = {2},
     PAGES = {109--127},
      ISSN = {0065-1036},
   MRCLASS = {11G18 (11G05 11R37 14G35)},
  MRNUMBER = {4228297},
MRREVIEWER = {Guy Fowler},
       DOI = {10.4064/aa180717-9-6},
       URL = {https://doi.org/10.4064/aa180717-9-6},
}

@incollection{GrossSatake,
    AUTHOR = {Gross, Benedict H.},
     TITLE = {\href{https://doi.org/10.1017/CBO9780511662010.006}{On the {S}atake isomorphism}},
 BOOKTITLE = {Galois representations in arithmetic algebraic geometry
              ({D}urham, 1996)},
    SERIES = {London Math. Soc. Lecture Note Ser.},
    VOLUME = {254},
     PAGES = {223--237},
 PUBLISHER = {Cambridge Univ. Press, Cambridge},
      YEAR = {1998},
   MRCLASS = {22E50 (11F70)},
  MRNUMBER = {1696481},
MRREVIEWER = {David Manderscheid},
       DOI = {10.1017/CBO9780511662010.006},
       URL = {https://doi.org/10.1017/CBO9780511662010.006},
}

@article {Murty, 
    AUTHOR = {Murty, M. Ram and Murty, V. Kumar},
     TITLE = {Mean values of derivatives of modular {$L$}-series},
   JOURNAL = {Ann. of Math. (2)},
  FJOURNAL = {Annals of Mathematics. Second Series},
    VOLUME = {133},
      YEAR = {1991},
    NUMBER = {3},
     PAGES = {447--475},
      ISSN = {0003-486X},
   MRCLASS = {11F67 (11G05 11G40)},
  MRNUMBER = {1109350},
MRREVIEWER = {Daniel Bump},
       DOI = {10.2307/2944316},
       URL = {https://doi.org/10.2307/2944316},
}

@misc{explicitdescent,
  title = {\href{https://arxiv.org/abs/2310.01677}{Explicit {H}ecke descent for special cycles}},
  author = {Shah, Syed Waqar Ali},
  year = {2025},
  eprint = {2310.01677},
  archivePrefix = {arXiv},
  primaryClass = {math.NT},
  url = {https://arxiv.org/abs/2310.01677},
  note = {to appear in \emph{Algebra \& Number Theory}}
}

@misc{Jetchev,
    title={Hecke and {G}alois Properties of Special Cycles on Unitary {S}himura varieties},
    author={Dimitar P. Jetchev},
    year={2014},
    eprint={1410.6692},
    archivePrefix={arXiv},
    primaryClass={math.NT}
}

@incollection {grosskoly,
    AUTHOR = {Gross, Benedict H.},
     TITLE = {\href{https://doi.org/10.1017/CBO9780511526053.009}{Kolyvagin's work on modular elliptic curves}},
 BOOKTITLE = {{$L$}-functions and arithmetic ({D}urham, 1989)},
    SERIES = {London Math. Soc. Lecture Note Ser.},
    VOLUME = {153},
     PAGES = {235--256},
 PUBLISHER = {Cambridge Univ. Press, Cambridge},
      YEAR = {1991},
   MRCLASS = {11G10 (11G40)},
  MRNUMBER = {1110395},
       DOI = {10.1017/CBO9780511526053.009},
       URL = {https://doi.org/10.1017/CBO9780511526053.009},
}

@article {Anticyclo,
    AUTHOR = {Graham, Andrew and Shah, Syed Waqar Ali},
     TITLE = {\href{ https://doi.org/10.1112/plms.12566}{Anticyclotomic {E}uler systems for unitary groups}},
   JOURNAL = {Proc. Lond. Math. Soc. (3)},
  FJOURNAL = {Proceedings of the London Mathematical Society. Third Series},
    VOLUME = {127},
      YEAR = {2023},
    NUMBER = {6},
     PAGES = {1577--1680},
      ISSN = {0024-6115},
   MRCLASS = {11F55 (11F67 11F80 11R23)},
  MRNUMBER = {4673434},
}

@article {LSZ,
    AUTHOR = {Loeffler, David and Skinner, Christopher and Zerbes, Sarah
              Livia},
     TITLE = {\href{https://doi.org/10.4171/jems/1124}{Euler systems for {${\rm GSp}(4)$}}},
   JOURNAL = {J. Eur. Math. Soc. (JEMS)},
  FJOURNAL = {Journal of the European Mathematical Society (JEMS)},
    VOLUME = {24},
      YEAR = {2022},
    NUMBER = {2},
     PAGES = {669--733},
      ISSN = {1435-9855},
   MRCLASS = {11F46 (11F67 11F80 11R23)},
  MRNUMBER = {4382481},
MRREVIEWER = {Meng Fai Lim},
       DOI = {10.4171/jems/1124},
       URL = {https://doi.org/10.4171/jems/1124},
}

@book {darmon,
    AUTHOR = {Darmon, Henri},
     TITLE = {Rational points on modular elliptic curves},
    SERIES = {CBMS Regional Conference Series in Mathematics},
    VOLUME = {101},
 PUBLISHER = {Published for the Conference Board of the Mathematical
              Sciences, Washington, DC; by the American Mathematical
              Society, Providence, RI},
      YEAR = {2004},
     PAGES = {xii+129},
      ISBN = {0-8218-2868-1},
   MRCLASS = {11G40 (11F67 11F85 11G05 11G18)},
  MRNUMBER = {2020572},
MRREVIEWER = {Chandrashekhar Khare},
}

@book{PinkThesis,
    author    = "Richard Pink",
    title     = "\href{https://people.math.ethz.ch/~pink/dissertation.html}{Arithmetical compactification of mixed {S}himura varieties}",
    publisher = "Bonner Mathematische Schriften",
    volume   = "209",
    year = {1988}
}

@article{Lagsansky,
    AUTHOR = {Lansky, Joshua M.},
     TITLE = {\href{https://doi.org/10.2140/pjm.2001.197.97}{Decomposition of double cosets in {$ p$}-adic groups}},
   JOURNAL = {Pacific J. Math.},
  FJOURNAL = {Pacific Journal of Mathematics},
    VOLUME = {197},
      YEAR = {2001},
    NUMBER = {1},
     PAGES = {97--117},
      ISSN = {0030-8730},
   MRCLASS = {22E20},
  MRNUMBER = {1810210},
MRREVIEWER = {Jean-Francois Dat},
       DOI = {10.2140/pjm.2001.197.97},
       URL = {https://doi.org/10.2140/pjm.2001.197.97},
}

@article {CornutnormI,
    AUTHOR = {Cornut, Christophe},
     TITLE = {\href{http://aif.cedram.org/item?id=AIF_2009__59_6_2223_0}{Normes {$p$}-adiques et extensions quadratiques}},
   JOURNAL = {Ann. Inst. Fourier (Grenoble)},
  FJOURNAL = {Universit\'{e} de Grenoble. Annales de l'Institut Fourier},
    VOLUME = {59},
      YEAR = {2009},
    NUMBER = {6},
     PAGES = {2223--2254},
      ISSN = {0373-0956},
   MRCLASS = {11E95 (20G25 22E35)},
  MRNUMBER = {2640919},
MRREVIEWER = {Maarten Sander Solleveld},
       URL = {http://aif.cedram.org/item?id=AIF_2009__59_6_2223_0},
}

@incollection{Milne,
	author	= "J. S. Milne",
	title	= "{I}ntroduction to {S}himura {V}arieties",
	booktitle= "{H}armonic {A}nalysis, {T}he {T}race {F}ormula, and {S}himura {V}arieties",
	publisher= "American Mathematical Society",
	editor	= "{J}ames {A}rthur, {D}avid {E}llwood, {R}obert {K}ottwitz",
	volume	= "4",
	series	= "Clay Mathematics Proceedings",
	pages	= "265 -- 378",
	year	= "2003",
}

@incollection {KolyvaginES,
    AUTHOR = {Kolyvagin, V. A.},
     TITLE = {Euler systems},
 BOOKTITLE = {The {G}rothendieck {F}estschrift, {V}ol. {II}},
    SERIES = {Progr. Math.},
    VOLUME = {87},
     PAGES = {435--483},
 PUBLISHER = {Birkh\"{a}user Boston, Boston, MA},
      YEAR = {1990},
   MRCLASS = {11R34 (11G05 11G40 11R29)},
  MRNUMBER = {1106906},
MRREVIEWER = {Karl Rubin},
}

@book {Miyake,
    AUTHOR = {Miyake, Toshitsune},
     TITLE = {Modular forms},
    SERIES = {Springer Monographs in Mathematics},
   EDITION = {English},
      NOTE = {Translated from the 1976 Japanese original by Yoshitaka Maeda},
 PUBLISHER = {Springer-Verlag, Berlin},
      YEAR = {2006},
     PAGES = {x+335},
      ISBN = {978-3-540-29592-1; 3-540-29592-5},
   MRCLASS = {11F11 (11F25 11F72)},
  MRNUMBER = {2194815},
}

@book {Rick,
    AUTHOR = {Miranda, Rick},
     TITLE = {\href{https://doi.org/10.1090/gsm/005}{Algebraic curves and {R}iemann surfaces}},
    SERIES = {Graduate Studies in Mathematics},
    VOLUME = {5},
 PUBLISHER = {American Mathematical Society, Providence, RI},
      YEAR = {1995},
     PAGES = {xxii+390},
      ISBN = {0-8218-0268-2},
   MRCLASS = {14Hxx (14-01 30F99)},
  MRNUMBER = {1326604},
MRREVIEWER = {R. F. Lax},
       DOI = {10.1090/gsm/005},
       URL = {https://doi.org/10.1090/gsm/005},
}

@article {CornutnormII,
    AUTHOR = {Cornut, Christophe},
     TITLE = {\href{https://doi.org/10.1007/s00229-011-0442-0}{On {$p$}-adic norms and quadratic extensions, {II}}},
   JOURNAL = {Manuscripta Math.},
  FJOURNAL = {Manuscripta Mathematica},
    VOLUME = {136},
      YEAR = {2011},
    NUMBER = {1-2},
     PAGES = {199--236},
      ISSN = {0025-2611},
   MRCLASS = {11E95 (20E42 20G25 22E35)},
  MRNUMBER = {2820402},
MRREVIEWER = {Guy Rousseau},
       DOI = {10.1007/s00229-011-0442-0},
       URL = {https://doi.org/10.1007/s00229-011-0442-0},
}

@misc{JNS,
    title={Split {E}uler systems for conjugate-dual {G}alois representations},
    author={Jetchev, Dimitar and Nekovář, Jan and Skinner, Christopher},
    note = {in preparation}, 
}

@article {WilesFermat,
    AUTHOR = {Wiles, Andrew},
     TITLE = {\href{https://doi.org/10.2307/2118559}{Modular elliptic curves and {F}ermat's last theorem}},
   JOURNAL = {Ann. of Math. (2)},
  FJOURNAL = {Annals of Mathematics. Second Series},
    VOLUME = {141},
      YEAR = {1995},
    NUMBER = {3},
     PAGES = {443--551},
      ISSN = {0003-486X},
   MRCLASS = {11G05 (11D41 11F11 11F80 11G18)},
  MRNUMBER = {1333035},
MRREVIEWER = {Karl Rubin},
       DOI = {10.2307/2118559},
       URL = {https://doi.org/10.2307/2118559},
}

@article {Conradetal,
    AUTHOR = {Breuil, Christophe and Conrad, Brian and Diamond, Fred and
              Taylor, Richard},
     TITLE = {\href{https://doi.org/10.1090/S0894-0347-01-00370-8}{On the modularity of elliptic curves over {$\mathbf{Q}$}: wild
              3-adic exercises}},
   JOURNAL = {J. Amer. Math. Soc.},
  FJOURNAL = {Journal of the American Mathematical Society},
    VOLUME = {14},
      YEAR = {2001},
    NUMBER = {4},
     PAGES = {843--939},
      ISSN = {0894-0347},
   MRCLASS = {11G05 (11F80 11G07 14G35)},
  MRNUMBER = {1839918},
MRREVIEWER = {Karl Rubin},
       DOI = {10.1090/S0894-0347-01-00370-8},
       URL = {https://doi.org/10.1090/S0894-0347-01-00370-8},
}

@incollection {KingsBloch,
    AUTHOR = {Kings, Guido},
     TITLE = {\href{https://doi.org/10.5802/jtnb.396}{The {B}loch-{K}ato conjecture on special values of
              {$L$}-functions. {A} survey of known results}},
      NOTE = {Les XXII\`emes Journ\'{e}es Arithmetiques (Lille, 2001)},
   JOURNAL = {J. Th\'{e}or. Nombres Bordeaux},
  FJOURNAL = {Journal de Th\'{e}orie des Nombres de Bordeaux},
    VOLUME = {15},
      YEAR = {2003},
    NUMBER = {1},
     PAGES = {179--198},
      ISSN = {1246-7405},
   MRCLASS = {11G40 (11R23)},
  MRNUMBER = {2019010},
MRREVIEWER = {Th\cfac{o}ng Nguy\cftil{e}n-Quang-\Dbar \cftil{o}},
       DOI = {10.5802/jtnb.396},
       URL = {https://doi.org/10.5802/jtnb.396},
}

@misc{FontaineOuyang,
      title={\href{https://www.imo.universite-paris-saclay.fr/~fontaine/galoisrep.pdf}{Theory of $p$-adic Galois representations}}, 
      author={Fontaine, Jean M. and Ouyang, Yi},
      primaryClass={math.NT},
      note = {preprint},
      year = {}, 
      url={https://www.imo.universite-paris-saclay.fr/~fontaine/galoisrep.pdf}, 
}

@article {loe,
    AUTHOR = {Loeffler, David},
     TITLE = {\href{https://doi.org/10.5802/jtnb.1186}{Spherical varieties and norm relations in {I}wasawa theory}},
   JOURNAL = {J. Th\'{e}or. Nombres Bordeaux},
  FJOURNAL = {Journal de Th\'{e}orie des Nombres de Bordeaux},
    VOLUME = {33},
      YEAR = {2021},
    NUMBER = {3, part 2},
     PAGES = {1021--1043},
      ISSN = {1246-7405},
   MRCLASS = {11F67 (11R23 14M27)},
  MRNUMBER = {4402388},
MRREVIEWER = {Matteo Longo},
       DOI = {10.5802/jtnb.1186},
       URL = {https://doi.org/10.5802/jtnb.1186},
}

@book {MilneElliptic,
    AUTHOR = {Milne, James S.},
     TITLE = {\href{https://doi.org/10.1142/11870 }{Elliptic curves}},
      NOTE = {Second edition [of  2267743]},
 PUBLISHER = {World Scientific Publishing Co. Pte. Ltd., Hackensack, NJ},
      YEAR = {[2021] \copyright 2021},
     PAGES = {x+308},
      ISBN = {[9789811221859]; [9789811221835]},
   MRCLASS = {14H52 (11G05 11G40)},
  MRNUMBER = {4235740},
}

@book {VignerasBook,
    AUTHOR = {Vign\'{e}ras, Marie-France},
     TITLE = {Repr\'{e}sentations {$l$}-modulaires d'un groupe r\'{e}ductif
              {$p$}-adique avec {$l\ne p$}},
    SERIES = {Progress in Mathematics},
    VOLUME = {137},
 PUBLISHER = {Birkh\"{a}user Boston, Inc., Boston, MA},
      YEAR = {1996},
     PAGES = {xviii and 233},
      ISBN = {0-8176-3929-2},
   MRCLASS = {22E35 (20G25 22-02 22E50)},
  MRNUMBER = {1395151},
MRREVIEWER = {Chon Hu Cheng},
}

@book {BushHenn,
    AUTHOR = {Bushnell, Colin J. and Henniart, Guy},
     TITLE = {\href{https://doi.org/10.1007/3-540-31511-X}{The local {L}anglands conjecture for {$\rm GL(2)$}}},
    SERIES = {Grundlehren der mathematischen Wissenschaften [Fundamental
              Principles of Mathematical Sciences]},
    VOLUME = {335},
 PUBLISHER = {Springer-Verlag, Berlin},
      YEAR = {2006},
     PAGES = {xii+347},
      ISBN = {978-3-540-31486-8; 3-540-31486-5},
   MRCLASS = {22E50 (11-02 11S37 22-02)},
  MRNUMBER = {2234120},
MRREVIEWER = {Alexandru Ioan Badulescu},
       DOI = {10.1007/3-540-31511-X},
       URL = {https://doi.org/10.1007/3-540-31511-X},
}

@incollection {Rohrlich,
    AUTHOR = {Rohrlich, David E.},
     TITLE = {\href{https://doi.org/10.1007/978-1-4612-1974-3_3}{Modular curves, {H}ecke correspondence, and {$L$}-functions}},
 BOOKTITLE = {Modular forms and {F}ermat's last theorem ({B}oston, {MA},
              1995)},
     PAGES = {41--100},
 PUBLISHER = {Springer, New York},
      YEAR = {1997},
   MRCLASS = {11G18 (11F32 11F66 11G40)},
  MRNUMBER = {1638476},
}

@incollection {CorVastal,
    AUTHOR = {Cornut, Christophe and Vatsal, Vinayak},
     TITLE = {\href{https://doi.org/10.1017/CBO9780511721267.005}{Nontriviality of {R}ankin-{S}elberg {$L$}-functions and {CM}
              points}},
 BOOKTITLE = {{$L$}-functions and {G}alois representations},
    SERIES = {London Math. Soc. Lecture Note Ser.},
    VOLUME = {320},
     PAGES = {121--186},
 PUBLISHER = {Cambridge Univ. Press, Cambridge},
      YEAR = {2007},
   MRCLASS = {11G18 (11F66 11F70 11G15)},
  MRNUMBER = {2392354},
MRREVIEWER = {Benjamin V. Howard},
       DOI = {10.1017/CBO9780511721267.005},
       URL = {https://doi.org/10.1017/CBO9780511721267.005},
}

@incollection {RibetSerre,
    AUTHOR = {Ribet, Kenneth A. and Stein, William A.},
     TITLE = {\href{https://doi.org/10.1090/pcms/009/04}{Lectures on {S}erre's conjectures}},
 BOOKTITLE = {Arithmetic algebraic geometry ({P}ark {C}ity, {UT}, 1999)},
    SERIES = {IAS/Park City Math. Ser.},
    VOLUME = {9},
     PAGES = {143--232},
 PUBLISHER = {Amer. Math. Soc., Providence, RI},
      YEAR = {2001},
   MRCLASS = {11F80 (11F66 11G05)},
  MRNUMBER = {1860042},
MRREVIEWER = {Chandrashekhar Khare},
       DOI = {10.1090/pcms/009/04},
       URL = {https://doi.org/10.1090/pcms/009/04},
}

@incollection {Beilinson,
    AUTHOR = {Beilinson, A. A.},
     TITLE = {\href{https://doi.org/10.1090/conm/055.1/862627}{Higher regulators of modular curves}},
 BOOKTITLE = {Applications of algebraic {$K$}-theory to algebraic geometry
              and number theory, {P}art {I}, {II} ({B}oulder, {C}olo.,
              1983)},
    SERIES = {Contemp. Math.},
    VOLUME = {55},
     PAGES = {1--34},
 PUBLISHER = {Amer. Math. Soc., Providence, RI},
      YEAR = {1986},
   MRCLASS = {11G40 (11F67 11R70 14C35 14G10 18F25 19F27)},
  MRNUMBER = {862627},
MRREVIEWER = {Glenn Stevens},
       DOI = {10.1090/conm/055.1/862627},
       URL = {https://doi.org/10.1090/conm/055.1/862627},
}

@article {Ribet,
    AUTHOR = {Ribet, K. A.},
     TITLE = {\href{https://doi.org/10.1007/BF01231195}{On modular representations of {${\rm Gal}(\overline{\bf
              Q}/{\bf Q})$} arising from modular forms}},
   JOURNAL = {Invent. Math.},
  FJOURNAL = {Inventiones Mathematicae},
    VOLUME = {100},
      YEAR = {1990},
    NUMBER = {2},
     PAGES = {431--476},
      ISSN = {0020-9910},
   MRCLASS = {11G18 (11F32 11F80 11S37)},
  MRNUMBER = {1047143},
MRREVIEWER = {Glenn Stevens},
       DOI = {10.1007/BF01231195},
       URL = {https://doi.org/10.1007/BF01231195},
}

@article {Nekovar,
    AUTHOR = {Nekov\'{a}\v{r}, Jan},
     TITLE = {\href{https://doi.org/10.24033/asens.2374}{Eichler-{S}himura relations and semisimplicity of \'{e}tale
              cohomology of quaternionic {S}himura varieties}},
   JOURNAL = {Ann. Sci. \'{E}c. Norm. Sup\'{e}r. (4)},
  FJOURNAL = {Annales Scientifiques de l'\'{E}cole Normale Sup\'{e}rieure. Quatri\`eme
              S\'{e}rie},
    VOLUME = {51},
      YEAR = {2018},
    NUMBER = {5},
     PAGES = {1179--1252},
      ISSN = {0012-9593},
   MRCLASS = {14G35 (11G18 14F20)},
  MRNUMBER = {3942040},
MRREVIEWER = {Gregorio Baldi},
       DOI = {10.24033/asens.2374},
       URL = {https://doi.org/10.24033/asens.2374},
}

@incollection {BlasiusRogawski,
    AUTHOR = {Blasius, Don and Rogawski, Jonathan D.},
     TITLE = {\href{https://doi.org/10.1090/pspum/055.2/1265563}{Zeta functions of {S}himura varieties}},
 BOOKTITLE = {Motives ({S}eattle, {WA}, 1991)},
    SERIES = {Proc. Sympos. Pure Math.},
    VOLUME = {55},
     PAGES = {525--571},
 PUBLISHER = {Amer. Math. Soc., Providence, RI},
      YEAR = {1994},
   MRCLASS = {11F32 (11F70 11G18 14G35)},
  MRNUMBER = {1265563},
MRREVIEWER = {Joachim Schwermer},
       DOI = {10.1090/pspum/055.2/1265563},
       URL = {https://doi.org/10.1090/pspum/055.2/1265563},
}

@misc{Milne1990b,
  author       = {Milne, J.~S.},
  note         = {Letter to {P}. {D}eligne (8 pages), available at \url{https://www.jmilne.org/math/articles/1990b.pdf}},
  year         = {1990},
  url          = {https://www.jmilne.org/math/articles/1990b.pdf}
}

@incollection {MilneJacobian,
    AUTHOR = {Milne, J. S.},
     TITLE = {Jacobian varieties},
 BOOKTITLE = {Arithmetic geometry ({S}torrs, {C}onn., 1984)},
     PAGES = {167--212},
 PUBLISHER = {Springer, New York},
      YEAR = {1986},
   MRCLASS = {14H40},
  MRNUMBER = {861976},
}

@article {KatoEuler,
    AUTHOR = {Kato, Kazuya},
     TITLE = {Euler systems, {I}wasawa theory, and {S}elmer groups},
   JOURNAL = {Kodai Math. J.},
  FJOURNAL = {Kodai Mathematical Journal},
    VOLUME = {22},
      YEAR = {1999},
    NUMBER = {3},
     PAGES = {313--372},
      ISSN = {0386-5991},
   MRCLASS = {11G40 (11F67 11F85 11R23 19F15)},
  MRNUMBER = {1727298},
MRREVIEWER = {Alexey A. Panchishkin},
       DOI = {10.2996/kmj/1138044090},
       URL = {https://doi.org/10.2996/kmj/1138044090},
}

@article {Perrin-Riou,
    AUTHOR = {Perrin-Riou, Bernadette},
     TITLE = {\href{https://doi.org/10.5802/aif.1655}{Syst\`emes d'{E}uler {$p$}-adiques et th\'{e}orie d'{I}wasawa}},
   JOURNAL = {Ann. Inst. Fourier (Grenoble)},
  FJOURNAL = {Universit\'{e} de Grenoble. Annales de l'Institut Fourier},
    VOLUME = {48},
      YEAR = {1998},
    NUMBER = {5},
     PAGES = {1231--1307},
      ISSN = {0373-0956},
   MRCLASS = {11R23 (11G40 11R42)},
  MRNUMBER = {1662231},
MRREVIEWER = {Amnon Besser},
       DOI = {10.5802/aif.1655},
       URL = {https://doi.org/10.5802/aif.1655},
}

@book {Bump,
    AUTHOR = {Bump, Daniel},
     TITLE = {Automorphic forms and representations},
    SERIES = {Cambridge Studies in Advanced Mathematics},
    VOLUME = {55},
 PUBLISHER = {Cambridge University Press, Cambridge},
      YEAR = {1997},
     PAGES = {xiv+574},
      ISBN = {0-521-55098-X},
   MRCLASS = {11F70 (11F41 11R39 22E50 22E55)},
  MRNUMBER = {1431508},
MRREVIEWER = {Solomon\ Friedberg},
       DOI = {10.1017/CBO9780511609572},
       URL = {https://doi.org/10.1017/CBO9780511609572},
}

@unpublished{compactinduction,
  author    = {Syed Waqar Ali Shah},
  title     = {Compact induction, norm relations and spherical varieties},
  year      = {2026},
  note      = {in preparation}
}

@article {Jannsen1988,
    AUTHOR = {Jannsen, Uwe},
     TITLE = {\href{https://doi.org/10.1007/BF01456052}{Continuous \'{e}tale cohomology}},
   JOURNAL = {Math. Ann.},
  FJOURNAL = {Mathematische Annalen},
    VOLUME = {280},
      YEAR = {1988},
    NUMBER = {2},
     PAGES = {207--245},
      ISSN = {0025-5831},
   MRCLASS = {14F20 (11G25)},
  MRNUMBER = {929536},
MRREVIEWER = {Wayne Raskind},
       DOI = {10.1007/BF01456052},
       URL = {https://doi.org/10.1007/BF01456052},
}

@article {ChenSun,
    AUTHOR = {Chen, Fulin and Sun, Binyong},
     TITLE = {Uniqueness of twisted linear periods and twisted {S}halika
              periods},
   JOURNAL = {Sci. China Math.},
  FJOURNAL = {Science China. Mathematics},
    VOLUME = {63},
      YEAR = {2020},
    NUMBER = {1},
     PAGES = {1--22},
      ISSN = {1674-7283,1869-1862},
   MRCLASS = {22E50},
  MRNUMBER = {4047168},
       DOI = {10.1007/s11425-018-9502-y},
       URL = {https://doi.org/10.1007/s11425-018-9502-y},
}

@book {Vakil,
    AUTHOR = {Vakil, Ravi},
     TITLE = {The rising sea---foundations of algebraic geometry},
 PUBLISHER = {Princeton University Press, Princeton, NJ},
      YEAR = {[2025] \copyright 2025},
     PAGES = {xxiii+662},
      ISBN = {978-0-691-26866-8; 978-0-691-26867-5; 978-0-691-26868-2},
   MRCLASS = {14-01},
  MRNUMBER = {4942570},
}

@misc{Bellaiche,
  author       = {Joël Bellaïche},
  title        = {An introduction to the conjecture of {B}loch and {K}ato},
  howpublished = {Lecture notes, Clay Math Institute Summer School, 2009},
  note         = {Available at \url{https://virtualmath1.stanford.edu/~conrad/BSDseminar/refs/BKintro.pdf}},
  year         = {2009}
}

@book {Morel,
    AUTHOR = {Morel, Sophie},
     TITLE = {\href{https://doi.org/10.1515/9781400835393}{On the cohomology of certain noncompact {S}himura varieties}},
    SERIES = {Annals of Mathematics Studies},
    VOLUME = {173},
      NOTE = {With an appendix by Robert Kottwitz},
 PUBLISHER = {Princeton University Press, Princeton, NJ},
      YEAR = {2010},
     PAGES = {xii+217},
      ISBN = {978-0-691-14293-7},
   MRCLASS = {11F75 (11F72 11G18 14G35)},
  MRNUMBER = {2567740},
MRREVIEWER = {Eva Viehmann},
       DOI = {10.1515/9781400835393},
       URL = {https://doi.org/10.1515/9781400835393},
}

@article {Bultel,
    AUTHOR = {B\"{u}ltel, Oliver},
     TITLE = {\href{https://doi.org/10.1515/crll.2002.020}{The congruence relation in the non-{PEL} case}},
   JOURNAL = {J. Reine Angew. Math.},
  FJOURNAL = {Journal f\"{u}r die Reine und Angewandte Mathematik. [Crelle's
              Journal]},
    VOLUME = {544},
      YEAR = {2002},
     PAGES = {133--159},
      ISSN = {0075-4102},
   MRCLASS = {14G35 (11F46 11G18)},
  MRNUMBER = {1887893},
MRREVIEWER = {Rolf Berndt},
       DOI = {10.1515/crll.2002.020},
       URL = {https://doi.org/10.1515/crll.2002.020},
}

@article {Wedhorn,
    AUTHOR = {B\"{u}ltel, Oliver and Wedhorn, Torsten},
     TITLE = {\href{https://doi.org/10.1017/S1474748005000253}{Congruence relations for {S}himura varieties associated to
              some unitary groups}},
   JOURNAL = {J. Inst. Math. Jussieu},
  FJOURNAL = {Journal of the Institute of Mathematics of Jussieu. JIMJ.
              Journal de l'Institut de Math\'{e}matiques de Jussieu},
    VOLUME = {5},
      YEAR = {2006},
    NUMBER = {2},
     PAGES = {229--261},
      ISSN = {1474-7480},
   MRCLASS = {11G18 (14G35 14K10)},
  MRNUMBER = {2225042},
MRREVIEWER = {Ambrus P\'{a}l},
       DOI = {10.1017/S1474748005000253},
       URL = {https://doi.org/10.1017/S1474748005000253},
}

@article {Scholze-Shin,
    AUTHOR = {Scholze, Peter and Shin, Sug Woo},
     TITLE = {\href{https://doi.org/10.1090/S0894-0347-2012-00752-8}{On the cohomology of compact unitary group {S}himura varieties
              at ramified split places}},
   JOURNAL = {J. Amer. Math. Soc.},
  FJOURNAL = {Journal of the American Mathematical Society},
    VOLUME = {26},
      YEAR = {2013},
    NUMBER = {1},
     PAGES = {261--294},
      ISSN = {0894-0347},
   MRCLASS = {11G18 (11F70 11F80 11S37 14G35 22E50)},
  MRNUMBER = {2983012},
MRREVIEWER = {Wei Zhang},
       DOI = {10.1090/S0894-0347-2012-00752-8},
       URL = {https://doi.org/10.1090/S0894-0347-2012-00752-8},
}

@misc{CaraianiShin,
      title={\href{https://arxiv.org/abs/2311.13382}{Recent progress on Langlands reciprocity for $\mathrm{GL}_n$: Shimura varieties and beyond}}, 
      author={Ana Caraiani and Sug Woo Shin},
      year={2023},
      eprint={2311.13382},
      archivePrefix={arXiv},
      primaryClass={math.NT},
      url={https://arxiv.org/abs/2311.13382}, 
}

@misc{SiYing,
      title={\href{https://arxiv.org/abs/2006.11745}{Eichler-Shimura relations for {S}himura varieties of {H}odge type}}, 
      author={Si Ying Lee},
      year={2021},
      eprint={2006.11745},
      archivePrefix={arXiv},
      primaryClass={math.NT},
      url={https://arxiv.org/abs/2006.11745}, 
}

@book {Shimurabook,
    AUTHOR = {Shimura, Goro},
     TITLE = {Introduction to the arithmetic theory of automorphic
              functions},
    SERIES = {Publications of the Mathematical Society of Japan},
    VOLUME = {11},
      NOTE = {Reprint of the 1971 original,
              Kan\^{o} Memorial Lectures, 1},
 PUBLISHER = {Princeton University Press, Princeton, NJ},
      YEAR = {1994},
     PAGES = {xiv+271},
      ISBN = {0-691-08092-5},
   MRCLASS = {11Fxx (11-02 11G05 11G40)},
  MRNUMBER = {1291394},
}

@article {BailyBorel,
    AUTHOR = {Baily, Jr., W. L. and Borel, A.},
     TITLE = {\href{https://doi.org/10.2307/1970457}{Compactification of arithmetic quotients of bounded symmetric
              domains}},
   JOURNAL = {Ann. of Math. (2)},
  FJOURNAL = {Annals of Mathematics. Second Series},
    VOLUME = {84},
      YEAR = {1966},
     PAGES = {442--528},
      ISSN = {0003-486X},
   MRCLASS = {32.65},
  MRNUMBER = {216035},
MRREVIEWER = {A. Kor\'{a}nyi},
       DOI = {10.2307/1970457},
       URL = {https://doi.org/10.2307/1970457},
}

@incollection {MilneBirk,
    AUTHOR = {Milne, J. S.},
     TITLE = {The action of an automorphism of {${\bf C}$} on a {S}himura
              variety and its special points},
 BOOKTITLE = {Arithmetic and geometry, {V}ol. {I}},
    SERIES = {Progr. Math.},
    VOLUME = {35},
     PAGES = {239--265},
 PUBLISHER = {Birkh\"{a}user Boston, Boston, MA},
      YEAR = {1983},
   MRCLASS = {11G18 (14K15 14K22)},
  MRNUMBER = {717596},
MRREVIEWER = {Paul B. Garrett},
}

@article {PollackShahGsp6,
    AUTHOR = {Pollack, Aaron and Shah, Shrenik},
     TITLE = {\href{https://doi.org/10.1353/ajm.2018.0018}{The spin {$L$}-function on {$\rm GSp_6$} via a non-unique
              model}},
   JOURNAL = {Amer. J. Math.},
  FJOURNAL = {American Journal of Mathematics},
    VOLUME = {140},
      YEAR = {2018},
    NUMBER = {3},
     PAGES = {753--788},
      ISSN = {0002-9327},
   MRCLASS = {22E40 (11F30)},
  MRNUMBER = {3805018},
MRREVIEWER = {Jean Raimbault},
       DOI = {10.1353/ajm.2018.0018},
       URL = {https://doi.org/10.1353/ajm.2018.0018},
}

@article {OWR,
     TITLE = {\href{https://doi.org/10.4171/OWR/2018/30}{Algebraische {Z}ahlentheorie}},
      NOTE = {Abstracts from the workshop held June 24--30, 2018,
              Organized by Guido Kings, Ramdorai Sujatha, Eric Urban and
              Otmar Venjakob},
   JOURNAL = {Oberwolfach Rep.},
  FJOURNAL = {Oberwolfach Reports},
    VOLUME = {15},
      YEAR = {2018},
    NUMBER = {2},
     PAGES = {1785--1844},
      ISSN = {1660-8933},
   MRCLASS = {11-06},
  MRNUMBER = {3941537},
       DOI = {10.4171/OWR/2018/30},
       URL = {https://doi.org/10.4171/OWR/2018/30},
}

@article {FriedbergJacquet,
    AUTHOR = {Friedberg, Solomon and Jacquet, Herv\'{e}},
     TITLE = {\href{https://doi.org/10.1515/crll.1993.443.91}{Linear periods}},
   JOURNAL = {J. Reine Angew. Math.},
  FJOURNAL = {Journal f\"{u}r die Reine und Angewandte Mathematik. [Crelle's
              Journal]},
    VOLUME = {443},
      YEAR = {1993},
     PAGES = {91--139},
      ISSN = {0075-4102,1435-5345},
   MRCLASS = {11F70 (22E55)},
  MRNUMBER = {1241129},
MRREVIEWER = {Stephen\ Gelbart},
       DOI = {10.1515/crll.1993.443.91},
       URL = {https://doi.org/10.1515/crll.1993.443.91},
}

@article {Dimitrov,
    AUTHOR = {Dimitrov, Mladen and Januszewski, Fabian and Raghuram, A.},
     TITLE = {\href{https://doi.org/10.1112/s0010437x20007551}{{$L$}-functions of {$\rm{GL}_{2n}$}: {$p$}-adic properties
              and non-vanishing of twists}},
   JOURNAL = {Compos. Math.},
  FJOURNAL = {Compositio Mathematica},
    VOLUME = {156},
      YEAR = {2020},
    NUMBER = {12},
     PAGES = {2437--2468},
      ISSN = {0010-437X,1570-5846},
   MRCLASS = {11F67 (11F55 11F70 11S40)},
  MRNUMBER = {4199213},
MRREVIEWER = {Matteo\ Longo},
       DOI = {10.1112/s0010437x20007551},
       URL = {https://doi.org/10.1112/s0010437x20007551},
}

@book {Weiltopology,
    AUTHOR = {Weil, Andr\'{e}},
     TITLE = {Adeles and algebraic groups},
    SERIES = {Progress in Mathematics},
    VOLUME = {23},
      NOTE = {With appendices by M. Demazure and Takashi Ono},
 PUBLISHER = {Birkh\"{a}user, Boston, MA},
      YEAR = {1982},
     PAGES = {iii+126},
      ISBN = {3-7643-3092-9},
   MRCLASS = {10C30 (12A82 12A85 20G35)},
  MRNUMBER = {670072},
MRREVIEWER = {K. F. Lai},
}

@inproceedings {DeligneVar,
    AUTHOR = {Deligne, Pierre},
     TITLE = {Vari\'{e}t\'{e}s de {S}himura: interpr\'{e}tation modulaire, et techniques
              de construction de mod\`eles canoniques},
 BOOKTITLE = {Automorphic forms, representations and {$L$}-functions
              ({P}roc. {S}ympos. {P}ure {M}ath., {O}regon {S}tate {U}niv.,
              {C}orvallis, {O}re., 1977), {P}art 2},
    SERIES = {Proc. Sympos. Pure Math., XXXIII},
     PAGES = {247--289},
 PUBLISHER = {Amer. Math. Soc., Providence, RI},
      YEAR = {1979},
   MRCLASS = {10D20 (14D20 14G25 14K15)},
  MRNUMBER = {546620},
MRREVIEWER = {James Milne},
}

@misc{stacks-project,
  author       = {The {Stacks project authors}},
  title        = {The Stacks project},
  howpublished = {\url{https://stacks.math.columbia.edu}},
  year         = {2025},
}

@inproceedings {Delignemodular,
    AUTHOR = {Deligne, P.},
     TITLE = {Formes modulaires et repr\'{e}sentations de {${\rm GL}(2)$}},
 BOOKTITLE = {Modular functions of one variable, {II} ({P}roc. {I}nternat.
              {S}ummer {S}chool, {U}niv. {A}ntwerp, {A}ntwerp, 1972)},
    SERIES = {Lecture Notes in Math., Vol. 349},
     PAGES = {55--105},
 PUBLISHER = {Springer, Berlin-New York},
      YEAR = {1973},
   MRCLASS = {10D15 (22E50 22E55)},
  MRNUMBER = {347738},
MRREVIEWER = {Roger Howe},
}

@book {Poonen,
    AUTHOR = {Poonen, Bjorn},
     TITLE = {\href{https://doi.org/10.1090/gsm/186}{Rational points on varieties}},
    SERIES = {Graduate Studies in Mathematics},
    VOLUME = {186},
 PUBLISHER = {American Mathematical Society, Providence, RI},
      YEAR = {2017},
     PAGES = {xv+337},
      ISBN = {978-1-4704-3773-2},
   MRCLASS = {14G05 (11G35)},
  MRNUMBER = {3729254},
MRREVIEWER = {Daniel Loughran},
       DOI = {10.1090/gsm/186},
       URL = {https://doi.org/10.1090/gsm/186},
}

@article {Conradtopology,
    AUTHOR = {Conrad, Brian},
     TITLE = {\href{https://doi.org/10.4171/LEM/58-1-3}{Weil and {G}rothendieck approaches to adelic points}},
   JOURNAL = {Enseign. Math. (2)},
  FJOURNAL = {L'Enseignement Math\'{e}matique. Revue Internationale. 2e S\'{e}rie},
    VOLUME = {58},
      YEAR = {2012},
    NUMBER = {1-2},
     PAGES = {61--97},
      ISSN = {0013-8584},
   MRCLASS = {14-02 (01A60 01A70 14-03 14G25)},
  MRNUMBER = {2985010},
MRREVIEWER = {J\"{o}rg Jahnel},
       DOI = {10.4171/LEM/58-1-3},
       URL = {https://doi.org/10.4171/LEM/58-1-3},
}

@article {Kiernan,
    AUTHOR = {Kiernan, Peter and Kobayashi, Shoshichi},
     TITLE = {\href{https://doi.org/10.1007/BF01425496}{Satake compactification and extension of holomorphic mappings}},
   JOURNAL = {Invent. Math.},
  FJOURNAL = {Inventiones Mathematicae},
    VOLUME = {16},
      YEAR = {1972},
     PAGES = {237--248},
      ISSN = {0020-9910},
   MRCLASS = {32H20},
  MRNUMBER = {310297},
MRREVIEWER = {W. Barth},
       DOI = {10.1007/BF01425496},
       URL = {https://doi.org/10.1007/BF01425496},
}

@article {Fouquet,
    AUTHOR = {Fouquet, Olivier},
     TITLE = {\href{https://doi.org/10.1112/S0010437X12000619}{Dihedral {I}wasawa theory of nearly ordinary quaternionic
              automorphic forms}},
   JOURNAL = {Compos. Math.},
  FJOURNAL = {Compositio Mathematica},
    VOLUME = {149},
      YEAR = {2013},
    NUMBER = {3},
     PAGES = {356--416},
      ISSN = {0010-437X},
   MRCLASS = {11F33 (11F80 11G18 11G40 11R23)},
  MRNUMBER = {3040744},
MRREVIEWER = {Stefano Vigni},
       DOI = {10.1112/S0010437X12000619},
       URL = {https://doi.org/10.1112/S0010437X12000619},
}

@book {Milneetale,
    AUTHOR = {Milne, J. S.},
     TITLE = {\'{E}tale cohomology},
    SERIES = {Princeton Mathematical Series},
    VOLUME = {33},
      NOTE = {Reprint of [ 0559531]},
 PUBLISHER = {Princeton University Press, Princeton, NJ},
      YEAR = {2025] \copyright 1980},
     PAGES = {xiii+323},
      ISBN = {9780691273792; 9780691273785; 9780691273778},
   MRCLASS = {14-02 (14F20 18F99)},
  MRNUMBER = {4904233},
}

@article {Carayol,
    AUTHOR = {Carayol, Henri},
     TITLE = {\href{http://www.numdam.org/item?id=CM_1986__59_2_151_0}{Sur la mauvaise r\'{e}duction des courbes de {S}himura}},
   JOURNAL = {Compositio Math.},
  FJOURNAL = {Compositio Mathematica},
    VOLUME = {59},
      YEAR = {1986},
    NUMBER = {2},
     PAGES = {151--230},
      ISSN = {0010-437X},
   MRCLASS = {11G18 (11G15 14H25 14K22)},
  MRNUMBER = {860139},
MRREVIEWER = {Ernst-Ulrich Gekeler},
       URL = {http://www.numdam.org/item?id=CM_1986__59_2_151_0},
}

@incollection {NekovarCM,
    AUTHOR = {Nekov\'{a}\v{r}, Jan},
     TITLE = {\href{https://doi.org/10.1017/CBO9780511721267.014}{The {E}uler system method for {CM} points on {S}himura curves}},
 BOOKTITLE = {{$L$}-functions and {G}alois representations},
    SERIES = {London Math. Soc. Lecture Note Ser.},
    VOLUME = {320},
     PAGES = {471--547},
 PUBLISHER = {Cambridge Univ. Press, Cambridge},
      YEAR = {2007},
   MRCLASS = {11G15 (11F80 11G18)},
  MRNUMBER = {2392363},
MRREVIEWER = {Benjamin V. Howard},
       DOI = {10.1017/CBO9780511721267.014},
       URL = {https://doi.org/10.1017/CBO9780511721267.014},
}

@article {CorVatsal2,
    AUTHOR = {Cornut, C. and Vatsal, V.},
     TITLE = {\href{https://ems.press/journals/dm/articles/8965120}{C{M} points and quaternion algebras}},
   JOURNAL = {Doc. Math.},
  FJOURNAL = {Documenta Mathematica},
    VOLUME = {10},
      YEAR = {2005},
     PAGES = {263--309},
      ISSN = {1431-0635},
   MRCLASS = {11G15 (11G18)},
  MRNUMBER = {2148077},
MRREVIEWER = {Benjamin V. Howard},
}

@incollection {Ngo-Genestier,
    AUTHOR = {Genestier, Alain and Ng\^{o}, Bao Ch\^{a}u},
     TITLE = {Lectures on {S}himura varieties},
 BOOKTITLE = {Autour des motifs---\'{E}cole d'\'{e}t\'{e} {F}ranco-{A}siatique de
              {G}\'{e}om\'{e}trie {A}lg\'{e}brique et de {T}h\'{e}orie des
              {N}ombres/{A}sian-{F}rench {S}ummer {S}chool on {A}lgebraic
              {G}eometry and {N}umber {T}heory. {V}olume {I}},
    SERIES = {Panor. Synth\`eses},
    VOLUME = {29},
     PAGES = {187--236},
 PUBLISHER = {Soc. Math. France, Paris},
      YEAR = {2009},
   MRCLASS = {14G35 (11G18 14C30 14K05)},
  MRNUMBER = {2730658},
MRREVIEWER = {Eva Viehmann},
}

@incollection {Scholl,
    AUTHOR = {Scholl, A. J.},
     TITLE = {\href{https://doi.org/10.1017/CBO9780511662010.011}{An introduction to {K}ato's {E}uler systems}},
 BOOKTITLE = {Galois representations in arithmetic algebraic geometry
              ({D}urham, 1996)},
    SERIES = {London Math. Soc. Lecture Note Ser.},
    VOLUME = {254},
     PAGES = {379--460},
 PUBLISHER = {Cambridge Univ. Press, Cambridge},
      YEAR = {1998},
   MRCLASS = {11G40 (11G16 11G18 11S25 14F30 19F15 19F27)},
  MRNUMBER = {1696501},
MRREVIEWER = {Amnon Besser},
       DOI = {10.1017/CBO9780511662010.011},
       URL = {https://doi.org/10.1017/CBO9780511662010.011},
}
\Addresses
\end{document}